\documentclass[leqno,11pt]{amsart}
\usepackage{amssymb, amsmath}
\usepackage[usenames,dvipsnames]{xcolor}
\usepackage{color,ulem,textcomp}

\numberwithin{equation}{section}
\usepackage{amsmath, amsfonts, amssymb, amsthm, amscd, graphicx}
\usepackage{enumerate}
\usepackage{graphicx}
\usepackage{cancel}

\theoremstyle{definition}
\theoremstyle{plain}
\newtheorem{thm}{Theorem}[section]
\newtheorem{Prop}[thm]{Proposition}
\newtheorem{lem}[thm]{Lemma}
\newtheorem{Cor}[thm]{Corollary}

 \newtheorem{Thm}{Theorem}[section]
 \newtheorem{Rmk}[thm]{Remark}
 \newtheorem{Lem}[thm]{Lemma}
  \newtheorem{Def}[thm]{Definition}

\def\sumetage#1#2{
\sum_{\scriptstyle {#1}\atop\scriptstyle {#2}} }

 \def\N {\mathbb{N}}
\def\R {\mathbb{R}}
\def\D {\mathbb{D}}
\def\T {\mathbb{T}}
\def\Z {\mathbb{Z}}

\def\cA {\mathcal{A}}

\def\cC {\mathcal{C}}

\def\cD {\mathcal{D}}
\def\cK {\mathcal{K}}

\def\cL {\mathcal{L}}

\def\cQ {\mathcal{Q}}
\def\cP {\mathcal{P}}

\def\cS {\mathcal{S}}
\def\cZ {\mathcal{Z}}

\def\cR {\mathcal{R}}
\def\cT {\mathcal{T}}

\def\eps {{\varepsilon}}
\def\e {{\varepsilon}}

\def\indc {{\bf 1}}

\def\la {\langle}
\def\ra {\rangle}

\def\d {{\partial}}

\newcommand{\Span}{\operatorname{span}}

\newcommand{\Ker}{\operatorname{Ker}}

\newcommand{\ba}{\begin{aligned}}
\newcommand{\ea}{\end{aligned}}

\newcommand{\be}{\begin{equation}}
\newcommand{\ee}{\end{equation}}

\def\i{\bar\imath }

\numberwithin{equation}{section}

\begin{document}


\title[an $L^2$ analysis of the Boltzmann-Grad limit] 
{From hard sphere dynamics to the Stokes-Fourier equations: an $L^2$ analysis of the Boltzmann-Grad limit}
\author{Thierry Bodineau, Isabelle Gallagher and Laure Saint-Raymond}


\begin{abstract}
We derive the linear acoustic and Stokes-Fourier equations as the limiting dynamics of a system of $N$ hard spheres of diameter $\eps$  in two space dimensions, when~$N\to \infty$, $\eps \to 0$, $N\eps =\alpha \to \infty$, using the linearized Boltzmann equation as an intermediate step.
Our proof is based on   Lanford's  strategy \cite{lanford}, and on the pruning procedure developed in~\cite{BGSR1} to improve the convergence time  to all kinetic times with a quantitative control which allows us to reach also hydrodynamic time scales. 
The main novelty here is that  uniform~$L^2$ a priori estimates  combined with a subtle symmetry argument  provide a weak version of chaos, in the form of a cumulant expansion describing the asymptotic decorrelation between the particles. A refined geometric analysis of recollisions is also required in order to discard the possibility of multiple recollisions.
\end{abstract}

\maketitle

\section{Introduction to the Boltzmann-Grad limit and statement of the  result}

The sixth problem  raised  by Hilbert in 1900 on the occasion of the International Congress of Mathematicians  addresses the question of the axiomatization of mechanics, and more precisely of describing the transition between atomistic and continuous models for gas dynamics by rigorous mathematical convergence results.
Even though it is quite restrictive (since only perfect gases can be considered by this process), Hilbert further suggested  using Boltzmann's kinetic equation as an intermediate step to understand the  appearance  of irreversibility and dissipative mechanisms \cite{hilbert0}.
The derivation of the Boltzmann equation was then formalized in  the pioneering work of Grad \cite{grad}.

\medskip
 
A huge amount of literature has been devoted to these asymptotic problems, but up to now they remain still largely open. Important breakthroughs \cite{diperna-lions,BGL2} have allowed for a complete study of some hydrodynamic limits of the Boltzmann equation, especially in incompressible viscous regimes leading to the Navier-Stokes equations (see \cite{GSR} for instance). Note that other regimes such as the compressible Euler limit (which is the most immediate from a formal point of view) are still far from being understood \cite{caflisch, liu}.

But, at this stage, the main obstacle seems actually to come from the other step, namely  the derivation of the Boltzmann equation from a system of interacting particles: the best result to this day concerning this low density limit which is due to Lanford in the case of hard-spheres~\cite{lanford} (see also \cite{CIP,uchiyama,GSRT,PSS,PS} for a complete proof)  is indeed valid only for short times, i.e. breaks down before any relaxation can be observed.

\begin{Thm}\label{lanford-thm}
Consider a system of $N$ hard-spheres of diameter $\eps$ on $\T^d = [0,1]^d$ (with $d\geq 2$), initially ``independent" and identically distributed with density 
$f_0$ such that 
 $$\big\| f_0 \exp(\mu + \frac\beta 2|v|^2 ) \big\|_{L^\infty(\T^d_x \times \R^d_v)}\leq 1 \, ,$$
 for some $\beta>0, \mu \in \R$.
 
Fix $\alpha >0$, then,  in the Boltzmann-Grad limit~$N \to \infty$ with $N \eps^{d-1} = \alpha$, the first marginal  density converges almost everywhere to the solution of the Boltzmann equation
\begin{equation}
\label{boltzmann}
\begin{aligned}
& \d_t f +v\cdot \nabla_x f= \alpha \, Q(f,f),\\
& Q(f,f)(v):=\iint_{{\mathbb S}^{d-1} \times \R^d} [f(v')f(v'_1)-f(v)f(v_1)]  \,\big((v-v_1)  \cdot \nu\big)_+  \, dv_1 d\nu  \, ,\\
& v'=v +\nu \cdot (v_1-v)  \, \nu \, , 
	\quad v'_1=v_1 - \nu \cdot(v_1-v) \, \nu \,  ,
 \end{aligned}
\end{equation}
on a time interval $[0, C(\beta,\mu) /\alpha]$.
 As the propagation of chaos holds, the empirical measure converges in law to 
a density given by the  solution of the Boltzmann equation.
 \end{Thm}
By independent we mean here that the correlations, which are only  due to the non overlapping condition, vanish asymptotically as $\eps \to 0$.

\medskip

The main reason why the convergence is not known to hold for longer time intervals is 
that the nonlinearity in the Boltzmann equation~(\ref{boltzmann}) is treated as if the equation was of the type~$\partial_t f = \alpha f^2$:   the cancellations between gain and loss terms in~$Q(f,f)$ are yet to be   understood. The only information we are able to get on these compensations comes from the stationarity of the canonical equilibrium measure.
In this work, we consider very small fluctuations around such equilibria and show that the convergence is valid for
all kinetic times with a quantitative control which allows us to reach also hydrodynamic time scales.

\subsection{Setting of the problem}

\subsubsection{The model}

In the following, we consider only the case of  dimension $d = 2$ 
(we refer the reader to Section \ref{openpbs} for a discussion of the difficulties to generalize our proof in higher dimensions).
We are interested   in describing the macroscopic behavior of a gas consisting of~$N$ hard spheres of diameter $\eps$ in a  periodic domain~${ \T^2} = [0,1]^2$ of~$ \R^2$, with positions and velocities~$(x_i, v_i)_{1\leq i \leq N} $ in~$ ({ \T^2} \times \R^2)^N$, the dynamics of which is given by
\begin{equation}
\label{hard-spheres1}
{dx_i\over dt} =  v_i \,,\quad {dv_i\over dt} =0 \quad \hbox{ as long as \ } |x_i(t)-x_j(t)|>\eps 
\quad \hbox{for \ } 1 \leq i \neq j \leq N
\, ,
\end{equation}
with specular reflection at a collision 
\begin{equation}
\label{hard-spheres2}
\begin{aligned}
\left. \begin{aligned}
 v_i'& := v_i - \frac1{\eps^2} (v_i-v_j)\cdot (x_i-x_j) \, (x_i-x_j)   \\
 v_j'& := v_j + \frac1{\eps^2} (v_i-v_j)\cdot (x_i-x_j) \, (x_i-x_j) 
\end{aligned}\right\} 
\quad  \hbox{ if } |x_i(t)-x_j(t)|=\eps\,.
\end{aligned}
\end{equation}

By macroscopic behavior, we mean that we   look for  a statistical description both taking the limit $N\to \infty$ and averaging  on the initial configurations.

	\medskip
	
  {  Denote~$X_N := (x_1,\dots,x_N) \in \T^{2N}$, $V_N := (v_1,\dots,v_N) \in \R^{2 N}$
and~$Z_N:= (X_N,V_N) \in \D^N$ where~$ \D^N:=  \T^{2N} \times \R^{2N}$.} Defining the Hamiltonian  
$$
H_N (V_N):= \frac12  \sum_{i=1}^N |v_i|^2  \, ,
$$
we consider the Liouville equation in the $4N$-dimensional phase space~{   
\begin{equation}
\label{eq: phase space}
{\mathcal D}_{\eps}^{N} := \big\{Z_N \in \D^N \, / \,  \forall i \neq j \, ,\quad |x_i - x_j| > \eps \big\} \, .
\end{equation}
The Liouville equation is the following
}
	$$
	\d_t f_N +\{ H_N, f_N\} =0\,,
$$
or in other words
\begin{equation}
\label{Liouville}
\d_t f_N +V_N \cdot \nabla_{X_N} f_N =0\,,
\end{equation}
	with specular reflection on the boundary, meaning that if~$Z_N$ belongs to~$\d {\mathcal D}_{\eps}^{N+}(i,j)$  then we impose that
\begin{equation}\label{tracecondition}
	 f_N (t,Z_N) = f_N (t,Z_N' )  \, ,
\end{equation} 
 where~$X_N'=X_N$ and~$v_k'=v_k$ if~$k \neq i,j$ while~$(v_i',v_j')$ are given by~(\ref{hard-spheres2}). We have also defined
\begin{align}
\label{eq: domain boundary}
\d {\mathcal D}_{\eps}^{N\pm }(i,j) := \Big \{Z_N \in \D^N \, / \,  &|x_i-x_j| = \eps \, , \quad \pm (v_i-v_j) \cdot (x_i- x_j) >0 \\
& \mbox{and}  \quad\forall (k,\ell) \in  [1,N]^2\setminus \{(i,j)\} \, , k \neq \ell \, ,     |x_k-x_\ell| > \eps\Big\} \, .
\nonumber
\end{align}
In the following we assume that~$f_N$ is symmetric under permutations of the~$N$ particles, meaning that the particles are exchangeable, and we  define~$f_N$ on the whole phase space~$\D^N$ by setting~$f_N \equiv 0$ on~$\D^N \setminus{\mathcal D}_{\eps}^{N}$.

\medskip
   We recall, as shown in~\cite{alexanderthesis} for instance, that
the set of initial configurations leading to   ill-defined characteristics (due to clustering of collision times, or collisions involving more than two particles) is of measure zero in~${\mathcal D}_\eps^N$.

\medskip
   In the following we shall denote by~${\bf \Psi}_N$ the solution operator to the ODE~(\ref{hard-spheres1}-\ref{hard-spheres2}) and by~${\bf S}_N$   the group associated with free transport in $\cD_\eps^N$ with specular reflection on the boundary.  In other words, for a function $\varphi_N$ defined on $\cD_\eps^N$, we write
$$
{\bf S}_N (\tau) \varphi_N (Z_N) = \varphi_N \big({\bf \Psi}_N (-\tau) Z_N\big) \, .
$$

\subsubsection{The BBGKY and Boltzmann hierarchies}\label{presentationhierarchies}
We are interested in the limiting behaviour of the previous system when~$N\to \infty$ and~$\eps \to 0$ under the Boltzmann-Grad scaling~$N\eps  =\alpha$, with~$\alpha = O(1)$ or diverging slowly to infinity. 
The quantities which are expected to have finite limits in the Boltzmann-Grad limit  are   the marginals 
$$ 
	f_N^{(s)} (t,Z_s): = \int _{\D^{N-s}}f_N(t,Z_N) dz_{s+1}\dots dz_N
$$
for every  $s$ fixed ($s<N$).

\medskip
 
 A formal computation based on Green's formula (see~\cite{CIP,Simonella,GSRT} for instance)  leads to the following  BBGKY  hierarchy for~$s<N$
	\begin{equation}
\label{eq: BBGKY}
	(\d_t +\sum_{i=1}^s v_i\cdot \nabla_{x_i} ) f_N^{(s)} (t,Z_s) = \alpha \big( C_{s,s+1} f_N^{(s+1)}\big) (t,Z_s)
	\end{equation}
	on~$\cD_\eps^s$, 	with the boundary condition as in~(\ref{tracecondition})
	$$f_N^{(s)} (t,Z_s) = f_N^{(s)} (t,Z_s') \hbox{ on }\d D_\eps^{s +}(i,j)  \, .$$

  	The collision term is defined by
		\begin{equation}
	\label{BBGKYcollision}
		\begin{aligned}
	&\big( C_{s,s+1} f_N^{(s+1)}\big) (Z_s)	  :=   (N-s) \eps \alpha^{-1}\\
&\times \Big( \sum_{i=1}^s \int_{{\mathbb S} \times \R^ 2}  f_N^{(s+1)}(\dots, x_i, v_i',\dots , x_i+\eps \nu, v'_{s+1}) \big((v_{s+1}-v_i) \cdot \nu\big)_+ d\nu dv_{s+1}  \\
& \quad    - \sum_{i=1}^s \int_{{\mathbb S} \times \R^ 2}    f_N^{(s+1)}(\dots, x_i, v_i,\dots , x_i+\eps \nu, v_{s+1}) \big((v_{s+1}-v_i) \cdot \nu\big)_- d\nu dv_{s+1} \Big)\\
&\mbox{with} \quad v_i ' := v_i  -  (v_i -v_{s+1})\cdot \nu \,  \nu \,   ,\quad v_{s+1}' := v_{s+1} +  (v_i -v_{s+1})\cdot \nu \,  \nu \, ,
\end{aligned}
	\end{equation}
	where~${\mathbb S}$ denotes the unit sphere in~$\R^2$.
	  Note that  the collision integral is split into two terms according to the sign of $(v_i-v_{s+1}) \cdot \nu $ and we used the trace condition on $\d {\mathcal D}_\eps^{N+}(i,s+1)$
		to  express all quantities in terms of pre-collisional configurations: in the following we shall also use the notation 
$$\begin{aligned}
C_{s,s+1}^{i,+}  f_{s+1} (Z_s)	 :=   (N-s) \eps \alpha^{-1} \int   f_{s+1}(\dots, x_i, v_i',\dots , x_i+\eps \nu, v'_{s+1})  \Big((v_{s+1}-v_i) \cdot \nu\Big)_+ d\nu dv_{s+1} \, , \\
C_{s,s+1}^{i,-}   f_{s+1}(Z_s) :=   (N-s) \eps  \alpha^{-1} \int   f_{s+1}(\dots, x_i, v_i,\dots , x_i-\eps \nu, v_{s+1})  \Big((v_{s+1}-v_i) \cdot \nu\Big)_+ d\nu dv_{s+1} \, , 
\end{aligned}
$$
so that
\begin{equation}
\label{eq: collision Cs}
 C_{s,s+1} =  \sum_{i=1}^s  (C_{s,s+1}^{i,+} -C_{s,s+1}^{i,-}   ) \, .
\end{equation}

		 The closure  for $s=N$ is given by the Liouville equation (\ref{Liouville}). 
		 
		 \medskip
		  
			To obtain the Boltzmann hierarchy, we  compute the formal  limit of the transport and collision operators when~$\eps$ goes to~$ 0$. Recalling that~$(N-s) \eps    \sim \alpha$, the limit hierarchy is  given by
\begin{equation}
\label{hier: Botzmann}
	(\d_t +\sum_{i=1}^s v_i\cdot \nabla_{x_i} ) f^{(s)} (t,Z_s) = \alpha \big( \bar C_{s,s+1} f^{(s+1)}\big) (t,Z_s) \,,
	\end{equation}
in~$\D^s$, where~$\bar C_{s,s+1}$ are   the limit collision operators defined by
$$
	\begin{aligned}
	\big(\bar  C_{s,s+1} f ^{(s+1)}\big) (Z_s)      &:=   \sum_{i=1}^s \int  f^{(s+1)}(\dots, x_i, v_i',\dots , x_i , v'_{s+1}) \Big((v_{s+1}-v_i) \cdot \nu\Big)_+ d\nu dv_{s+1} \\
&\quad-  \sum_{i=1}^s \int  f^{(s+1)}(\dots, x_i, v_i,\dots , x_i , v_{s+1}) \Big((v_{s+1}-v_i) \cdot \nu\Big)_+ d\nu dv_{s+1} \,.
	\end{aligned}
$$

\subsubsection{Initial data and closures for the Boltzmann hierarchy}

 Consider chaotic initial data of the form~$(f_0^{\otimes s})_{s\in \N^*}$, with
$$
f_0^{\otimes s}(Z_s):= \prod_{i=1}^s f_0(z_i) 
\quad \text{with} \quad
\int_\D f_0(z) dz = 1
\, ,
$$
and denote by $f(t)$ the solution of
the  nonlinear Boltzmann equation \eqref{boltzmann} which can be rewritten as
$$ (\d_t+v\cdot \nabla_x ) f =\alpha \bar C_{1,2} f^{\otimes 2}\, ,  \qquad f_{|t = 0} = f_0\,.$$
Then
an easy computation shows that $(f(t)^{\otimes s})_{s\in \N^*}$ is a chaotic solution  to the Boltzmann hierarchy, whose first marginal is nothing else than~$f(t)$.
Note that, even though it may look  like a very particular case, it is somehow generic as any symmetric initial datum may in fact be decomposed as a superposition of chaotic distributions (this is known as the Hewitt-Savage theorem, see~\cite{HS}).
This means that the Boltzmann hierarchy, even though consisting of linear equations,  encodes   nonlinear phenomena. In the absence of suitable uniform  a priori estimates, we therefore may expect the solution to blow up after a finite time. This is actually the main obstacle to get a rigorous derivation of the Boltzmann equation over time intervals larger than the mean free time~$O(1/\alpha)$.
 \smallskip
  
A different structure of initial datum can lead to other types of equations.
Recall that the Maxwellian
$$ M_\beta (v) := {\beta \over 2\pi}   \exp \left(-\beta {|v|^2\over 2} \right)$$  
is  an equilibrium for the Boltzmann dynamics, so that $(M_\beta ^{\otimes s})_{s \geq 1}$ is 
 a stationary solution to the Boltzmann hierarchy. 
Consider an initial datum which is a perturbation of this stationary  solution
\begin{equation}
\label{eq: initial boltz}
f^{(s)} _0 (Z_s) = M_\beta^{\otimes s} (V_s)\sum_{i=1}^s  g_{\alpha,0} (z_i)\,   ,
\end{equation}
where we added a dependency of $g_{\alpha,0}$ on $\alpha$ for later purposes.
This form is stable under the limit dynamics \cite{BLLS} so that a solution to the Boltzmann hierarchy 
\eqref{hier: Botzmann} is
\begin{equation}
\label{solboltz}
f^{(s)} (t,Z_s) = M_\beta^{\otimes s} (V_s)\sum_{i=1}^s  g_\alpha (t,z_i)
\end{equation}
where $g_\alpha$ is a solution of the linearized Boltzmann equation
\begin{equation}
\label{lBoltz}
\begin{aligned}
&(\d_t+v\cdot \nabla_x ) g_\alpha = - \alpha \cL _\beta g_\alpha \, ,\\
& \cL_\beta \,  g_\alpha (v):= 
-{1\over M_\beta} \bar C_{1,2}  (M_\beta  \otimes M_\beta g_\alpha  +  M_\beta g_\alpha  \otimes M_\beta )(v) \\
& \qquad \qquad = \int M_\beta (v_1) \Big(g_\alpha (v)+ g_\alpha (v_1) - g_\alpha (v') - g_\alpha (v_1')\Big)\Big((v_{1}-v) \cdot \nu\Big)_+ d\nu dv_{1} \, ,
\end{aligned}
\end{equation}
with initial datum $g_{\alpha,0}$. The functional space~$L^2(dx M_\beta dv)$ is natural to study the linearized Boltzmann equation,  because the associate norm is a Lyapunov functional for~(\ref{lBoltz})
 (see Appendix~\ref{appendixBoltzlin}).
As we will heavily use it later on, we introduce the following notation, for~$p = 1,2$:
for any function~$g_s$ defined on~$\D^s$,
\begin{equation}
\label{eq: Lp}
\|g_s\|_{L^p_\beta (\D^s)}:= \Big(\int M^{\otimes s}_\beta(V_s) |g_s|^p (Z_s) \, dZ_s\Big)^\frac1p \, .
\end{equation}

\bigskip

We now turn  to the particle dynamics and discuss the counterpart of the initial datum \eqref{eq: initial boltz}.
The  Gibbs measure 
\begin{equation}
\label{eq: fonction partition}
M_{N,\beta} ( Z_N) := \frac1 {\cZ_N}  \indc_{\cD_\eps^N} (X_N) \,  M_\beta^{\otimes N} (V_N)\, ,
\qquad 
\cZ_N: = \int_{\T^{2N}} \prod_{1 \leq i\neq j \leq N}\indc_{|x_i-x_j| > \eps } \, dX_N 
\end{equation}
is invariant for the dynamics. An idea to get such linear asymptotics as~(\ref{solboltz}) is to  consider small fluctuations around  an equilibrium of the form
$$
f_{N,0} (Z_N) = M_{N,\beta} ( Z_N)  \prod_{i=1}^N   \big(1+ \delta  g_{\alpha,0} (z_i) \big)   \,.
$$
However whatever the smallness of $\delta$, such a  sequence of initial data is never a small correction to 
$M_{N,\beta}$.
Thus, we shall tune the size of the perturbation with $N$
\begin{align}
f_{N,0} (Z_N) & = M_{N,\beta} ( Z_N)  \prod_{i=1}^N   \big(1+ \frac\delta N g_{\alpha,0} (z_i) \big) 
\nonumber \\
& = M_{N,\beta} ( Z_N) +  \frac\delta N M_{N,\beta} ( Z_N) \sum_{i=1}^N   g_{\alpha,0} (z_i) 
 + O( \delta^2) \,.
\label{eq: perturbation 1/N}
\end{align}
At the first order in $ {\delta}{ }$, we recover an initial datum for  the BBGKY hierarchy of the   form  \eqref{eq: initial boltz}
\begin{equation}
\label{initial data}
f_{N,0} (Z_N) = M_{N,\beta}(Z_N)  \sum_{i=1}^N   g_{\alpha,0}(z_i) \quad  \hbox{  with } 
\quad \int M_\beta g_{\alpha,0}(z) dz =0\, .
\end{equation}
This initial datum records only the perturbation and it is no longer a probability measure. In particular 
$$
\int  f_{N,0} (Z_N) dZ_N = 0 \, ,
$$
and this property is preserved by the Liouville equation~(\ref{Liouville}).
The question is then to know if the solution of the BBGKY hierarchy obeys a form similar  to~(\ref{solboltz}), at least approximately,   and if  one can obtain good enough bounds  in~$L^2$ spaces  to prove long-time convergence of the marginals  to~$f^{(s)}$ defined in~(\ref{solboltz}).

\begin{Rmk}
Note that another type of (non symmetric) perturbation was dealt with in \cite{BGSR1}, namely an initial datum of the form
	\begin{equation}
	\label{linear-initialdata}
	 f_{N,0} (Z_N) = M_{N,\beta}(Z_N) g_0(z_1)\,.
	 \end{equation}
	 This  describes the motion of a tagged particle in a background close to equilibrium, and we have shown that it satisfies asymptotically the linear Boltzmann equation, and the tagged particle dynamics converges to the Brownian motion in the diffusive limit. However the proof is less complicated since all quantities of interest are uniformly controlled in $L^\infty$, which will not be the case with the initial datum \eqref{initial data}.
	\end{Rmk}

	\subsection{Statement of the results}

		\subsubsection{Low density limit}
	
Our main result is the following.
\begin{Thm}
\label{long-time}
Consider~$N$ hard spheres on the space~$\D = \T^2 \times \R^2$, initially distributed 
according to  $f_{N,0}$ defined as in \eqref{initial data} where~$  g_{\alpha, 0}$ is a bounded, Lipschitz function on~$\D$ 
with zero average, and satisfying the following bound for some constant $C_1$
\begin{equation}
\label{eq: Linfty initial}
\| g_{\alpha, 0}\| _{W^{1,\infty}} \leq  C_1\exp(C_1\alpha^2)\,.
\end{equation}
Then the one-particle distribution~$f_N^{(1)}(t, z)$ 
is close to $M_\beta(v) g_\alpha  ( t , z) $, where $g_\alpha  ( t , z) $ is the solution of the linearized Boltzmann equation~{\rm(\ref{lBoltz})}
		  with initial datum~$g_{\alpha,0} (z)$.
		  
		  \medskip
  
More precisely,  there exists   a non negative constant~$C$ such that  for all~$T>1$ and all~$\alpha > 1$,  in the limit~$N \to \infty$,~$N\varepsilon \alpha^{-1}= 1$,  
\begin{equation}
\label{eq: approx temps gd}
\sup_{t \in [0,T]} \big\|  f_N^{(1)}(t)  - M_\beta  g _\alpha ( t ) 
\big\|_{L^2(\D)} \leq    \frac{ T ^2 e^{C \alpha^2} } { \sqrt{\log\log N}} \, \cdotp
\end{equation}
\end{Thm}  
Note that the $L^\infty$-convergence to the solution of the linearized equation was established in \cite{BLLS} following Lanford's strategy.
This convergence was derived for short times, but in any dimension $d \geq 3$. 
The generalization out of  equilibrium was then established in \cite{S3}.
 
\medskip

Following \cite{BLLS}, Theorem \ref{long-time} can also be interpreted as the limit of  time correlations in the fluctuation field at equilibrium.
Let $h$ be a smooth function in $\T^2 \times \R^2$ such that  $\int M_\beta h(z) dz =0$, then the fluctuation field $\zeta^N$ can be tested against 
$h$ at time $t$
\begin{eqnarray*}
\zeta^N (h,Z_N(t) ) := \frac{1}{\sqrt{N}} \sum_{i =1}^N h ( z_i(t) ) \, ,
\end{eqnarray*}
where $Z_N(t)$ stands for the particle configuration at time $t$. 
The equilibrium covariance of the fluctuation field at different times, say 0 and $t$, is given by
\begin{eqnarray*}
{\mathbb E}_{M_{N, \beta}} \left( \zeta^N ( h,Z_N(0))   \zeta^N (\tilde h,Z_N(t))  \right) 
= \int_{\T^{2N} \times \R^{2N}}  
M_{N, \beta}(Z_N) \, \zeta^N ( h,Z_N(0)) \zeta^N (\tilde h,Z_N(t)) \, ,
\end{eqnarray*}
for all smooth functions $h, \tilde h$ in $\T^2 \times \R^2$ with mean 0.
Using an initial datum of the form~\eqref{initial data}
$$
f_{N,0} (Z_N) = M_{N,\beta}(Z_N)  \sum_{i=1}^N    h(z_i) \quad  \hbox{  with } 
\quad \int M_\beta h (z) dz =0 \, ,
$$
the covariance can be rewritten,  thanks to the  exchangeability of the particles, as
\begin{align*}
{\mathbb E}_{M_{N, \beta}} \left( \zeta^N (h,Z_N(0))   \zeta^N (\tilde h,Z_N(t))  \right)  & = 
\int_{\T^{2 N} \times \R^{2 N}} dZ_N \, 
f_{N,0} (Z_N) \, \frac{ \sum_{i =1}^N \tilde h ( z_i(t) )}{N} \\
& = \int_{\D} dz_1 f_N^{(1)}(t, z_1) \tilde h ( z_1 ) \,  .
\end{align*}
Thus the limiting time covariance is related to the convergence of the first marginal $f_N^{(1)}$ and the following corollary
is an immediate consequence of Theorem \ref{long-time}.
\begin{Cor}
\label{cor: covariance}
Fix $\alpha >0$ and let $h, \tilde h$ be two functions  in $L^2_\beta (\D)$ with mean 0 with respect to $M_\beta dv dx$. 
Then for any $t \geq 0$, the time covariance converges 
in the Boltzmann-Grad  limit~$N \to \infty$,~$N\varepsilon \alpha^{-1}= 1$ 
\begin{align*}
\lim_{N \to \infty} & 
{\mathbb E}_{M_{N, \beta}} \left( \zeta^N (h,Z_N(0))   \zeta^N (\tilde h,Z_N(t))  \right) \\
& = \int_{\T^2 \times \R^2} dz \, M_\beta(v)   \exp \big( - t ( v \cdot \nabla_x + \alpha \cL_\beta)  \big) h (z) \; \tilde h(z),
\end{align*}
where $v \cdot \nabla_x + \alpha \cL_\beta$ is the operator associated with the linearized Boltzmann equation \eqref{lBoltz}.
\end{Cor}
Correlation functions are cornerstones of statistical mechanics and besides the case of mean field models, 
mathematical results on these correlations are sparse in the context of classical interacting $n$-body systems
(see nevertheless \cite{LPS} for an explicit computation in the case of one dimensional hard rods).
The convergence of the fluctuation field (for arbitrary time) to a 
 stationary Ornstein-Uhlenbeck  process
was derived in \cite{Rez} for a related microscopic dynamics with random collisions.
A similar convergence of the fluctuation field for the Hamiltonian dynamics is conjectured  in~\cite{S2}, but 
its derivation would require a better understanding of the emergence of the noise arising from the deterministic evolution.

			\subsubsection{Hydrodynamic limits}

Once Theorem~\ref{long-time} is known, it is  possible to    take the limit~$\alpha \to \infty$ while conserving a small error on the right-hand side of~(\ref{eq: approx temps gd}). Using the classical convergence of the linearized Boltzmann equation to the acoustic equation (see Appendix A), one   infers the following result.
\begin{Cor}
\label{acoustics-cor}
Consider~$N$ hard spheres on the space~$\D =  \T^2 \times \R^2$, initially distributed 
according to  $f_{N,0}$ defined as in \eqref{initial data} with a sequence $(g_{\alpha, 0})$ 
of functions satisfying the assumptions of Theorem~{\rm\ref{long-time}} and
converging in $L^2_\beta (\D)$  as $\alpha$ diverges to
$$
g_0 (x,v):=  \rho_0 (x) +  \sqrt \beta \, u_0(x) \cdot v + \frac{\beta |v|^2 - 2}2 \, \theta_0(x)  
\quad \text{with} \quad \int_{\T^2} \rho_0 (x) dx = 0 \, .
$$
Then as~$N \to \infty$, $N\varepsilon = \alpha  \to \infty$ much slower than~$ \sqrt {\log \log \log N}$, the distribution~$f_N^{(1)}(t )$ 
 converges  in~$L^2(\D)$-norm to~$M_\beta g (t) $ with
$$
g  (t,x,v):=  \rho(t,x) +  \sqrt \beta \, u (t,x) \cdot v + \frac{\beta |v|^2 - 2}2 \theta (t,x) \, ,$$
where~$(\rho, u,\theta)$ satisfies the acoustic equations
$$
\left\{\begin{aligned}
&\d_t \rho + \frac1{ \sqrt \beta} \,  \nabla_x \cdot u = 0\\
&\d_t u + \frac1{ \sqrt \beta} \, \nabla_x ( \rho+\theta)  = 0   \\
&\d_t \theta  + \frac1{ \sqrt \beta} \,  \nabla_x \cdot u  = 0 \\
\end{aligned}\right.
$$
with initial datum $(\rho_0, u_0,\theta_0)$.
\end{Cor}

\bigskip

It is even possible to rescale time as~$t = \alpha \tau$ and to take the limit~$\alpha \to \infty$. For well-prepared initial data, we then obtain the following diffusive approximation by the Stokes-Fourier dynamics.

\begin{Cor}
\label{stokes-cor}
Consider~$N$ hard spheres on the space~$\D = \T^2 \times \R^2$, initially distributed 
according to  $f_{N,0}$ defined in \eqref{initial data} with a sequence $(g_{\alpha, 0})$ 
of functions satisfying the assumptions of Theorem~{\rm\ref{long-time}}  and
converging in $L^2_\beta$ as $\alpha\to \infty $  to
$$
g_0 (x,v):=  \sqrt \beta \, u_0(x) \cdot v + \frac{\beta |v|^2 - 4}2 \, \theta_0(x)  \, , \quad \nabla_x \cdot u_0 = 0\, . 
$$
Then { as~$N \to \infty$, $N\varepsilon = \alpha  \to \infty$ much slower 
than~$ \sqrt {\log \log \log N}$}, the distribution~$f_N^{(1)}(\alpha \tau )$ converges in~$L^2 (\D)$ 
norm to~$M_\beta  g (\tau) $ with
$$
g  (\tau,x,v):=  \sqrt \beta \,  u (\tau,x) \cdot v + \frac{\beta |v|^2 -  4 }2 \theta (\tau,x) \, ,
$$
where~$(u,\theta)$ satisfies the Stokes-Fourier equations
\begin{equation}
\label{stokes}
\left\{\begin{aligned}
&\partial_\tau u - \frac1{ \sqrt \beta} \,  \mu_\beta  \Delta_x u = 0   \\
&\nabla_x \cdot u = 0\\
&\partial_\tau \theta - \frac1{ \sqrt \beta} \,  \kappa_\beta  \Delta_x \theta = 0 \\
\end{aligned}\right.
\end{equation}
with initial datum $( u_0,\theta_0)$,
and
$$ 
\begin{aligned}
\mu_\beta :=\frac14\int \Phi _\beta \cL_\beta^{-1} \Phi_\beta  M_\beta(v) dv \quad \hbox{ with } \quad\Phi_\beta (v):= \beta^2( v\otimes v - {|v|^2\over 2 } {\rm{Id}})\, ,\\
\kappa_\beta :=\frac14 \int \Psi_\beta  \cL_\beta^{-1} \Psi_\beta  M_\beta(v) dv\quad \hbox{ with }\quad \Psi_\beta (v):= \sqrt \beta \,  v\left( \beta{|v|^2\over 4} - 1\right) \, ,
\end{aligned}
$$
where the operator $\cL_\beta$ was introduced in \eqref{lBoltz}. 
\end{Cor}

\begin{Rmk}  In the case of  general, ill-prepared initial data, the asymptotics is also well known \cite{golse-levermore}. Details are provided in Appendix~{\rm\ref{appendixBoltzlin}}. 
\end{Rmk}

\bigskip

\noindent
{\bf Acknowledgements.}
We would like to thank Herbert Spohn and Sergio Simonella for their careful reading of our paper and very useful suggestions.

\section{Strategy of the proof}

In  what follows, we   focus on the proof of Theorem  \ref{long-time}, as it is the new contribution of this work.
Even though it follows some ideas introduced in \cite{BGSR1},   it represents a real improvement of what has been done up to now:
\begin{itemize}
\item First of all, we are able  to capture a fluctuation of order $O(1/N)$  around an equilibrium~\eqref{eq: perturbation 1/N}, and in particular there is no more positivity.
\item Second, we deal with a much weaker functional setting than the $L^\infty$ framework of Lanford's strategy \cite{lanford}, which leads to major difficulties to give sense to the collision operator 
(defined as an integral over a singular set).
\item The strategy developed here to bypass this obstacle uses crucially the exchangeability to get  a weak version of chaos independently of the precise structure of the initial datum. This seems to be an  important conceptual progress.
\end{itemize}

  Let us  recall that, up to now, all the results regarding the low density limit of deterministic systems of particles have been established following Lanford's strategy \cite{lanford}.  In this section, we  describe the main objects involved in the proof, and the pruning procedure introduced in~\cite{BGSR1}. We then show the main differences between our setting and that of~\cite{BGSR1} and finally explain how to adapt the pruning procedure to our setting.

\subsection{The series expansion}
\label{seriesexpansion}

The starting point is the series expansion obtained by iterating Duhamel's formula for the BBGKY hierarchy (\ref{eq: BBGKY})
  \begin{equation}
 \label{duhamel1}
 \begin{aligned}
 f^{(s)} _N(t) =\sum_{n=0}^{N-s}  \alpha^n  \int_0^t \int_0^{t_{s+1}}\dots  \int_0^{t_{s+n-1}}  {\bf S}_s(t-t_{s+1}) C_{s,s+1}  {\bf S}_{s+1}(t_{s+1}-t_{s+2}) C_{s+1,s+2}   \\
\dots  {\bf S}_{s+n}(t_{s+n})     f^{(s+n)}_{N,0} \: dt_{s+n} \dots d t_{s+1} \, ,
\end{aligned}
\end{equation}
where~${\bf S}_s$ denotes the group associated with free transport in $\cD_\eps^s$ with specular reflection on the boundary. 
By abuse of notation, the term $n=0$ in \eqref{duhamel1} should be interpreted as ${\bf S}_{s}(t)     f^{(s)}_{N,0}$ as $n$ records the number of collision operators up to time 0.
Denoting by~$  {\bf S}_s^0$ the free flow, one can derive formally the limiting Boltzmann hierarchy
\begin{equation}
 \label{duhamel1boltzmann}
 \begin{aligned}
 f^{(s)}  (t) =\sum_{n\geq 0}   \alpha^n  \int_0^t \int_0^{t_{s+1}}\dots  \int_0^{t_{s+n-1}}  {\bf S}^0_s(t-t_{s+1}) \bar C_{s,s+1}  {\bf S}^0_{s+1}(t_{s+1}-t_{s+2}) \bar C_{s+1,s+2}   \\
\dots  {\bf S}^0_{s+n}(t_{s+n})     f^{(s+n)}_{0} \: dt_{s+n} \dots d t_{s+1} \, ,
\end{aligned}
\end{equation}
and one aims  at proving the convergence of one hierarchy to the other. 
 
\bigskip
These series expansions have graphical representations which play a key role in the analysis as explained  first in~\cite{lanford,CIP,Simonella,GSRT,PSS,PS}.
This interpretation in terms of collision trees is described below.

\medskip

Let us  extract   combinatorial information from the iterated Duhamel formula~(\ref{duhamel1}). We describe the adjunction of new particles (in the backward dynamics)   by  ordered trees.

 \begin{Def}[Collision trees]
 \label{trees-def}
 Let $s>1$ be fixed. An (ordered) collision tree  $a \in \cA_s$ is defined by a family $(a(i)) _{2\leq i \leq s}$ with $a(i) \in \{1,\dots, i-1\}$.
 \end{Def}
  Note that~~$|\cA_s| \leq (s-1) !$.

\medskip
  Once we have fixed a collision tree $a \in \cA_s$, we can reconstruct pseudo-dynamics starting from any point in the one-particle phase space $z_1= (x_1, v_1) \in \T^2\times \R^2$ at time $t$.

\begin{Def}[Pseudo-trajectory]
\label{pseudotrajectory}
Given~$z_1 \in \T^2\times \R^2$, $t>0$ and a collision tree~$a \in \cA_s$, consider a collection of times, angles and velocities~$(T_{2,s}, \Omega_{2,s}, V_{2,s}) = (t_i, \nu_i, v_i)_{2\leq i\leq s}$ with~$0\leq t_s\leq\dots\leq t_2\leq t$. We then define recursively the pseudo-trajectories in terms of the  backward BBGKY dynamics as follows
\begin{itemize}
\item in between the  collision times~$t_i$ and~$t_{i+1}$   the particles follow the~$i$-particle backward flow with specular reflection;
\item at time~$t_i^+$,  particle   $i$ is adjoined to particle $a(i)$ at position~$x_{a(i)}(t_i^+) + \eps \nu_i$ and 
with velocity~$v_i$, provided~$|x_i-x_j(t_i^+)| > \eps$ for all~$j < i$ with~$ j \neq a(i)$.  
If $(v_i - v_{a(i)} (t_i^+)) \cdot \nu_i >0$, velocities at time $t_i^-$ are given by the scattering laws
\begin{equation}
\label{V-def}
\begin{aligned}
v_{a(i)}(t^-_i) &= v_{a(i)}(t_i^+) - (v_{a(i)}(t_i^+)-v_i) \cdot \nu_i \, \nu_i  \, ,\\
v_i(t^-_i) &= v_i+ (v_{a(i)}(t_i^+)-v_i) \cdot \nu_i \,  \nu_i \, .
\end{aligned}
\end{equation}
\end{itemize}
 We denote  by~$z_i(a, T_{2,s}, \Omega_{2,s}, V_{2,s},\tau)$ the position and velocity of the particle labeled~$i$, at time~$\tau$ (provided~$\tau< t_i$). The  configuration obtained at the end of the tree, i.e. at time 0, is~$Z_s(a, T_{2,s}, \Omega_{2,s}, V_{2,s},0)$.

\medskip

Similarly, we define  the pseudo-trajectories associated with the Boltzmann hierarchy. 
These pseudo-trajectories evolve according to the backward Boltzmann dynamics as follows
\begin{itemize}
\item in between the  collision times~$t_i$ and~$t_{i+1}$   the particles follow the~$i$-particle backward free flow;
\item at time~$t_i^+$,  particle   $i$ is adjoined to particle $a(i)$  at exactly the same position $x_{a(i)}(t_i^+)$. Velocities are given by the laws {\rm(\ref{V-def})}.
\end{itemize}
We denote $\bar Z_s(a, T_{2,s}, \Omega_{2,s}, V_{2,s},0)$ the initial configuration.
\end{Def}
The definition of a pseudo-trajectory  in the BBGKY dynamics is   subject to the fact that particles cannot overlap. This is 
recorded in the next definition.
\begin{Def}[Non overlapping sets]
\label{def: overlap}
Given~$z_1 \in \T^2\times \R^2$ and a collision tree~$a \in \cA_s$,  the non-overlapping set is defined by
\begin{align*}
G_s(a) := \Big\{(T_{2,s}, \Omega_{2,s}, V_{2,s}) & \in  \cT_{2,s} \times {\mathbb S}^{s-1} \times \R^{2(s-1)} \, \Big| \\
&  \text{there exists \ }   Z_s(a, T_{2,s}, \Omega_{2,s}, V_{2,s},0) \, \text{ a pseudo-trajectory}\Big \} \, ,
\end{align*}
denoting
$$
\cT_{2,s} :=\big\{ (t_i)_{2\leq i\leq s} \in [0,t]^{s-1} \,/\, 0\leq t_s\leq\dots\leq t_2\leq t\big\}\,.
$$
\end{Def}
The following semantic distinction will be important later on.
\begin{Def}[Collisions/Recollisions]
\label{def: recollisions}
In the BBGKY hierarchy, the term \emph{collision} will be used only for the creation of a new particle, i.e. for a branching in the collision trees. 
A shock between two particles in the  backward BBGKY dynamics will be called a \emph{recollision}.
\end{Def}
 
Note that no recollision occurs in the Boltzmann hierarchy as the particles have zero diameter.

\medskip

With these notations, the iterated Duhamel formula (\ref{duhamel1}) for the first marginal ($s=1$) can be rewritten 
 \begin{equation}\label{duhamelpseudotraj}
 \begin{aligned}
 f^{(1)} _N(t) =\sum_{s=1}^{N}   (N-1) \dots\big(N-(s-1)\big) \eps^{s-1} \sum_{a \in \cA_s}   \int_{G_s(a)}& dT_{2,s} d\Omega_{2,s}    dV_{2,s}\Big( \prod_{i=2}^s \big( (v_i -v_{a(i)} (t_i)) \cdot \nu_i \Big)  
 \\
 &\quad  \times 
f_{N,0}^{(s)}\big (Z_s(a, T_{2,s}, \Omega_{2,s}, V_{2,s},0)\big)  \, ,
 \end{aligned}
\end{equation}
while in the limit
 \begin{equation}\label{duhamelpseudotrajlim}
 \begin{aligned}
 f^{(1)} (t) =\sum_{s=1}^{\infty}    \alpha^{s-1} \sum_{a \in \cA_s}   \int_{\cT_{2,s} \times {\mathbb S}^{s-1} \times \R^{2(s-1)}} &dT_{2,s}   d\Omega_{2,s}   dV_{2,s}  \Big( \prod_{i=2}^s  (v_i -v_{a(i)} (t_i)) \cdot \nu_i \Big)  \\
 &\quad \times f_{0}^{(s)}\big (\bar Z_s(a, T_{2,s}, \Omega_{2,s}, V_{2,s},0)\big)  \, .
 \end{aligned}
\end{equation}

\subsection{Lanford's strategy}\label{lanfordstrategy}

Lanford's proof relies   then on two steps~:
\begin{itemize}
\item[(i)] proving a short time bound for the series~(\ref{duhamelpseudotraj})  expressing the correlations of the system of $N$ particles  and a similar bound for  the corresponding quantities associated with the Boltzmann hierarchy;
\item[(ii)] proving the convergence of each term of the series, i.e. proving that the BBGKY and Boltzmann pseudo-trajectories $Z_s(a, T_{2,s}, \Omega_{2,s}, V_{2,s},0)$ and $\bar Z_s(a, T_{2,s}, \Omega_{2,s}, V_{2,s},0)$ stay close to each other,  outside a set of parameters~$(t_i, \nu_i, v_i)_{2\leq i \leq s}$ of vanishing measure.
 \end{itemize}
Note that step~(i) alone is responsible for the fact that the low density limit is only known to hold for short times (of the order of~$1/\alpha$). This is  due to the fact that the uniform bound is essentially obtained by  replacing the hierarchy by the one related to an equation of the type~$\partial_t f = \alpha f^2$, neglecting all cancellations  present in the collision term.

More precisely, defining the operator associated with the series \eqref{duhamel1}
\begin{equation} \begin{aligned}
Q_{s,s+n} (t) :=   \alpha^n  \int_0^t \int_0^{t_{s+1}}\! \! \! \! \! \dots \! \!  \int_0^{t_{s+n-1}}  {\bf S}_s(t - t_{s+1}) C_{s,s+1}  {\bf S}_{s+1}(t_{s+1}-t_{s+2})  
  C_{s+1,s+2}\dots    \\
  \dots  {\bf S}_{s+n}(t_{s+n})   \: dt_{{s+n}} \dots dt_{s+1} 
\label{eq: iterated collision operator}
\end{aligned}
\end{equation}
we   overestimate all contributions by considering rather the operators~$|Q_{s,s+n}|$ defined by
\begin{equation} \begin{aligned}
|Q_{s,s+n}| (t) :=   \alpha^n  \int_0^t \int_0^{t_{s+1}}\! \! \! \!  \! \dots \! \!  \int_0^{t_{s+n-1}}  {\bf S}_s( t-t_{s+1}) |C_{s,s+1}|  {\bf S}_{s+1}(t_{s+1}-t_{s+2})  
|  C_{s+1,s+2}|\dots    \\
  \dots  {\bf S}_{s+n}(t_{s+n})   \: dt_{{s+n}} \dots dt_{s+1} 
\label{eq: iterated collision operator abs}
\end{aligned}
\end{equation}
where $C_{s,s+1}$ in~\eqref{eq: collision Cs} is replaced by
$$
| C_{s,s+1} |f_{s+1}:=  \sum_{i=1}^s  (C_{s,s+1}^{i,+} +C_{s,s+1}^{i,-})|f_{s+1}|    \, .
$$
In the same way for the Boltzmann hierarchy, the iterated collision operator is denoted by
\begin{equation}
\label{barQ}
 \begin{aligned}
\bar Q_{s,s+n} (t)  :=  \alpha^n  \int_0^t \int_0^{t_{s+1}}\! \! \! \! \! \dots \! \!   \int_0^{t_{s+n-1}}  {\bf S}^0_{s}(t-t_{s+1})   \bar C_{s,s+1}  {\bf S}^0_{s+1}(t_{s+1}-t_{s+2})  
 \bar  C_{s+1,s+2}\dots  \\
   \dots  {\bf S}^0_{s+n}(t_{s+n})   \: dt_{{s+n}} \dots dt_{s+1}
\end{aligned}
\end{equation}
which is bounded from above by
\begin{equation} 
\begin{aligned}
|\bar Q_{s,s+n}| (t) :=   \alpha^n  \int_0^t \int_0^{t_{s+1}}\! \! \! \! \! \dots \! \!   \int_0^{t_{s+n-1}}  {\bf S}^0_s( t-t_{s+1}) |\bar C_{s,s+1}|  {\bf S}^0_{s+1}(t_{s+1}-t_{s+2})  
|  \bar C_{s+1,s+2}|\dots   \nonumber\\
  \dots  {\bf S}_{s+n}(t_{s+n})   \: dt_{{s+n}} \dots dt_{s+1} \, ,
\end{aligned}
\end{equation}
where~$ |\bar C_{s,s+1}|$ is defined as~$| C_{s,s+1} |$ above.

\bigskip

\noindent
  {\bf Notation.} $ $ From now on, we shall denote by~$C$ a constant which may change from line to line, and which may depend on~$\beta$, but not on $N$ and~$\alpha$. 
We will also write~$A \ll B$ for~$A \leq CB$ if the constant $C$ is small enough, and similarly~$A \gg B$ if~$A \geq CB$ and the constant $C$ is large enough (uniformly in all the relevant parameters). Finally we write~$B_R^s$ for the ball of~$\R^{2s}$ of radius~$R$, and~$B_R=B_R^1$.
  
  \bigskip

We   have the following continuity estimates  (see~\cite{CIP,GSRT, BGSR1}). 
\begin{Prop}
\label{estimatelemmacontinuity} 
There is a constant~$C$  such that for all $s,n\in \N^*$ and all~$h,t\geq 0$, the operator~$|Q|$   satisfies the following continuity estimates: if~$g_s,g_{s+n}$ belong to~$L^\infty(\D^s)$ and~$L^\infty(\D^{s+n})$ respectively, then
	$$	\forall z_1 \in \D \, , \,  
		\begin{aligned}
	& \big( |Q_{1,s}| (t)M_{s,\beta}g_s\big) (z_1) \leq (C\alpha t) ^{s-1} M_{3\beta/4} (z_1) \|g_{s}\|_{L^\infty(\D^{s})}\\
	& \big(|Q_{1,s }| (t) \,  |Q_{s,s+n} | (h)M_{s+n,\beta}g_{s+n} \big)(z_1) \! \leq \!   (C\alpha)^{s+n - 1} t^{s-1}h^n
	M_{3\beta/4} (z_1)\|g_{s+n}\|_{L^\infty(\D^{s+n})} \, .
	\end{aligned}
$$
Similar estimates hold for~$|\bar Q|$.
\end{Prop}

\begin{proof}[Sketch of proof]
The estimate is simply obtained from the fact that the transport operators preserve the Gibbs measures, along with the continuity of the elementary collision operators~:
\begin{itemize}
\item  the transport operators satisfy the identities
$$
	 {\mathbf  S}_k (t) M_{k,\beta} = 	M_{k,\beta}
	$$
\item  the collision operators satisfy the following bounds in the Boltzmann-Grad 
scaling~$N \e = \alpha $  (see \cite{GSRT})
$$
|{C}_{k,k+1} | 
M_{k+1,\beta } (Z_k)
  \leq C
   \Big( k\beta^{-\frac12} + \sum_{1 \leq i \leq k} |v_i|\Big) M_{k,\beta } (Z_k) \, ,
$$
almost everywhere on $\cD_\eps^k$.
\end{itemize}

Estimating the operator $|Q_{s,s+n}|(h)$ follows from piling together those inequalities (distributing the  exponential weight evenly on each occurence of a collision term). We  notice indeed  that by the Cauchy-Schwarz inequality
\begin{equation}
\label{CS}
\begin{aligned}
\sum_{1 \leq i \leq k} |v_i| \exp\Big(- \frac \beta{8n} |V_k|^2\Big) &\leq \left( k\frac {4n} \beta\right) ^{\frac12} \left( \sum_{1 \leq i \leq k} \frac\beta {4n} |v_i|^2 \exp\Big(- \frac \beta{4n} |V_k|^2\Big)\right)^{1/2} \\
&\leq \Big( \frac{4nk}{e\beta}\Big)^{1/2} \leq \frac{2}{\sqrt{ e\beta}} (s+n)  \, ,
\end{aligned}
\end{equation}
where the last inequality comes from the fact that  $k \leq s+n$. 
Each collision operator gives therefore a loss of $ C \beta ^{-1/2}  (s+n) $ together with a loss on the exponential weight, while 
the integration with respect to time provides  a factor $h^n/n!$. By Stirling's formula, we have
$${ (s+n)^n\over n!} \leq  \exp \left(  n \log {n+s \over n }  + n\right) \leq \exp ( s+n) \,.$$
As a consequence
$$
|Q_{s,s+n}| (h) M_{s +n,\beta} (Z_s) \leq C^{s+n} \, (\alpha h) ^{n} M_{s,3\beta/4} (Z_s) \, .
$$
The proof of 
Proposition \ref{estimatelemmacontinuity} follows from this upper bound.
\end{proof}


 The iteration of the first estimate in Proposition~\ref{estimatelemmacontinuity} is the key to the local wellposedness of the hierarchy (see~\cite{CIP,GSRT}) : we indeed prove that,
 if the initial datum satisfies
 $$ |f_{N,0}^{(s)} | \leq \exp (\mu s) M_{s,\beta}$$
the series expansion (\ref{duhamel1}) converges (uniformly in $N$)
  on a time such that~$t \alpha \ll 1$.

\subsection{The pruning procedure introduced   in~\cite{BGSR1}}
\label{pruning}

We recall now a strategy devised in \cite{BGSR1} in order  to control the growth of collision trees.
The idea is to introduce some sampling in time with a (small) parameter $h>0$.
Let~$\{n_k\}_{k \geq 1}$ be a  sequence of integers, typically~$n_k = 2^k$.
We  then study the dynamics up to time  $t := K h$ for some large integer $K$, by splitting the time interval~$[0,t]$  into~$K$ intervals of size~$h$,  and controlling the number of collisions  on each interval. 
In order to discard  trajectories with a large number of collisions in the iterated Duhamel formula, we define collision trees ``of controlled size" by the condition that they have strictly less than $n_k$ branch points on the  interval~$[t-kh ,t-(k-1) h]$.
Note that by construction, the trees are actually  followed ``backwards", from time~$t$ (large)   to time~$0$. So
we decompose the iterated Duhamel formula~(\ref{duhamel1}), in the case~$s=1$, by writing
\begin{equation}
\label{series-expansion}
\begin{aligned}
&f_N^{(1)} (t)    : =   \sum_{j_1=0}^{n_1-1}\! \!   \dots \!  \! \sum_{j_K=0}^{n_K-1} Q_{1,J_1} (h )Q_{J_1,J_2} (h )
 \dots  Q_{J_{K-1},J_K} (h )       f^{(J_K)}_{N,0}   \\
&\qquad  +\sum_{k=1}^K \; \sum_{j_1=0}^{n_1-1} \! \! \dots \! \! \sum_{j_{k-1}=0}^{n_{k-1}-1}\sum_{j_k \geq n_k} \; 
 Q_{1,J_1} (h ) \dots  Q_{J_{k-2},J_{k-1}} (h ) \,  Q_{J_{k-1},J_{k}} (h)     f^{(J_{k})}_N(t-kh) 
    \, ,
 \end{aligned}
\end{equation}
with~$J_0:=1$,~$J_k :=1+ j_1 + \dots +j_k$. The first term on the right-hand side corresponds to the smallest trees,  and the second term is the remainder:  it represents trees with super exponential branching, i.e. having at least~$n_k$ collisions during the last time lapse, of size~$h$. One proceeds in a similar way for the Boltzmann hierarchy~(\ref{duhamel1boltzmann}).

\medskip
   The main argument of~\cite{BGSR1} consists in 
proving that  the remainder is small, even for large~$t$ (but small~$h$).
This was achieved in~\cite{BGSR1} to derive the linear Boltzmann equation  with initial datum of the form (\ref{linear-initialdata}).
In that case, the maximum principle  ensures that the $L^\infty$ norm of the marginals are bounded at all times 
\begin{equation}
\label{estimatelinearcase}
\big | f_{N}^{(s)} (t,Z_s) \big | \leq        C^s  M_{N,\beta}^{(s)} (Z_s) \, .
\end{equation}
Combining this uniform bound with the $L^\infty$ estimate on the collision operator given in Proposition~\ref{estimatelemmacontinuity}, one 
can  gain smallness thanks to the factor~$h^{j_k}$ which controls the occurence of 
$j_k$ collisions in the last time interval.

\medskip
   The conclusion of the proof in the linear case (see~\cite{BGSR1}) then comes from a comparison of the BBGKY  and the Boltzmann pseudo-trajectories, through a geometric argument showing that recollisions  are events with small probability (compared to the~$O(1)$ norm of the datum in~$L^\infty$), once~$K$ is fixed.

\subsection{A priori estimates}
One of the main differences here with~\cite{BGSR1}  is that the initial datum is no longer~$O(1)$ in~$L^\infty$.
We summarize below the estimates at our disposal for the initial datum~$ f_{N,0} $  defined in~(\ref{initial data}) and the associate solution~$f_N$ to the Liouville equation~(\ref{Liouville}), compared with~\cite{BGSR1}.

\medskip

{\it $L^\infty$-estimates.}
First, one has clearly
\begin{equation}
\label{Linfty-init}
\big | f_{N,0} (Z_N)\big | \leq N \| g_{\alpha,0}\|_{L^\infty(\D)}  \,  M_{N,\beta}(Z_N)  \,.
\end{equation}
 From the maximum principle, we deduce from~(\ref{Linfty-init}) that
 for all~$t \in \R$,
 \begin{equation}
\label{Linfty}
\big | f_{N} (t,Z_N) \big | \leq N \| g_{\alpha,0}\|_{L^\infty(\D)} \,   M_{N,\beta}(Z_N) \, .
\end{equation}
A classical result on the exclusion  (see Lemma 6.1.2 in \cite{GSRT}) shows the following control on the 
partition function introduced in \eqref{eq: fonction partition}
\begin{equation}
\label{eq: ZNZN-k}
\forall 1 \leq s \leq N, \qquad 
{\mathcal Z}_N^{-1} {\mathcal Z}_{N-s}  \leq C(1-C\alpha\eps)^{-s}  \leq C \exp (Cs\alpha\eps) \,  ,
\end{equation}
so   from \eqref{Linfty}, the marginals satisfy 
\begin{align}
\big | f_{N}^{(s)} (t,Z_s) \big |\leq   
 N  M_{N, \beta}^{(s)}(Z_s)   \; \|g_{\alpha,0}\|_{L^\infty(\D)} 
 \leq      N C^s\exp (Cs\alpha \eps)  M_{\beta}^{\otimes s}(Z_s)   
\; \|g_{\alpha,0}\|_{L^\infty(\D)} \, . 
\label{Linfty-marginals}
\end{align}
This should be compared with the counterpart in the linear case, given in~(\ref{estimatelinearcase}) : there is a factor~$N$ difference between the two estimates.

\smallskip

Much better estimates can be obtained at   initial time by using the explicit structure 
of the measure $f_{N,0 }$ defined by~\eqref{initial data}. In particular the
discrepancy between the marginals~$f_{N,0}^{(s)}$ and 
$f^{(s)} _0$ defined in~(\ref{eq: initial boltz})
can be evaluated.
\begin{Prop}
\label{exclusion-prop2}
There exists~$C>1$ such that  as $N\to \infty $ in the scaling $N \e = \alpha \ll 1 /\eps $
$$
\forall s \leq N, \qquad 
\Big| \Big(f_{N,0 }^{(s)}  - f^{(s)} _0 \Big) (Z_s)
\indc_{{\mathcal D}_{\eps}^s}(X_s)
\Big|  \leq C ^s   \alpha^3
\eps      M_\beta ^{\otimes s}(V_s)  \|g_{\alpha,0} \|_{L^\infty}  \,.
$$
As a consequence, if~$\alpha^3 \eps \ll 1$ then the initial data are bounded by
\begin{equation}
\label{eq: linfty initial}
\forall s \leq N, \qquad 
\big| f_{N,0}^{(s)}  (Z_s) \big|  \leq C^s   
M_\beta ^{\otimes s}(V_s)  \|g_{\alpha,0} \|_{L^\infty} \, .
\end{equation}
\end{Prop}
The proof of this Proposition can be found in Appendix \ref{sec: appendix Initial data estimates}. 
A similar statement was derived in \cite{BLLS}. Note that contrary to estimate \eqref{estimatelinearcase} in the linear case, we are unable to propagate
the initial estimate~(\ref{eq: linfty initial}) in time and to improve \eqref{Linfty-marginals}.

\medskip

{\it $L^2$-estimates.} In our setting the~$L^2_\beta$-norm  (defined in \eqref{eq: Lp}) is better behaved than the~$L^\infty$ norm. 
One of the specificities of  dimension 2 is the fact that the normalizing factor ${\mathcal Z}_N^{-1}$ is uniformly bounded in $N$.
From (\ref{eq: ZNZN-k}), we indeed deduce that under the Boltzmann-Grad scaling~$N\eps =\alpha$, one has 
\begin{equation}
\label{eq: ZN}
{\mathcal Z}_N^{-1} \leq C \exp (C \alpha^2) \, . 
\end{equation}
This upper bound and the definition of~$f_{N,0}$ in~\eqref{initial data} lead to
\begin{equation}
\label{L2-init}
\begin{aligned}
\int {f_{N,0} ^2 \over M_{N,\beta} }(Z_N) dZ_N 
& \leq   C\exp (C \alpha^2)  \int  M_{\beta}^{\otimes N} (Z_N)
\left( \sum_{i=1}^N   g_{\alpha,0}(z_i) \right)^2  dZ_N  \\
& \leq   CN\exp (C \alpha^2)\| g_{\alpha,0}\| _{L^2_\beta(\D)}^2    \, ,
\end{aligned}
\end{equation}		
where we used in the last inequality that~$g_{\alpha,0}$ is mean free with respect to the measure~$M_\beta dz$ due to~\eqref{initial data}. 
The weighted~$L^2$ norm is therefore~$O(\sqrt N)$. Since the Liouville equation is conservative, we obtain from~(\ref{L2-init}) that
\begin{equation}
\label{L2}
\int {f_{N} ^2 \over M_{N,\beta} }(t,Z_N) dZ_N \leq   CN\exp (C \alpha^2) \| g_{\alpha,0}\| _{L^2_\beta(\D)}^2 \,.
\end{equation}	
The $L^2$ bound  (\ref{L2}) is in some sense more accurate than (\ref{Linfty}) since it comes from the  orthogonality at time~0 inherited from the structure of the initial datum.
In particular, if the function~$ f_{N} (t,Z_N)$ was of the same form as the initial datum for all times, meaning if
\begin{equation}
\label{wish} 
f_{N} (t,Z_N) = M_{N,\beta}(Z_N)  \sum_{i=1}^N   g_\alpha (t,z_i)  \hbox{  with }
\int M_\beta g_\alpha(t,z) dz =0\,,\end{equation}
 we would deduce a uniform $L^2$ estimate
on $ f_{N}^{(s)}(t)$. Unfortunately this structure is not preserved by the flow. However one inherits a trace of this structure, as will be  shown in  Proposition \ref{lemfNL2sym 0}.

\subsection{Estimate of the collision operators in~$L^2$}\label{introproblemL2}
Proving an analogue of Proposition~\ref{estimatelemmacontinuity} in an~$L^2$ setting is not an easy task, since  one cannot compute the trace of an~$L^2$ function on a hypersurface. However (and that is actually the way to get around a similar difficulty in~$L^\infty$, see~\cite{Simonella,GSRT}) composing the collision integral with free transport and integrating over time is a way of replacing the integral over the unit sphere by an integral over a volume using a change of variables of the type
\begin{equation}\label{changevarintro}
(Z_s,\nu_{s+1},v_{s+1},t) \mapsto Z_{s+1} = (Z_s-V_s t,x_s + \eps \nu_{s+1} - v_{s+1}t,v_{s+1})
\end{equation}
(with scattering if need be).
Using this idea one can hope to prove some kind of continuity estimate of~$Q_{s,s+n}$ in~$L^2$, but two additional difficulties arise:

\begin{enumerate}
\item\label{introtransportnotfree} the transport operators appearing in~$Q_{s,s+n}$ are not free transport operators since recollisions are possible, so the change of variables~(\ref{changevarintro}) cannot be used directly.  If there is a fixed number of recollisions then 
one can still use a similar argument but  if there is no control on the number of collisions then this method fails.

\item\label{introlossepsilon} Computing an~$L^\infty$  bound on the collision operator~$C_{s,s+1}$ gives rise to the size of the circular boundary, hence~$\eps$,  which compensates exactly (up to a factor~$\alpha$) the factor~$(N-s)$; but in~$L^2$ one only can recover~$\eps^\frac12$, so there remains a factor~$N^\frac12$. Typically one can expect in general an estimate of the type
$$
 \big\| |Q_{1,s}| (t) g_s\big\|_{L^2_\beta}\leq (C\alpha  t) ^{s-1}  \|g_{s}\|_{L^2_\beta} N^{\frac{s-1}2}
$$
so this power of~$N$ will need to be compensated (see Section \ref{decompositionL2andCss+1}).
\end{enumerate}

\subsection{Decomposition of the BBGKY solution}
Starting from decomposition~(\ref{series-expansion}),  we need to analyze differently the trajectories with more or less than  1 recollision in order to control the remainder.
 This is due to the fact that as explained in Paragraph~\ref{introproblemL2} (Point~(\ref{introtransportnotfree})), the estimates  in~$L^2_\beta$ of the collision operators~$Q_{s,s+n}$ require  a precise control on the number of recollisions.

\medskip

Our strategy consists in adapting~(\ref{series-expansion}) in two ways: first we   truncate energies by defining 
\begin{equation}
\label{eq: truncated velocities}
\forall s \geq 1, \qquad 
{\mathcal V}_s := \big  \{ V_s \in \R^{2s} \; \big| \; \;  
| V_s |^2 \leq C_0 |\log \eps| \big\} ,
\end{equation}
for some constant~$C_0$ to be specified later in Proposition \ref{lem: truncation}.
Second we
decompose
\begin{equation}\label{decompositionmarginal}
  f_N^{(1)} (t) = f^{(1,K)}_N(t) +  R_N^{K} (t)
\end{equation}
with the leading contribution
$$\begin{aligned} 
  f^{(1,K)}_N(t)     : =   \sum_{j_1=0}^{n_1-1}\! \!   \dots \!  \! \sum_{j_K=0}^{n_K-1} Q_{1,J_1} (h )Q_{J_1,J_2} (h )
 \dots  Q_{J_{K-1},J_K} (h ) \, \big(  f^{(J_K)}_{N,0}    
 \indc_{{\mathcal V}_{J_K} }\big),
 \end{aligned}
$$
with~$n_k = 2^k n_0$ for some~$n_0$ to be specified, and where~$J_0:=1$,~$J_k :=1+ j_1 + \dots +j_k$. 
 The decomposition above is reminiscent of  \eqref{series-expansion}, except that the velocities have been truncated in the dominant term $f^{(1,K)}_N$.

We then split the remainder into three parts according to the number of recollisions in the pseudo-trajectories (see Definition \ref{def: recollisions}) and a fourth part to take into account   large velocities
\begin{equation}\label{reste}
 R_N^{K} (t)=R_N^{K,0} (t)+R_N^{K,1} (t)+R_N^{K,>} (t) + R_N^{K,vel} (t) \, .
\end{equation}

\medskip

$\bullet$ $ $ We first  introduce a truncated transport operator up to the first collision. Let us rewrite Liouville's equation \eqref{Liouville} for $s$ particles with a different boundary condition
\begin{equation*}
\d_t \varphi_s +V_s \cdot \nabla_{X_s} \varphi_s =0
\qquad \text{with} \quad \varphi_s (t, Z_s ) = 0 
\quad \text{for} \quad Z_s \in \bigcup_{i,j \leq s} \d {\mathcal D}_{\eps}^{s+}(i,j) \, .
\end{equation*}
The corresponding semi-group is denoted by 
$\widehat  {\mathbf S}_{s}^0$ and it coincides with the free flow~${\mathbf S}^0_s$ up to the first recollision
\begin{eqnarray*}
\left( \widehat  {\mathbf S}_{s}^0 (\tau) \varphi_s  \right) (Z_s) = 
\begin{cases}
\left( {\mathbf S}^0_s (\tau) \varphi_s \right) (Z_s)  \qquad  & \text{if no recollision occurs in $[0,\tau]$} \, , \\
0  & \text{otherwise} \, .
\end{cases}
\end{eqnarray*}
We define the operator $Q^0_{s,s+n} (t)$ by replacing ${\mathbf S}_{s}$ by $\widehat  {\mathbf S}_{s}^0$ in the iterated collision operator~$Q_{s,s+n} (t)$ given in~\eqref{eq: iterated collision operator}
\begin{equation}\begin{aligned}
Q^0_{s,s+n} (t) :=   \alpha^n  \int_0^t \int_0^{t_{s+1}}\dots  \int_0^{t_{s+n-1}}  \widehat  {\mathbf S}_{s}^0 (t-t_{s+1}) C_{s,s+1}  \widehat  {\mathbf S}_{s+1}^0 (t_{s+1}-t_{s+2})  
  C_{s+1,s+2}\dots    \\
  \dots  \widehat  {\mathbf S}_{s+n}^0   (t_{s+n})   \: dt_{{s+n}} \dots dt_{s+1} \, .
\label{eq: iterated collision operator free}
\end{aligned}
\end{equation}
With this definition, we set
\begin{equation}
\label{reste0}
  R_N^{K,0} (t)  := \sum_{k=1}^K \; \sum_{j_1=0}^{n_1-1} \! \! \dots \! \! \sum_{j_{k-1}=0}^{n_{k-1}-1}\sum_{j_k \geq n_k} \; 
 Q^0_{1,J_1} (h ) \dots  Q^0_{J_{k-1},J_{k}} (h )  \big( f_N^{(J_k)}(t-kh ) \indc_{{\mathcal V}_{J_k} }\big)\, .
\end{equation}

\medskip
$\bullet $ $ $ In a similar way, we  define  pseudo-dynamics involving exactly one recollision.
\begin{eqnarray*}
\left( \widehat  {\mathbf S}_{s}^1 (\tau) \varphi_s  \right) (Z_s) = 
\begin{cases}
\left( {\mathbf S}_s (\tau) \varphi_s \right) (Z_s)  \qquad  & \text{if exactly one recollision occurs in $[0,\tau]$}\, , \\
0  & \text{otherwise} \, .
\end{cases}
\end{eqnarray*}
Note that, contrary to $\widehat  {\mathbf S}_{s}^0 (\tau)$, the operator $ \widehat  {\mathbf S}_{s}^1 (\tau)$ is not a semi-group, as the dynamics keeps memory  of   past events.
In particular, there is no infinitesimal generator.

We then define the operator $Q^1_{s,s+n} (t)$ by replacing ${\mathbf S}_{s}$ by $\widehat  {\mathbf S}_{s}^0$ in the iterated collision operator $Q_{s,s+n} (t)$, except for one iteration
$$
\begin{aligned}
Q^1_{s,s+n} (t) :=   \alpha^n  \sum_{j= 0}^n \int_0^t \int_0^{t_{s+1}}\dots  \int_0^{t_{s+n-1}}  \widehat  {\mathbf S}_{s}^0 (t-t_{s+1}) C_{s,s+1}  \widehat  {\mathbf S}_{s+1}^0 (t_{s+1}-t_{s+2})  
  C_{s+1,s+2}\dots   \nonumber\\
  \dots C_{s+j-1, s+j} \widehat  {\mathbf S}_{s+j}^1   (t_{s+j} - t_{s+j-1})  \dots   \widehat  {\mathbf S}_{s+n}^0   (t_{s+n})   \: dt_{{s+n}} \dots dt_{s+1} \, .
\label{eq: iterated collision operator 1rec}
\end{aligned}
$$

With this definition, we set
\begin{equation}
\label{reste1}
\begin{aligned}
  R_N^{K,1} (t)  := \sum_{k=1}^K\sum_{ \ell =1}^k  \; \sum_{j_1=0}^{n_1-1} \! \! \dots \! \! \sum_{j_{k-1}=0}^{n_{k-1}-1}\sum_{j_k \geq n_k} \; 
 Q^0_{1,J_1} (h ) \dots Q^1_{J_{\ell-1}, J_\ell} (h) \\
 \dots  Q^0_{J_{k-1},J_{k}} (h )     \big( f_N^{(J_k)}(t-kh ) \indc_{{\mathcal V}_{J_k} }\big) \, .
\end{aligned}
\end{equation}

$\bullet$ $ $ The contribution of large velocities, i.e. those which are not in  ${\mathcal V}_{J_K}$, is 
\begin{align}
  R_N^{K, vel} (t) &  :=
   \sum_{j_1=0}^{n_1-1}\! \!   \dots \!  \! \sum_{j_K=0}^{n_K-1} Q_{1,J_1} (h )Q_{J_1,J_2} (h )
 \dots  Q_{J_{K-1},J_K} (h ) \, \left(  f^{(J_K)}_{N,0}    
 \indc_{{\mathcal V}_{J_K}^c }\right)  \nonumber \\
+ &   \sum_{k=1}^K \; \sum_{j_1=0}^{n_1-1} \! \! \dots \! \! \sum_{j_{k-1}=0}^{n_{k-1}-1}\sum_{j_k \geq n_k} \; 
 Q _{1,J_1} (h ) \dots  Q _{J_{k-1},J_{k}} (h ) 
 \left( f_N^{(J_k)}(t-kh )  \indc_{ {\mathcal V}_{J_k}^c }\right) .
\label{reste cut}
\end{align}

$\bullet$ $ $ We finally define
\begin{equation}
\label{reste>}
R_N^{K,>} (t)  :=  R_N^{K} (t) -R_N^{K,0} (t)  -R_N^{K,1} (t) - R_N^{K, vel} (t)  \,  ,
\end{equation}
which by definition corresponds to pseudo-dynamics involving at least two recollisions, with truncated velocities.
 
 \bigskip

Using the notation (\ref{barQ}), the counterpart of $f^{(1,K)} _N(t) $ for the Boltzmann hierarchy is 
$$ 
\bar f^{(1,K)}(t)   :=\sum_{j_1=0}^{n_1-1}\! \!   \dots \!  \! \sum_{j_K=0}^{n_K-1} \bar Q_{1,J_1} (h )\bar Q_{J_1,J_2} (h )
 \dots  \bar Q_{J_{K-1},J_K} (h ) \,     \big(  f^{(J_K)}_{0}\indc_{{\mathcal V}_{J_K} }\big) \, ,
$$
and we define also
$$ 
\bar R ^K (t) =  \sum_{k=1}^K \; \sum_{j_1=0}^{n_1-1} \! \! \dots \! \! \sum_{j_{k-1}=0}^{n_{k-1}-1}\sum_{j_k \geq n_k} \; 
 \bar Q_{1,J_1} (h ) \dots \bar  Q_{J_{k-1},J_{k}} (h )   \big( f^{(J_k)}(t-kh ) \indc_{{\mathcal V}_{J_k} }\big)
 $$
and
$$\begin{aligned}
\bar R ^{K, vel} (t)   :=
   \sum_{j_1=0}^{n_1-1}\! \!   \dots \!  \! \sum_{j_K=0}^{n_K-1} 
\bar Q_{1,J_1} (h )  \, \bar Q_{J_1,J_2} (h )
 \dots  \bar  Q_{J_{K-1},J_K} (h ) \, \left(  f^{(J_K)}_{0}    
 \indc_{ {\mathcal V}_{J_K}^c  }\right)    \\
+     \sum_{k=1}^K \; \sum_{j_1=0}^{n_1-1} \! \! \dots \! \! \sum_{j_{k-1}=0}^{n_{k-1}-1}\sum_{j_k \geq n_k} \; 
 \bar Q_{1,J_1} (h ) \dots  \bar Q_{J_{k-1},J_{k}} (h ) 
 \left( f^{(J_k)}(t-kh )  \indc_{ {\mathcal V}_{J_k}^c  }\right) .
\end{aligned}
$$
 \bigskip

Section~\ref{convergencepartieprincipales}  deals with the convergence of the main part $ f^{(1,K)}_N(t)$ defined in~\eqref{series-expansion}. 
Since the initial datum is well behaved (see Proposition \ref{exclusion-prop2}), the proof of this convergence essentially follows the same lines as in \cite{BGSR1}. 
In the proof of Proposition \ref{main-prop}, we shall however improve the estimates of \cite{BGSR1} on the measure of trajectories having at least one recollision, as they will be the first step to control multiple recollisions.

\medskip

Section~\ref{decompositionL2andCss+1} is the main breakthrough of this paper, as it shows how exchangeability combined with the $L^2$ estimate provides a very weak chaos property (see Proposition \ref{lemfNL2sym 0}). 
We   then explain, in Proposition \ref{L2-est}, how to use this structure to compensate the expected loss explained in Paragraph~\ref{introproblemL2} (Point~(\ref{introlossepsilon})), and to obtain  an estimate on $R_N^{K,0}$, corresponding to pseudo-trajectories with super exponential branching but without recollision. This $L^2$ continuity estimate uses crucially the integration with respect to time of the free  transport (see Paragraph~\ref{introproblemL2}, Point~(\ref{introtransportnotfree})). 
Section~\ref{lanfordL2} is a refinement of this argument to estimate the remainder $R_N^{K,1}$ when there is one recollision.
In fact, the same argument holds with any finite number of recollisions.  \medskip

Section~\ref{recollisions}  deals with $R_N^{K>}$, which corresponds to multiple recollisions (Proposition \ref{propR>}). In this case, the extra smallness coming from the geometric control of multiple recollisions  compensates exactly the $O(N)$ divergence of the $L^\infty
$-bound \eqref{Linfty}. 
The proof  relies on delicate geometric estimates which are detailed in Appendix~\ref{geometricallemmasappendix}.
This allows one to control the remainder~$R_N^{K>}$ 
by using $L^\infty$ estimates from Proposition \ref{estimatelemmacontinuity}. Note that the critical number of recollisions depends on the dimension, it is 1 only in the simple case of dimension $d=2$.  
The~$L^\infty$-bound \eqref{Linfty} is also used in Section~\ref{sec: large velocities} to control $R_N^{K, vel}$, i.e. the large velocities.

\medskip

Finally, we conclude the proof in Section~\ref{sec: conclusion} and state some open problems.

\bigskip

The parameters~$\alpha$ and~$K$ will be tuned at the very end of the proof (see Section~\ref{sec: conclusion}) but one may keep in mind that
$$
K = {T\over h} \ll \log | \log \eps | \quad \mbox{and} \quad \alpha \ll  \sqrt{\log | \log \eps |}\,.
$$

\section{Convergence of the principal parts}
\label{convergencepartieprincipales}
%
%
%
%
%

We recall that  the principal part of the iterated Duhamel formula (\ref{duhamel1}) for the first marginal is given by (\ref{series-expansion})
$$ f^{(1,K)}_N(t)     : =   \sum_{j_1=0}^{n_1-1}\! \!   \dots \!  \! \sum_{j_K=0}^{n_K-1} Q_{1,J_1} (h )Q_{J_1,J_2} (h )
 \dots  Q_{J_{K-1},J_K} (h ) \, 
    \big(  f^{(J_K)}_{N,0}\indc_{{\mathcal V}_{J_K} }\big)
   \, ,
$$
and its counterpart  for the Boltzmann hierarchy is 
$$ 
\bar f ^{(1,K)}(t)   :=\sum_{j_1=0}^{n_1-1}\! \!   \dots \!  \! \sum_{j_K=0}^{n_K-1} \bar Q_{1,J_1} (h )\bar Q_{J_1,J_2} (h )
 \dots  \bar Q_{J_{K-1},J_K} (h ) \,     \big(  f^{(J_K)}_{0}\indc_{{\mathcal V}_{J_K} }\big) \, .
$$
From now on, the exponential growth of the collision trees will be controlled by the sequence
$$
n_k :=    2^k n_0\, ,
$$
for some large integer $n_0$   to be tuned later (see Section~\ref{conclusion of the proof superexpmultiple}). 

\medskip

The error $f^{(1,K)} _N- \bar f ^{(1,K)}  $ can be estimated as follows.
\begin{Prop}
\label{main-prop}
Assume that $g_{\alpha,0}$ satisfies the Lipschitz bound  \eqref{eq: Linfty initial} then,  under the Boltzmann-Grad scaling~$N\eps = \alpha \gg1$, we have for all~$T>1$ and~$t\in [0,T]$, \begin{equation}
\label{mainpart-est}
\Big\|   f^{(1,K)} _N(t)-\bar f ^{(1,K)} (t) \Big\|_{L^2(\D)} 
\leq   
\exp(C\alpha^2)  (C\alpha  T)^{   2^{K+1} n_0}  \left(     \eps | \log\eps|^{10}  + {\eps \over \alpha }  
 \right)   \, .
\end{equation}
\end{Prop}
 The key step of the proof is {Proposition \ref{prop: 1 recoll}
 where the contribution of   recollisions in the pseudo-trajectories associated with 
$f^{(1,K)} _N$ is shown to be negligible.}
Once the recollisions have been neglected and overlaps have been removed, the pseudo-trajectories in both hierarchies are comparable  and the rest of the proof is rather straightforward (see Section~\ref{subsec: Convergence of pseudo-trajectories}). 

\smallskip

In the rest of this section, we assume that $g_{\alpha,0}$ satisfies the Lipschitz bound  \eqref{eq: Linfty initial}.

\subsection{Geometric control of recollisions}
\label{sec: Geometric control of recollisions}

We are   going to prove that pseudo-trajectories involving recollisions contribute very little to~$f^{(1,K)}_{N}$ so that  $ {\bf S}_{s}$ can be replaced by the free transport~$\widehat {\bf S}^0_{s}$, up to a small error. 
With the notation~\eqref{eq: iterated collision operator free}, $f^{(1,K)}_{N}$ can be  decomposed as follows: 
$$
f^{(1,K)}_N = f^{(1,K), 0}_N +f^{(1,K), \geq}_N    
$$
with
\begin{align}
\label{eq: fN,cut 0}
f^{(1,K), 0}_{N}(t)  : =
\sum_{j_1=0}^{n_1-1}\! \!   \dots \!  \! \sum_{j_K=0}^{n_K-1}  Q^0_{1,J_1} (h )Q^0_{J_1,J_2} (h )
\dots  Q^0_{J_{K-1},J_K} (h ) \,     \big(  f^{(J_K)}_{N,0}\indc_{{\mathcal V}_{J_K} }\big)\end{align}
and the remainder   encodes the occurence of at least one recollision
\begin{align}
\label{eq: fN,cut 1}
f^{(1,K), \geq}_{N}   := f^{(1,K)}_{N} - f^{(1,K), 0}_{N}  \, .
\end{align}

\begin{Prop} 
\label{prop: 1 recoll}
The contribution of pseudo-dynamics involving (at least) a recollision is bounded by 
$$
\forall t\in [0,T] \, , \quad  | f^{(1,K), \geq}_N (t,z_1)  |  \leq
 \exp (C\alpha^2)  \big( CT \alpha \big)^{  2^{K+1} n_0} 
 \eps \big| \log \eps \big|^{10} M_{\beta/2}(v_1) \, .
$$ 
\end{Prop} 
The core of the proof is based on a careful analysis of   recollisions detailed in Section~\ref{sectionlocalcondition} below.
The proof of  Proposition \ref{prop: 1 recoll} is   completed in Section
\ref{par4.2}. Thanks to the energy cut-off ${\mathcal V}_{J_K}$, we  assume, in the rest of this section, that all energies are bounded by~$C_0 |\log \eps|$.

\subsubsection{A local condition for a recollision}
\label{sectionlocalcondition}

We start by writing a geometric condition for a recollision which involves only two collision integrals: this corresponds to writing a local condition, which will then be incorporated to the other collision integral estimates in Section~\ref{par4.2}. 
The following notions of {\it pseudo-particles} and {\it parents} will be useful. These notions are depicted in Figures~\ref{fig: pseudo-particle} and~\ref{fig: parent}.
 \begin{Def}[Pseudo-particles]
 \label{pseudoparticle} 
 Given a tree $a \in \cA_s$ and $i\leq s$, we define recursively, moving towards the root, the pseudo-particle $\i$  associated with the particle $i$ to be
\begin{itemize}
 \item $\i = i$ as long as $i$ exists,
 \item $\i = a(i)$ when $i$ disappears, and as long $a(i)$ exists,
 \item $\i = a\big (a(i)\big)$ when $a(i) $ disappears, and as long as this latter   exists,
  \item...
\end{itemize}
When there is no  possible confusion, we shall denote abusively by $i$ the pseudo-particle. 
\end{Def} 
  
\medskip 
 Contrary to the case of a particle in a collision tree, whose trajectory stops at its creation time, the trajectory of a pseudo-particle exists for all times. At each collision time the pseudo-particle is liable to be deviated through a scattering operator, and may jump of a distance $\eps$ in space  (see Figure \ref{fig: pseudo-particle}).

\begin{figure} [h] 
\centering
\includegraphics[width=5.5cm]{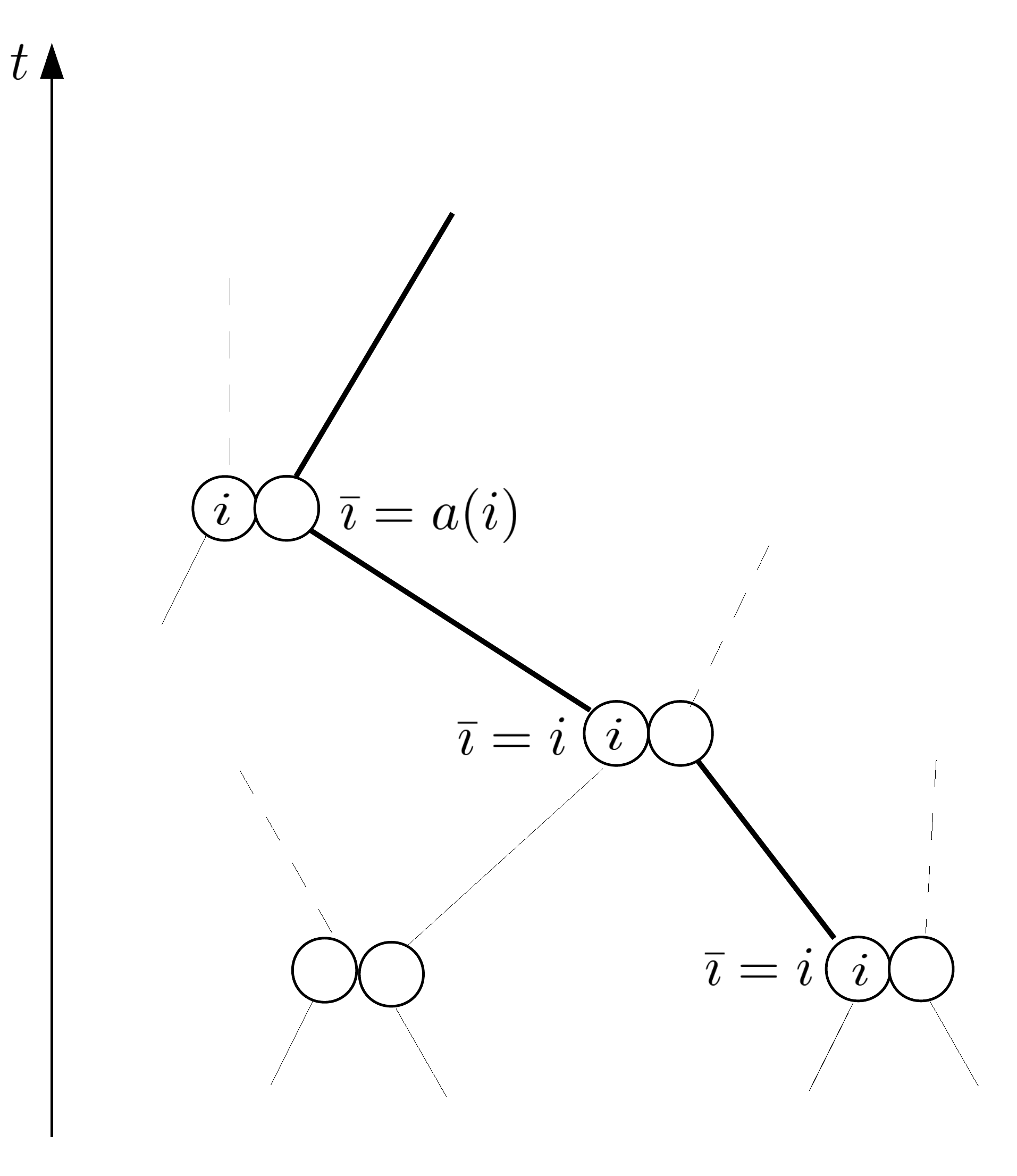} 
\caption{\small 
A collision tree is depicted with the trajectory of the pseudo-particle~$\i$   thickened.  The pseudo-particle $\i$ coincides with $i$ up to the creation time of $i$, moving up to the root, it then coincides with $a(i)$ and so on. Each change of label induces a shift by~$\eps$ of the pseudo-particle $\i$. 
}
\label{fig: pseudo-particle}
\end{figure}

Each collision leading to the deviation of a pseudo-particle brings a new degree of freedom which will be essential to control the trajectories later on. This degree of freedom is associated with a new particle which we call {\it parent}.

\begin{figure} [h] 
\centering
\includegraphics[width=6cm]{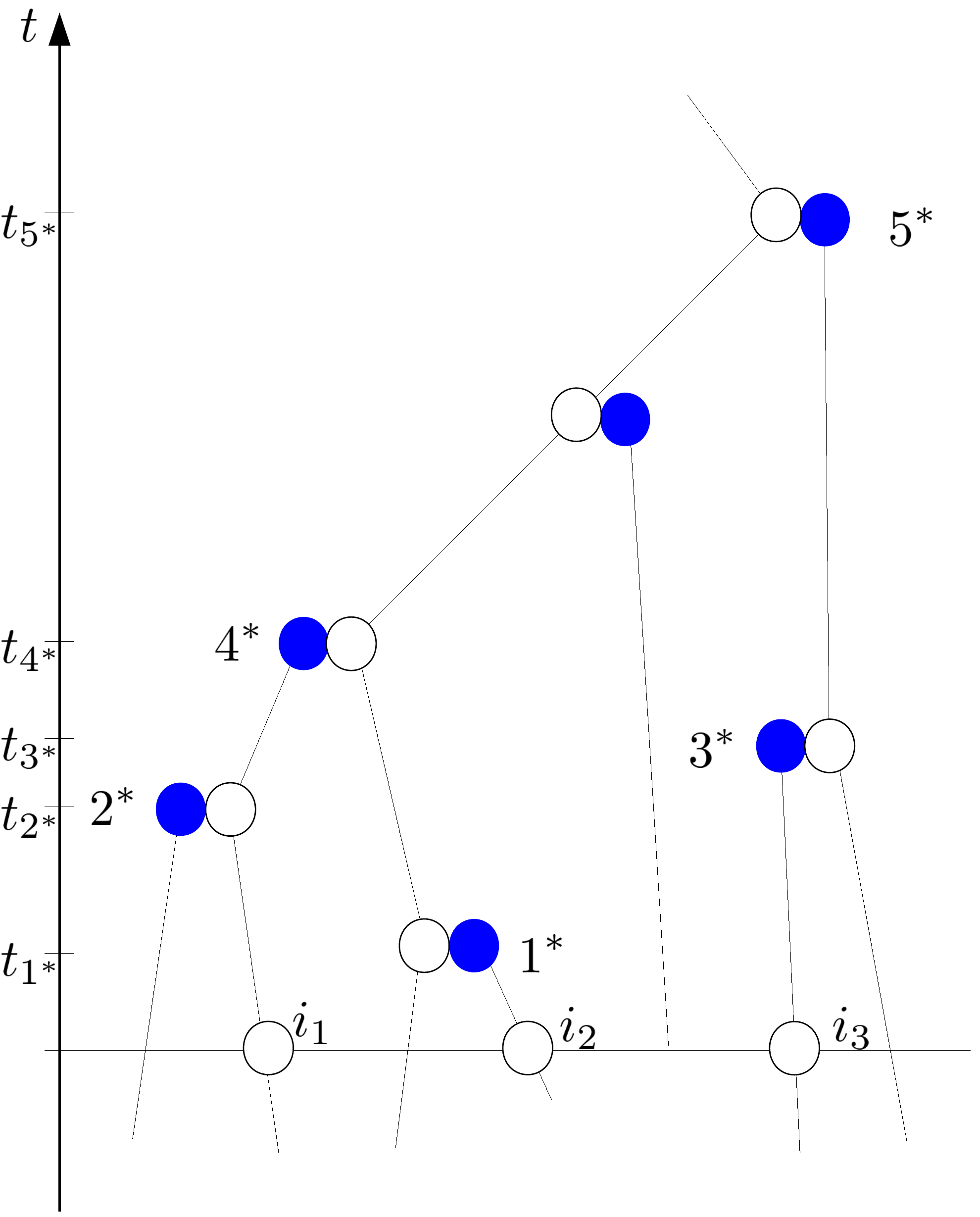} 
\caption{\small 
The set~${\mathcal I}$ consists in~$\{i_1,i_2,i_3\}$. The parents are~$1^*,\dots,5^*\dots$ Note that between times~$t_{4^*}$ and~$t_{5^*}$ a particle has been created but with no scattering so it is not a parent.
}
\label{fig: parent}
\end{figure}

\begin{Def}[Parent]
 \label{parent} 
 Given a collision tree $a \in \cA_s$  and a height in this tree, we consider a subset~${\mathcal I}$ of particles at that height. We define~$({n^*})_{n\in \N}$ the sequence {of branching points in $a$} at which one of the pseudo-particles associated with the particles in ${\mathcal I}$  is deviated. The family~$1^*,2^*,\dots$ of   particles created in these collisions are the parents of the set~${\mathcal I}$. 
Note that the particles $1^*,2^*,\dots$ may coincide with the pseudo-particles  (see Figure {\rm\ref{fig: parent}}).
\end{Def} 
 Note that we disregard times~$t_k$ at which the pseudo-particles  encounter a new particle~$k$ with no scattering (see Figure \ref{fig: parent}).

\medskip
  
A recollision between two particles~$i$ and $j$  imposes  strong constraints  on the history of these particles, especially on the last two collisions at times  $t_{1^*}$ and $t_{2^*}$ with the particles~$1^*$ and~$2^*$ which are the first parents of $i,j$ (see Figure~\ref{Figiiiaiib}(i)). These constraints can be expressed by different equations according to the recollision scenario (each scenario  will be indexed by a number $p$).   We can then   prove the smallness of  the  collision integral associated with particle~$1^*$ (with the measure~$ |(v_{1^*}-v_{a(1^*)} ( t_{{1^*}} ) )\cdot \nu_{{1^*}}|d  t_{{1^*}} d \nu_{{1^*}}dv_{{1^*}}$), with a singularity at small relative velocities which can be integrated out using  the collision integral with respect to particle~$ 2^*$.  The final result is the following.

\begin{Prop}
\label{recoll1-prop}
Fix a final configuration of bounded energy~$z_1 \in \T^2 \times  B_R$ with~$1 \leq R^2 \leq C_0 |\log \eps|$,  a time $1\leq t \leq C_0|\log \eps|$ and  a collision tree~$a \in \cA_s$ with~$s \geq 2$.

For all types of recollisions $p=0,1,2$, and all sets of parents $\sigma \subset \{2,\dots, s\}$ with $|\sigma| = 1 $ if $p =0$ and $|\sigma| = 2 $ if $p=1,2$, there exist sets of bad parameters
$\cP_1(a, p,\sigma)\subset \cT_{2,s} \times {\mathbb S}^{s-1} \times \R^{2(s-1)}$ such that
\begin{itemize}
\item
$\cP_1 (a, p,\sigma)$ is parametrized only in terms of   $(t_m, v_m, \nu_m)$ for $m \in \sigma$ and  $m < \min \sigma$;
\item  its measure is  small in $( t_m, v_m, \nu_m)_{m\in \sigma}$ uniformly with respect to the other parameters
\begin{equation}
 \label{cP-est}
 \int \indc _{\cP_1(a,p,\sigma)  } \displaystyle  \prod_{m\in \sigma} \,
 \big | \big(v_{m}-v_{a(m)} ( t_{{m}} ) )\cdot \nu_{{m}} \big| d  t_{{m}} d \nu_{{m}}dv_{{m}}   
  \leq CR^{7}s  t^3\eps \, |\log \eps|^3\,;
\end{equation}
\item  any pseudo-trajectory starting from $z_1$ at $t$, with total energy bounded by $R^2$ and involving at least one recollision is parametrized by 
$$( t_n, \nu_n, v_n)_{2\leq n\leq s }\in  \bigcup _{p=0}^2 \bigcup _\sigma  \cP_1(a, p,\sigma)\,.$$
\end{itemize}
\end{Prop}

\begin{proof} 
Consider a pseudo-trajectory starting from $z_1$ at $t$, with total energy bounded by~ $R^2$ and involving at least one recollision.
Let  $i$ and $ j$ be the particles involved in the first recollision. Denote by $\theta$ the label of the time interval $] t_{\theta+1},t_\theta[$ where this recollision occurs, and  by  $1^*, 2^*$  the indices in $\{2,\dots , s\}$ of the first two parents of  the set~$\{i,j\}$ starting at height $\theta$.

\begin{Rmk}
\label{rem: vide} 
Notice  that if the recollision takes place  between the two first particles at play before any other collision, then there is actually no such parameter~$2^*$, but in this case only the first scenario (involving just one parent) will be possible.  From now on we shall always assume that there are enough degrees of freedom as needed for the computations, since if that is not the case the result will follow simply by integrating over less variables.
\end{Rmk}

\medskip
 \noindent
\underbar{Self-recollision (case $p=0$).}\label{selfrecollpage}
If the collision at time $t_{1^*}$  involves $i$ and~$j$, a recollision may occur due to the periodicity (see Figure~\ref{coccinellerecoll}). In this case, the parent $1^*$ is $i$ or $j$.

\begin{figure} [h] 
   \centering
   \includegraphics[width=6cm]{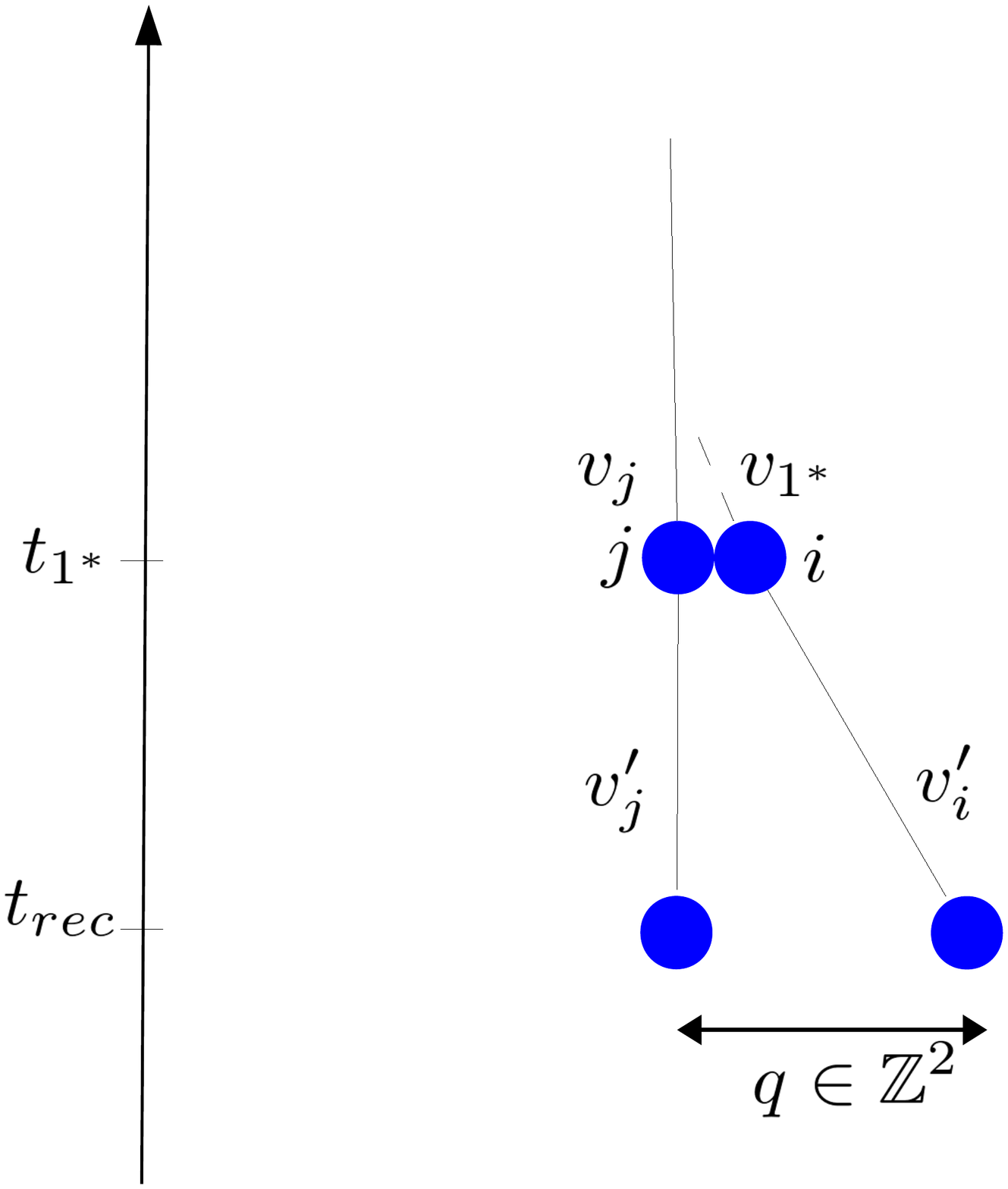}  \qquad\qquad  \includegraphics[width=6cm]{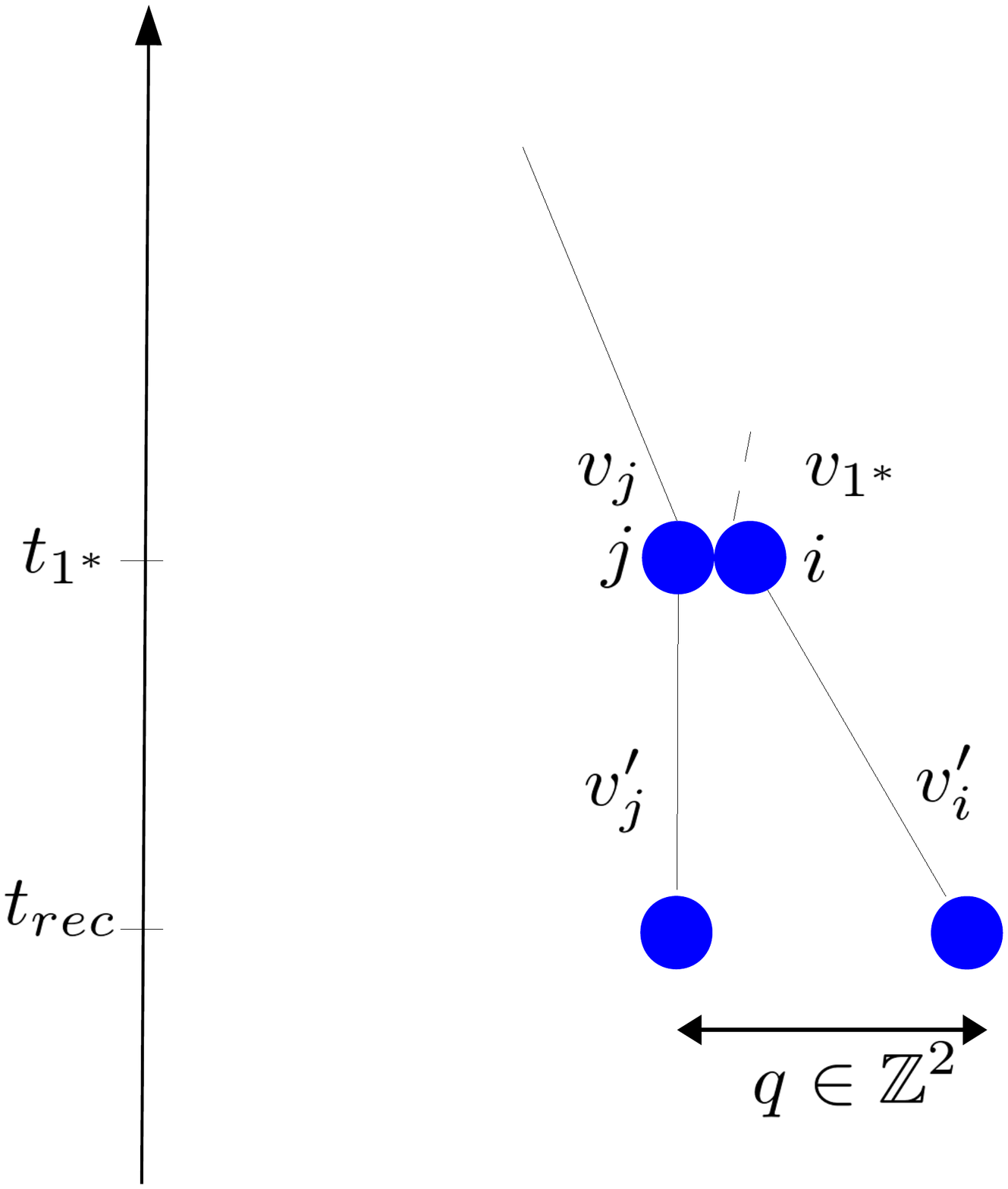} 
   \caption{ A self-recollision between $i,j$ ($p =0$) is due to periodicity; on the left the collision at time~$t_{1^*} $ is without scattering, on the right it is with scattering.
   }
\label{coccinellerecoll}
\end{figure}

 This has a very small cost, we indeed have for some recollision time $t_{rec} \geq 0$ and~$\nu_{rec}$ in ${\mathbb S}$
\begin{equation}
\label{self-recoll}
\eps \nu_{1^*} + (v'_i - v'_j)  (t_{rec} - t_{1^*}) = \eps \nu_{rec} + q \hbox{ for some } q\in \Z^2\setminus \{0\}
\end{equation}
assuming  for instance that particle $j$ has been created at time $t_{1^*}$ with velocity $v_{1^*}$, and denoting by $v'_i, v'_j$ the velocities after the collision.

\medskip

$ \bullet $ $ $   In the absence of scattering at time~$t_{1^*}$, we have  $v'_i= v_i $ and $v'_j = v_{1^*}$, and the equation~(\ref{self-recoll}) for   self recollision implies that $v_{1^*}-v_i$ has to belong to a cone $C(q, 2\eps)$ of opening $\eps$.
Because of the assumption that the   total energy is bounded by~$R^2$,  
\begin{align*} 
\int \indc _{ \{ v_{1^*}-v_i \in C(q,2\eps) \cap B_{2R} \} }
 \, \, \big | \big(v_{1^*}-v_{a(1^*)} ( t_{{1^*}} ) )\cdot \nu_{{1^*}}  \big |  d  t_{1^*} d \nu_{{1^*}}dv_{{1^*}}
 \leq C\eps R^3 t \, ,
\end{align*}
where $a(1^*) = i$.

\medskip

$ \bullet $ $ $  In the case with scattering, recall that
$$ v'_i-v'_j   = (v_i- v_{1^*}) - 2(v_i-v_{1^*}) \cdot \nu_{1^*} \nu_{1^*} .$$
Equation (\ref{self-recoll}) for the self recollision implies that~$ v'_i-v'_j $
  has to belong to $C(q,2\eps)$. For each fixed $\nu_{1^*}$, we conclude that  $v_i - v_{1^*}$ is in the  cone $S_{\nu_{1^*}}C(q,2\eps)$  (obtained from $C(q,2\eps)$ by symmetry with respect to $\nu_{1^*}$).
Because of the assumption that the   total energy is bounded by $R^2$, we have as in the previous case
\begin{align*} 
 \int \indc _{ \{ v_{1^*}-v_i \in S_{\nu_1^*} C(q,2\eps) \cap B_{2R}  \} }
 \big | \big(v_{1^*}-v_{a(1^*)} ( t_{{1^*}} ) )\cdot \nu_{{1^*}} \big|  d  t_{1^*} d \nu_{{1^*}}dv_{{1^*}} 
\leq C\eps R^3 t \,  .
\end{align*}

\medskip

 Note that, since the total energy is assumed to be bounded by $R^2 $ and we consider a finite time interval $[0,t]$ with $t\geq 1$, the  number of $q$'s for which the set 
is not empty is at most~$O\big(R^2t^2 \big)$.

 In order to obtain a bad set which depends only on the upper structure of the tree $ (a(i))_{i <1^*}$ and of the parameters  $ (t_i, v_i, \nu_i)_{i <1^*}$, we define  $\cP_1(a, 0, \{ 1^*\})$ as  the union of the previous sets over all possible $|q | \leq Rt +1$ and  all possible  $a(1^*) <1^*$.
 Summing over all these contributions, we end up with an upper bound for the scenario $p =0$
 \begin{equation}
 \label{self-recoll-est}
 \int \indc _{\cP_1(a, 0, \{ 1^*\})} \,  \,  \big | \big(v_{1^*}-v_{a(1^*)} ( t_{{1^*}} ) )\cdot \nu_{{1^*}} \big | d  t_{1^*} d \nu_{{1^*}}dv_{{1^*}} \leq C\eps  s R^5 t^3 \, .
\end{equation}

\bigskip
 \noindent
\underbar{Geometry of the first recollision.}
Without loss of generality,  we  may now assume  that  time~$ t_{1^*}$ corresponds to the deviation/creation of  the pseudo-particle $i$ and  that at $ t_{1^*}$ the collision does not involve both $i$ and~$j$. From now on, we denote by~$i$ and~$j$ the pseudo-particles, even if the actual particles may have disappeared through a collision (see Definition~\ref{pseudoparticle}).
 
Denote by $z_i$ and $z_j$ the (pre-collisional) configuration of pseudo-particles $i$ and~ $j$ at time~$t_{2^*}$.

 \begin{figure} [h] 
   \centering
 \includegraphics[width=5.7cm]{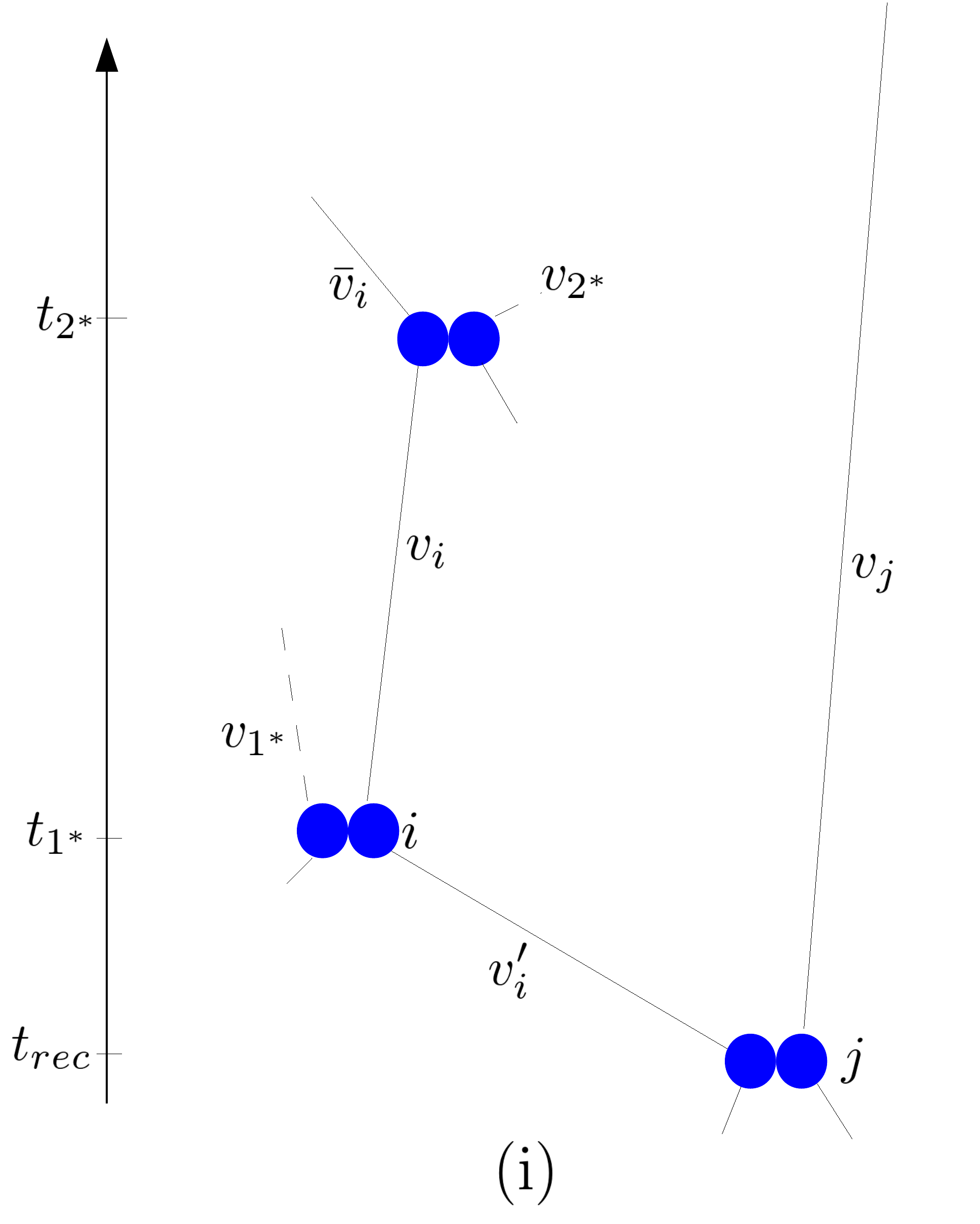}   \qquad  \quad\quad  \quad   \quad\includegraphics[width=5cm]{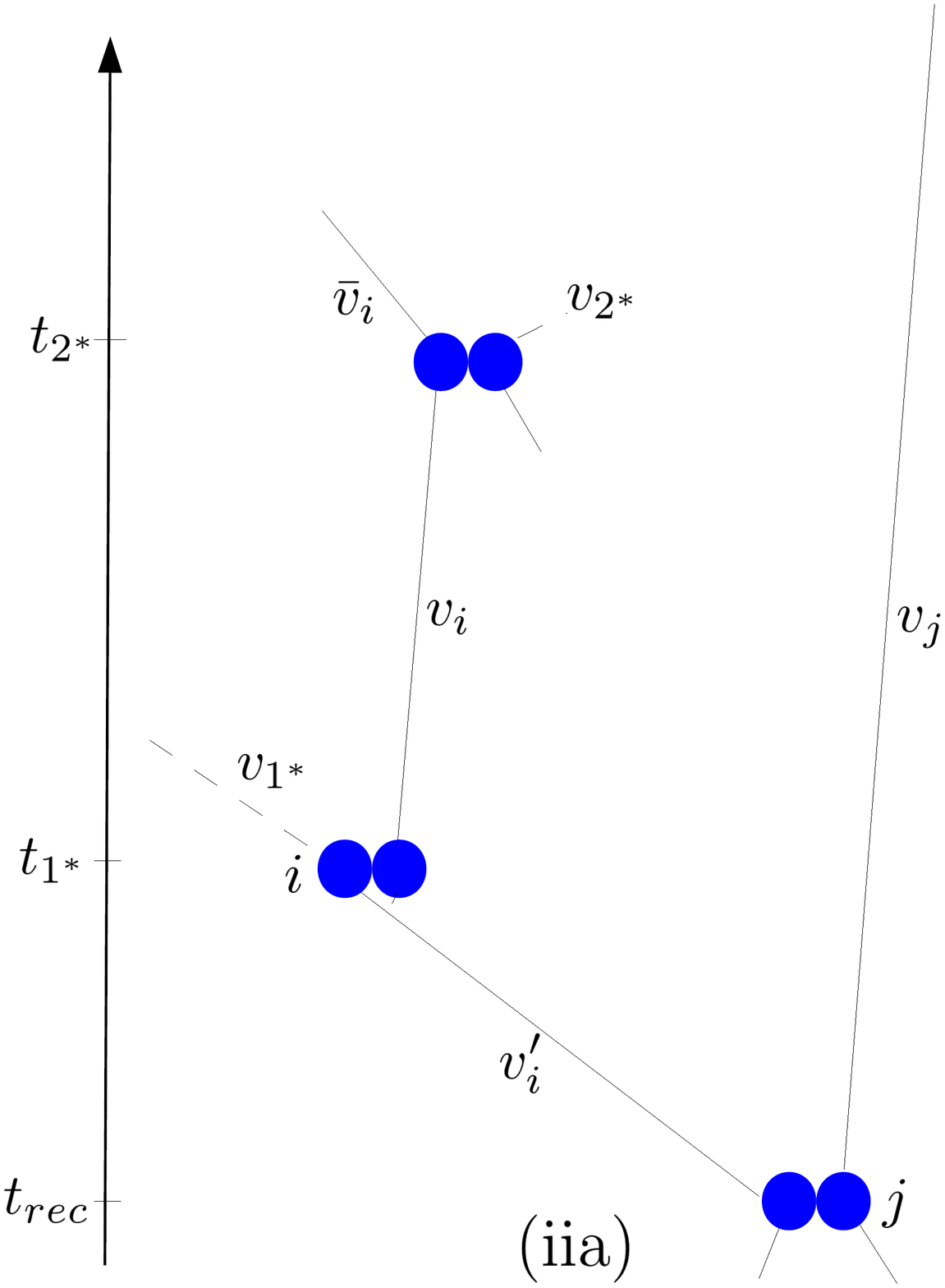}  
 \qquad  \quad     \includegraphics[width=5cm]{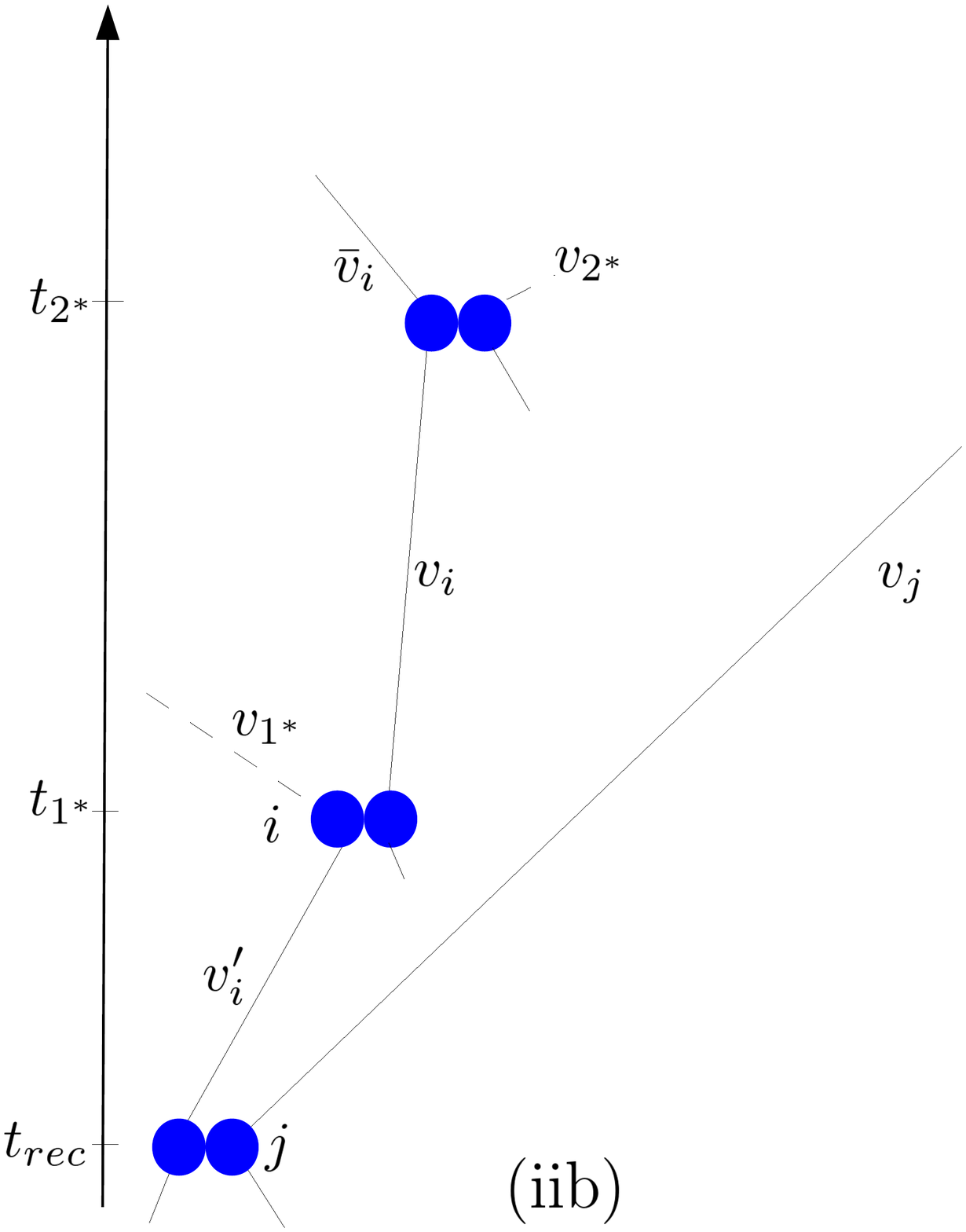}
   \caption{The two collisions at times $t_{1^*}$ and $t_{2^*}$ leading to the recollision between the pseudo particles $i$ and $j$ are depicted. Three different cases can occur if the first collision involves $i$ :  the particle $i$ can be deflected (i), or created without scattering (iia) or with scattering (iib). 
These three cases can also occur for the recollision at $2^*$ but only one is depicted each time. } \label{Figiiiaiib}
\end{figure}

\medskip
The condition for the recollision to hold in the backward dynamics at a time $t_{rec} \geq 0$
 then states
\begin{equation}
\label{eq: cond 1ere recoll}
(x_{i} - x_{j}) + (t_{1^*} - t_{2^*}) (v_{i} - v_{j})  + (t_{rec} - t_{1^*}) (  v'_i- v_{j}) =\eps \nu_{rec} +q\,,
\end{equation}
for some  $\nu_{rec} \in {\mathbb S}$, and $q \in \Z^2$.
As noticed previously, since the total energy is assumed to be bounded by $R^2$ and we consider a finite time interval $[0,t]$ with $t\geq 1$, the  number of $q$'s for which the set 
is not empty is at most $O\big(R^2 t^2 \big)$. Let us now fix $q$ and prove that the condition implies that $(t_{1^*},v_{1^*}, \nu_{1^*})$  is in a small domain depending only on $x_i-x_j$, $v_i$, $v_j$ and $q$.

As previously we   consider separately 
\begin{itemize}
\item   the case when the particle $i$  already exists before $t_{1^*}$ (as depicted in Figure~\ref{Figiiiaiib}(i)) : the velocity of particle $i$ after $t_{1^*}$ (in the backward dynamics) is then
$$ 
v'_i = v_i -  \big( (v_i- v_{1^*})\cdot \nu_{1^*} \big)  \, \nu_{1^*}\, . 
\eqno {\rm (i) }
$$

\item   the case when the particle $i$  was created at $t_{1^*}$ :  we then get\label{iia}
$$ v'_i = v_{1^*},   \eqno  {\rm(iia)  }$$
if $(v_{1^*}, \nu_{1^*}, v_i)$ is a precollisional configuration as on Figure~\ref{Figiiiaiib}(iia), and
$$
v'_i = v_{1^*} + \big( (v_i- v_{1^*})\cdot \nu_{1^*} \big) \,  \nu_{1^*}\,,  
\eqno {\rm(iib)}  
$$
if $(v_{1^*}, \nu_{1^*}, v_i)$ is a post-collisional configuration as on Figure~\ref{Figiiiaiib}(iib).  
\end{itemize}

\medskip
 
We   denote
$$\delta x := \frac1\eps ( x_{i} - x_{j}-q) \quad \hbox{ in case (i),  and}\quad  \delta x := \frac1\eps ( x_{i} - x_{j}-q)+\nu_{1^*}\quad  \hbox{  in case (ii)}\, .$$

Next we  decompose~$\delta x $ into a component along~$v_{i} - v_{j}$ and an orthogonal component, by writing 
$$
\delta x =  {\lambda \over \eps} (v_{i} - v_{j})  + \delta x_\perp \quad \hbox{ with } \quad\delta x_\perp \cdot (v_i - v_{j}) =0  \,  ,
$$
and we further rescale time as
\begin{equation}
\label{eq: tau1}
\tau_1 :=- \frac1\eps (t_{1^*} - t_{2^*}+\lambda ) \, ,\qquad  \tau_{rec} :=-\frac1\eps (t_{rec} - t_{1^*})\,.
\end{equation}
Note that we have used the hyperbolic scaling invariance (by scaling the space and time variables by~$\eps$), and that only the bounds on $\tau_1$ depend now on $\eps$
$$ |v_i-v_j|  \, |\tau_1| \leq \frac1\eps |v_i-v_j| t+ |\delta x| \leq {CRt\over \eps}\,\cdotp$$
We shall gain a factor $\eps$ on the integral in  time, thanks to the change of variable~$t_{1^*}  \mapsto \tau_1$.

\medskip
  In these new variables, the equation for the recollision can   be restated as follows
\begin{equation}
\label{rec-equation}
v'_{i} - v_{j} = \frac1{\tau_{rec} }  \delta x_\perp - \frac{\tau_1}{\tau_{rec}} (v_{i} - v_{j})-  \frac1{\tau_{rec} } \nu_{rec} \, .
\end{equation}
By using \eqref{eq: small distance cut off prime} with $M = R$, we can restrict to the case 
$|\tau_1| \, |v_{i} - v_{j}| \geq R$ so that 
$$
\big| \delta x_\perp - \tau_1  (v_{i} - v_{j}) \big| \gg 1,
$$ 
as $ \delta x_\perp \perp ( v_{i} - v_{j} )$.
 Since the total energy is bounded by $R^2$, the left-hand side of \eqref{rec-equation}  is bounded by $2R$, and we get that
\begin{equation}
\label{sizetaurec}
\frac1{|\tau_{rec} |} \leq  {4R\over |\tau_1| | v_{i} - v_{j}|} \, \cdotp
\end{equation}

Given $ \delta x_\perp$ and  $\tau_1 (v_{i} - v_{j})$, the relation \eqref{rec-equation}  forces  $v'_{i} - v_{j}$ to belong to a rectangle~${\mathcal R}(\delta x_\perp, v_i-v_j, \tau_1, q) $ of main axis $\delta x_\perp - \tau_1 (v_{i} - v_{j})$ and  of size~$2R \times \left( 2 R \min \left( \frac{4}{|\tau_1| | v_{i} - v_{j}|}, 1 \right) \right)$. The length~$2R$ is a consequence of the cut-off on the velocities.
The following lemma provides an upper bound on this constraint.
\begin{Lem}
\label{rec-eq} 
Fix~$t \geq 1$,  $\delta x_\perp \in \R^2 $, $v_{i} ,v_{j} \in B_R$ with~$1\leq R^2 \leq C_0 |\log \eps | $, and~$1\leq t \leq C_0 |\log \eps |$. Then 
$$
\int_{B_R \times {\mathbb S} \times [-  Ct/\eps,  Ct/\eps] } 
\indc_{  \{ v'_{i} - v_{j} \in {\mathcal R}(\delta x_\perp, v_i-v_j, \tau_1, q)  \} }  \;  \big |  (v_{1^*}-v_{i} ) \cdot  \nu_{{1^*}}\big|   d \tau_1d\nu_{1^*}   dv_{1^*} 
 \leq \frac{C R^3   ( \log \eps)^2}{ |v_i-v_j| } \,\cdotp
$$
\end{Lem}
\begin{proof}[Proof of Lemma {\rm\ref{rec-eq}}]
Applying \eqref{rectangle0} of Lemma~\ref{scattering-lem 2} page~\pageref{scattering-lem 2}, we deduce that 
\begin{align*}
\int\! \!  \indc_{  \{ v'_{i} - v_{j} \in {\mathcal R}(\delta x_\perp, v_i-v_j, \tau_1, q)  \} }  
& 
\big |  (v_{1^*}-v_{i} ) \cdot  \nu_{{1^*}}\big|    \,  dv_{1^*} d\nu_{1^*}  \\
& \leq C R^3    \min \left( \frac{4}{|\tau_1| | v_{i} - v_{j}|}, 1 \right)\Big ( |\log (|\tau_1| | v_{i} - v_{j}| )| + \log R \Big)\\
& \leq C R^3  |\log \eps |  \min \left( \frac{4}{|\tau_1| | v_{i} - v_{j}|}, 1 \right),
\end{align*}
   recalling that~$R^2+t \ll |\log \eps|$.
Integrating with respect to $| v_{i} - v_{j}| \, |\tau_1|$ up to $Rt/\eps $, we obtain that 
$$
\int  \indc_{  \{ v'_{i} - v_{j} \in {\mathcal R}(\delta x_\perp , v_i-v_j, \tau_1, q)  \} }  \; 
  \, \,  \big |  (v_{1^*}-v_{i} ) \cdot  \nu_{{1^*}}\big|   |v_i-v_j| d \tau_1  dv_{1^*} d\nu_{1^*} 
 \leq C R^3   ( \log \eps)^2 .
$$
This completes Lemma~{\rm\ref{rec-eq}}. 
\end{proof}

In Lemma \ref{rec-eq}, the measure of the set leading to a recollision is evaluated in terms of the variable~$\tau_1$.
Going back to the variables $(v_{1^*}, \nu_{1^*}, t_{1^*})$ and summing over all possible~$q$, we therefore obtain  
\begin{equation}
\label{proofProp4.3} 
\int \indc_{\{ v'_{i} - v_{j} \in \cup_q {\mathcal R}(\delta x_\perp, v_i-v_j, \tau_1, q) \} } \,  \,   \big |  (v_{1^*}-v_{i} ) \cdot  \nu_{{1^*}}\big|   d t_{1^*}  dv_{1^*} d\nu_{1^*} \leq CR^5t^2 {\eps  |\log \eps|^2\over |v_i-v_j|}\,\cdotp
\end{equation}
 On the other hand, a direct computation shows that
$$
  \int  \, \big |  (v_{1^*}-v_{i} ) \cdot  \nu_{{1^*}}\big|  d t_{1^*}  dv_{1^*} d\nu_{1^*} \leq CR^3 t \, ,
  $$
  so using the fact that $R\geq 1 $, $t\geq 1$, we find
\begin{equation}\label{estimate1recol1/vi-vj} 
\begin{aligned}
   \int \indc_{ \{ v'_{i} - v_{j} \in \cup_q {\mathcal R}(\delta x_\perp, v_i-v_j, \tau_1, q) \} } \,   \,  \big |  (v_{1^*}-v_{i} ) \cdot  \nu_{{1^*}}\big|  d t_{1^*}  dv_{1^*} d\nu_{1^*}
   \\ \leq CR^5 t^2  
   \min \Big( {\eps |\log \eps|^2 \over |v_i-v_j|} \,  , 1 \Big)\,\cdotp
\end{aligned}
\end{equation}

\bigskip
 \noindent
\underbar{Integration of the singularity.}\label{integrationsingulatity}

Now we need to integrate out the singularity $1/ |v_{i} - v_{j}|  $, 
when the parameters of the preceding collision~$(t_{2^*}, v_{2^*}, \nu_{2^*})$ range over $[0,t] \times B_R \times {\mathbb S}$.
Denote by $\bar v_i$ the velocity of particle $i$ before the collision with $2^*$ (see Figure \ref{Figiiiaiib}).
From \eqref{carleman3 joint} in Lemma~\ref{joint-scattering-lem} page~\pageref{joint-scattering-lem}, we know that the singularity $1/ |v_{i} - v_{j}|$ is integrable if particles $i,j$ are related through the same collision. Otherwise Inequality~(\ref{carleman2}), from Lemma~\ref{scattering-lem}, implies that 
$$
 \int  \min \Big(
 \frac{ \eps|\log \eps|^2  }{|v_j-  v_i|},1
 \Big) \,  \,   \big |  (v_{2^*}-  \bar v_{i} ) \cdot  \nu_{{2^*}}\big|  d t_{2^*}  dv_{2^*} d\nu_{2^*} \leq C t
 R^2\eps | \log \eps | ^3\,,
 $$
and together with~(\ref{estimate1recol1/vi-vj})   this implies that 
$$\begin{aligned}
 \int \indc _{\{ v'_{i} - v_{j} \in \cup_q {\mathcal R}(\delta x_\perp,v_i-v_j, \tau_1, q) \} } 
 \displaystyle  \prod_{m=1^*,2^*} \,\big | \big(v_{m}-v_{a(m)} ( t_{{m}} ) )\cdot \nu_{{m}} \big| 
 d  t_{{m}} d \nu_{{m}}dv_{{m}}   
  \leq CR^{7} t^3\eps \, |\log \eps|^3\,.
\end{aligned}
$$

\medskip
Now we would like to define  bad sets which are parametrized only by $( t_m, v_m, \nu_m)$ for~$m=1^* $ or $m \leq 2^*$.

\smallskip

\noindent
- Suppose that $2^*$ is not the parent of $1^*$ (which we will refer to as scenario $p=1$). Then by construction $1^*$ will branch on one of the labels less than $2^*$. There are exactly two particles~ $a(1^*)$ and~$a(2^*)$ associated with the parents of $1^*, 2^*$ and the recollision will  take place  among these four  particles. By construction,  the choice of parameters for $1^*, 2^* $  leading to a recollision of type $p=1$ can be determined only from the configurations of the particles~$a(1^*),a(2^*)$ at height $2^* -1$.

The bad set associated with the previous scenario (labelled $p=1$) is denoted  $\cP_1(a, 1, \{1^*, 2^*\})$ and defined as the union  of the previous sets.
We end up with the estimate
 \begin{equation}
 \label{recoll1-est}
 \int \indc _{\cP_1(a, 1, \{1^*, 2^*\}) }\,  \,  \prod_{m=1^*, 2^*} 
 \,\big | \big(v_{m}-v_{a(m)} ( t_{{m}} ) )\cdot \nu_{{m}} \big | d  t_{{m}} d \nu_{{m}}dv_{{m}}   
  \leq CR^{7}   t^3\eps \, |\log \eps|^3\,.
\end{equation}

\noindent \label{pagewiths}
- If $2^*$ is the parent of $1^*$, we  have --  by definition --  a recollision of type $p=2$. Only one particle involved in the recollision is fixed (it can be either~$1^*$ or~$a(1^*)$) and the second    recolliding particle $j$ is just an obstacle which  has to be chosen among the particles with label less than~$2^*$. Note that this obstacle is just transported freely between time~$t_{2^*}$ and the time of the recollision.

We then define $\cP_1(a, 2, \{1^*, 2^*\})$ as the union over all possible choices of  $j<2^*$ of the previous sets.
This leads to the   estimate
 \begin{equation}
 \label{recoll2-est}
 \int \indc _{\cP_1(a, 2, \{1^*, 2^*\})} \,  \,  \prod_{m=1^*, 2^*} 
 \,\big | \big(v_{m}-v_{a(m)} ( t_{{m}} ) )\cdot \nu_{{m}} \big | d  t_{{m}} d \nu_{{m}}dv_{{m}}   
  \leq CR^{7} s  t^3\eps \, |\log \eps|^3\,.
\end{equation}

\medskip

Note that $\cP_1(a, p, \{1^*, 2^*\})$ is empty if the parent of $1^*$ has a label greater than $2^*$.
This ends the proof of the proposition.\end{proof}

\begin{Rmk}
Estimate~{\rm(\ref{cP-est})} involves a loss with respect to~$\eps$ of the order~$|\log \eps|^3$. The above proof shows that 
the integration in time over the first parent produces a first loss in~$|\log \eps|^2$ (one of which is linked to the scattering operator), while the other power is due to a possible singularity in relative velocities, which needs to be integrated out thanks to the second parent, and the scattering operator again induces a ~$|\log \eps|$ loss. 
 \end{Rmk}

\subsubsection{Global estimate}
\label{par4.2}

To estimate the global error due to recollisions, we have to incorporate  the  estimate  provided in  Proposition \ref{recoll1-prop}
with all the other collision integrals.  
We   use the fact that we have now a tree with $s-2$ or $s-3$  branching points, neglecting the constraints
 that~$(t_j)_{j\in \sigma}$ have to be  properly chosen in between other collision times, and also the constraint on the distribution of collision times on the different time intervals~$[t-kh, t-(k-1)h]$.
\begin{Prop}
\label{restrictedtrees}
We fix~$z_1 \in \T^2 \times  \R^2$,  
~$p \in \{0,1,2\}$ and 
 a set $\sigma \subset \{1, \dots, s\}$  of  at most~{\rm 2} indices.
We consider the sets $\cP_1 (a,p, \sigma)$ introduced in Proposition~{\rm\ref{recoll1-prop}} and we denote by~$\eta := C s R^7 t^3 \eps |\log \eps|^3$ the right-hand side of \eqref{cP-est}.
Then for $t\geq 1$, one has 
\begin{equation}
\label{eq: L1 estimate petit}
\begin{aligned} 
 \sum_{a \in \cA_s}   \int \indc _{\cT_{2,s}}    
 \indc_{\cP_1(a, p,\sigma)} \,  \Big( \prod_{i=2}^s \,  \big| (v_i -v_{a(i)} (t_i))\cdot \nu_i \big| \Big)  & M_\beta^{\otimes s} \, 
  dT_{2,s}  d\Omega_{2,s}   dV_{2,s} \\
 &   \leq   (Ct)^{s-3}  s^2 \eta M_{\beta/2}(v_1) \,.
  \end{aligned}
\end{equation}
If we further specify that the last $n$ collision times have to be in an interval of length $h\leq 1$
(this constraint is denoted by $\cT_{s-n+1, s}^h$)
\begin{equation}
\label{eq: L1 estimate petit h}
\begin{aligned}
 \sum_{a \in \cA_{s}}   \int \indc _{\cT_{2,s}}
 \indc_{\cT_{s-n+1, s}^h}   \indc_{\cP_1(a, p,\sigma)}
 \; \Big( \prod_{i=2}^{s}   \,   \big| (v_i -v_{a(i)} (t_i))\cdot \nu_i \big|\Big)  M_\beta^{\otimes s } \, dT_{2,s}  d\Omega_{2,s}   dV_{2,s}    \\
\leq   (Ct)^{s-n-1}  (Ch)^{n-2} s^2 \eta  M_{\beta/2}(v_1) \, .
\end{aligned}
\end{equation}
\end{Prop}

\begin{proof}

We only consider the cases $p = 1,2$ which are the most delicate.
Proposition \ref{restrictedtrees} is a consequence of the estimates on the collision operators (see Proposition 
\ref{estimatelemmacontinuity})  for the particles which are not in $\sigma$ and the smallness estimate
\eqref{cP-est} for the particles in $\sigma$. 
These estimates can be decoupled by using Fubini's theorem and the fact that the sets $\cP_1(a, z_1,\sigma)$ do not depend on the whole trajectory but only on the parameters with labels less than $2^*$ as well as on the parameters associated with $1^*$.

In order to evaluate \eqref{eq: L1 estimate petit}, we 
  first perform the integration with respect to all the velocities and angles with labels larger than $2^*$ except those of the particle $1^*$.
Recall that~$\cP_1(a,z_1, \sigma)$ is independent of these parameters.
We can use the same estimates as in the proof of Proposition~\ref{estimatelemmacontinuity} 
$$\sum_{( a(j))_{j > 2^* \atop j \not \in 1^*} }\left( \prod_{i > 2^*, \atop i \not = 1^*}  \,   \big| (v_i -v_{a(i)} (t_i))\cdot \nu_i \big| \right)  
M_{ \beta/4}^{\otimes s } (V_s)\leq  (C s)^{s- 2^* -2}   \,.
$$
and integrate over each label in $\tilde \sigma = \{ i >2^*, \quad i \not= 1^*\}$
\begin{equation}
\label{eq: feuilles arbre}
\sum_{( a(j))_{j\in \tilde \sigma} }\int \left( \prod_{i\in \tilde \sigma }  \,   \big| (v_i -v_{a(i)} (t_i))\cdot \nu_i \big| \right)  
M_{ \beta}^{\otimes s } (V_s)dV_{\tilde \sigma} d\Omega_{\tilde \sigma}  \leq  (C s)^{s- 2^* -2}  M_{ 3\beta/4 }^{\otimes {2^*} }(V_{2^*}) \,.
\end{equation}
This bound takes into account the combinatorics of the trees up to $2^*$.
Note that the upper bound \eqref{eq: feuilles arbre}
overestimates \eqref{eq: L1 estimate petit} as we are also counting trees for which 
the branchings in between $2^*$ and $1^*$ may not be compatible with the conditions imposed by a recollision. This does not matter as the constraint on the recollision has already been  encoded in $\cP_1(a, z_1,\sigma)$ which we will use next.

The previous step removed all the dependency on the collision trees below the level $2^*$ and we can now use estimate \eqref{cP-est} and integrate over $1^*, 2^*$ (keeping frozen the parameters of the labels before $2^*$)
$$
\sum_{a(1^*), a(2^*)  }
 \int \indc _{\cP_1(a,p,\sigma)  }  \left( \prod_{i \in \sigma  }  \,   \big| (v_i -v_{a(i)} (t_i))\cdot \nu_i \big| \right)  
dT_\sigma dV_\sigma d\Omega_\sigma \leq  s^2 \eta 
 \,,$$
 uniformly with respect to all parameters $( t_i,v_i,\nu_i)_{i<2^*}$.
The factor $s^2$ in the inequality comes from the choices of $a(1^*), a(2^*)$.

Once the constraint on the recollision has been taken into account, the remaining part of the tree before $2^*$ can be estimated 
 by using the estimates from Proposition \ref{estimatelemmacontinuity}.
This leads to an extra factor $(C s)^{2^*-1}$.

It remains to  integrate over the times $(t_i)_{i \not \in \sigma}$ and we can simply remove the constraint on the times labelled by $\sigma$.
We distinguish two cases :

\begin{itemize}
\item
In \eqref{eq: L1 estimate petit}, the time constraint $\cT_{2,s}$
boils down to integrating over a simplex of dimension~$(s-1)-2$, the volume of which is
$${t^{s-3}\over (s-3)!} \leq C^{s}  { t^{s-3}\over s^{s-3}}  $$
by Stirling's formula.
\item 
In \eqref{eq: L1 estimate petit h}, we have to add the condition that the last $n$ times are in an interval of length $h\leq 1$. For $t \geq 1$, the worst situation is when all times $(t_i)_{i\in \sigma}$ are in this small time interval, as we loose the corresponding smallness. More precisely, we get
$${t^{s-1-n}\over (s-1-n)!}{h^{n-2}\over (n-2)!} \leq C^{s}  { t^{s-1-n} h^{n-2} \over s^{s-1-2}} \,\cdotp$$
\end{itemize}
This completes the proof of Proposition \ref{restrictedtrees}.
\end{proof}

\begin{proof}[Proof of Proposition {\rm\ref{prop: 1 recoll}}]

Given   $z_1 \in \T^2 \times B_R$, the set of parameters leading to pseudo-trajectories with at least one recollision is partitioned 
into subsets $\cP_1(a,p,{\sigma})$ 
(see Proposition \ref{recoll1-prop}). We therefore have
\begin{equation}
\label{eq: ensemble pathologique}
\begin{aligned}
\Big |& f^{(1,K), \geq}_{N}(t,z_1) \Big| 
\leq   \sum_{j_1=0}^{n_1-1}\! \!   \dots \!  \! \sum_{j_K=0}^{n_K-1} \alpha^{J_K-1} 
\sum_{ a\in \cA_{J_K} } \sum_{ p, \sigma} 
   \int  \indc _{\cT_{2,J_k}}
  \indc_{\cP_1(a,p, \sigma)} 
 \\
&   \quad \times  \left( \prod_{i=2}^{J_K}  \,   \big| (v_i -v_{a(i)} (t_i))\cdot \nu_i \big|\right)      \big(  f^{(J_K)}_{N,0}\indc_{{\mathcal V}_{J_K} }\big)
\,    dT_{2,J_K}  d\Omega_{2,J_K}   dV_{2,J_K} \, .
\end{aligned}
\end{equation}
We have seen in \eqref{eq: linfty initial} that the marginals of the  initial datum 
are dominated by a Maxwellian
$$
\big| f_{N,0}^{(J_K)}  (Z_{J_K}) \big|  \leq C^{J_K}   M_\beta^{\otimes J_K}(V_{J_K})  
\|g_{\alpha,0} \|_{L^\infty} \, .
$$
Thus  \eqref{eq: L1 estimate petit}  can be applied to estimate $f^{(1,K), \geq}_{N}$
\begin{align*}
\Big | f^{(1,K), \geq}_{N}(t,z_1) \Big| 
\leq 
\|g_{\alpha,0} \|_{L^\infty} 
\sum_{j_1=0}^{n_1-1}\! \!   \dots \!  \! \sum_{j_K=0}^{n_K-1}
C^{J_K}    \alpha^{J_K-1} J_K^5 t^{J_K} \eps \big| \log \eps \big|^\frac{19}2 M_{\beta/2}(v_1)\, ,
\end{align*}
where the parameter $\eta$ in \eqref{eq: L1 estimate petit}   has been estimated by using that~$1 \leq  R^2\leq C_0 |\log \eps|$ and~$1 \leq  t \leq C_0 |\log \eps|$.
Note that compared to \eqref{eq: L1 estimate petit}, an extra factor $J_K^2$ was added to take into account  the sum over the possible choices for $\sigma$.

Now recalling that~$n_k =  2 ^kn_0$ we have
\begin{equation}
\label{eq: nk 2k}
J_K \leq     2^{K+1} n_0
\quad \text{and} \quad
\sum_{j_1=0}^{n_1-1}\! \!   \dots \!  \! \sum_{j_K=0}^{n_K-1}
\leq \prod_{i = 1}^K n_i \leq n_0^K 2^{K^2}\, ,
\end{equation}
so thanks to Assumption \eqref{eq: Linfty initial} on the initial datum~$g_{\alpha,0}$, we conclude
\begin{align*}
\Big | f^{(1,K), \geq}_{N}(t,z_1) \Big| 
\leq  \exp (C\alpha^2) 2^{4K^2}n_0^{5+K}  \big( CT \alpha \big)^{  2^{K+1}n_0} 
 \eps \big| \log \eps \big|^\frac{19}2 M_{\beta/2}(v_1)\, .
\end{align*}
Since~$2^{K^2}n_0^{5+K}  \ll C^{2^K}$, this completes the proof of Proposition \ref{prop: 1 recoll} (bounding~$ \big| \log \eps \big|^\frac{19}2$ by~$ \big| \log \eps \big|^{10}$ to simplify).
\end{proof}

\subsection{Proof of Proposition \ref{main-prop}}
\label{subsec: Convergence of pseudo-trajectories}

Each term in the decomposition \eqref{eq: fN,cut 1}
$$
f^{(1,K)}_{N}(t) = f^{(1,K), 0}_{N}(t) + f^{(1,K), \geq}_{N}(t) 
$$
can be interpreted as a restriction of the domain of integration  of the times, velocities and   deflection angles. For $f^{(1,K), \geq}_{N}$, the pseudo-trajectories associated with a tree $a$ are integrated over the sets~$\cP_1(a, p, \sigma)$ as in \eqref{eq: ensemble pathologique}, instead they are integrated outside these sets in $f^{(1,K), 0}_N$. As a consequence the pseudo-trajectories in $f^{(1,K), 0}_N$ have no recollision.

\medskip

A similar decomposition holds for the Boltzmann hierarchy:  we distinguish whether the     pseudo-trajectories     lie on  the non-overlapping sets~$G_s(a)$ or not (see Definition~\ref{def: overlap}), and whether they lie on the pathological sets~$\cP_1(a, p, \sigma)$ or not (this splitting is artificial as there are no recollisions in the Boltzmann hierarchy, however it will be useful to 
compare the different contributions). Recalling
\begin{equation*}
\bar f ^{(1,K)}(t)  : =
\sum_{j_1=0}^{n_1-1}\! \!   \dots \!  \! \sum_{j_K=0}^{n_K-1}  \bar Q_{1,J_1} (h ) \bar Q_{J_1,J_2} (h )
\dots  \bar Q_{J_{K-1},J_K} (h ) \,     \big(  f^{(J_K)}_{0}\indc_{{\mathcal V}_{J_K} }\big) \, ,
\end{equation*}
let us write  
$$
\bar f ^{(1,K)} =\bar f ^{(1,K), 0}   + \bar f ^{(1,K), \geq}   +\bar f^{(1,K), \rm{overlap}}  \,,
$$
where~$\bar f ^{(1,K), 0} (t) $ corresponds to   restricting the pseudo-trajectories  to the sets of parameters~$^c{}\cP_1(a, p,\sigma) \cap G_s(a)$, while~$ \bar f ^{(1,K), \geq} (t) $ corresponds to the restriction to~$ \cP_1(a, p,\sigma) \cap G_s(a)$, and finally~$\bar f ^{(1,K), \rm{overlap}} (t)$ corresponds to the restriction to~$^c{} G_s(a)$.
 As a consequence of Proposition \ref{prop: 1 recoll}, 
the term $\bar f ^{(1,K), \geq} $ is negligible
\begin{align}
\label{eq: recoll Boltzmann hierarchy}
\Big | \bar f ^{(1,K), \geq} (t,z_1) \Big| 
\leq  \exp (C\alpha^2)   \big( CT \alpha \big)^{    2^{K+1}n_0 } 
 \eps \big| \log \eps \big|^{10}M_{\beta/2}(v_1)\, .
\end{align}
Similarly we claim that
\begin{align}
\label{eq: overlap Boltzmann hierarchy}
\Big | \bar f ^{(1,K),  \rm{overlap}} (t,z_1) \Big| 
\leq  \exp (C\alpha^2)   \big( CT \alpha \big)^{    2^{K+1}n_0 } 
 \eps \big| \log \eps \big|^{10}M_{\beta/2}(v_1)\, .
\end{align}
Indeed we notice that by definition
$$
^c{}G_s \subset  \widetilde{^c{}G_s} :=\Big\{(T_{2,s}, \Omega_{2,s}, V_{2,s}) \in  \cT_{2,s} \times {\mathbb S}^{s-1} \times \R^{2(s-1)} \, / \, \exists  i \,  , \, \exists k,\ell < i \, / \, |x_k(t_i) - x_\ell (t_i)| \leq 2 \eps\Big \} \, .
$$
If~$(T_{2,s}, \Omega_{2,s}, V_{2,s}) $ belongs to~$  \widetilde{^c{}G_s}$ and if~$i$ is the smallest integer such that
\begin{equation}\label{likearecoll2eps}
\exists k,\ell < i \, / \, |x_k(t_i) - x_\ell (t_i)| \leq 2 \eps \, ,
\end{equation}
then either the corresponding pseudotrajectory before time~$t_i$ (which exists by definition of~$i$) has suffered at least one recollision, and the result is a consequence of the proof of Proposition~ \ref{prop: 1 recoll}; or the condition~(\ref{likearecoll2eps}) can itself be interpreted as a ``recollision" (with~$\eps$ replaced by~$2\eps$) and the computations leading to  Proposition~\ref{prop: 1 recoll} may   again be reproduced exactly.  So~(\ref{eq: overlap Boltzmann hierarchy}) follows.

The last step to conclude Proposition \ref{main-prop} is to evaluate the difference $f^{(1,K), 0}_{N}(t) - \bar f ^{(1,K), 0} (t)$.
Once   recollisions and overlaps have been excluded, the only discrepancies between the BBGKY and the Boltzmann  pseudo-trajectories come from the micro-translations due to the diameter~$\eps$ of the colliding 
particles (see Definition \ref{pseudotrajectory}). 
At the initial time, the error between the two configurations is at most $O(s\eps)$ after $s$ collisions (see \cite{GSRT, BGSR1})
\begin{equation}
\label{eq: shift donnees initiales}
\big |\bar X_{s}(a, T_{2,s}, \Omega_{2,s}, V_{2,s},0) - X_{s}(a, T_{2,s}, \Omega_{2,s}, V_{2,s},0) \big|
\leq Cs \eps \, .
\end{equation}
The discrepancies are only for   positions, as   velocities remain equal in both hierarchies.
These configurations are then evaluated either on the marginals of the 
 initial datum~$f_{N,0}^{(s)} $ or of $f^{(s)} _0$ which are close to each other thanks to Proposition
 \ref{exclusion-prop2}.

\medskip

The main discrepancy between $f^{(1,K), 0}_{N}$ and $\bar f^{(1,K), 0} $ depends on 
\begin{align*}
& \Big| f_{0 }^{(s)} \Big( \bar Z_{s}(a, T_{2,s}, \Omega_{2,s}, V_{2,s},0) \Big)
- 
f_{N,0 }^{(s)} \Big( Z_{s}(a, T_{2,s}, \Omega_{2,s}, V_{2,s},0) \Big) \Big|\\
& \qquad \leq 
\Big| f_{0 }^{(s)} \Big( \bar Z_{s}(a, T_{2,s}, \Omega_{2,s}, V_{2,s},0) \Big)
- 
f_{0 }^{(s)} \Big( Z_{s}(a, T_{2,s}, \Omega_{2,s}, V_{2,s},0) \Big) \Big|\\
& \qquad  \qquad +
\Big| f_{0 }^{(s)} \Big( Z_{s}(a, T_{2,s}, \Omega_{2,s}, V_{2,s},0) \Big)
- 
f_{N,0 }^{(s)} \Big( Z_{s}(a, T_{2,s}, \Omega_{2,s}, V_{2,s},0) \Big) \Big| \, .
\end{align*}
By the assumption \eqref{eq: Linfty initial}, $g_{\alpha,0}$ has a Lipschitz bound
$\exp ( C \alpha^2)$, thus combining \eqref{eq: shift donnees initiales} and 
the estimate of Proposition \ref{exclusion-prop2}, we get
\begin{align*}
\Big| f_{0 }^{(s)} \Big( \bar Z_{s}(a, T_{2,s}, \Omega_{2,s}, V_{2,s},0) \Big)
- 
f_{N,0 }^{(s)} \Big( Z_{s}(a, T_{2,s}, \Omega_{2,s}, V_{2,s},0) \Big) \Big|
\leq C^s \exp ( C \alpha^2) s \eps M_\beta ^{\otimes s}(V_s) \, .
\end{align*}

The last source of discrepancy between the formulas defining $f^{(1,K), 0}_{N}$ and $\bar f_N^{(1,K), 0} $
comes from the prefactor~$(N-1)\dots (N-s+1) \eps^{s-1}$ which has been replaced by $\alpha^{s-1}$.
For fixed $s$, the corresponding error is 
$$ \Big( 1- {(N-1) \dots (N- s+1) \over N^{s-1}}\Big)  \leq C {s^2 \over N}\leq Cs^2 {\eps \over \alpha} $$
which, combined with the bound on the collision operators,  leads to an error of the form 
\begin{equation}
\label{error3}
(C  \alpha t)^{s-1} s^2{\eps \over \alpha}   \,\cdotp
\end{equation}
Summing the previous bounds gives
\begin{equation}
\begin{aligned}
\label{eq: erreur 0N cut}
\Big | f^{(1,K), 0}_{N}(t,z_1) -\bar f ^{(1,K), 0} (t,z_1) \Big|
& \leq  \exp ( C \alpha^2)M_\beta(v_1) 
\sum_{j_1=0}^{n_1-1}\! \!   \dots \!  \! \sum_{j_K=0}^{n_K-1}
(C  \alpha t)^{J_K-1} \left( J_K^2 \, {\eps \over \alpha} + 
  J_K \eps \right) \\
& \leq  \exp (C\alpha^2) M_\beta(v_1) 
  \big( CT \alpha \big)^{    2^{K+1}n_0}
 \left( 2^{2 (K+1)} \, {\eps \over \alpha} + 
  2^{K+1} \eps \right) ,  
\end{aligned}
\end{equation}
where we used the bounds \eqref{eq: nk 2k}  for the sequence $n_k = 2^kn_0$ .

\medskip

Finally Proposition~\ref{main-prop} follows by combining  
\begin{itemize}
\item Proposition \ref{prop: 1 recoll} and \eqref{eq: recoll Boltzmann hierarchy} to control the recollisions,
\item (\ref{eq: overlap Boltzmann hierarchy}) to control overlaps in the pseudo-trajectories,
\item (\ref{eq: erreur 0N cut}) to control  the difference in the parts without recollisions.
\end{itemize}
The result is proved. \qed

\section{Symmetry and $L^2$ bounds}
\label{decompositionL2andCss+1}

In this section, we  prove an upper bound on the contribution of   
super exponential collision trees without recollisions introduced in \eqref{reste0}
\begin{equation*}
 R_N^{K, 0} (t) :=  \sum_{k=1}^K  \sum_{j_i <n_i \atop i\leq k-1} \sum_{j_k \geq n_k} Q^0_{1,J_1} (h ) \dots  Q^0_{J_{k-1},J_{k}} (h ) \big( f^{(J_K)}_N (t-kh)\indc_{{\mathcal V}_{J_K} } \big)\, .
 \end{equation*}
\begin{Prop}
\label{RNK0} 
Given~$T>1$,  $\gamma\ll 1$ and~$C$ a large enough constant (independent of~$\gamma$ and~$T$),  the parameters are tuned as follows 
\begin{equation}
\label{eq: parameters RNK0}
h \leq {\gamma^2 \over  \exp (C\alpha^2) T^3 } \, ,
\qquad n_k =  2^k n_0   \, .
\end{equation}
Then, under the Boltzmann-Grad scaling~$N\eps = \alpha \gg1$, we have for $t \in [0,T]$
\begin{equation}
\label{R0-est}
\Big\| R_N^{K, 0} (t) \Big\|_{L^2(\D)}  \leq \gamma  \, .
\end{equation}
\end{Prop}
The main step to derive Proposition \ref{RNK0} is to replace the $L^\infty$ estimates on the collision kernel (Proposition~\ref{estimatelemmacontinuity}) by $L^2$ estimates. 
To do this, we   first establish an $L^2_\beta$ decomposition of the marginals $f_N^{(s)}(t)$ (Proposition~\ref{lemfNL2sym 0} in Section~\ref{decompositionL2})
and then an $L^2$ counterpart of Proposition \ref{estimatelemmacontinuity} (Proposition~\ref{L2-est} in Section~\ref{sectionstabcss+1}).
%

\subsection{Structure of symmetric functions in $L^2$}
\label{decompositionL2}
  
We   prove in Proposition \ref{lemfNL2sym 0} that a structure similar to~(\ref{wish}) is intrinsic 
to   symmetric functions with suitable $L^2$ bounds (the argument does not involve dynamics).
As the density $f_N(t)$ of the particle system is symmetric and admits $L^2$ bounds uniform in time, we  can then  deduce that the higher order correlations of the marginals~$f_N^{(s)}(t, Z_s) $ are small in $L^2$ for any time.
This is a key ingredient in the proof of the main theorem. 
\medskip

The following proposition states a general decomposition of symmetric functions in $L^2_\beta$.
\begin{Prop}
\label{lemfNL2sym 0}
Let~$f_N$ be a mean free, symmetric function such that $f_N/ M_\beta^{\otimes N} \in L^2_\beta(\D^N)$. 
There exist symmetric functions $g_N^m$  on $\D^m$ for $1\leq m\leq N$  such that
for all $s \leq N$,  the marginal of order~$s$ satisfies
\begin{equation}
\label{fNs-exp}
f_N^{(s)}( Z_s) = M_\beta^{\otimes s} (V_s) \sum_{m=1}^s  \sum _{\sigma\in {\mathfrak S}^m_s } g_N^{{m}} (Z_\sigma)\,,
\end{equation}
where ${\mathfrak S}^m_s$ denotes the set of all parts of $\{1,\dots ,s\}$ with $m$ elements, and $ \binom{N}{m}  $  is its cardinal. Moreover 
$$
\| g_N^{m}  \|^2_{ L^2_\beta (\D^m)} 
\leq {  { 1 \over  \binom{N}{m}}  }
  \|f_N/M_\beta^{\otimes N}\|^2_{L^2_\beta (\D^N)} \,.
$$
\end{Prop}
Combining  (\ref{eq: ZN})  and \eqref{L2}, we see that at any time $t \geq 0$
\begin{align}
\label{eq: L2 upper bound}
\int {f_{N} ^2 \over M_\beta^{\otimes N} }(t,Z_N) dZ_N 
\leq \frac{1}{{\mathcal Z}_N}
\int {f_{N} ^2 \over M_{N,\beta} }(t,Z_N) dZ_N \leq 
CN\exp (C \alpha^2) \| g_{\alpha,0}\| _{L^2_\beta(\D)}^2 \, .
\end{align}
Thus Proposition \ref{lemfNL2sym 0} applies to the solution $f_N(t)$ of the Liouville equation and
for all $s \leq N$,  the marginal of order~$s$ satisfies
\begin{equation}
\label{eq: decomposition a t}
f_N^{(s)}(t, Z_s) = M_\beta^{\otimes s} (V_s) \sum_{m=1}^s  \sum _{\sigma\in {\mathfrak S}^m_s } g_N^{{m}} (t,Z_\sigma)\,,
\end{equation}
with 
\begin{equation}
\label{gi-est}
\forall t \geq 0, \qquad 
\| g_N^{m}(t) \|^2_{ L^2_\beta (\D^m)} \leq {CN\exp(C\alpha^2) \over \binom{N}{m} } 
\|g_{\alpha,0} \|_{  L^2_\beta  (\D)}^2 \, .
\end{equation} 
Although the definition is not exactly the usual one (due to the linear setting), we will call cumulant of order $m$  the function $g_N^{m}$ as it encodes the correlations of order $m$. It is indeed defined by some exhaustion procedure (which is somehow comparable to the Calder\' on-Zygmund decomposition), which ensures that the average of $g_N^{m}$ with respect to any of its coordinate is zero. In other words, all correlations of order less than $m-1$ have been removed.

Note that  the size of the correlations between several particles has been quantified by Pulvirenti, Simonella  \cite{PS} for chaotic initial data.  As in \eqref{gi-est}, the bounds obtained in \cite{PS} decrease with the degree of the correlations, however these estimates hold only for short times and moderate $m$ as they are valid even far from equilibrium.

The decomposition \eqref{eq: decomposition a t} can be understood as a projection of 
$f_N$ onto the reference measure~$M_\beta^{\otimes N}$ and the terms in 
\eqref{gi-est} are small because $f_N$ is close to $M_\beta^{\otimes N}$ in the $L^2$ sense
\eqref{eq: L2 upper bound}. In $d \geq 3$, the estimate \eqref{eq: L2 upper bound} no longer holds
(even for $f_N = M_{N,\beta}$) as the corrections induced by the exclusion are too large.
Thus to generalize the previous decomposition in $d \geq 3$, one would need to replace the reference 
measure $M_\beta^{\otimes N}$ by a more suitable one.

\begin{proof}[ Proof of Proposition~{\rm\ref{lemfNL2sym 0}}]

Define
$$
g_N^{m} (Z_m) := \sum_{k=1}^m (-1)^{m-k} \sum_{ {\sigma} \in {\mathfrak S}_m^k}
{f_N^{(k)} \over M_\beta^{\otimes k}} (Z_{\sigma}) \,.
$$

\medskip

{\bf Step 1.}   The identity
\begin{equation}
\label{fN-dec}
 {f_N\over M_\beta^{\otimes N}} (Z_N)  
 =   \sum_{m=1}^N  \sum_{ {\sigma} \in {\mathfrak S}_N^m} g_N^{m} (Z_{\sigma})
 \end{equation}
comes from a simple application of Fubini's theorem.
We indeed have
$$
\begin{aligned}
\sum_{m=1}^N  \sum_{ {\sigma} \in {\mathfrak S}_N^m} g_N^{m} (Z_{\sigma})
&= \sum_{m=1}^N  \sum_{ {\sigma} \in {\mathfrak S}_N^m}
\sum_{k=1}^m (-1)^{m-k} \sum_{ {\tilde \sigma} \in {\mathfrak S}_m^k}
{f_N^{(k)} \over M_\beta^{\otimes k}} (Z_{\tilde \sigma})\\
&= \sum _{k =1}^N  \sum_{ {\tilde \sigma} \in {\mathfrak S}_N^k} {f_N^{(k)} \over M_\beta^{\otimes k}} (Z_{\tilde \sigma}) \sum_{m=k}^N (-1)^{m-k}   \binom{N-k}{m-k},
\end{aligned}
$$
since the number of possible $\sigma$ with $m$ elements having $\tilde \sigma$ as a subset is $ \binom{N-k}{m-k}$.

\medskip
  For $k<N$, we have
$$\sum_{m=k}^N (-1)^{m-k}  \binom{N-k}{m-k} = \sum_{m=0}^{N-k} (-1)^m \binom{N-k}{m} = 0^{N-k} = 0 \,  ,$$
while for $k=N$ we just obtain 1. We therefore get (\ref{fN-dec}).

\medskip

{\bf Step 2.} 
 We prove now that 
 \begin{equation}\label{orthogonality}
\int g_N^{m} (Z_m) M_\beta(v_\ell)\, dz_\ell  = 0 \, , \quad  1 \leq \ell \leq m\,. 
\end{equation}
Given~$1 \leq \ell \leq m$, one can split the sum over~$ {\sigma} \in {\mathfrak S}_m^k$ into two pieces, depending on whether~$\ell $ belongs to~$   {\sigma}$ or not
$$ 
\begin{aligned}
&\int g_N^{m} (Z_m) M_\beta(v_\ell)\, dz_\ell  \\
&\quad = \sum_{k=1}^m (-1)^{m-k} \sumetage{ {\sigma} \in {\mathfrak S}_m^k}{\ell \in {\sigma}}
\int {f_N^{(k)} \over M_\beta^{\otimes k}} (Z_{\sigma})  M_\beta (v_\ell) dz_\ell 
+  
\sum_{k=1}^{m-1} (-1)^{m-k} \sumetage{ {\sigma} \in {\mathfrak S}_m^k}{\ell \notin {\sigma}}
\int {f_N^{(k)} \over M_\beta^{\otimes k}} (Z_{\sigma})  M_\beta (v_\ell) dz_\ell \\
&\quad = \sum_{k'=0}^{m-1} (-1)^{m-k'+1} \sum_{ {\sigma} \in {\mathfrak S}_{m-1}^{k'} \atop \ell \notin {\sigma}}
{f_N^{(k')} \over M_\beta^{\otimes k'}} (Z_{\sigma})   
+  \sum_{k=1}^{m-1} (-1)^{m-k} 
\sumetage{ {\sigma} \in {\mathfrak S}_m^k}{\ell \notin {\sigma}}
 {f_N^{(k)} \over M_\beta^{\otimes k}} (Z_{\sigma})  \, .
\end{aligned}
$$
The conclusion follows from the fact that the case~$k' = 0$ corresponds to 
$$
\int {f_N^{(1)} \over M_\beta} (z_\ell)  M_\beta (v_\ell) dz_\ell 
= \int f_N  (Z_N) d Z_N = 0 \, .
$$
Hence we obtain
$$
\int g_N^{m}  (Z_m) M_\beta(v_\ell)\, dz_\ell  = 0 \, .
$$  
The identity \eqref{fNs-exp} follows by  integrating (\ref{fN-dec}) with respect to $M_\beta^{\otimes (N-s)} dz_{s+1}\dots dz_{N}$
$$
f_N^{(s)}(Z_s) = M_\beta^{\otimes s} \sum_{m=1}^s  \sum _{\sigma\in {\mathfrak S}^m_s } g_N^{m} (Z_\sigma)\,.
$$

\medskip

{\bf Step 3.}  It remains to establish estimate (\ref{gi-est}).
From (\ref{fN-dec})  and the orthogonality condition~(\ref{orthogonality}), we also deduce that
$$
\begin{aligned}
\int {f_N^2 \over M^{\otimes N}_\beta} dZ_N &= \int M_\beta^{\otimes N}  
\left(  \sum_{m=1}^N \sum _{ {\sigma} \in {\mathfrak S}_N^m} g_N^{m} (Z_{\sigma})\right)^2 dZ_N 
= \sum_{m=1}^N \sum _{ {\sigma} \in {\mathfrak S}_N^m}\int M_\beta^{\otimes N}  \left(  g_N^{m} (Z_{\sigma})\right)^2dZ_N\\
& =  \sum_{m=1}^N  \binom{N}{m} \| g_N^{m} \|^2_{L^2_\beta (\D^m)}\,.
\end{aligned}
$$
This ends the proof of Proposition \ref{lemfNL2sym 0}.
\end{proof}

\begin{Rmk}
The decomposition \eqref{fNs-exp} shows that the higher order correlations decrease in~$L^2$-norm according to the number of particles. 
This is a step towards proving local equilibrium, but these estimates are not strong enough to deduce directly that the equation on the first marginal can be closed because  the collision operator is too   singular.  
\end{Rmk}

\subsection{$L^2$ continuity estimates for the iterated collision operators}
\label{sectionstabcss+1}

We will now establish an $L^2$ estimate for  $Q^0_{1,J} (t)$ (see Proposition~\ref{L2-est}). As explained in the introduction (see Paragraph \ref{introproblemL2}), it involves a loss in $\eps$, which will be exactly compensated by the 
decay of the $L^2_\beta$-norm (\ref{gi-est}) in the expansion (\ref{fNs-exp}).
This shows that the  structure (\ref{wish}) is partly preserved by the collision-transport operators, as long as there is no recollision.

\subsubsection{Statement of the result and strategy of the proof}
Let us first introduce some notation.
As in \eqref{eq: iterated collision operator abs}
 for $| Q_{s,s+n}| (t)$, the operator $| Q^0_{s,s+n}| (t)$ is obtained by considering 
the sum $C^+_{s,s+1} + C^-_{s,s+1}$ instead of the difference. 
Let $g_m \in L^2_\beta (\D^m)$, we set for ${\sigma} \in {\mathfrak S}_s^m$
\begin{equation}
\label{notationsigma}
g_{m,\sigma} (Z_s) = g_m (Z_\sigma) \, .
\end{equation}
The key estimate is given by the following proposition. Note that the bound provided in~(\ref{Q-est1}) is not the best one can prove (in terms of the way the powers of~$t$ and~$h$ are divided) but suffices for our purposes.
\begin{Prop}
\label{L2-est}
There is a constant~$C$ (depending only on $\beta$) such that for all $J,n\in \N^*$ and all~$ t\geq 1, h\in [0,t] $, the operator~$|Q^0|$   satisfies the following continuity estimate 
	\begin{equation}
	\label{Q-est1}
		\begin{aligned}
	&
	\Big \|   |Q_{1,J}^0| (t) \, |Q^0_{J,J+n}|(h) \, \sum _{\sigma\in {\mathfrak S}^m_{J+n} } {  \indc_{{\mathcal V}_{J+n} }  }M_\beta^{\otimes (J +n)} \big| g_{m, \sigma} \big| \Big\|_{L^2(\D)}\\
	& \qquad \qquad \qquad \qquad \qquad 
	\leq   (C\alpha)^{J+n-1} t^{J+n/2 -1}h^{n/2}{ \| g_m\|_{L^2_\beta ( \D^m)}  \over \sqrt{ \eps^{m-1} m!} }\,\cdotp
\end{aligned}
\end{equation}
\end{Prop}

\begin{proof} 
To simplify  the analysis, especially the treatment of large velocities, we define  modified collision operators
\begin{equation}\begin{aligned}
\label{Cb-def}
 \big(C^{b,\pm}_{s,s+1}h^{s+1}\big)(Z_s)  := 
\frac{(N-s) \eps}{\alpha}  \sum_{i=1}^s & \int_{{\mathbb S} \times \R^2} h^{s+1} 
(Z_{s+1}^{\pm,i,s+1}) 
{\big((v_i-v_{s+1}) \cdot \nu \big)_+\over 1+ |v_i-v_{s+1}|}  \, d\nu dv_{s+1} \, , \\
\big(C^{q,\pm}_{s,s+1}h^{s+1}\big)(Z_s)  := 
\frac{(N-s) \eps}{\alpha} \sum_{i=1}^s &  \int_{{\mathbb S} \times \R^2} h^{s+1} (Z_{s+1}^{\pm,i,s+1}) 
 \\
 &  \qquad  \times (1+ |v_i- v_{s+1}|) {\big((v_i-v_{s+1}) \cdot \nu \big)_+}  \, d\nu dv_{s+1}  ,
\end{aligned}
\end{equation}
where $Z_{s+1}^{\pm,i,s+1}$ denotes the configuration after the collision between $i$ and $s+1$ as in \eqref{BBGKYcollision}
\begin{align*}
Z_{s+1}^{-,i,s+1} &: = (x_1, v_1, \dots, x_i, v_i, \dots , x_i - \e \nu , v_{s+1}) \, ,\\
Z_{s+1}^{+,i,s+1} &: = (x_1, v_1, \dots, x_i, v_i', \dots , x_i + \e \nu , v_{s+1}') \, .
\end{align*} 
By construction, $C^{b,\pm}_{s,s+1}$ has a bounded collision cross-section and $C^{q,\pm}_{s,s+1}$ has  a collision cross-section with quadratic growth in $v$. Defining accordingly $|Q^{b,0}_{1,J}|$ and 
$|Q^{q,0}_{1,J}|$, we have by the Cauchy-Schwarz inequality
$$
\begin{aligned}
 \Big| \, |Q_{1,J}^0| (t) \, |Q^0_{J,J+n}|(h) \,&\sum _{\sigma\in {\mathfrak S}^m_{J+n} } M_\beta ^{\otimes (J +n) }  {  \indc_{{\mathcal V}_{J+n} }  } \big| g_{m, \sigma} \big| \Big| \\
&  \leq \Big(\sum _{\sigma\in {\mathfrak S}^m_{J+n} } | Q_{1,J}^{q,0}| (t) \,|Q^{q,0}_{J,J+n}|(h) \,  
M_\beta ^{\otimes (J +n) }  
\Big)^{1 /2} \\
& \qquad \times
\Big( | Q_{1,J}^{b,0} | (t) \, |Q^{b,0}_{J,J+n}|(h) \sum _{\sigma\in {\mathfrak S}^m_{J+n} } 
M_\beta ^{\otimes (J +n) }  g_{m, \sigma}^2 \Big)^{1 /2},
\end{aligned}
$$
 where the velocity cut-off ${\mathcal V}_{J+n}$ has been dropped. Thus we find directly
\begin{equation}
\begin{aligned}
\label{nosmeilleursamis}
 \Big| \, |Q_{1,J}^0| (t) \, |Q^0_{J,J+n}|(h) \,&\sum _{\sigma\in {\mathfrak S}^m_{J+n} } 
 M_\beta ^{\otimes (J +n) }  
 \big| g_{m, \sigma} \big| \Big| \\
&  \leq   2^{\frac{J+n}{2}}  \Big(  |Q_{1,J}^{q,0}|(t)  \,|Q^{q,0}_{J,J+n}|(h) \, 
M_\beta ^{\otimes (J +n) }   \Big)^{1 /2}\\
&\qquad \times
\Big( |Q_{1,J}^{b,0}| (t)  \, |Q^{b,0}_{J,J+n}|(h) \, \sum _{\sigma\in {\mathfrak S}^m_{J+n} } 
M_\beta ^{\otimes (J +n) } g_{m, \sigma}^2  \Big)^{1 /2}   \, .
\end{aligned}
\end{equation}

$ \bullet $ $ $  
The first factor can be bounded in $L^\infty$ as in Proposition \ref{estimatelemmacontinuity}.
\begin{Prop}
\label{prop: estimatelemmacontinuity quad} 
There is a constant~$C$ (depending only on $\beta$) such that for all $J,n\in \N^*$ and all~$h,t\geq 0$, the operator~$|Q^{q,0}|$   satisfies the following continuity estimates 
	\begin{equation}
	\label{Qquad-estinfty}
	\forall z_1 \in \D \, , \qquad 
|Q_{1,J }^{q,0}| (t) \,  |Q_{J,J+n}^{q,0} | (h)M_\beta^{\otimes (J+n)} (z_1) 
\leq   (C\alpha t)^{J - 1} (C\alpha h)^{n} 
	M_{3\beta/4} (z_1) \, .
	\end{equation}
\end{Prop}
The proof is omitted as it is similar to the one of Proposition \ref{estimatelemmacontinuity}  (we just have to skip the Cauchy-Schwarz estimate in (\ref{CS})).
Note that the quadratic growth in the collision cross-section is  critical in the sense that it is the highest possible power  giving an admissible loss estimate.

\medskip

Thus \eqref{nosmeilleursamis} can be bounded as  follows
\begin{equation}
\begin{aligned}
\label{nosmeilleursamis bis}
& \int_\D \Big( \, |Q_{1,J}^0| (t) \, |Q^0_{J,J+n}|(h) \, \sum _{\sigma\in {\mathfrak S}^m_{J+n} } M_\beta^{\otimes (J +n)} {  \indc_{{\mathcal V}_{J+n} }  } \big| g_{m, \sigma} \big| \Big)^2  \, 
dz_1 \\
&  \quad \leq (C\alpha t)^{J-1}(C\alpha h)^{n} \int_\D  |Q_{1,J}^{b,0}| (t) \, |Q^{b,0}_{J,J+n}|(h)  \sum _{\sigma\in {\mathfrak S}^m_{J+n} } M_\beta^{\otimes (J+n)} g_{m, \sigma}^2  \, dz_1  
\, .  
\end{aligned}
\end{equation}

$ \bullet $ $ $  
The second factor can be bounded from above by relaxing the conditions on the distribution of times 
to retain only  that the collision times  have to satisfy
$$
0\leq t_{J+n-1} \leq  \dots \leq t_J \leq \dots \leq  t_2 \leq t+h \leq 2t \, .
$$
In other words, we have
$$|Q^{b,0}_{1,J }| (t)  \, |Q^{b,0}_{J,J+n}|(h)  \leq  |Q^{b,0}_{1,J+n }| (2t)   \, .$$
This is suboptimal in the sense that it implies that powers of~$h$ will be traded for powers of~$t$ but the smallness thanks to~$h$   already present on the right-hand side of~(\ref{nosmeilleursamis bis}) will be enough for our purposes.
To establish Proposition \ref{L2-est}, it is then enough to prove the following
proposition which will be applied to $g_m^2$. 
\begin{Prop}
\label{L1-est}
Let $\varphi_m(Z_m)$ be a nonnegative symmetric function in ~$L^1_\beta (\D^m)$.
For $J \geq m$, we have for any time $t \geq 1$
\begin{equation}
\label{L1-cont}
\int_\D dz\,  |Q_{1,J}^{b,0}| (t) \,  \sum _{\sigma\in {\mathfrak S}^m_{J} } M_\beta ^{\otimes J } 
\varphi_{m,\sigma}   \leq {(C\alpha t) ^{J -1}  \over m! \eps^{m-1} }  
\| \varphi_m \|_{L^1_{\beta}(\D^m)} \, .
\end{equation} 
\end{Prop} 
Thus this completes the derivation of Proposition \ref{L2-est}.
\end{proof}

\medskip

The idea of the proof of Proposition~\ref{L1-est} is to proceed by iteration: Lemma~\ref{lem: elementary-step}
in Paragraph~\ref{stabstructure} shows that the structure is preserved through an integrated in time transport-collision operator, the proof of Proposition \ref{L1-est} 
is then completed in Paragraph~\ref{subsec: Iterated L1 continuity estimates}.

\subsubsection{Evolution of the structure {\rm(\ref{fNs-exp})} under the BBGKY dynamics}\label{stabstructure}

In order to prove Proposition \ref{L1-est}, we first state and prove a key lemma on the collision kernel which will be used recursively in Section \ref{subsec: Iterated L1 continuity estimates}
to prove Proposition~\ref{L1-est}.  In order to decouple the time integrals, we introduce an exponential weight (which will play essentially  the same role as the Laplace transform).

\begin{lem}
\label{lem: elementary-step}
Fix~$t>0$ and~$1 \leq m \leq s +1 \leq J$, and let $\varphi_m$ be a nonnegative symmetric function in $L^1_\beta(\D^m)$.
Then there are two symmetric functions $\Phi^{(m)}_m$ and $\Phi^{(m)}_{m-1}$ defined on $\D^m$ and 
$\D^{m-1}$ such  that  with notation~{\rm(\ref{notationsigma})}
$$\begin{aligned}
\int_0^{+\infty}     d\tau  \,   & e^{-{J \tau\over t} } |C^{b, \pm}_{s,s+1}| 
\widehat {\mathbf S}_{s+1}^0 (\tau)  \, \Big( M_\beta ^{\otimes (s+1)} 
\sum_{\sigma\in {\mathfrak S}^m_{s+1} } \varphi_{m, \sigma} \Big)   \\
& \qquad 
\leq M_{\beta} ^{\otimes s} (V_{s})
\Big( \sum_{\sigma\in {\mathfrak S}^m_{s} }  \Phi^{(m)}_{m,\sigma}
+ \sum_{\sigma\in {\mathfrak S}^{m-1} _{s} }  \Phi^{(m)}_{m-1,\sigma}\Big)  \, .
\end{aligned}
$$
Furthermore, they satisfy 
\begin{align}
\label{eq: Phim}
\| \Phi^{(m)}_m \| _{L^1_{\beta}(\D^m)} &  \leq C  t  \| \varphi_m \| _{L^1_{\beta}(\D^m)} \\
\label{eq: Phim-1}
\| \Phi^{(m)}_{m-1} \| _{L^1_{\beta}(\D^{m-1})}  & \leq {C \over \eps (m - 1) }  \| \varphi_m \| _{L^1_{\beta}(\D^m)} 
\end{align}
and $ \Phi_{s+1}^{(s+1)} =\Phi_0^{(1)} = 0$. 
\end{lem}

\begin{proof}
To simplify the notation, we   drop the superscript $(m)$ throughout the proof.

Let~$\sigma := (i_1,\dots ,i_m)$ be a collection of ordered indices in~$\{1,\dots, s+1\}$. We   first analyze the term involving  
$\varphi_{m,\sigma}$
and then conclude by summing over all possible~$\sigma$'s. 

\medskip

In the following, we shall use the notation~$Z_{s}^{<i>} $ for the configuration in~$\D^{s-1}$ defined by
$$
Z_{s}^{<i>} := (z_1,\dots,z_{i-1},z_{i+1},\dots,z_s) \, .
$$

{

When applying  the collision operator $|C^{b, \pm}_{s,s+1}|$ 
to~$\widehat  {\mathbf S}_{s+1}^0 (\tau)M_{\beta}^{\otimes (s+1)} \varphi_{m, \sigma}$, four different situations occur depending on whether the colliding particles $s+1$ and $i$  belong to $\sigma$ or not. 
Indeed recall that the collision operator  consists mainly in integrating one of the variables, namely~$x_{s+1}$, on a hypersurface~$|x_i-x_{s+1} |= \eps$ for some~$1 \leq i \leq s$.
Thus the collision  may add some dependency in the arguments of $g_{m,\sigma}$.

\medskip

\begin{itemize}
\item 
If~$z_{s+1}$ does not belong to $\sigma$, i.e. the variables of~$\varphi_{m,\sigma}$:
\begin{itemize}
\item 
either~$z_i$ does not belong to $\sigma$ and in that case essentially nothing happens as the collision does not affect the variables in $\sigma$ and the transport operator is an isometry in~$L^1$.
\item or~$z_i$ does belong to $\sigma$ and in that case~$v_i$ is modified by the scattering operator but that will be shown to be harmless thanks to the energy conservation and a change of variables by the scattering operator.
\end{itemize}
\item
If~$z_{s+1}$ does   belong to $\sigma$:
\begin{itemize}
\item either~$z_i$ does not belong  to $\sigma$  then this is quite similar to the second case above,
\item or~$z_i$ belongs to $\sigma$ then by integration on the hypersurface a variable is lost (and that case alone accounts for the term~$ \Phi^{(m)}_{m-1}$ in the lemma).
\end{itemize}
\end{itemize}

\medskip

We turn now to a detailed analysis of these cases.

}

\medskip

\noindent 
{\bf Case 1.}  $s+1 \notin \sigma$:

This case  corresponds to $\sigma \in {\mathfrak S}^m_{s}$ ($m \leq s$) and will contribute partly  to the function $\Phi_m$.
Recall that  $\varphi_{m,\sigma}$ depends only on the coordinates $Z_\sigma$ indexed by $\sigma$.

\noindent
$\bullet$
Define the contribution~$\Phi_{\sigma} ^{1,\pm} $ corresponding to collisions between two particles of the background~:
 $$
 \begin{aligned}
\Phi_{\sigma} ^{1,\pm} (Z_s)& := \int_0^{+\infty} d\tau \,  e^{-{J \tau\over t} }  
\widehat  {\mathbf S}_{s}^0 (\tau) 
\Big( \sumetage{i=1}{i \notin \sigma}^s 
M_{\beta} ^{\otimes (s-1)}  \, \varphi_{m,\sigma} \Big) (V_{s}^{<i>}, X_\sigma )\\
  & \quad \times 
  \int_{{\mathbb S} \times \R^2} M_{\beta} ^{\otimes 2}(v_i^{\pm,i,s+1}, v_{s+1}^{\pm,i,s+1})  {\big((v_i-v_{s+1}) \cdot \nu \big)_+\over 1+|v_i- v_{s+1}|}  \, d\nu dv_{s+1}  \, .
 \end{aligned} 
 $$
 Notice that by energy conservation
\begin{equation}
\label{eq: conservation energy}
M_{\beta} ^{\otimes 2}(v_i^{\pm,i,s+1}, v_{s+1}^{\pm,i,s+1})  = M_{\beta} ^{\otimes 2}(v_i , v_{s+1} ) \,.
\end{equation}
As the collision kernel is bounded, we deduce that
$$
 \Phi_{\sigma} ^{1,+}  (Z_s) +  \Phi_{\sigma} ^{1,-}  (Z_s)
 \leq C M_{\beta} ^{\otimes s} (V_{s} ) \Phi_m^1 (Z_\sigma) \, ,
$$
where $\Phi_m^1$ is the first contribution to $\Phi_m$
$$
\Phi_m^1 (Z_m) := 2
(s-m) \int_0^{+\infty} d\tau e^{-{J \tau\over t} } \;  \widehat  {\mathbf S}_{m}^0 (\tau)  \varphi_m(Z_m)   \, .
$$
Let us compute the~$L^1_\beta$ norm of~$\Phi_m^1 $.  
Note that  $\widehat  {\mathbf S}_{m}^0$ assigns the value 0 if a configuration has a recollision in the time interval $[0,\tau]$, so
\begin{equation}
\label{Sm0leqSm}   
\widehat {\mathbf S}_{m}^0 (\tau) \leq  {\mathbf S}_{m} (\tau) .
\end{equation}
Since~$\varphi_m \geq 0$ and~${\mathbf S}_m$ assigns the value 0 to configurations which initially overlap, we find  for $\tau \geq 0$ 
$$
\begin{aligned}
\int M_{\beta} ^{\otimes m} (V_m ) \widehat  {\mathbf S}_{m}^0 (\tau) \varphi_m(Z_m) dZ_m 
&\leq 
\int M_{\beta} ^{\otimes m} (V_m )  {\mathbf S}_m (\tau) \varphi_m(Z_m) dZ_m \\
& \leq
\int M_{\beta} ^{\otimes m} (V_m ) \varphi_m(Z_m) dZ_m \, ,
\end{aligned}
$$
where we used that the transport preserves the Lebesgue measure.
Finally,  we deduce that 
\begin{align}
 \| \Phi_m^1 \| _{L^1_{\beta}(\D^m)} 
& =  2
(s-m)
\int_0^{+\infty } d\tau e^{-{J \tau\over t} } \,  
\int M_{\beta} ^{\otimes m} (V_m ) \widehat  {\mathbf S}_{m}^0 (\tau) 
\varphi_m(Z_m) dZ_m  \nonumber\\
&  \leq C   {(s-m) \over J} t \| \varphi_m \| _{L^1_{\beta}(\D^m)}
 \leq C    t \| \varphi_m \| _{L^1_{\beta}(\D^m)} \, ,
\label{eq: Phi_m^1}
\end{align}
where we used that $s \leq J$.

\bigskip
\noindent
$\bullet$
It remains to understand what happens when the collision involves one of the  particles in~$\sigma$, i.e. $i\in (i_1, \dots, i_m)$. From the energy conservation \eqref{eq: conservation energy} and 
the fact that the collision kernel is bounded, we have 
\begin{equation}\begin{aligned}
\label{eq: Phi2}
M_{\beta} ^{\otimes s} (V_{s})  
 &\sum_{\ell =1}^m    \int_0^{+\infty} d\tau e^{-{J \tau\over t} } \,  \int_{{\mathbb S} \times \R^2}  
 d\nu dv_{s+1}   \, M_{\beta} (v_{s+1}) \\
  &  \times
\Big(  \widehat  {\mathbf S}_{m}^0 (\tau)  \varphi_m\Big) (Z_\sigma^{<i_\ell>},x_{i_\ell}  , v_{i_\ell}^{\pm, i_\ell,s+1} )
 { \big((v_{i_\ell} -v_{s+1}) \cdot \nu \big)_+   \over   1+ |v_{i_\ell} - v_{s+1}|}   
\leq M_{\beta} ^{\otimes s} (V_{s})  \Phi_m^{2,\pm} (Z_\sigma) \,,
\end{aligned}
\end{equation}
where 
$$
\Phi_m^{2,\pm} (Z_m) 
:=  \int_0^{+\infty} d\tau e^{-{J \tau\over t} } \,  \tilde \Phi^{2,\pm}_m   (\tau, Z_m) \,,
$$
with
$$ 
\tilde \Phi^{2,\pm}_m   (\tau, Z_m) 
:= \sum_{\ell=1}^m \int_{{\mathbb S} \times \R^2}  dv_{s+1} d\nu M_\beta (v_{s+1} ) \,
\Big( \widehat  {\mathbf S}_{m}^0 (\tau)
\varphi_m \Big) (Z_m^{<\ell>}, x_{\ell}, v_{\ell}^{\pm, \ell, s+1})  \,.
$$
The function $\tilde \Phi^{2,\pm}_m$ is symmetric with respect to the coordinates $Z_m$.
{ Using again the conservation of energy,
we have
 $$\begin{aligned}
\int M_{\beta} ^{\otimes m} (V_m ) \tilde \Phi^{2,\pm}_m  (\tau, Z_m)dZ_m   &= \sum_{\ell=1}^m \int dZ_m M_{\beta} ^{\otimes m} (V_m )
\int_{{\mathbb S} \times \R^2}  dv_{s+1} d\nu M_\beta (v_{s+1} ) \\
& \qquad\qquad\qquad\qquad\qquad\qquad \Big( \widehat  {\mathbf S}_{m}^0 (\tau)
\varphi_m\Big) (Z_m^{<\ell>}, x_{\ell}, v_{\ell}^{\pm, \ell, s+1})\\
& =  \sum_{\ell=1}^m\int  \int_{{\mathbb S} \times \R^2}dZ_mdv_{s+1} d\nu M_{\beta} ^{\otimes (m-1)} (V_m^{< \ell >} )
M_\beta^{\otimes 2} (v_{\ell}^{\pm, \ell, s+1},v_{s+1}^{\pm, \ell, s+1}) \\
& \qquad\qquad\qquad\qquad\qquad\qquad \Big( \widehat {\mathbf S}_{m}^0 (\tau)
\varphi_m\Big)  (Z_m^{<\ell>}, x_{\ell}, v_{\ell}^{\pm, \ell, s+1})\,.
\end{aligned}
$$
Since the change of variables
\begin{equation}
\label{eq: isometrie}
( \nu,v_{\ell},v_{ s+1}) \mapsto ( \nu,v_{\ell}^{\pm, \ell,s+1},v_{ s+1}^{\pm, \ell,s+1})
\end{equation}
is an isometry and using~(\ref{Sm0leqSm}), we deduce that for any $\tau \ge 0$,  }
\begin{equation}
\label{eq: integrale isometrie}
\int M_{\beta} ^{\otimes m} (V_m ) \tilde \Phi^{2,\pm}_m  (\tau, Z_m)dZ_m 
 \leq C   m \int M_{\beta} ^{\otimes m} (V_m ) \,  \varphi_m(Z_{m}) dZ_m\,.
\end{equation}
Then, integrating with respect to time and using that $m \leq J$, we get
\begin{equation}
\begin{aligned}
\label{eq: Phi_m^2}
\| \Phi_m^{2,\pm} \| _{L^1_{\beta}(\D^m)}
&=
\int_0^{+\infty} d\tau e^{-{J \tau\over t} }
\int  M_{\beta} ^{\otimes m} (V_m) \tilde \Phi_m^{2,\pm} (\tau, Z_m) dZ_m   \\
& \leq C  {mt\over J} \| \varphi_m \| _{L^1_{\beta}(\D^m)}  \leq C  t \| \varphi_m \| _{L^1_{\beta}(\D^m)}  \, .
\end{aligned}
\end{equation}
From \eqref{eq: Phi2}, this gives a  second contribution to $\Phi_m$ for any $\sigma \in {\mathfrak S}^m_{s}$. 

\bigskip

\noindent  
{\bf Case 2.} $s+1 \in \sigma$ :
 
 As previously, we have to distinguish if the collision with~$s+1$ involves a particle $i \notin \sigma$ or $i \in \sigma $. The first case will lead to a third contribution to~$\Phi_m$ and the second case to the term $\Phi_{m-1}$.

\medskip

\noindent
$\bullet$
We define the contribution of the collisions with particles outside $\sigma$ as
\begin{equation}
\begin{aligned}
\label{eq: Psi_m-1^1}
\Psi_{\sigma}^{1,\pm} (Z_s) &:= 
  \sumetage{i=1}{i \notin \sigma }^s  M_{\beta}^{\otimes (s-1)} (V_{s}^{<i>}) \,  
\int_0^{+\infty} e^{-{J \tau\over t} } d\tau  \,  \int_{{\mathbb S} \times \R^2} M_{\beta}  ^{\otimes 2}(v_i^{\pm,i,s+1}, v_{s+1}^{\pm,i,s+1 })  \\
  & \qquad  \times
 \widehat  {\mathbf S}_{m}^0 (\tau) \varphi_m (Z_{\sigma}^{<s+1>}  , x_i \pm \e \nu , v_{s+1}^{\pm, i,s+1}){\big((v_i-v_{s+1}) \cdot \nu \big)_+\over  1+ |v_i - v_{s+1}|}  \, d\nu dv_{s+1}   \, .  
\end{aligned}
\end{equation}
As the collision kernel is bounded and using the energy conservation \eqref{eq: conservation energy}, 
we get  
\begin{align*}
\Psi_{\sigma}^{1,\pm} (Z_s) 
\leq  M_{\beta}^{\otimes s} (V_{s}) \,
\sumetage{i=1}{i \notin \sigma }^s  \psi_m^\pm   (Z_{\sigma}^{<s+1>}, z_i) \, ,
\end{align*}
with 
$$
\psi_m^\pm   (Z_{m-1}, z_i) 
:= \int_0^{+\infty} e^{-{J \tau\over t} } d\tau  \, \int_{{\mathbb S} \times \R^2}
 \widehat  {\mathbf S}_{m}^0 (\tau) \varphi_m (Z_{m-1}, x_i\pm \eps \nu  , v_i^{\pm, i, s+1}) M_\beta (v_{s+1} ) dv_{s+1} d\nu \,.
$$

We follow now the same arguments as in \eqref{eq: integrale isometrie} to  
compute the~$L^1_\beta$ norm of~$\psi_m^\pm$.
Using first the space translation invariance,   then the isometry \eqref{eq: isometrie} and finally~(\ref{Sm0leqSm})  and the fact that the transport preserves the Lebesgue measure, we get
\begin{align*}
& \int dZ_m \, M_{\beta} ^{\otimes m} (V_m )
\int_{{\mathbb S} \times \R^2}
\Big(  \widehat  {\mathbf S}_{m}^0 (\tau) \varphi_m \Big)(Z_{m-1}, x_m \pm \eps \nu  , v_{s+1}^{\pm, m, s+1}) M_\beta (v_{s+1} ) dv_{s+1} d\nu\\
& \qquad = \int dZ_m \, M_{\beta} ^{\otimes m} (V_m )
\int_{{\mathbb S} \times \R^2}
\Big(  \widehat  {\mathbf S}_{m}^0 (\tau) \varphi_m \Big)(Z_{m-1}, x_m  , v_{s+1} ^{\pm, m, s+1}) M_\beta (v_{s+1} ) dv_{s+1} d\nu\\
&  \qquad  = \int dZ_m \, M_{\beta} ^{\otimes m} (V_m )
\int_{{\mathbb S} \times \R^2}
\Big(  \widehat  {\mathbf S}_{m}^0 (\tau) \varphi_m \Big)(Z_{m-1}, x_m  , v_{s+1}) M_\beta (v_{s+1} ) dv_{s+1} d\nu\\  
&  \qquad  \leq   \int dZ_m \, M_{\beta}^{\otimes m} (V_m )\int \varphi_m (Z_{m-1}, x_m  , v_{s+1}) M_\beta (v_{s+1} ) dv_{s+1} d\nu  \leq C  \| \varphi_m \| _{L^1_{\beta}(\D^m)} \, .
\end{align*}
Finally the time integral leads to 
\begin{align*}
 \| \psi_m^\pm   \| _{L^1_{\beta}(\D^m)}  \leq C  \frac{t}{J} \| \varphi_m \| _{L^1_{\beta}(\D^m)} \, .
\end{align*}

Note that $\psi_m^\pm   (Z_{m-1}, z_i) $ is only  symmetric over the variables $Z_{m-1}$ and not as a function on $\D^m$. 
However the function 
$$
Z_s \to \sum_{\sigma' \in {\mathfrak S}_{s}^{m-1}} \sum_{i \notin\sigma' }   \psi_m^\pm   (Z_{\sigma'  }, z_i) 
$$
is symmetric. Thus 
one can check that 
$$
\sum_{\sigma'  \in {\mathfrak S}_{s}^{m-1}} \sum_{i \notin \sigma' }   \psi_m^\pm   (Z_{\sigma'  }, z_i) 
\leq m \sum_{\sigma \in {\mathfrak S}_{s}^{m}} \widehat \psi_m^\pm  (Z_{\sigma}) \,,
$$
where $\widehat \psi_m^\pm$ is the symmetric version of $\psi_m^\pm$~:
$$ \widehat \psi_m^\pm (Z_m) =\frac1m \sum_{i=1}^m \psi_m^\pm (Z_m^{<i>}, z_i)\,.$$

Finally, the function $\Phi_m^{3,\pm} (Z_m) := m \widehat \psi_m^\pm  (Z_m)$ provides an upper bound for 
\eqref{eq: Psi_m-1^1}
$$
\sumetage{\sigma\in {\mathfrak S}^m_{s+1}}{s+1 \in \sigma }
\Psi_{\sigma}^{1,\pm} (Z_s) 
\leq \sum_{\sigma\in {\mathfrak S}^m_{s}}
\Phi_m^{3,\pm} (Z_\sigma)
$$
with 
\begin{equation}
\label{eq: Phi_m^3}
 \| \Phi_m^{3,\pm}   \| _{L^1_{\beta}(\D^m)}  \leq C  \frac{m}{J} t \| \varphi_m \| _{L^1_{\beta}(\D^m)}
  \leq C   t \| \varphi_m \| _{L^1_{\beta}(\D^m)} \, .
\end{equation}
This defines the third contribution to $\Phi_m := \Phi_m^1 + \Phi_m^{2,\pm} + \Phi_m^{3,\pm}$.
Thus the upper bound~\eqref{eq: Phim} on the $L^1_\beta$-norm of $\Phi_m$ follows from the estimates~\eqref{eq: Phi_m^1}, \eqref{eq: Phi_m^2}  and \eqref{eq: Phi_m^3}.

\bigskip

\noindent
$\bullet$
It remains to understand what happens when the collision involves two particles in $\sigma$, 
i.e. when~$i, s+1\in \sigma$.
This is a more delicate situation, as we need to take a trace on the function~$\varphi_m$. The transport operator will be the key to using nevertheless an~$L^1$ bound on~$\varphi_m$. We set
\begin{align}
\Psi_{\sigma}^{2,\pm} (Z_\sigma^{<s+1>})& := 
\sum_{i \in I_{m-1}}  M_{\beta}^{\otimes (s-1)} (V_{s}^{<i>}) \, \int_0^{+\infty} 
d\tau e^{-{J \tau\over t} } \,
\int_{{\mathbb S} \times \R^2}   M_{\beta}^{\otimes 2}(v_{i}^{\pm,i,s+1}, v_{s+1}^{\pm,i,s+1 })
\nonumber \\
& \qquad \times
   \Big(  \widehat  {\mathbf S}_{m}^0 (\tau) \varphi_m \Big) (Z_\sigma^{<i,s+1>}  ,x_{i}, v_{i}^{\pm, i ,s+1}, x_{i}  \pm \e \nu , v_{s+1}^{\pm, i,s+1}){\big((v_{i} -v_{s+1}) \cdot \nu \big)_+\over 1+ |v_{i} - v_{s+1}|}  \, d\nu dv_{s+1}  \nonumber \\
& \leq  M_{\beta}^{\otimes s} (V_{s}) \Phi_{m-1} (Z_\sigma^{<s+1>})\,,    
\label{Psi_m-1,2,pm}
\end{align}
where
\begin{equation*}
\Phi_{m-1} (Z_{m-1}) := \sum_{i =1}^{m-1}   \psi^{i,\pm}_{m-1} (Z_{m-1}) \,,
\end{equation*}
with 
\begin{align*}
\psi^{i,\pm}_{m-1} (Z_{m-1})  := 
\int_0^{+\infty} 
d\tau  &
\int_{{\mathbb S} \times \R^2}   d\nu dv_{m}   \, M_{\beta}(v_{m}) 
{\big((v_{i} -v_{m}) \cdot \nu \big)_+} 
\\
& \qquad \times
 \Big(    \widehat  {\mathbf S}_{m}^0 (\tau) \varphi_m\Big)  (Z_{m-1}^{<i>}   ,x_{i}, v_{i}^{\pm, i ,m}, x_{i}  \pm \e \nu , v_{m}^{\pm, i,m})\, .
\end{align*}
The function $\Phi_{m-1} $ is symmetric but not the functions $\psi^{i,\pm}_{m-1}$.
The inequality \eqref{Psi_m-1,2,pm} comes from the fact that the denominator $(1+ |v_{i} - v_{m}|)$ has been removed and the exponential factor $e^{-{J \tau\over t} }$ bounded by 1. 
As we shall see, the time integral is still converging thanks to the cut-off on the transport operator $\widehat  {\mathbf S}_{m}^0$.

\medskip

We compute now the $L^1_\beta$-norm of $\Phi_{m-1}$.
Since the scattering transform
$$ (v_i, v_{m}, \nu) \mapsto (v'_i, v'_{m}, \nu)$$
is bijective and has unit Jacobian, it is  enough to study the simple case
\begin{equation}
\begin{aligned}
\label{eq: psi l,n,+}
\psi^{i,+}_{m-1} (Z_{m-1})  = 
\int_0^{+\infty}  d\tau \, &
\int_{{\mathbb S} \times \R^2}   d\nu dv_{m}   \, M_{\beta}(v_{m}) 
{\big((v_{i} -v_{m}) \cdot \nu \big)_+}\\
& \qquad \times \Big(  \widehat  {\mathbf S}_{m}^0 (\tau) \varphi_m \Big) (Z_{m-1}^{<i>}   ,x_{i}, v_{i}, x_{i}  \pm \e \nu , v_{m})\, ,
\end{aligned}
    \end{equation}
  where we have used again the conservation of energy.
Define the maximal subset $\mathcal{S}^{i,m}$ of the space~$\mathbb{D}^{m-1}  \times {\mathbb S} \times 
\R^2 \times \R$ such that 
for any initial datum~$(Z_{m-1},  x_{i}  + \e \nu, v_{m} )$ in~$\mathcal{S}^{i,m}$ no recollision takes place in the time interval $[0,\tau]$. On the domain~$\mathcal{S}^{i,m}$, the map 
\begin{eqnarray}
\label{changevariables}
\Gamma^{i,m}: \qquad \qquad \mathcal{S}^{i,m} \qquad & \mapsto & \qquad \mathbb{D}^m\\
(Z_{m-1},  \nu, v_{m}, \tau ) & \mapsto &  {\bf \Psi}(-\tau)
(Z_{m-1},  x_{i}  + \e \nu, v_{m} ) \nonumber
\end{eqnarray}
is injective. This would not be true for the transport map without the restriction to $\mathcal{S}^{i,m}$ due to the periodic structure of $\mathbb{D}^m$. However, for any $Z_m$ in the range $\mathcal{R}^{i,m}$ of the map $\Gamma^{i,m}$, the time $\tau$ is uniquely determined as  the first collision time in the flow starting from $Z_m$. 
This collision will take place between $i$ and $m$ because the possibility of any other collision has been excluded.  All the other parameters can be determined from ${\bf \Psi}(\tau) (Z_m)$.

Given $j \in \{ 1, \dots, m\} \setminus \{i\}$, we denote by $\omega^{j,m}$ the permutation which swaps the coordinates~$z_j,z_m$ of $Z_m$. Then~$\Gamma^{i,j}= \omega^{j,m} \circ \Gamma^{i,m}$.  
These maps are of the same nature, however the ranges $\mathcal{R}^{i,j}$, $\mathcal{R}^{i',j'}$ are disjoint as soon as~$\{i,j\} \neq \{i',j'\}$. 
Indeed for any configuration~$Z_m$ in~$\bigcup_{j \not = i} \mathcal{R}^{i,j}$, one can recover the associated map, as the first collision in the flow starting from~$Z_m$ will take place between $i$ and $j$. Once again this is possible because we considered the truncated transport dynamics associated with the flow 
$\widehat  {\mathbf S}^0$.
The last important feature is that the change of variables $\Gamma^{i,m}$ maps the  measure~$\big((v_{i} -v_{m}) \cdot \nu \big)_+ \, \eps d\nu dv_{m}d\tau dZ_{m-1}$ to~$dZ_m$. 
Thus we can rewrite \eqref{eq: psi l,n,+} as
\begin{align*}
 \| \Phi_{m-1}  \|_{L^1_\beta (\mathbb{D}^{m-1})}  
 &= \sum_{i =1}^{m-1}  \| \psi^{i,\pm}_{m-1}  \|_{L^1_\beta (\mathbb{D}^{m-1})} \\
&   = 
 \sum_{i =1}^{m-1} \int_{\cS^{i,m}} d Z_{m-1}
 d \tau    d\nu dv_{m}   \,  M_{\beta}^{\otimes (m)} (V_m)
{\big((v_{i} -v_{m}) \cdot \nu \big)_+} \\
&   \qquad \qquad 
\times  \varphi_m \Big( \Gamma^{i,m} ( Z_{m-1}^{<i>}   ,x_{i}, v_{i}, \nu , v_{m}, \tau) \Big)\\
&    
=  \frac{1}{\eps}  \sum_{i =1}^{m-1}\int_{\cR^{i,m}} d Z_{m} M_{\beta}^{\otimes m} (V_m) \varphi_m \big(  Z_m \big) \\
&    
=  \frac{1}{\eps}  \sum_{i =1}^{m-1} \frac{1}{m-1} \sum_{j \not = i} \int_{\cR^{i,j}} d Z_{m} M_{\beta}^{\otimes m} (V_m) \varphi_m \big(  Z_m \big)\\
&    \leq \frac{1}{\eps} \frac{2}{m-1}   \| \varphi_m \|_{L^1_\beta (\mathbb{D}^m)} \, ,
\end{align*}
where we used that the sets $( \cR^{i,j} )_{i\neq j}$ cover at most twice   $\mathbb{D}^m$.

\smallskip

Finally we notice that~$\Phi_{m}^{m} = 0$ because there is no loss in the number of particles only if one of the particles~$z_i$ and~$z_{m}$  corresponding to the collision integral is not part of the variables of~$\Phi_{m}$, which is impossible since it is defined on~$\D^{m}$. Similarly~$\Phi_{0}^{1} = 0$
because there is a loss in the number of variables only if the two variables of the collision kernel are part of  the variables of the function considered, which is impossible if the function only depends on one variable.

\smallskip

This completes the bound \eqref{eq: Phim-1} and ends the proof of Lemma \ref{lem: elementary-step}. 
\end{proof}

\subsubsection{Iterated $L^1$ continuity estimates}
\label{subsec: Iterated L1 continuity estimates}

To evaluate the norm of $|Q^{b,0} _{1,J}| (t)$ and prove Proposition \ref{L1-est}, we   use recursively  Lemma \ref{lem: elementary-step}.

\begin{proof}[End of the proof of Proposition {\rm\ref{L1-est}}]

The quantity to be controlled  is of the form
\begin{align*}
&\int_{\D} dz\,  |Q_{1,J}^{b,0}| (t)  M_\beta ^{\otimes J} \varphi_{m,\sigma} (z)   \\
&  = 
 \alpha^{J-1} \int_{\D} dz\!  \int_0^t \int_0^{t_{2}}\! \dots\!   \int_0^{t_{J-1}} \! \! \!   dt_{{J}} \dots dt_{2}
\widehat  {\mathbf S}^0_1(t-t_{2}) |C^b_{1,2}|  \widehat  {\mathbf S}^0 _{2}(t_{2}-t_{3})   |C^b_{2,3}|  
\dots  \widehat  {\mathbf S}^0 _{J}(t_{J}) M_\beta ^{\otimes J} \varphi_{m,\sigma} (z) \\
&  = 
 \alpha^{J-1} \int_{\D} dz\!  \int_0^t \int_0^{t_{2}}\! \dots\!   \int_0^{t_{J-1}}  \! \! \! dt_{{J}} \dots dt_{2}
 |C^b_{1,2}|  \widehat  {\mathbf S}^0 _{2}(t_{2}-t_{3})   |C^b_{2,3}|  
\dots  \widehat  {\mathbf S}^0 _{J}(t_{J}) M_\beta ^{\otimes J} \varphi_{m,\sigma}  (z) \, .
\end{align*}
Rewriting the   time integrals in terms of the time increments $\tau_i = t_i - t_{i+1}$ with the constraint~$ \tau_{2} + \dots + \tau_{J} \leq t $, we get
\begin{align*}
&\int_{\D} dz\,  |Q_{1,J}^{b,0}|(t) M_\beta ^{\otimes J} \varphi_{m,\sigma} (z)   \\
& = 
 \alpha^{J-1} \int_{\D} dz\!   \int_0^\infty \int_0^\infty\! \dots \!  \int_0^\infty\! \! \!  d\tau_{{J}} \dots d\tau_{2} \indc_{ \{ \tau_{2} +\dots +\tau_{J } \leq t  \}}  |C^b_{1,2}|  \widehat  {\mathbf S}^0 _{2}(\tau_2)   |C^b_{2,3}|  
\dots  \widehat  {\mathbf S}^0 _{J}(\tau_{J}) M_\beta ^{\otimes J} \varphi_{m,\sigma} (z) \, .
\end{align*}
This constraint can be removed by using the inequality
$$
\indc_{ \{ \tau_{2} +\dots +\tau_{J } \leq t  \}}
 \leq \exp \Big( J \big(1 - {\tau_{2} +\dots +\tau_{J} \over t} \big) \Big)
$$
which allows one to decouple the time integrals and  to deal  with the elementary operators
$$ 
\int_0^{+\infty} e^{-J{\tau_{s+1}  \over t}} |C^b_{s,s+1}| S_{s+1} (\tau_{s+1} )d\tau_{s+1} 
$$
separately. A factor $e^J$ is lost in this decoupling procedure.

\medskip

We proceed now by applying $J-1$ times the estimates of Lemma \ref{lem: elementary-step}.
One iteration transforms a symmetric sum of functions $\varphi_\ell$ depending on $\ell$ variables into 
similar sum of functions $\Phi^{(\ell)}_\ell, \Phi^{(\ell)}_{\ell-1}$ depending on $\ell$ or $\ell-1$ variables with the following exceptions
\begin{itemize}
\item $\Phi^{(\ell)}_\ell = 0$ if $\ell = s+1$,
\item $\Phi^{(\ell)}_{\ell - 1} = 0$ if $\ell = 1$.
\end{itemize}

We recall the bounds \eqref{eq: Phim} and \eqref{eq: Phim-1}
\begin{align*}
\| \Phi^{(\ell)}_\ell \| _{L^1_{\beta}(\D^\ell)} &  \leq C  t  \| \varphi_\ell \| _{L^1_{\beta}(\D^\ell)} \, , \qquad
\| \Phi^{(\ell)}_{\ell-1} \| _{L^1_{\beta}(\D^{\ell-1})}  \leq {C \over \eps (\ell - 1) }  \| \varphi_\ell \| _{L^1_{\beta}(\D^\ell)} \, .
\end{align*}
As the number of variables has to be  dropped exactly by $m-1$, the $J-1$ iterations will lead to a sum of at most $\binom{J-1}{m-1} \leq 2^J$ terms. We therefore end up with
\begin{align*}
\int_{\D} dz\,  |Q_{1,J}^{b,0}| (t) \, 
\Big( \sum _{\sigma\in {\mathfrak S}^m_{J} } M_\beta^{\otimes J} \varphi_{m,\sigma} \Big)(z)    
& \leq 
(C\alpha)^{J-1} \, t^{J-m} \frac{1}{\eps^{m-1} (m-1)!}  
 \| \varphi_m \|_{L^1_{\beta}(\D^m)}   \, ,
\end{align*}
which is the expected estimate (bounding~$t^{J-m}$ by~$t^{J-1}$ and changing the constant~$C$).
\end{proof}

\subsection{Proof of Proposition \ref{RNK0}}
\label{subsec: Proposition RNK0}

%
%

This Proposition is a straightforward consequence of Propositions \ref{lemfNL2sym 0} and  \ref{L2-est}. We have only to sum over all elementary contributions.

\medskip
 
$\bullet$
Fix $k$, $j_i <n_i$ for each $ i\leq k-1$ and $j_k \geq n_k$. 

By relaxing the conditions on the distribution of times to retain only  the constraint on 
the time increments
$$
\begin{aligned}
\tau_2 +\dots +\tau_{J_{k-1}} \leq (k-1) h \leq t \, ,\\
\tau_{J_{k-1} +1} +\dots + \tau_{J_k} \leq h\,,
\end{aligned}
$$
it is enough to consider the upper bound
$$
| Q^0_{1,J_1} |(h ) \dots  |Q^0_{J_{k-1},J_{k}} |(h )  
\leq  | Q^0_{1,J_{k-1}} |(t) \,  | Q^0_{J_{k-1}, J_k} |(h) \, .
$$
From  the uniform $L^2$ estimates (\ref{gi-est}) following from Proposition \ref{lemfNL2sym 0} and Stirling's formula, we deduce that
$$
\| g_N^{m}( t-kh) \|_{L^2_\beta(\D^m)}^2 \leq {CN\exp (C\alpha^2)\over  \binom{N}{m}}  
\leq {C^m \, m ! \exp (C\alpha^2) \over N^{m-1}  }  \, \cdotp
$$
Then, by    Proposition \ref{L2-est}, we conclude that 
$$
\begin{aligned}
\Big(\int \Big( |Q^0_{1,J_1}| (h ) \dots  |Q^0_{J_{k-1},J_{k}} | (h )  \sum _{\sigma\in {\mathfrak S}^m_{J_k} } M_\beta ^{\otimes J_k} {  \indc_{{\mathcal V}_{J_K} }  } \big|g_{N,\sigma}^{m}  ( t- k h)\big| \Big)^2 dz_1 \Big) ^\frac12\\
\leq (C\alpha)^{J_k} \exp (C\alpha^2) t^{J_{k-1}+j_k/2  } h^{j_k/2} \, ,
\end{aligned}
$$
with the notation $g_{N,\sigma}^{m}  (t',Z_{J_k}) = g_N^{m}  (t', Z_\sigma)$.
We then sum over all $m\in \{1,\dots,  J_k\}$  to get
$$\Big(\int \Big( |Q^0_{1,J_1} | (h ) \dots  |Q^0_{J_{k-1},J_{k}}| (h ) \, | f^{(J_K)}_N (t-kh) | {  \indc_{{\mathcal V}_{J_K} }  }\Big)^2
dz_1 \Big)^\frac12 \leq (C\alpha)^{J_k} \exp (C\alpha^2) t^{J_{k-1}+\frac{j_k}2  } h^{\frac{j_k}2  }\,.
$$

\medskip
 
$\bullet$ For $\gamma$ small, the  scaling assumption \eqref{eq: parameters RNK0} implies in particular that $\alpha^2 th \ll 1$ and that $\alpha ^2t^{3/2} h^{1/2}  \ll 1$, recalling that~$t \geq 1$. Thus summing over all $j_k \geq n_k$ leads to 
\begin{equation}
\begin{aligned}
\sum_{j_k \geq n_k}\Big(\int \big( |Q^0_{1,J_1} | (h ) &\dots  |Q^0_{J_{k-1},J_{k}}| (h ) \; | f^{(J_K)}_N (t-kh) | {  {  \indc_{{\mathcal V}_{J_K} }  } }\big)^2 dz_1 \Big)^\frac12 \\
& \, \leq  \exp (C\alpha^2)  (C\alpha)^{J_{k-1}+n_k } t^{J_{k-1} +n_k/2} h^{n_k/2}
\\
& \leq  \exp (C\alpha^2) (C\alpha)^{2 n_k } t^{{3 \over 2} n_k} h^{\frac{1}{2}n_k}, 
\label{eq: sommation RNK0 1}
\end{aligned}
\end{equation}
where we used that $J_{k-1} \leq n_k$ as $j_\ell \leq n_\ell = 2^\ell n_0$. 

Taking the sum over all possible $j_i$ as in \eqref{eq: nk 2k}, we get at most $C^k 2^{k^2}$ such terms.
From the scaling assumption \eqref{eq: parameters RNK0} and the fact that $\alpha \geq 1$, one can choose  $h \leq  {\gamma^2}/{8 C \exp (C\alpha^2)  \alpha^4  T^3}$. This implies that
\begin{equation}
\label{eq: sommation RNK0 2}
\Big(\int_\D dz_1  \big|  R_N^{K, 0} (t,z_1) \big|  ^2  \Big)^\frac12
\leq   e^{C\alpha^2} \sum_{k=1}^K 2^{k^2} \big(C \alpha^4 t^3 h)^{\frac{1}{2}n_k} 
 \leq \gamma  \, ,
\end{equation}
and  Proposition \ref{RNK0} follows.    \qed

\subsection{Super exponential branching for the Boltzmann  pseudo-dynamics}
\label{subsec: super exp boltz}

It remains then to estimate similarly the contribution of the super-exponential branching collision trees in the 
Boltzmann  pseudo-dynamics  
$$ 
\bar R^K (t) :=  \sum_{k=1}^K \; \sum_{j_1=0}^{n_1-1} \! \! \dots \! \! \sum_{j_{k-1}=0}^{n_{k-1}-1}\sum_{j_k \geq n_k} \; 
 \bar Q_{1,J_1} (h ) \dots \bar  Q_{J_{k-1},J_{k}} (h ) \big( f^{(J_K)}  (t-kh){  \indc_{{\mathcal V}_{J_K} }  }\big)  \, . 
$$
We can state a result analogous to Proposition \ref{RNK0} 
\begin{Prop}
\label{RNK0 boltz} 
Given~$T>1$,  $\gamma\ll 1$ and~$C$ a large enough constant (independent of~$\gamma$ and~$T$),  the parameters are tuned as follows 
\begin{equation}
\label{eq: parameters RNK0 boltz}
h \leq \frac{\gamma^2}{C \alpha^4  T^3} \, ,
\qquad n_k =  2^k n_0 \,.
\end{equation}
Then, we have for $t\in [0,T]$
\begin{equation}
\label{R0-est boltz}
\Big\| \bar R ^K (t) \Big\|_{L^2(\D)}  \leq \gamma  \, .
\end{equation}
\end{Prop}
 
\begin{proof}
{  At this stage, the constraint ${\mathcal V}_{J_K}$ is purely cosmetic and it can be removed.} 
We   use the fact that the solution \eqref{solboltz} of the Boltzmann hierarchy is explicit 
\begin{equation*}
f^{(s)} (t,Z_s) = M_\beta^{\otimes s} (V_s)\sum_{i=1}^s  g_\alpha (t,z_i) \, ,
\end{equation*}
where $g_\alpha$ solves the linear Boltzmann equation \eqref{lBoltz}  and is smooth.
In particular, the weighted $L^2$ norm is a Lyapunov functional for the linearized Boltzmann equation, so  
\begin{equation}
\label{eq: lyapunov}
 \forall t \geq 0, \quad \int M_\beta  g_\alpha^2(t,z) dz 
 \leq \int M_\beta  g_{\alpha,0}^2(z) dz\,.
\end{equation}
The collision operators are decomposed into  $\bar  C_{s,s+1}^{b, \pm}$ and  $\bar  C_{s,s+1}^{q, \pm}$
as in \eqref{Cb-def}.
Then, following the same arguments as in the proof of Lemma \ref{lem: elementary-step} (case 1), we get
for any continuous nonnegative function~$\varphi$ in $L^1_\beta (\D)$
 $$  
 \bar  C_{s,s+1}^{b, \pm} M_\beta^{\otimes (s+1)}  \sum_{i=1}^{s+1} \varphi ( z_i) 
 = s M_\beta^{\otimes s}  \sum_{i=1}^{s} \tilde \varphi ( z_i)  
 $$
where 
$$ 
\int M_\beta \tilde \varphi  ( z) dz \leq C   \int M_\beta  \varphi ( z) dz \, .
$$
By iteration and integration with respect to time which leads to a factor
$\displaystyle  {t^{J-1}}/{(J-1)!}$, we deduce that 
$$
\begin{aligned}
& \int  d z_1  |\bar Q^b_{1,J}|(t) 
 \Big( M_\beta^{\otimes J} \sum_{i=1}^{J } \varphi (z_i)\Big) 
\leq (C\alpha t)^{J-1}  \int M_\beta \varphi  (z) dz \,.
 \end{aligned}
$$
The previous estimate can be applied to the explicit form of the Boltzmann hierarchy.
Combining this upper bound with Lanford's estimate for 
$|\bar Q^q_{1,J}|(t) \, |\bar Q^q_{J,J+n} | (h) M_\beta^{\otimes (J+n)}$, 
we get by the Cauchy-Schwarz inequality as in \eqref{nosmeilleursamis} 
\begin{align*}
& \Big\| \bar Q_{1,J_1} (h ) \dots \bar  Q_{J_{k-1},J_{k}} (h ) {   \indc_{{\mathcal V}_{J_K} } } f^{(J_K)}(t - (k-1) h)
\Big\|_{L^2(\D)} \\
& \qquad \qquad 
\leq 
(C\alpha t)^{J_k-1}(C\alpha h)^{j_k/2} 
\left( \int M_\beta g_\alpha ^2 (t-(k-1)h, z) dz\right)^{1/ 2}\\
& \qquad \qquad 
\leq 
(C\alpha t)^{J_{k-1} + j_k -1}(C\alpha h)^{j_k/2}  \| g_{\alpha,0} \|_{L^2_\beta ( \D) } \, ,
\end{align*}
where we used \eqref{eq: lyapunov} in the last inequality.

\medskip

We proceed as in \eqref{eq: sommation RNK0 1}, \eqref{eq: sommation RNK0 2}
and  sum over $j_k \geq n_k$, $j_i<n_i$ for $i \leq k-1$,
$$
\Big\| \bar R^K (t) \Big\|_{{ L^2(\D)}} 
\leq \sum_{k=1}^K 2^{k^2} \big(C \alpha^4 t^3 h)^{\frac{1}{2}n_k} 
\| g_{\alpha,0} \|_{L^2_\beta ( \D) } \leq \gamma \, ,
$$ 
where the last inequality follows from the condition $h \leq  {\gamma^2}/(8 C \alpha^4  T^3)$.
This completes the proof of Proposition \ref{RNK0 boltz}.
\end{proof}


\section{Control of super exponential trees with one recollision}
\label{lanfordL2}

In this section, we   show how to modify the proof of Proposition \ref{RNK0} to take into account a finite number of recollisions (actually one here, but the argument could easily be extended to an arbitrary, finite number), and prove the following estimate for $R_N^{K,1}$.

\begin{Prop}
\label{RNK1} 
Under the Boltzmann-Grad scaling~$N\eps = \alpha \gg1$ and with the previous notation, we have for~$T>1$ and all~$t \in [0,T]$, assuming
$$h \leq \frac{\gamma^2}{C \alpha^4  T^3} $$
that
$$
\Big\|   R_N^{K, 1}(t) \Big\|_{L^2(\D)} 
\leq    \gamma     
{ \eps^{1/ 2}  | \log \eps |^{  6} \over h}  \,  \cdotp
$$

\end{Prop}

Given a function $g_N^{m}$, let us call {  distinguished} the particles which are in the argument of $g_N^{m}$ and the others are the background particles. 
Proposition \ref{L1-est} cannot be applied as a black box: 
indeed, the structure~(\ref{fNs-exp})  is not exactly preserved by the transport operator at the time of recollision if there is scattering between one distinguished particle and one particle of the background.
We have therefore to extend Lemma   \ref{lem: elementary-step}     to incorporate the case of one recollision.
The point is  to modify locally the decomposition (\ref{fNs-exp}) to ensure that the recollision will always involve either two particles of the background or two distinguished particles, in which case it is easy to adapt the   proof of  Proposition \ref{L1-est} .

\subsection{Extension of Lemma   \ref{lem: elementary-step} to the case of one recollision  }

Note that in the pseudo dynamics describing the operator       
$$ |C^{b, \pm}_{s,s+1}| 
\widehat {\mathbf S}_{s+1}^1 (\tau)$$
 there is exactly one collision occurring at the initial time and  the particles evolve in straight lines with the exception of the two recolliding particles.

\begin{lem}
\label{lem: elementary-step 2}
Fix $t>0$, $1 \leq m \leq s +1 $ and let $\varphi_m$ be a nonnegative symmetric function in ~$L^1_\beta (\D^m)$. 
Then there are three symmetric functions $\Phi_m^{(m)}$, $\Phi_{m-1}^{(m)}$ and $\Phi_{m+1}^{(m)}$ defined respectively on~$\D^m$, $\D^{m-1}$ and $\D^{m+1}$ such  that 
$$\begin{aligned}
\int_0^{t}  &   d\tau  \,   e^{-{J \tau\over t} } |C^{b, \pm}_{s,s+1}| 
\widehat {\mathbf S}_{s+1}^1 (\tau)  \, \Big( M_\beta ^{\otimes (s+1)} 
\, {  \indc_{{\mathcal V}_{s+1}} } \,
\sum_{\sigma\in {\mathfrak S}^m_{s+1} } \varphi_{m, \sigma} \Big)   \\
& \qquad \leq M_{\beta} ^{\otimes s} (V_{s})
\Big( \sum_{\sigma\in {\mathfrak S}^m_{s} }  \Phi_{m,\sigma}^{(m)}
+ \sum_{\sigma\in {\mathfrak S}^{m-1} _{s} }  \Phi_{m-1,\sigma}^{(m)}
+ \sum_{\sigma\in {\mathfrak S}^{m+1} _{s} }  \Phi_{m+1,\sigma}^{(m)}\Big)\,  , 
\end{aligned}
$$
where ${\mathcal V}_{s+1}$  was introduced in \eqref{eq: truncated velocities}.
Furthermore, they satisfy 
\begin{align}
\label{eq: Phim recol}
\| \Phi_m^{(m)} \| _{L^1_{\beta}(\D^m)} &  \leq C s^2 t \, \big| \log \eps \big| \| \varphi_m \| _{L^1_{\beta}(\D^m)} \\
\label{eq: Phim-1 recol}
\| \Phi_{m-1}^{(m)} \| _{L^1_{\beta}(\D^{m-1})}  & \leq {C  \over \eps (m - 1) }  \| \varphi_m \| _{L^1_{\beta}(\D^m)}\\
\label{eq: Phim+1 recol}
\| \Phi_{m+1}^{(m)} \| _{L^1_{\beta}(\D^{m+1})}  & \leq {C  s^3 t} \eps \, \big| \log \eps \big| \| \varphi_m \| _{L^1_{\beta}(\D^m)}  
\end{align}
with $\Phi_0^{(1)} = \Phi_{s+1}^{(s)}  = \Phi_{s+1}^{(s+1)} = \Phi_{s+2}^{(s+1)}  = 0$. 
\end{lem}
Unlike Lemma \ref{lem: elementary-step} which is iterated, the previous lemma will be used only once, 
thus there is no need to establish sharp bounds.

\begin{proof}
To simplify notation we drop the superscript~$(m)$ in the proof. We follow the main steps of the proof of Lemma~\ref{lem: elementary-step}.

\medskip

{\bf Step 1.} {\it Localization of the transport operators.}

 Let us first fix $(i,j)$ the pair of recolliding particles and denote by 
$\widehat {\mathbf S}_{s+1} ^{1, (i,j)}(\tau)$ the corresponding transport operator. 
For a given $\sigma \in {\mathfrak S}^m_{s+1} $, we have to distinguish two cases.

\medskip

{\it Case 1.} {\it $(i,j)$ belongs to $\sigma$ or $\sigma^c$}.

If $i,j \notin \sigma$, we have
\begin{equation}
\label{eq: localisation 1}
\widehat {\bf S}_{s+1} ^{1, (i,j)}(\tau)  \, \indc_{{\mathcal V}_{s+1}} \, M_\beta^{\otimes (s+1) } \varphi_m (Z_\sigma) \leq  M_\beta^{\otimes (s+1)} \widehat {\bf S}^0_m (\tau) \varphi_m (Z_\sigma )\, ,
\end{equation}
where the transport $\widehat {\bf S}^0_m$ acts only on the $m$ particles in $\sigma$.
The distribution is therefore unchanged.

If $i,j \in \sigma$, we have
\begin{equation}
\label{eq: localisation 2}
\widehat {\bf S}_{s+1} ^{1, (i,j)} (\tau) M_\beta^{\otimes (s+1) } \, \indc_{{\mathcal V}_{s+1}} \, \varphi_m (Z_\sigma) \leq  M_\beta^{\otimes (s+1)} \widehat {\bf S}^1_m (\tau) \varphi_m (Z_\sigma )\,.
\end{equation}
In this case then  the recollision involves two distinguished particles, so the distribution is modified by the scattering. However since the scattering preserves the measure $ dvdv_1 d\nu$, both the $L^\infty$ and $L^1$ norms will be unchanged. Note that in both cases the velocity cut-off has been neglected.

Compared to the previous section, there is however one issue: if there is no recollision, then a point of the phase space cannot be in the image~${\bf S}^0_{m} (\tau) (\d {\mathcal D}_\eps^{m,\pm}) (i,j)$ for two different pairs~$(i,j)$, and that fact was  the key argument to get the suitable $L^1$ estimate for $\Phi_{m}^{( m-1)}$ previously
(without loosing a factor $m^2$). In the current situation as there is exactly one recollision,  for any point in $\cD_\eps^s$ there exists a unique parametrization by one point of the boundary $\cD_\eps^s$ and one time. It is obtained by  using the backward flow, going through the first collision (which is the recollision) and reaching  another point of the boundary with a different (longer) time.

So in the end in both cases the analysis is exactly like the one performed in the previous section.

%
%
%
 
\medskip

{\it Case 2.} {\it $i$ belongs to $\sigma^c$ and $j$ to $\sigma$ (or the symmetric situation)}.

Note first that this situation can only occur when $m< s+1$.

The recollision in the transport $\widehat {\bf S}_{s+1} ^{1, (i,j)}(\tau)$ induces a correlation between the  particles $z_i, z_j$ so   the structure  with $m$ distinguished particles and $s+1-m$ particles at equilibrium  is not stable anymore. 
The idea is then to add particle $i$ to the set of distinguished particles. 
But in order to keep some of the structure, we then need to gain  additional smallness (since~$\| \varphi_m\|_{L^1_\beta}$ is expected to decay roughly as $\eps^{m-1}$, adding a variable requires gaining a power of~$\eps$).

 For any $\tau \leq t$, a configuration $Z_{s+1}$ obtained by backward transport $\widehat {\bf S}_{s+1} ^{1, (i,j)}(\tau)$ will necessarily belong to the set 
\begin{equation}
\label{eq: set P}
\cP_{(i,j)} := \big\{ Z_{s+1} \in \D^{s+1} \; \Big| \; \;  
\exists \, u \leq t, \; \;    d( x_i + u v_i  , x_j + u v_j)\leq \eps  \big\} \, ,
\end{equation}
where $d$ denotes the distance on the torus.
Note that this set does not depend on $\tau \leq t$.
We then define  a new function with $m+1$ variables which will encompass the constraint on the recollision
\begin{equation}
\label{eq: function  P}
\psi_{m +1,\sigma^{{ <j>}}}^{i,j} ( Z_{\sigma}, z_i ):= \varphi_m ( Z_\sigma ) \indc_{\cP_{(i,j)}} ( z_j, z_i )
\, \indc_{{\mathcal V}_{m+1}} ( Z_{\sigma}, z_i )\, ,
\end{equation}
with a velocity cut-off acting on the $m+1$ variables.

We are going to check that 
\begin{equation}
\label{eq: bound psi m +1 -ij}
\|\psi_{m +1,\sigma^{{ <j>}}}^{i,j}  \|_{L^1_{\beta}(\D^{m+1})} 
\leq Ct \eps \, \big| \log \eps \big|  \|  \varphi_m \|_{L^1_{\beta}(\D^m)} \, .
\end{equation}
Thus the extra factor $\eps \big| \log \eps \big|$ will compensate partly the factor $1/\eps$ corresponding to the shift from $m$ to $m+1$.
To prove \eqref{eq: bound psi m +1 -ij}, we first freeze the coordinates $Z_\sigma$. 
Integrating first over~$z_i$, we recover 
the factor $Ct\eps \, \big| \log \eps \big|$ from the constraint $\cP_{(i,j)}$ (as all energies are bounded by~$C_0 |\log \eps|$),
and then \eqref{eq: bound psi m +1 -ij} after integrating over the other coordinates.

The transport operator can be localized on~$m+1$ variables
\begin{align*}
\widehat {\bf S}_{s+1}^{1, (i,j)}(\tau)  \, \indc_{{\mathcal V}_{s+1}} \, M_\beta^{\otimes (s+1) } \varphi_m (Z_\sigma) 
& \leq  \widehat {\bf S}_{s+1} ^{1, (i,j)}(\tau)   \, M_\beta^{\otimes (s+1) } 
\varphi_m (Z_\sigma)  \indc_{\cP_{(i,j)}} ( z_j, z_i )
\, \indc_{{\mathcal V}_{m+1}}
\\
& \leq M_\beta^{\otimes (s+1)} \widehat {\bf S}^1_{m+1} (\tau)\psi_{m +1,\sigma^{{ <j>}}}^{i,j}  ( Z_{\sigma}, z_i ) \, ,
\end{align*}
where we used that $\varphi_m \geq 0$.

The function \eqref{eq: function  P} is not symmetric with respect  to the $i$ and $j$  variables. 
Thus to recover the symmetry, we   bound it from above by 
$$
\psi_{m+1} (Z_{m+1}) := \sum_{k,\ell \leq m+1\atop k\neq \ell}
 \varphi_m ( Z_{m+1} ^{<k>}) \, \indc_{\cP_{(k,\ell)}} ( z_k, z_\ell) \indc_{{\mathcal V}_{m+1}} \, .
$$
In this way, a factor  $m^2 \leq s^2$ has been lost compared to \eqref{eq: bound psi m +1 -ij}
\begin{equation}
\label{eq: bound psi m +1}
\|\psi_{m +1} \|_{L^1_{\beta}(\D^{m+1})} 
\leq Cts^2  \eps \, \big| \log \eps \big|  \|  \varphi_m \|_{L^1_{\beta}(\D^m)}  \, .
\end{equation}
Finally, we can write 
\begin{equation}
\label{eq: localisation 3}
\widehat {\bf S}_{s+1}^{1, (i,j)}(\tau)  \, \indc_{{\mathcal V}_{s+1}} \, M_\beta^{\otimes (s+1) } \varphi_m (Z_\sigma) 
\leq M_\beta^{\otimes (s+1)} \widehat {\bf S}^1_{m+1} (\tau) \, \psi_{m +1} ( Z_{\sigma}, z_i ).
\end{equation}

\medskip
 
{\bf Step 2.} {\it Reduction to the estimates of Lemma {\rm\ref{lem: elementary-step}}.}

Using the estimates \eqref{eq: localisation 1}, \eqref{eq: localisation 2} and \eqref{eq: localisation 3}, we get
$$
\begin{aligned}
 |C^{b, \pm}_{s,s+1}| & 
\widehat {\mathbf S}_{s+1}^1 (\tau)  \, \Big( M_\beta ^{\otimes (s+1)} 
\, \indc_{{\mathcal V}_{s+1}} \,
\sum_{\sigma\in {\mathfrak S}^m_{s+1} } \varphi_{m, \sigma} \Big) \\
& \leq   |C^{b, \pm}_{s,s+1}| \Big( 
 M_\beta ^{\otimes (s+1)} 
\sum_{\sigma\in {\mathfrak S}^m_{s+1} } (\widehat {\mathbf S}_{m}^0 (\tau) + \widehat {\mathbf S}_{m}^1 (\tau))   \varphi_{m, \sigma} \Big) \\
& \qquad +|C^{b, \pm}_{s,s+1}|  \Big(  M_\beta ^{\otimes (s+1)} 
\sum_{\tilde \sigma\in {\mathfrak S}^{m+1}_{s+1} }  \widehat {\mathbf S}_{m+1}^1 (\tau) \psi_{m+1, \tilde \sigma}\Big) \,.
\end{aligned}
$$
The global cut-off on the velocities has been removed and the transport operator localized so that  the  proof of Lemma \ref{lem: elementary-step} can be applied. Note that the first term in the right-hand side will contribute to $\Phi_m$ and $\Phi_{m-1}$, while the second term will   contribute to $\Phi_{m+1}$ and~$\Phi_{m}$. In the latter case, an argument of the function $\psi_{m+1}$ is dropped  and the factor $1/\eps$ is compensated (up to a logarithmic loss in $\eps$) thanks to
the estimate \eqref{eq: bound psi m +1}.
We therefore end up with 
$$\begin{aligned}
\int_0^{t}  &   d\tau  \,   e^{-{J \tau\over t} } |C^{b, \pm}_{s,s+1}| 
\widehat {\mathbf S}_{s+1}^1 (\tau)  \, \Big( M_\beta ^{\otimes (s+1)} \, \indc_{{\mathcal V}_{s+1}} \,
\sum_{\sigma\in {\mathfrak S}^m_{s+1} } \varphi_{m, \sigma} \Big)   \\
& \qquad \leq M_{\beta} ^{\otimes s} (V_{s})
\Big( \sum_{\sigma\in {\mathfrak S}^m_{s} }  \Phi_{m,\sigma}
+ \sum_{\sigma\in {\mathfrak S}^{m-1} _{s} }  \Phi_{m-1,\sigma}
+ \sum_{\sigma\in {\mathfrak S}^{m+1} _{s} }  \Phi_{m+1,\sigma}\Big)  \, , 
\end{aligned}
$$
with 
$$
\begin{aligned}
\| \Phi_{m-1} \| _{L^1_{\beta}(\D^{m-1})}  & \leq {C  \over \eps (m - 1) }  \| \varphi_m \| _{L^1_{\beta}(\D^m)}
\\
\| \Phi_m \| _{L^1_{\beta}(\D^m)} &  \leq C s^2 t \, \big| \log \eps \big| \| \varphi_m \| _{L^1_{\beta}(\D^m)} \\
\| \Phi_{m+1} \| _{L^1_{\beta}(\D^{m-1})}  & \leq {C  s^2 t} \eps \, \big| \log \eps \big| \| \varphi_m \| _{L^1_{\beta}(\D^m)} ,
\end{aligned}
$$
with $\Phi_0 =0$ if~$m=1$ and~$ \Phi_s = \Phi_{s+1} = 0$ if~$m=s$ or $m=s+1$.  
This is exactly the expected estimate.
\end{proof}

\subsection{Estimate of $R_N^{K,1}$  (super exponential branching with exactly one recollision)}

The proof of Proposition \ref{RNK1} follows   the same lines as the proof of Proposition \ref{RNK0}.
With the notation \eqref{Cb-def}, the iterated collision operators with quadratic and bounded collision kernels are denoted by $|Q^{q,1} _{1,J}|$, $|Q^{b,1} _{1,J}|$. The proof is split into three steps.

\medskip
 
{\bf Step 1.} {\it Evaluating  the norm of $|Q^{b,1} _{1,J}|(t)$ in $L^1_\beta$.}

We   use recursively   Lemma 
\ref{lem: elementary-step}, together with one iteration of Lemma \ref{lem: elementary-step 2}. 
Using as previously the exponential to get rid of the constraint on the time increments, we have to control
a quantity of the form
\begin{align*}
&\int_{\D} dz\,  |Q_{1,J}^{b,1}| (t)  M_\beta ^{\otimes J} \varphi_{m,\sigma} \, \indc_{{\mathcal V}_{J}}   \\
& \qquad \qquad  \leq 
 \alpha^{J-1}e^J \sum_{ \ell=2}^J \int_{\D} dz\,  \int_0^\infty \int_0^\infty\dots  \int_0^\infty d\tau_{{J}} \dots d\tau_{2}  e^{  - J{\tau_2\over t} }  |C^b_{1,2}|  \widehat  {\mathbf S}^0 _{2}(\tau_2)  \dots \\
& \qquad \qquad  \qquad \qquad \dots  e^{  - J{\tau_\ell\over t} } \indc_{\tau_\ell \leq t}  |C^b_{\ell-1,\ell}|   \widehat  {\mathbf S}^1 _{\ell}(\tau_{\ell}) \dots    e^{  - J{\tau_J\over t} } \widehat  {\mathbf S}^0 _{J}(\tau_{J}) M_\beta ^{\otimes J} \varphi_{m,\sigma} \, \indc_{{\mathcal V}_{J}} \\
&  \qquad \qquad  \leq 
 \alpha^{J-1}e^J \sum_{ \ell=2}^J \int_{\D} dz\,  \int_0^\infty \int_0^\infty\dots  \int_0^\infty d\tau_{{J}} \dots d\tau_{2}  e^{  - J{\tau_2\over t} }  |C^b_{1,2}|  \widehat  {\mathbf S}^0 _{2}(\tau_2)  \dots \\
&  \qquad \qquad  \qquad \qquad \dots  e^{  - J{\tau_\ell\over t} } \indc_{\tau_\ell \leq t}  |C^b_{\ell-1,\ell}|   \widehat  {\mathbf S}^1 _{\ell}(\tau_{\ell})  \, \indc_{{\mathcal V}_{\ell}}  \dots    e^{  - J{\tau_J\over t} } \widehat  {\mathbf S}^0 _{J}(\tau_{J}) M_\beta ^{\otimes J} \varphi_{m,\sigma}, 
\end{align*}
where the cut-off on the velocities  in the second inequality applies only to the operator with one recollision
(by using the fact that the energy is preserved by the transport operators).

\medskip

We proceed now by applying $J-2$ times the estimates of Lemma \ref{lem: elementary-step}, and once the estimate of Lemma \ref{lem: elementary-step 2}.
When applying Lemma \ref{lem: elementary-step 2}, the number of variables may shift from $m$ to $m+1$, but for all other iterations we either stay with the same number variables, or shift from $m$ to $m-1$. As the number of variables has to be  dropped to 1, the total number of possible combinations is less than  $2^J$. We therefore end up with
\begin{equation}
\label{Q1L1-est}
\int_{\D} dz\,  |Q_{1,J}^{b,1}(t)| \sum _{\sigma\in {\mathfrak S}^m_{J} } M_\beta ^{\otimes J} \varphi_{m,\sigma} (Z_\sigma)    \, \indc_{{\mathcal V}_{J}}
\leq 
(C\alpha)^{J-1} \, t^{J-m} \frac{J^3 |\log \eps| }{\eps^{m-1} m! } 
 \| \varphi_m \|_{L^1_{\beta}(\D^m)} \, .
\end{equation}
This estimate is similar to the one of Proposition \ref{L1-est} with an extra factor $J^3|\log \eps|$.
To compensate this logarithmic divergence, we are going to adapt the $L^\infty$ estimates of  Proposition~\ref{prop: estimatelemmacontinuity quad} in order to gain a factor $\eps$ from the recollision.

\medskip
 
{\bf Step 2.} {\it Evaluating  the norm of $|Q^{q,1}|$ in $L^\infty$.}

Noticing that the recollision  takes place either in the last time interval or before, we get the decomposition
\begin{align}
\label{eq: gain gactor eps}
 |Q_{1,J}^{q,0}| (t) \,  |Q_{J, J+n}^{q,1}|(h) \, M_\beta ^{\otimes (J+n) }
 {  \indc_{{\mathcal V}_{J+n}}}
 & + |Q_{1,J}^{q,1}| (t) \,  |Q_{J, J+n}^{q,0}| (h) \, M_\beta ^{\otimes (J+n) }
  {  \indc_{{\mathcal V}_{J+n}}}
 \\
  &\leq (C\alpha t)^{J-1} (C\alpha h)^{n-2}(J+n) ^3 \eps |\log \eps |^{10} M_{5 \beta/8} (v_1)\,, 
\nonumber
\end{align}
where we used the refined estimate~(\ref{eq: L1 estimate petit h}) 
and the geometric estimates of Section~\ref{sec: Geometric control of recollisions}
in order to recover the factor $\eps$ from the recollision.
Combined with \eqref{Q1L1-est} and a Cauchy-Schwarz estimate as in \eqref{nosmeilleursamis}, we get
\begin{equation}
	\label{Q1-estinfty}
		\begin{aligned}
& \big\|  |Q_{1,J}^1|(t) \,   |Q^0_{J,J+n}|(h) \, M_{J+n,\beta}  \, \indc_{{\mathcal V}_{J + n}}
\sum _{\sigma\in {\mathfrak S}^m_{J} } | g_{m,\sigma} | \big\|_{{ L^2(\D)}} \\
&\quad 	+\big\|  |Q_{1,J}^0|(t) \,  |Q^1_{J,J+n}|(h) \, M_{J+n,\beta}  \, \indc_{{\mathcal V}_{J + n}}
\sum _{\sigma\in {\mathfrak S}^m_{J} } |g_{m,\sigma}| \big\|_{{ L^2(\D)}} \\ 
&\qquad \qquad 
\leq   (C\alpha t)^{J+n/2 -1}(C\alpha h)^{n/2 - 1}(J+n)^\frac32 \eps^{1 /2} |\log \eps|^{  11/2} { \| g_m\|_{L^2_\beta}  \over \sqrt{ \eps^{m-1} m!} } \, \cdotp
	\end{aligned}
	\end{equation}
The logarithmic loss in $\eps$ is  compensated by the extra $\eps^{1 /2}$ factor from \eqref{eq: gain gactor eps}.
Thus, we have obtained a counterpart of Proposition \ref{L2-est}.

\medskip
 
{\bf Step 3.} {\it Resummation.}

 The last step is then to sum over all possible contributions $k$,  $j_i <n_i$ for  $ i\leq k-1$, $j_k \geq n_k$,  and $m\leq  J_k$.
Recall from \eqref{gi-est} that
$$
\| g_N^{m}( t-kh) \|_{L^2_\beta}^2 \leq {CN\exp (C\alpha^2)\over  \binom{N}{m}}  
\leq {C^m \, m ! \exp (C\alpha^2) \over N^{m-1}  } \, \cdotp 
$$
Then, by    (\ref{Q1-estinfty}), we have   (rounding off the power of $\log \eps$)
$$
\begin{aligned}
\big\| |Q^0_{1,J_1} | (h ) \dots  |Q^0_{J_{k-1},J_{k}} | (h )  \sum _{\sigma\in {\mathfrak S}^m_{J_k} } &M_\beta ^{\otimes J_k} \, \indc_{{\mathcal V}_{J_k}} |g_N^{m}  ( Z_{\sigma}) | \, 
 \big\|_{{ L^2(\D)}}\\
& \leq (C\alpha)^{J_k} \exp (C\alpha^2) t^{J_{k-1}+j_k/2  } h^{j_k/2 - 1}\eps^{1 /2} |\log \eps|^{  6} \,.
 \end{aligned}
 $$
We then sum over all $m\in \{1,\dots,  J_k\}$  to get
$$
\begin{aligned}
\big\|  |Q^0_{1,J_1} | (h) \dots & |Q^0_{J_{k-1},J_{k}} | (h) \,  | f^{(J_K)}_N (t-kh)|
\, \indc_{{\mathcal V}_{J_k}} \big\|_{{ L^2(\D)}}\\ 
&  \leq (C\alpha)^{J_k}  \exp (C\alpha^2) t^{J_{k-1}+j_k/2  } h^{j_k/2-1 } \eps^{1 /2} |\log \eps|^{  6} \,.
 \end{aligned}
$$

Provided that $\alpha ^2t^{3/2}h^{1/2}   \ll 1$, we can first sum over all $j_k \geq n_k$, which leads to 
$$
\begin{aligned}
\sum_{j_k \geq n_k} \big\|  |Q^0_{1,J_1} | (h) \dots  |Q^0_{J_{k-1},J_{k}} | (h) \,  | f^{(J_K)}_N (t-kh)|
\, \indc_{{\mathcal V}_{J_k}} \big\|_{{ L^2(\D)}} \\
\leq  (C\alpha^2 t^{3/2} h^{1/2} )^{J_{k-1} } \exp (C\alpha^2) \frac{ \eps^{1 /2}}{h} |\log \eps|^{  6}
\,.
\end{aligned}$$
Taking the sum over all possible $j_i <    2^in_0$ for $i\leq k-1$, we get $O(2^{k^2})$ such terms.
We therefore end up with 
$$
\Big\|   R_N^{K, 1}(t) \Big\|_{L^2(\D)} 
\leq \gamma     \,  
{ \eps^{1/ 2}  |\log \eps|^{  6} \over h}  \,  \cdotp
$$
This concludes the proof of Proposition \ref{RNK1}. \qed


\section{Control of super-exponential trees with multiple recollisions}
\label{recollisions}

Recall that the remainder term~$R_N^K$ is a series expansion~(\ref{reste}) with elementary terms of the form 
$$\alpha^{J_k -1} Q_{1,J_1} (h ) \dots  Q_{J_{k-2},J_{k-1}} (h ) Q_{J_{k-1},J_k} (h)     f^{(J_k)}_N(t- kh) \, ,$$
which corresponds exactly  to collision trees having
\begin{itemize}
\item $j_i<n_i$ branching points on the first $k-1$ intervals ($i<k$);
\item $j_k \geq n_k$ branching points on the $k$-th interval;
\end{itemize}
and that $R_N^{K,>}$ is  the restriction of $R_N^K$ to pseudo-dynamics having more than one recollision, with energies bounded by~$C_0 |\log  \eps|$.

\medskip

The main result of this section is the following.
\begin{Prop}
\label{propR>} 
Let~$\gamma<1$ be given. Choose
$$
n_k =n_0\times  2^k, \qquad h \leq \frac{\gamma}{\exp (C\alpha^2) T^3} \, \cdotp
$$ 
Under the Boltzmann-Grad scaling~$N\eps = \alpha \gg1$, there holds for all~$t \in [0,T]$
$$
\Big\|   R_N^{K,>}(t) \Big\|_{{ L^2 (\D)}}   \leq \gamma  \, .
$$
\end{Prop}
The next two paragraphs are devoted to a quantitative estimate showing  that dynamics with more than one  recollision  are unlikely: the statement is given in Paragraph~\ref{geometric control statement}, and its proof is in Paragraphs~\ref{classificationsection}. Finally the proof of Proposition~\ref{propR>} appears in Paragraph~\ref{conclusion of the proof superexpmultiple}, combining the geometric argument with the time sampling.

\subsection{Geometric control of multiple recollisions: statement of the result}\label{geometric control statement}
Unlike in  Section~\ref{convergencepartieprincipales}, 
we   need very sharp estimates to compensate the divergence of order~$N$ 
of the~$L^\infty$ norm given in \eqref{Linfty}.
Thus we cannot afford to lose any power of $|\log \eps|$. In order to improve the bound obtained in Proposition~\ref{recoll1-prop} we shall estimate the size of trajectories having at least two recollisions rather than one. Indeed we recall that the three powers of~$|\log \eps|$
in the analysis of one recollision are due
to the integration in time of the constraint of having one recollision, as well as on the possibility of having small relative velocities
 (which are integrated out and create a loss of~$|\log \eps|$ through the scattering operator). Here we shall see that the presence of a second recollision leads to a finer geometric condition which produces in general a bound of the size~$\eps^\gamma$ with~$\gamma>1$, hence any power of~$|\log \eps|$ can be absorbed. However
in some degenerate situations this geometric condition is ineffective (for instance when some relative velocities are too small) and in that case some specific arguments must be used, which give just a power of~$\eps$ with no additional gain -- nor loss.

\medskip

 The  presence of multiple recollisions can be encoded in the domain of integration (collision times, impact parameter and velocity of the additional particles). 
 The analysis relies heavily on the computations leading to Proposition~\ref{recoll1-prop}, but the two recollisions may be intertwined so   more cases have to be considered.  
In the next paragraph we start by introducing a classification of the different situations that can lead to the first two recollisions in the dynamics.    The proof of the  technical aspects regarding each geometric case  is postponed to Appendix~\ref{geometricallemmasappendix}.

\medskip

 As in the case of Proposition~\ref{recoll1-prop}
we analyze all possible scenarios  leading to the occurrence of at least two recollisions. Each scenario is labeled by an index~$p$ and the total number of possible scenarios, whose exact value is irrelevant, is denoted in the following by~$p_0$.
 The next statement is the counterpart of Proposition~\ref{recoll1-prop}
in the case of two recollisions, and we use the notation
of that proposition. For any finite set of integers~$\sigma$ we denote by~$\sigma_-$ the smallest integer in~$\sigma$, and by~$\sigma_+$ the largest one. 
 \begin{Prop}
\label{lemrecoll}
Fix a final configuration of bounded energy~$z_1 \in \T^2 \times  B_R$ with~$1 \leq R^2 \leq C_0 |\log \eps|$,  a time $1\leq t \leq C_0|\log \eps|$ and  a collision tree~$a \in \cA_s$ with~$s \geq 2$.

For all types of recollisions $0 \leq p\leq p_0$, and all sets of parents $\sigma \subset \{2,\dots, s\}$ (where the length $|\sigma| \leq 5 $ depends only on $p$), there exist sets of bad parameters
$\cP_2(a, p,\sigma)\subset \cT_{2,s} \times {\mathbb S}^{s-1} \times \R^{2(s-1)}$  satisfying

\begin{itemize}
\item[(i)]  for $0 \leq p\leq 2$,
$$\cP_2(a, p,\sigma):=\Big \{(t_m, v_m, \nu_m) \in \cP_1(a,p, \sigma\setminus \{\sigma^+\} )\,/\, ( t_{\sigma^+}, v_{\sigma^+}, \nu_{\sigma^+}) \in  \cQ(a,p,\sigma) \Big\} $$
 where $\cQ(a,p,\sigma) $ is parametrized only in terms of   $(t_m, v_m, \nu_m)$ for   $m <  \sigma^+$ and satisfies
 $$\int \indc _{\cQ(a,p,\sigma)  } \displaystyle  \big| \big(v_{\sigma^+}-v_{a(\sigma^+)} ( t_{{\sigma^+}} ) )\cdot \nu_{{\sigma^+}}\big)\big | d  t_{{\sigma^+}} d \nu_{{\sigma^+}}dv_{{\sigma^+}}   
  \leq CR^5 s t^2 \eps^\frac14  |\log \eps|^2\,;
$$

 \item[(ii)]  for $2 < p \leq p_0$,
$\cP_2(a, p,\sigma)$ is parametrized only in terms of   $(  t_m, v_m, \nu_m)$ for $m \in \sigma$ and~$m <   \sigma_-$ and satisfies  \begin{equation}
 \label{cP2-est}
 \int \indc _{\cP_2(a,p,\sigma)  } \displaystyle  \prod_{m\in \sigma} \,\big | \big(v_{m}-v_{a(m)} ( t_{{m}} ) )\cdot \nu_{{m}}\big)\big | d  t_{{m}} d \nu_{{m}}dv_{{m}}   
  \leq C(Rt)^{r}s^2 \eps 
  \end{equation}
for some~$r>0$;
\end{itemize}
such that   any pseudo-trajectory starting from $z_1$ at $t$, with total energy bounded by $R^2$ and involving at least two recollisions, is parametrized by 
$$( t_n, \nu_n, v_n)_{2\leq n\leq s }\in  \bigcup _{p=0}^{p_0} \bigcup _\sigma  \cP_2(a, p,\sigma)\,.$$
\end{Prop}

 \subsection{Classification of multiple recollisions}\label{classificationsection}
 In the case of one recollision (recall Proposition~\ref{recoll1-prop}), 
the key to  the proof was to identify  two collisions related to that recollision, i.e. two degrees of freedom, for which  the constraints  due to the recollision lead to a set of small measure.
We   proceed in the same way here: we consider a pseudotrajectory involving at least two recollisions and denote by $(i,j)$ and~$(k,\ell)$ the particles involved in  the first two recollisions in the backward dynamics and by~$t_{rec}\in ]t_{\theta+1}, t_\theta[$ and~$\tilde  t_{rec} \in ]t_{\tilde \theta+1} ,t_{\tilde \theta}[ $ the corresponding recollision times;  note that the labels are not necessarily distinct, and neither are   the associate pseudo-particles, using the terminology introduced in Definition~\ref{pseudoparticle}. We denote  the  first parent  {(starting at height~$\theta$) of the recolliding particles~$(i,j)$ 
by $1^*$, and by $\tilde 1$ the first parent  (starting at height~$\tilde \theta$) of the recolliding particles~$(k,\ell)$. We define similarly~$2^*,3^*,\tilde 2, \tilde 3$ the other parents moving up the tree  to the root (they might not all be distinct).

 \bigskip

Without loss of generality we may assume that~$t_{\tilde 1} \leq t_{1^*}$ and that~$\tilde 1$ is the parent of~$\ell$.   
To classify the dynamics, we  shall    consider separately the cases~$t_{\tilde 1}  < t_{1^*}$ and~$t_{\tilde 1}  = t_{1^*}$.

 \subsubsection{Case 1: $t_{\tilde 1}  < t_{1^*}$}  \label{firstcase1<1}
 
 We denote by~$(x_k( t_{\tilde 1}),  v_k)$, $(x_\ell( t_{\tilde 1}), \bar v_\ell)$ the configurations of the pseudo-particles~$k$ and~$\ell$ at time $  t_{\tilde 1}^-$, and by~$v_{\ell}$ the velocities of  pseudo-particles~$\ell$ at time~$t_{\tilde 1}^+$.
 With that notation, let us write the condition for the recollision $(k,\ell)$ to hold:
 $$
 x_{k}( t_{\tilde 1}) - x_\ell ( t_{\tilde 1}) 
 + (v_{k} - v_\ell ) (\tilde t_{rec}- t_{\tilde 1})  = \eps  \tilde \nu_{rec} 
+\tilde q \, ,
$$
So defining as previously
$$
\delta \tilde x_{k \ell} ( t_{\tilde 1}) := \frac1\eps ( x_{k}( t_{\tilde 1})  - x_\ell( t_{\tilde 1}) -\tilde q)   
$$
and    
$$ 
\tilde \tau_{rec} :=-\frac1\eps (\tilde t_{rec} - t_{\tilde 1})\,  ,
$$
we find
 \begin{equation}
\label{reckl-equation}
v_{k} - v_\ell = \frac1{\tilde \tau_{rec}}  \delta \tilde x_{k \ell}  ( t_{\tilde 1}) -  \frac1{\tilde \tau_{rec} } \tilde \nu_{rec}  \, .
\end{equation}
To compute the dependency of $ \delta \tilde x_{k \ell}  ( t_{\tilde 1})$ in $ t_{\tilde 1}$, 
we will not use a decomposition as precise as~\eqref{rec-equation}, as the trajectories of $k, \ell$ may be modified by the first recollision in the time interval~$[t_{\tilde 1}, t_{\tilde 2}]$.
Since the trajectories of $k, \ell$ are piecewise linear, we retain only the information that 
\begin{equation}
\label{eq: derivative}
{d \over dt_{\tilde 1} } \delta \tilde x_{k \ell} ( t_{\tilde 1}) =\frac1\eps (v_k - \bar v_\ell)\,.
\end{equation}
We therefore consider two situations according to the size of the relative velocity~$| v_k-  \bar v_\ell|$.

\medskip

\noindent
\underbar{1.1 Relative velocities bounded from below.}  
Suppose that $| v_k-  \bar v_\ell| \geq \eps^\frac34$. 
The relation \eqref{reckl-equation} imposes that at time $t_{\tilde 1}$, the velocity 
$v_\ell$ has to belong to a cone $\cC (\delta \tilde x_{k \ell} ( t_{\tilde 1}))$ of axis $\delta \tilde x_{k \ell} ( t_{\tilde 1})$ with volume at most
$$
\min \left(  \frac{R}{ | \delta \tilde x_{k \ell} ( t_{\tilde 1})| }  , R^2 \right)  ,
$$
as the velocities are bounded by $R$.
Integrating first this constraint over $dv_{\tilde 1}d\nu_{\tilde 1}$ 
(as in the proof of Lemma \ref{rec-eq}) and then 
over $dt_{\tilde 1}$ by using the  derivative \eqref{eq: derivative},
we get 
\begin{align*}
& \int \indc_{\{ v_\ell \in \cC (\delta \tilde x_{k \ell} ( t_{\tilde 1})) \}}
\big | \big(v_{\tilde 1}- \bar v_{\ell} )\cdot \nu_{\tilde 1}\big)\big |
 dt_{\tilde 1} dv_{\tilde 1}d\nu_{\tilde 1}\\
& \qquad \qquad \leq 
R^3  \, \int  |\log | \delta \tilde x_{k \ell} ( t_{\tilde 1})| | \, 
\min \left(  \frac{1}{ | \delta \tilde x_{k \ell} ( t_{\tilde 1})|}  , R \right) 
\; \indc_{ \{ | v_k-  \bar v_\ell| \geq \eps^\frac34 \}} dt_{\tilde 1}\\
& \qquad \qquad \leq 
R^3 \eps |\log \eps |   \int_0^{Rt/\eps}
\min \left(  \frac{1}{ u}  , R \right) 
\frac{1}{| v_k - \bar v_\ell|}
\, \indc_{ \{ | v_k-  \bar v_\ell| \geq \eps^\frac34  \} } d u
\leq CR^3  \eps^\frac14  |\log \eps|^2.
\end{align*}
Summing over  all possible~$\tilde q$ gives a bound similar to (\ref{proofProp4.3}) 
$$
\int \indc_{\{ v_\ell \in \cC (\delta \tilde x_{k \ell} ( t_{\tilde 1})) \}} 
\big | \big(v_{\tilde 1}- \bar v_{\ell} )\cdot \nu_{\tilde 1}\big)\big |
dt_{\tilde 1} dv_{\tilde 1}d\nu_{\tilde 1}
\leq CR^5t^2 \eps^\frac14  |\log \eps|^2 \,  .
$$
Then we apply Proposition~\ref{recoll1-prop} to recollision~$(i,j)$ which involves   parents $1^*, 2^* < \tilde 1$ (if it is a self-recollision as depicted page~\pageref{selfrecollpage} then only one parent is involved). Note that this condition (characterized by $\cP_1(a,0,\{1^*\})$ or $\cP_1(a, p, \{1^*, 2^*\})$ with $p=1,2$) is independent from the previous one. We are therefore in the situation (i) of Proposition~\ref{lemrecoll}.

\medskip
\noindent
\underbar{1.2 Small relative velocities.}  In that case  we forget about     (\ref{reckl-equation}) and we   consider instead  the condition $| v_k-  \bar v_\ell| \leq \eps^\frac34$.
We then need to define to which degree the recollision between~$i$ and~$j$ affects the recollision between~$k$ and~$\ell$. There are four different possible situations:
 
\begin{itemize}
\item[(a.1)] $t_{rec} \in (\tilde t_{rec},  t_{\tilde 2})$ and~$k\in \{i,j\}$

\item[(a.2)] $t_{rec} \in (\tilde t_{rec},  t_{\tilde 2})$ and~$\ell \in \{i,j\}$

\item[(b)]$t_{rec} \in ( t_{\tilde 2},t_{\tilde 3})$ and $k\in \{i,j\}$ or $\ell \in \{i,j\}$ 

\item[(c)]$t_{rec} \notin (\tilde t_{rec},t_{\tilde 3})$, or $k,\ell \notin \{i,j\}$

 \end{itemize}
 examples of which are depicted in Figure~\ref{whereistrec}. 
   \begin{figure} [h] 
   \centering
\quad \includegraphics[width=5.7cm]{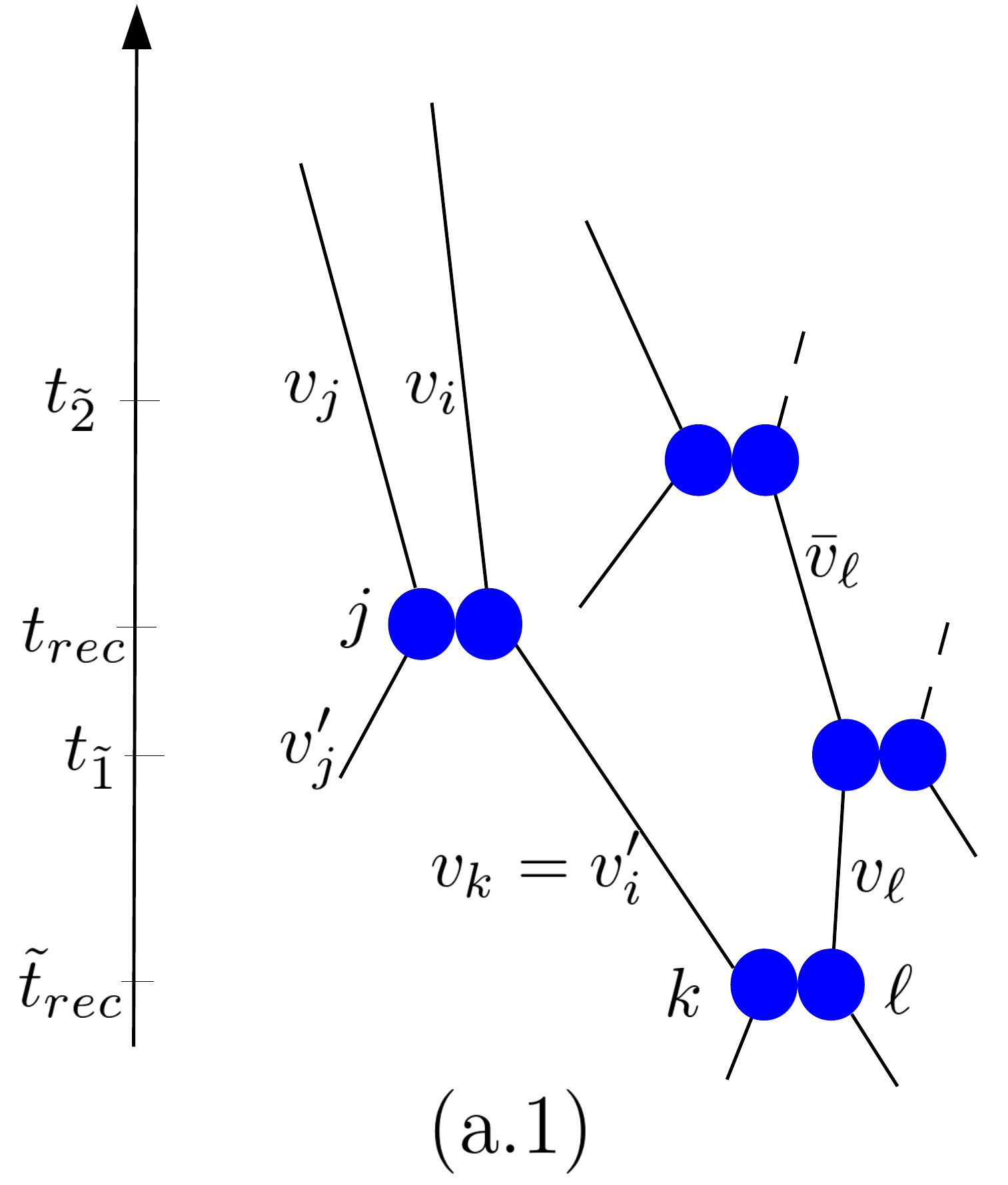}   \qquad    \quad\quad  \quad   \quad 
          \includegraphics[width=5.9cm]{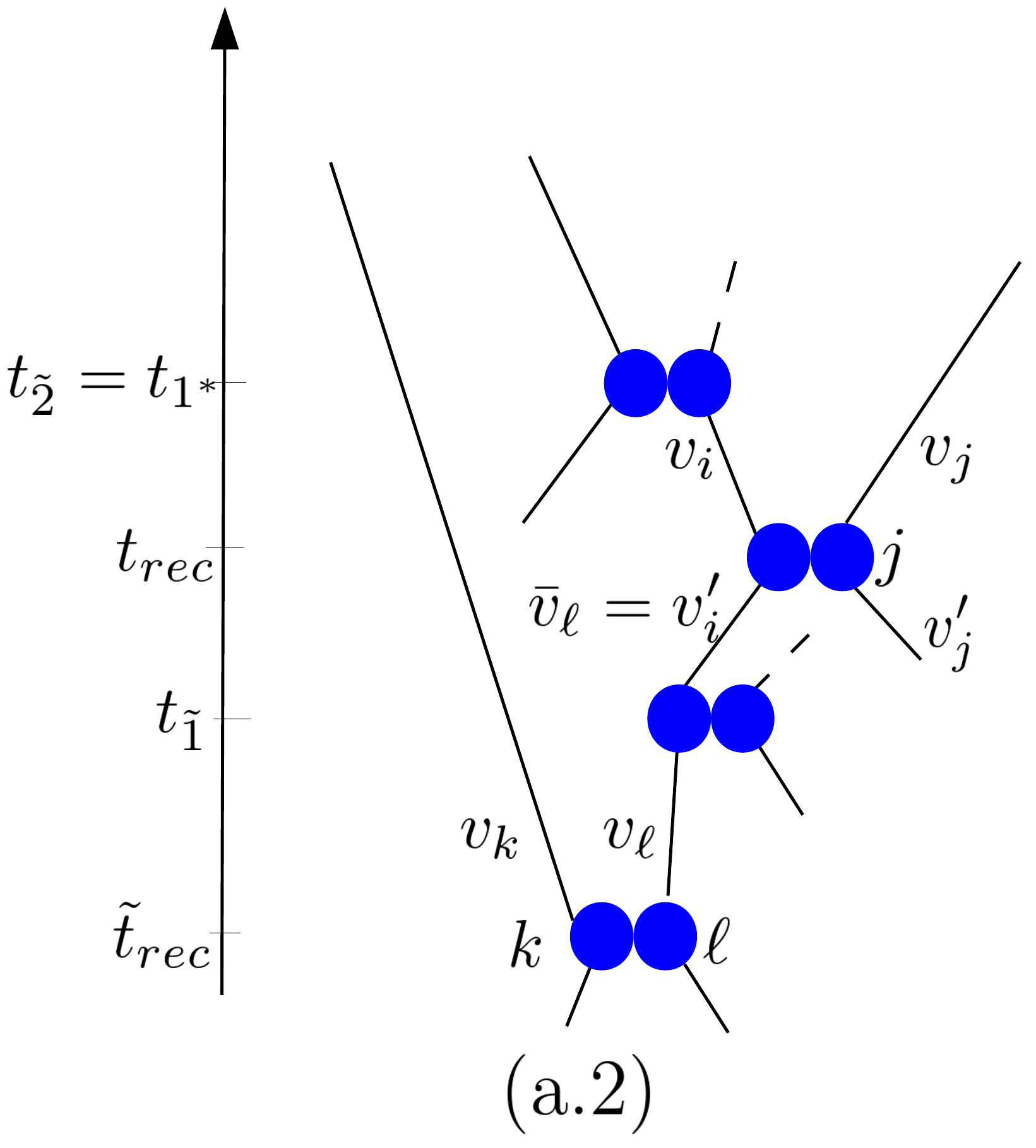}  \qquad    \quad\quad  \quad    \includegraphics[width=6cm]{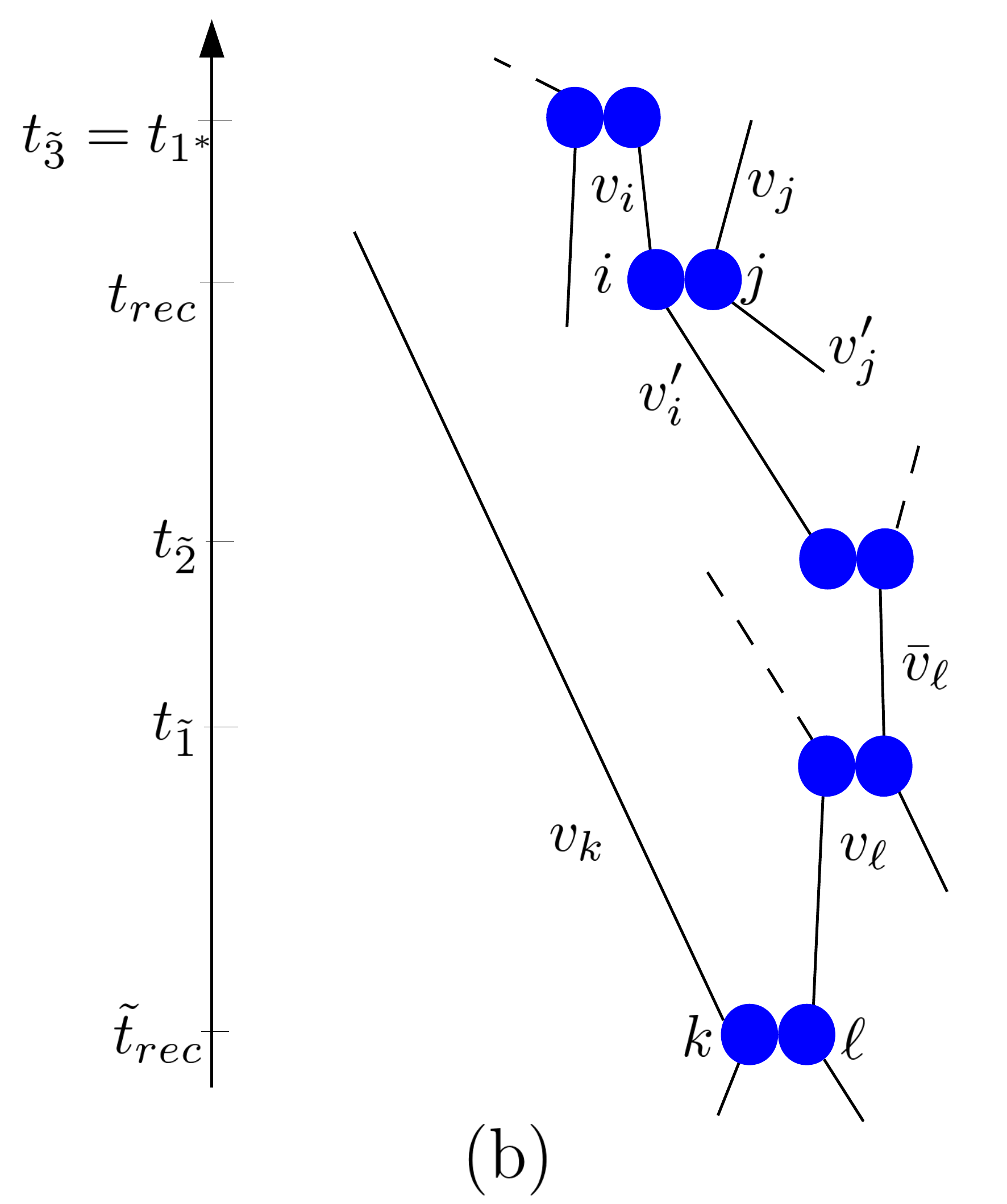} \qquad    \quad\quad  \quad   \quad\quad    \includegraphics[width=5.3cm]{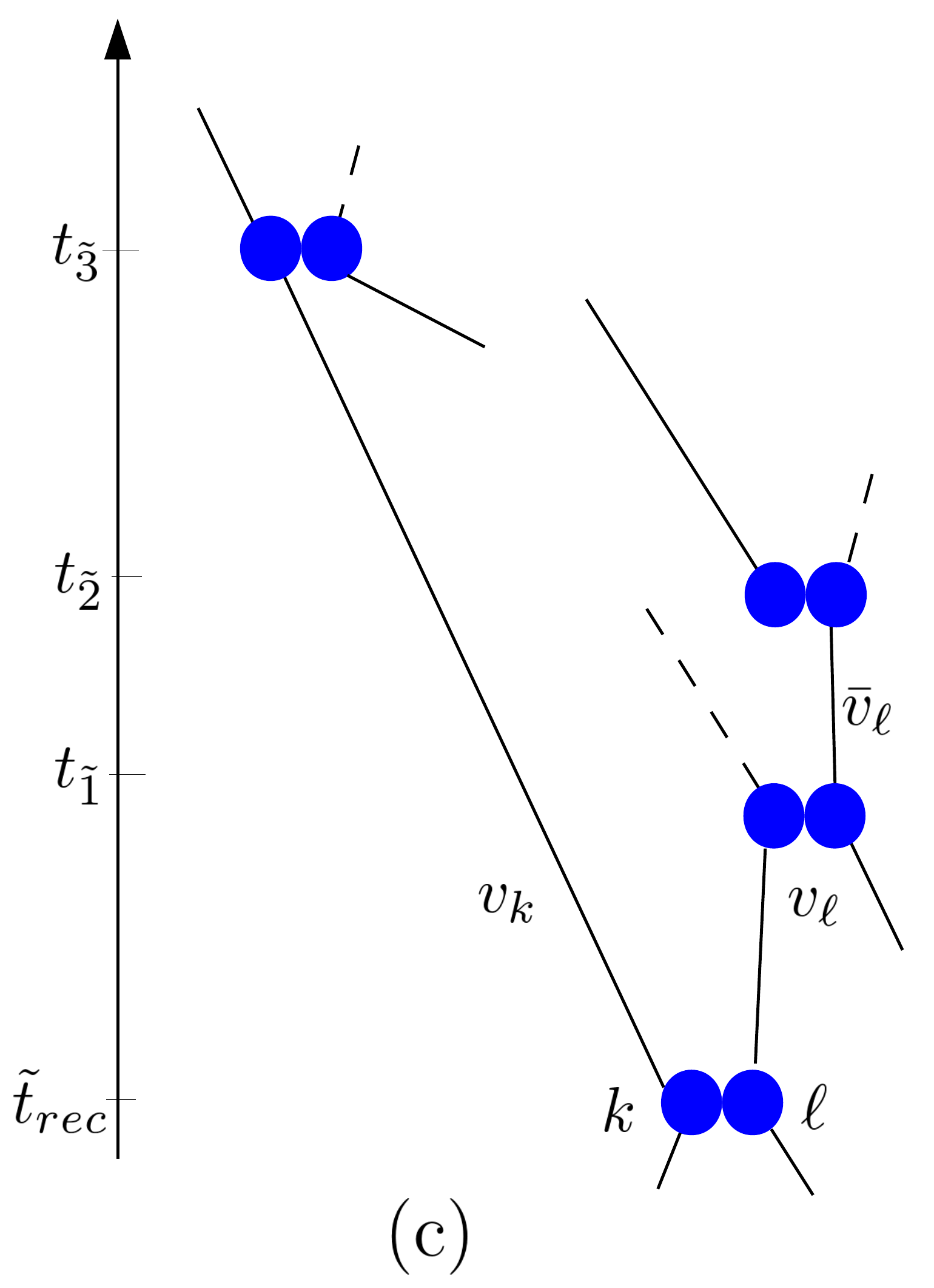}
   \caption{ Case $t_{\tilde 1}  < t_{1^*}$, Degree of intertwinement between both recollisions} \label{whereistrec}
\end{figure}

We therefore need to be more precise when describing the history  of $(k,\ell)$.
 We denote by~$(x_k,  x_\ell)$ the positions of the pseudo-particles~$k$ and~$\ell$ at time $  t_{\tilde 2}$.
We denote by~$( v_i,  v_j)$ (resp. $(v'_i, v'_j)$ )  the velocities of particles $i,j$ before the recollision (resp. after) the recollision $(i,j)$.

\begin{itemize}
\item[(a.1)] We need a more precise geometric argument to ensure both that the first recollision   occurs, and that it   produces an outgoing velocity~$v_i'$ or $v_j'$ in the ball $B(\bar v_\ell, \eps^{3/4})$.  This is provided in Lemma \ref{v-constraint}.
Note that this case can be degenerate if $\tilde 2 = 1^*$, which means that $\tilde 2$ is also the parent of $\ell$.  

\item[(a.2)] Again Lemma  \ref{v-constraint}
enables us to ensure both that the first recollision   occurs, and that it   produces an outgoing velocity~$v_i'$ or $v_j'$ in the ball $B(  v_k, \eps^{3/4})$.  
Note that this case can be degenerate if $\tilde 2$ is also the parent of $k$.

\item[(b)]  We know that $ \tilde 2 > 1^*, 2^*$. We then integrate over~$dt_{\tilde 2}dv_{\tilde 2}d\nu_{\tilde 2}$ the constraint of having small relative velocities $| v_k -  \bar v_\ell| \leq \eps^{3/4}$ and  by (\ref{rectangle0})  this gives a bound of the order of~$O( R^2 t \eps^{3/4}  |\log \eps|)$.
Then we apply Proposition~\ref{recoll1-prop} to recollision~$(i,j)$ which involves   parents $1^*, 2^* < \tilde 2$.  We are therefore in the situation~(i).

\item[(c)]  In that case  combining (\ref{preimage-sphere1}) and (\ref{carleman2}), we deduce that 
$$ \int \indc _{| v_k- \bar v_\ell| \leq \eps^\frac34 } \displaystyle  \prod_{m= \tilde2, \tilde 3} \,\big | \big(v_{m}-v_{a(m)} ( t_{{m}} ) )\cdot \nu_{{m}}\big)\big | d  t_{{m}} d \nu_{{m}}dv_{{m}}   
  \leq CR^5 t^2 \eps^{3/2} |\log \eps| \,.$$
  For any fixed $\{\tilde2, \tilde 3\}$, this scenario  will be labelled by $p=3$ hence 
$$\int \indc _{\cP_2(a,3,\sigma)  } \displaystyle  \prod_{m\in \sigma} \,
\big | \big(v_{m}-v_{a(m)} ( t_{{m}} ) )\cdot \nu_{{m}}\big)\big | d  t_{{m}} d \nu_{{m}}dv_{{m}}   
  \leq Cs R^5 t^2 \eps^{3/2} |\log \eps|\,. $$
  \end{itemize}
  Note the extra factor~$s$, which appears for the same reasons as explained page~\pageref{pagewiths}.

\subsubsection{Case 2: $t_{\tilde 1}  = t_{1^*}$} 
\label{secondcase1=1}

 This is a  very constrained case, as all the recolliding particles have the same first parent.
 We separate the analysis into two subcases.

\medskip
\noindent
\underbar{2.1 Parallel recollisions.}
 This case is depicted in  Figure~\ref{cherries}; the two recollisions take place with the same parent, but   there is no direct link between the two couples of recolliding pseudo-particles~$(i,j)$ and~$(k,\ell)$, meaning as previously that the trajectory of~$\ell$ and~$k$  between time~$t_{\tilde 1}=t_{1^*}$ and~$\tilde  t_{rec}$ is unaffected by that of~$i$ or~$j$ on the same time interval. 
  \begin{figure} [h] 
   \centering
    \includegraphics[width=5cm]{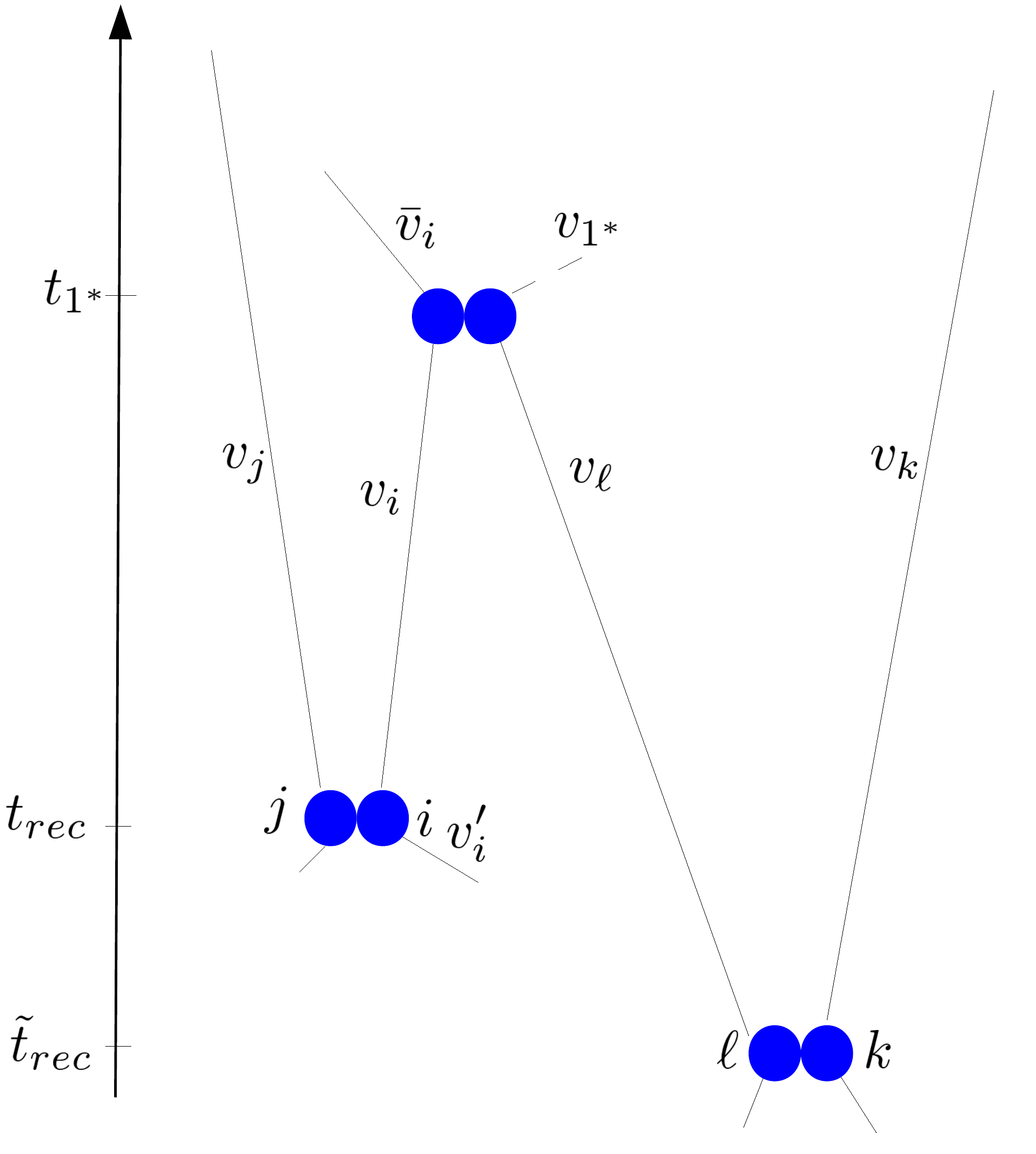}  
   \caption{$t_{\tilde 1}  = t_{1^*}$, parallel recollisions.\label{cherries}}
\end{figure}
     The analysis   is  postponed to      Lemma \ref{doublerecollision}.

 \medskip
\noindent
\underbar{2.2 Recollisions in chain.}
In this case   the two recollisions take place in chain (the trajectory of one of the recolliding particles~$k$ or~$\ell$ is affected by~$i$ or~$j$ between  time~$t_{\tilde 1}=t_{1^*}$ and~$\tilde  t_{rec}$)  (see Figure~\ref{murphyfig}). 
 \begin{figure} [h] 
   \centering
   \includegraphics[width=5.7cm]{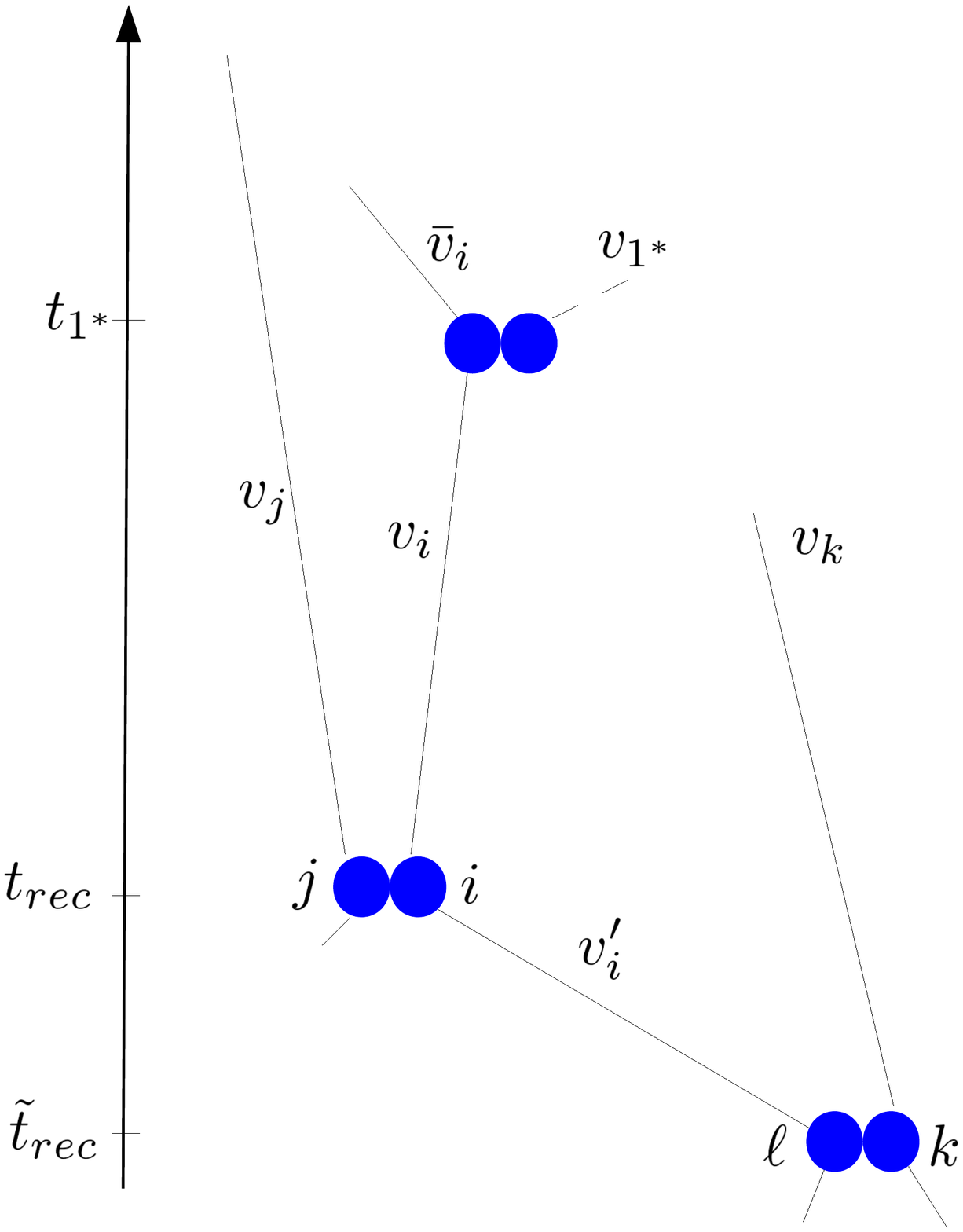}  
     \caption{$t_{\tilde 1}  = t_{1^*}$,  recollisions in chain. \label{murphyfig}}
\end{figure}
  
This case is dealt with in Lemma~\ref{Murphy} and Lemma \ref{self-recoll-periodic}  in Appendix~\ref{geometricallemmasappendix}

\subsubsection{Indexing with $\sigma$}  
\label{sec: indexing with sigma}

So far the two recollisions have been described in terms of the recolliding particles $i,j,k,\ell$, however in Proposition \ref{lemrecoll} the sets $\cP_2(a, p,\sigma)$ are not indexed by the recolliding particles but by the parents of these particles, i.e. by the degrees of freedom leading to these recollisions. Once the set $\sigma$ of parents is fixed, all the recolliding particles are not necessarily prescribed. Indeed for parallel recollisions or recollisions in chain, the set $\sigma$ could corresponds only to parents of $i$ in which case there is an extra combinatorial factor~$s^2$  for choosing the other two recolliding particles $j,k$.

\subsection{Estimate of $R_N^{K,>}$ (super exponential trees with multiple recollisions)}
\label{conclusion of the proof superexpmultiple}

Proposition \ref{propR>} comes   from a careful summation of all elementary contributions.
We   therefore need the following refinement of Proposition \ref{restrictedtrees}.  

\begin{Prop}
\label{restrictedtrees2}
{\color{black}
We fix~$z_1 \in \T^2 \times B_R$,   a label $p\leq p_0$ and a  set $\sigma \subset \{1, \dots, s\}$ of cardinal~$|\sigma |\leq 5$.
With the notation of Proposition~{\rm\ref{lemrecoll}}, denoting~$\eta =s^2 t^r \eps$, one has  for $t\geq 1$
\begin{equation}
\label{eq: L1 estimate petit2recoll}
\begin{aligned}
 \sum_{a \in \cA_s}   \int \indc _{\cT_{2,s}}    
 \indc_{\cP_2(a, p,\sigma)}
\,  \Big( \prod_{i=2}^s \,  \big| (v_i -v_{a(i)} (t_i))\cdot \nu_i \big| \Big)  & M_\beta^{\otimes s} \, 
  dT_{2,s}  d\Omega_{2,s}   dV_{2,s} \\
 &   \leq   s^5 (Ct)^{s-1}  \eta M_{\beta/2}(v_1) \,.
  \end{aligned}
\end{equation}
If we further specify that the last $n$ collision times have to be in an interval of length $h\leq 1$
(this constraint is denoted by $\cT_{s-n+1, s}^h$) then
\begin{equation}
\label{eq: L1 estimate petit2recoll h}
\begin{aligned}
 \sum_{a \in \cA_{s}}   \int \indc _{\cT_{2,s}}
 \indc_{\cT_{s-n+1, s}^h}  
 \indc_{\cP_2(a, p,\sigma)} 
 \; \Big( \prod_{i=2}^{s}   \,   \big| (v_i -v_{a(i)} (t_i))\cdot \nu_i \big|\Big)  M_\beta^{\otimes s } \, dT_{2,s}  d\Omega_{2,s}   dV_{2,s}    \\
\leq s^5  (Ct)^{s-n-1}  (Ch)^{n-5} \eta  M_{\beta/2}(v_1) \, .
\end{aligned}
\end{equation}
}
\end{Prop}

\begin{proof}

The proof of Proposition \ref{restrictedtrees2} follows the same lines as the one of Proposition \ref{restrictedtrees}. The additional difficulty is to control the divergence in $R^r$ in the estimate~\eqref{cP2-est} on the recollisions. 
To do so, we decompose the total energy into blocks 
\begin{align*}
& \sum_{a \in \cA_s}  \int \indc _{\cT_{2,s}} \indc_{\cP_2 (a, p , \sigma)} 
\left( \prod_{i=2}^s\, \big|  (v_{a(i)}(t_i)  - v_i) \cdot  \nu_{i} \big|    \right)  M_\beta^{\otimes s} dT_{2,s}  d\Omega_{2,s}   dV_{2,s}  \\
& \qquad\qquad 
\leq  \sum_{m= 1}^{C |\log |\log \eps||}
\sum_{a \in \cA_s}  \int \indc _{\cT_{2,s}}
\indc_{\cP_2 (a,p, \sigma)}   \indc_{\{ 2^{m-1} \leq |V_s| \leq  2^m \}}\\
& \qquad\qquad  \qquad\qquad \qquad\qquad
\times \left( \prod_{i=2}^s\, \big|  (v_{a(i)}(t_i)  - v_i) \cdot  \nu_{i} \big|    \right)  M_\beta^{\otimes s} dT_{2,s}  d\Omega_{2,s}   dV_{2,s} \, .
\end{align*}
Situations $(i)$ and $(ii)$ in Proposition \ref{lemrecoll} will be dealt with separately. We start with $\cP_2  (a,p, \sigma)$ for $p >2$ which depends only on the configurations at time $t_{\sigma_-}$ and on the parameters labelled by $\sigma$.
Using  Proposition~\ref{estimatelemmacontinuity}, the contribution of the trees after  $\sigma_-$ (without the labels in $\sigma$) is estimated as in \eqref{eq: feuilles arbre}
\begin{equation}
\label{eq: feuilles arbre bis}
\sum_{( a(j))_{j > \sigma_-} }\left( \prod_{i > \sigma_-, \atop i \not  \in \sigma}  \,   \big| (v_i -v_{a(i)} (t_i))\cdot \nu_i \big| \right)  
M_{5 \beta/ 6}^{\otimes s } (V_s)\leq  (C s)^{s- \sigma_--3}  M_{2\beta/3}^{\otimes \sigma_-} (V_{\sigma_-}) \,.
\end{equation}
Then for any $R = 2^m$, we integrate with respect to the $|\sigma| $ variables indexed by $\sigma$ and get 
\begin{equation}
\label{eq: divergence en R}
 \int  \indc_{\cP_2 (a,p, \sigma)} \indc_{|\hat V| \leq  R } 
\, \prod_{i \in \sigma} \, 
\big|  (v_{a(i)}(t_i)  - v_i) \cdot  \nu_{i} \big|  \,  dT_\sigma d\Omega_\sigma dV_\sigma    \leq  \eta \, R^r, 
\end{equation}
 with $\hat V = \{  v_i \quad i \leq \sigma_- \ \text{and} \ i \in \sigma\}$.
The main difference with the proof of Proposition \ref{restrictedtrees}  is that we use once again 
the Maxwellian tails to get 
\begin{equation}
\label{eq: control R}
\sup_{V_s}  \Big\{ \indc_{ \{ R/2 \leq |V_s| \leq  R \}} R^r  \, M_{\beta/6} ^{\otimes s }(V_s)  \Big\}
\leq C^s \exp ( - CR^2) \,,
\end{equation}
for some constant $C$ depending only on $r$ and $\beta$. This controls the divergence in $R$ arising in~\eqref{eq: divergence en R}.
An additional factor $s^{| \sigma|}$ takes into account the choices $a(i)$ for the labels $i$ belonging to~$\sigma$.
Finally the contribution of the collision trees before $\sigma_-$ can be estimated by Proposition~\ref{estimatelemmacontinuity} loosing an additional $\beta/6$ in the exponential weight.

\medskip

The contribution of the sets $\cP_2( a,p,\sigma)  
$ for $p \leq 2$ can be estimated with minor changes in the order of integration in order to decouple the contraints on both recollisions. First the combinatorics of the trees after $\sigma_+$ is estimated, then the geometric constraint  at $\sigma_+$  is evaluated. Then  the combinatorics of the trees up to $\max(\sigma \setminus \{\sigma_+\})$, and the (independent) constraints to be  $\cP_1(a,p,\sigma \setminus \{\sigma_+\})$ . Finally it remains to take into account  the contribution of  the rest of the labels.
The large velocities are bounded also as in \eqref{eq: control R}.

\medskip

The final step is to integrate with respect to the remaining time variables. We only retain the condition for the times $(t_i)_{i \notin \sigma}$.
\begin{itemize}
\item
In the first case, we get a simplex of dimension $s-1-|\sigma| $, the volume of which is
$${t^{s-1-|\sigma|}\over (s-1-|\sigma|)!} \leq C^{s}  { t^{s-1-|\sigma|}\over s^{s-1-|\sigma|}} \,,$$
by Stirling's formula.

\item In the second case, we have to add the condition that the last $n$ times have to be in an interval of length $h\leq 1 $.
The worst situation is when all times $(t_i)_{i\in \sigma}$ are in this small time interval, as we loose the corresponding smallness. More precisely, we get
$${t^{s-1-n}\over (s-1-n)!}{h^{n-|\sigma|}\over (n-|\sigma|)!} \leq C^{s}  { t^{s-1-n} h^{n-|\sigma|} \over s^{s-1-|\sigma|}} \,\cdotp$$
\end{itemize}
We thus conclude that for any $R$,
 $$ 
 \begin{aligned}
 \sum_{a \in \cA_s}   \int \indc _{\cT_{2,s}} \indc_{\cP_2 (a,p, \sigma) }    \indc_{ \{ R/2 \leq |V_s| \leq  R \}}  \left( \prod_{i=2}^s\, \big|  (v_{a(i)}(t_i)  - v_i) \cdot  \nu_{i} \big|   \right)  M_\beta^{\otimes s} dT_{2,s}  d\Omega_{2,s}   dV_{2,s}  \\
 \leq   s^4 (Ct)^{s-1}  \eta e^{-CR^2}  M_{\frac{\beta}2}(v_1)  \, , 
 \end{aligned}
 $$
and 
$$\begin{aligned}
  \sum_{a \in \cA_{s}}   \int \indc _{\cT_{2,s}}
&  \indc_{\cT_{s-n+1, s}^h} \indc_{\cP_2 (a,p, \sigma) }   \indc_{ \{ R/2 \leq |V_s| \leq  R \} }  \\
& \qquad \times  \left( \prod_{i=2}^{s} \, \big|  (v_{a(i)}(t_i)  - v_i) \cdot  \nu_{i} \big|   \right)  
  M_\beta^{\otimes s } dT_{2,s}  d\Omega_{2,s}   dV_{2,s} \\
&  \qquad \qquad \leq s^{| \sigma|}  (Ct)^{s-n-1}  (Ch)^{n- | \sigma|} \eta e^{-CR^2}M_{\frac{\beta}2}(v_1)  \, ,
 \end{aligned}
 $$
where all constants are independent of $R$. 
The factor $s^{| \sigma|}$ comes from the summation over the possible choices of $( a(i) )_{i \in \sigma}$.
Finally, the result follows by summing over $R= 2^m$.

\end{proof}

\begin{proof}[Proof of Proposition~{\rm\ref{propR>}}]
The occurrence of multiple recollisions in a collision tree of size $s$ can be estimated by summing over 
all the possible $\sigma$ and using Proposition \ref{restrictedtrees2} 
\begin{align*}
\sum_\sigma \sum_{a \in \cA_{s}}   \int \indc _{\cT_{2,s}}
 \indc_{\cT_{s-n+1, s}^h} \indc_{\cP_2(a,p, \sigma)  }  
 \left( \prod_{i=2}^{s}  \, \big|  (v_{a(i)}(t_i)  - v_i) \cdot  \nu_{i} \big|    \right) {  \indc_{{\mathcal V}_{s} }  } f_N^{(s)} (t - k h) dT_{2,s}  d\Omega_{2,s}   dV_{2,s}\\  
 \leq N \exp(C\alpha^2) s^{12}  (Ct)^{s-n-1}  (Ch)^{n-5} \, t^r \eps \,  M_{\frac{\beta}2}(v_1)  \, ,
\end{align*}
where the a priori $L^\infty$-bound \eqref{Linfty} has been used.
The factor~$s^{12}$ comes from the contribution~$s^2$ in the definition of~$\eta$,~$s^5$ in  Proposition \ref{restrictedtrees2}, and
from the fact that there are at most~$O(s^5)$ choices for the  elements of $\sigma$.


   \smallskip
 
Choosing~$n_k =   2^k n_0$, we then have, since~$\alpha^2 t h \ll 1$,  
\begin{align*}
\Big|  R_N^{K,>}(t,z_1) \Big| &\leq   M_{\frac{\beta}2}(v_1) N{\eps t^r \over h^5} \exp(C\alpha^2)  \sum_{k=1}^K  \sum_{j_1=0}^{n_1-1} \! \! \dots \! \! \sum_{j_{k-1}=0}^{n_{k-1}-1}\sum_{j_k= n_k}^{N-J_{k-1}} 
 (C   \alpha t)^{J_{k-1}} (C \alpha h)^{j_k }    J_k^{12}   \\
& \leq   M_{\frac{\beta}2}(v_1)  \exp(C\alpha^2){t^r   \over h^5}  \sum_{k=1}^K  \sum_{j_1=0}^{n_1-1} \! \! \dots \! \! \sum_{j_{k-1}=0}^{n_{k-1}-1}  n_k^{12} (C\alpha h)^{n_k} (C\alpha t)^{J_{k-1}} \\
& \leq   M_{\frac{\beta}2}(v_1)  \exp(C\alpha^2){t^r   \over h^5}  \sum_{k=1}^K 2^{k^2} \left( C\alpha^2 ht \right)^{  2^kn_0} \\
& \leq   M_{\frac{\beta}2}(v_1)  \exp(C\alpha^2){t^r   \over h^5}  (C\alpha^2 ht)^{n_0}  \,,
\end{align*}
and Proposition \ref{propR>}  follows with 
$h \leq  {\gamma}/{\exp (C\alpha^2) T^{3}}$ as soon as~$n_0$ is large enough.
Note that this is the only argument in which $n_0$ needs to be tuned.
\end{proof}


\section{Truncation of large velocities}
\label{sec: large velocities} 

In this section, we prove that   collision trees with   large velocities  contribute very little to the
iterated Duhamel series.
As a consequence, the error term $R_N^{K, vel}$ introduced in \eqref{reste cut} vanishes.
This holds also for the analogous term in the Boltzmann hierarchy
$$\begin{aligned}
\bar R ^{K, vel} (t)   :=
   \sum_{j_1=0}^{n_1-1}\! \!   \dots \!  \! \sum_{j_K=0}^{n_K-1} 
\bar Q_{1,J_1} (h )  \, \bar Q_{J_1,J_2} (h )
 \dots  \bar  Q_{J_{K-1},J_K} (h ) \, \left(  f^{(J_K)}_{0}    
 \indc_{|V_{J_K}|^2 > C_0 |\log \eps| }\right)    \\
+     \sum_{k=1}^K \; \sum_{j_1=0}^{n_1-1} \! \! \dots \! \! \sum_{j_{k-1}=0}^{n_{k-1}-1}\sum_{j_k \geq n_k} \; 
 \bar Q_{1,J_1} (h ) \dots  \bar Q_{J_{k-1},J_{k}} (h ) 
 \left( f^{(J_k)}(t-kh )  \indc_{|V_{J_k}|^2 > C_0 |\log \eps| }\right) .
\end{aligned}
$$
The contribution of the large energies can be estimated  by the following result.
\begin{Prop} 
\label{lem: truncation}
There exists a constant $C_0 \geq 0 $ such that for all $t\in [0,T]$ and  
$\alpha^2 h T \ll 1$
$$
\begin{aligned}
\left|  R_N^{K, vel} (t) \right|+ \left| \bar R ^{K, vel} (t) \right|
 \leq  \exp(C\alpha^2)n_0^K 2^{K^2}Ã(C\alpha  T)^{n_0 \cdot 2^K} \eps
M_{\beta/2}(z_1) ,
 \end{aligned}
$$
with the sequence $n_k =    2^kn_0$.
\end{Prop}

\begin{proof}
The remainders $R_N^{K, vel}$ \eqref{reste cut} and $\bar R ^{K, vel}$ are made of two contributions, the first one is an energy cut-off for the Duhamel series up to time 0 (with a number of collisions less than~$2^Kn_0$) and the second one is a truncation at an intermediate time corresponding to a large number of collisions.
We shall   consider only the BBGKY hierarchy as~$\bar R ^{K, vel}$ can be treated similarly.

\medskip

For the Duhamel series up to time 0, we notice that for $C_0$ large enough
$$ 
\begin{aligned}
\left|  f^{(J_k)}_N(0)  \indc_{|V_{J_k}|^2 \geq C_0 |\log \eps| } \right|
& \leq  C^{J_k} N  M^{\otimes J_k}_\beta  \indc_{|V_{J_k}|^2 \geq C_0 |\log \eps| }
\; \|g_{\alpha,0}\|_{L^\infty(\D)}
\\
& \leq  \exp(C\alpha^2) C^{J_k} N M_{5 \beta/ 6 }^{\otimes J_k}  \exp \left( - \frac{\beta}{12}  |V_{J_k}|^2\right)  \indc_{|V_{J_k}|^2 \geq C_0 |\log \eps| }\\
& \leq  \eps  \exp(C\alpha^2) C^{J_k} M_{5 \beta/6}^{\otimes J_k} \, .
\end{aligned}
$$
 Then using the fact that 
$$
\big| Q_{1,J_1} (h )Q_{J_1,J_2} (h )
 \dots  Q_{J_{k-1},J_k} (h ) \big| \leq  | Q_{1, J_{k-1}}| (t) \,| Q_{J_{k-1}, J_k } | (h)\, , 
$$
 together with Proposition \ref{estimatelemmacontinuity}, we get
 $$
  \begin{aligned}
 \sum_{k=1}^K  \sum_{j_i <n_i } &  \left| Q_{1,J_1} (h )Q_{J_1,J_2} (h )
 \dots  Q_{J_{k-1},J_k} (h ) \, \big( f^{(J_k)}_N(0)  \indc_{|V_{J_k}|^2 \geq C_0 |\log \eps| }\big) \right|\\
 & \leq  \exp(C\alpha^2)  \sum_{k=1}^K  \sum_{j_i <n_i } (C\alpha  t)^{J_{K}}  \eps M_{ \beta/2} (z_1)\\
& \leq \exp(C\alpha^2)  n_0^K 2^{K^2} \,  (C\alpha  T)^{  2^{K+1} n_0}   \eps  M_{ \beta/2} (z_1)\,.
 \end{aligned}
 $$

From the maximum principle \eqref{Linfty-marginals}, we further  deduce that for $C_0$ large enough
$$ 
\big |  f_{N}^{(J_k)} (t - k h)  \indc_{|V_{J_k}|^2
\geq C_0 |\log \eps| }  \big| 
  \leq  \eps  \exp(C\alpha^2) C^{J_k} M_{5 \beta/6}^{\otimes J_k} ,
$$
so
\begin{align*}
 \sum_{j_k\geq n_k} &  \left|  Q_{1,J_1} (h )Q_{J_1,J_2} (h )
 \dots  Q_{J_{k-1},J_k} (h ) \, \big( f^{(J_k)}_N(t-kh)  \indc_{|V_{J_k}|^2 
 \geq C_0 |\log \eps| }\big) \right| \\
  & \qquad \qquad  \qquad \qquad   \qquad \qquad 
  \leq  \exp(C\alpha^2) (C\alpha  t)^{J_{k-1}} (C\alpha h )^{n_k} \eps M_{ \beta/2} (z_1)\, ,
\end{align*}
as soon as $\alpha h \ll 1$.

Recalling~(\ref{eq: nk 2k}), since  $\alpha^2 t h \ll 1$,  we can sum the different contributions corresponding to a large number of collisions  
  $$
  \begin{aligned}
 \sum_{k=1}^K  \sum_{j_i <n_i \atop i \leq k-1} & \sum_{j_k\geq n_k} \left|  Q_{1,J_1} (h )Q_{J_1,J_2} (h )
 \dots  Q_{J_{k-1},J_k} (h ) \, \big( f^{(J_k)}_N(t-kh)  \indc_{|V_{J_k}|^2 \geq C_0 |\log \eps| }\big) \right|\\
 & \leq  \exp(C\alpha^2)  \sum_{k=1}^K  \sum_{j_i <n_i \atop i \leq k-1} (C\alpha  t)^{J_{k-1}} (C\alpha h )^{n_k} \eps M_{ \beta/2} (z_1)\\
 &\leq  \exp(C\alpha^2)  \,  \, \eps \sum_{k=1}^K n_0^k 2^{k^2} (C\alpha  ^2 t h )^{n_0 \, 2^k} M_{ \beta/2} (z_1)\\
& \leq  \exp(C\alpha^2)  \,  \eps  M_{ \beta/2} (z_1)\,.
 \end{aligned}
 $$
 
 Combining both estimates concludes the proof of Proposition \ref{lem: truncation}.
\end{proof}

\section{End of the proof of Theorem \ref{long-time}, and open problems}
\label{sec: conclusion}

\subsection{Proof of Theorem \ref{long-time}}
In this section, we gather all the error estimates obtained in the previous section and conclude the proof of   Theorem \ref{long-time}. Fix $T>1$ and $t \in [0,T]$.

\medskip

We recall that due to~(\ref{decompositionmarginal}) and~(\ref{reste}) we have
$$
  f_N^{(1)} (t) = f^{(1,K)}_N(t) +  R_N^{K} (t)
$$
and
$$
 R_N^{K} (t)=R_N^{K,0} (t)+R_N^{K,1} (t)+R_N^{K,>} (t) + R_N^{K,vel} (t) \, .
$$
Similarly
$$
    f^{(1)} (t) = \bar f ^{(1,K)}(t) +  \bar R ^{K} (t)+  \bar R ^{K,vel} (t) \, .
$$

\bigskip

From Proposition \ref{main-prop}, we know that  the difference between the dominant parts is 
$$
\Big\| f^{(1,K)} _N(t)-  \bar f^{(1,K)} (t)  \Big\|_{L^2} \leq    (C\alpha  T)^{2^{K+1} n_0} \exp  (C\alpha^2 ) \Big(  \eps | \log\eps|^{10}+ {\eps \over \alpha }    \Big)   \, .$$
This contribution will be small provided that the number of collisions is bounded by
\begin{equation}
\label{cond1}
K = {T\over h} \ll \log | \log \eps | ,\qquad \alpha \ll  \sqrt{\log | \log \eps |}\,.
\end{equation}

\bigskip

Let us now gather the estimates for the remainders, under the assumption that
\begin{equation}
\label{cond2}
h \leq  \frac{\gamma^2}{\exp (C\alpha^2) T^3 },
\end{equation}
for some $C$ large enough.

\medskip

 By Propositions \ref{RNK0} and \ref{RNK0 boltz}, we have
$$
\Big\| R_N^{K, 0} (t) \Big\|_{L^2}   \leq \gamma   
\quad \text{and} \quad
\Big\| \bar R ^K (t) \Big\|_{L^2(\D)}  \leq \gamma\,.
$$
  By Proposition \ref{RNK1}, the remainder for 1 recollision is bounded by 
$$
\Big\| R_N^{K, 1}  (t) \Big\|_{L^2} 
\leq    \eps^{1/ 2}  | \log \eps |^{6} {\gamma \over h}   .
$$
 From Proposition \ref{propR>}, the remainder for multiple recollisions is bounded by 
$$
\Big\| R_N^{K,>} (t) \Big\|_{L^2} 
 \leq \gamma\,.  $$

\medskip

 By Proposition~\ref{lem: truncation} the remainders for large velocities satisfy, as soon as
$\alpha h \ll 1$,
$$
\begin{aligned}
\Big\|  R_N^{K, vel} (t) \Big\|_{L^2} 
 +\Big\| \bar R ^{K, vel} (t) \Big\|_{L^2}   \leq  \exp(C\alpha^2) n_0^K 2^{K^2} (C\alpha  T)^{   2^{K+1} n_0} \eps  \,  ,
 \end{aligned}
$$
which  is small under~(\ref{cond1}), (\ref{cond2}).

\medskip

 The convergence estimate \eqref{eq: approx temps gd} is then obtained by combining conditions (\ref{cond1}) and (\ref{cond2})
\begin{align*} 
\Big\| f_N^{(1)}(t)-   f^{(1)}(t)  \Big\|_{L^2} &\leq  
\Big\| f^{(1,K)} _N(t)- \bar f ^{(1,K)} (t)  \Big\|_{L^2} 
+ \Big\| R_N^{K, 0} (t) \Big\|_{L^2}  + \Big\| \bar R ^K (t) \Big\|_{L^2} \\
&\qquad \qquad    + \Big\| R_N^{K, 1}  (t) \Big\|_{L^2} 
+ \Big\| R_N^{K,>} (t) \Big\|_{L^2} + \Big\|  R_N^{K,vel} (t) \Big\|_{L^2}  + \Big\| \bar  R ^{K,vel} (t) \Big\|_{L^2}  \\
&\leq  \frac{ \exp (C\alpha^2) T^2}{\sqrt {\log |\log \eps|}} \, \cdotp
\end{align*}
This concludes the proof of Theorem \ref{long-time}.
\qed

\subsection{Open problems}
\label{openpbs} 

In this final section, we collect some open  problems related to those treated in this paper.

\medskip

{\it Finite range potentials.}

We expect the same convergence results to hold if microscopic interactions are described by a repulsive compactly supported potential (instead of the singular hard-sphere interactions).
The proof then involves  truncated marginals and cluster estimates as in \cite{GSRT,PSS}. With the present scaling, there is however a difficulty to control  triple interactions, the size of which is critical (see the computations of Appendix~\ref{geometricallemmasappendix}). Note that the case of a potential, non compactly supported, is rigorously  analyzed for the first time in~\cite{ayi} in the linear case.

\medskip

{\it Higher dimensions.}

We also expect the convergence results to extend to higher dimensions and it has been proven for short times in 
\cite{BLLS}. However, there are two important simplifications in dimension~2. The first one is due to the fact that the inverse partition function associated with the exclusion is bounded uniformly in $N$, as shown in~(\ref{eq: ZN});  in particular this makes it possible to propagate somehow the initial form of the initial datum and to decompose the marginals of the solution in a quasi-orthogonal form; see Section~\ref{decompositionL2andCss+1}. The second one is related to the control of recollisions: we have seen in this paper (namely in Section~\ref{recollisions}) that the probability of having pseudo-dynamics with multiple recollisions is~$O(\eps)$, which balances exactly  the~$O(N)$ size of the~$L^\infty$ norm of the solution, and that  is not the case in higher dimension in the Boltzmann-Grad scaling since~$\eps \sim N^{\frac1{1-d}}$.

\medskip

{\it Spatial Domain.} 

The spatial domain we consider here is the torus $\T^2$, which is equivalent to a rectangular box with specular reflection on the boundary. To extend the analysis to more general domains, we would need a geometric property of the free flow on these domains, stating roughly  that the probability, in velocities, for two trajectories to approach at a distance $\eps$ on a fixed time interval $[0,T]$  is vanishing 
  in the limit $\eps \to 0$.

\medskip

{\it Dissipation.}  

The control on the higher order cumulants $g_N^m$ is  the key to improve the convergence time
with respect to Lanford's original argument. 
This estimate can be seen as playing   the role of the dissipation on the limiting equation.
We indeed have
$$
\frac1N\int {f_N^2 (t)\over M^{\otimes N}_\beta} dZ_N  =   \| g_N^1(t) \|^2_{L^2_\beta (\D)}
+ \sum_{m=2}^N { \binom{N}{m}\over N}  \| g_N^{m} (t) \|^2_{L^2_\beta (\D^m)}=\frac1N\int {f_{N,0}^2 \over M^{\otimes N}_\beta} dZ_N\,.
$$
to be compared to 
$$\| g(t) \|^2_{L^2_\beta (\D)}+\alpha \int_0^t\int M_\beta  g\cL_\beta g (s,x,v)dvdxds = \| g_0 \|^2_{L^2_\beta (\D)}$$
for the limiting equation.
\medskip

{\it Stochastic corrections.}

 In \cite{S3}, Spohn studied the stochastic fluctuations around the Boltzmann equation and computed the variance of the fluctuation field in a non-equilibrium state
\begin{eqnarray*}
\zeta^N (g,t) = \frac{1}{\sqrt{N}} \left( \chi^N(g,t) - \la  \chi^N(g,t) \ra \right)
\quad \text{with} \quad
\chi^N(g,t) = \sum_{i =1}^N g( z_i(t) ) \, ,
\end{eqnarray*}
where $g$ is a smooth function and $\la \cdot \ra$ stands for the mean.
It would be of great interest to prove that the limiting field is Gaussian and to derive, even for short times, the fluctuating hydrodynamics.

\appendix

\section{The linearized Boltzmann equation and its fluid limits}
\label{appendixBoltzlin}

For the sake of completeness, we recall here some by now classical results about the linearized Boltzmann equation  \eqref{lBoltz}
\begin{equation}
\label{lBoltzq}
		  \begin{aligned}
		 &{1\over \alpha^q}  \d_t g _\alpha +v\cdot \nabla_x g_\alpha= - \alpha\cL _\beta g_\alpha\\
		& \cL_\beta g (v) = \int M_\beta (v_1) \Big(g (v)+g (v_1) - g (v') - g (v_1')\Big)\big((v_{1}-v) \cdot \nu\big)_+ d\nu dv_{1} 
		 \end{aligned}
\end{equation}
and its hydrodynamic limits as $\alpha \to \infty$ (for $q=0,1$). 
The results below are valid in any dimension~$d \geq 2$, thus contrary to the rest of this article, 
we assume the space  dimension to be~$d$.
 
Because of the  scaling invariance of the collision kernel, we shall actually restrict our attention 
in the sequel to the case where  $M_\beta$ is the reduced centered Gaussian, i.e.
$\beta = 1$ (and we omit the subscript~$\beta$ in the following).
The collision operator \eqref{lBoltzq} will be denoted by $\cL$.

\subsection{The functional setting}

The linearized Boltzmann operator $\cL$ has been studied extensively (since it governs small solutions of the nonlinear Boltzmann equation). In the case of non singular cross sections, its spectral structure was described by Grad~\cite{grad}.
The main result is that it 
satisfies the Fredholm alternative
in a weighted $L^2$ space. In the following we define the collision frequency
 $$
 a(|v|) := \int M(v_1) \big((v_{1}-v) \cdot  \nu\big)_+ d\nu dv_{1} 
 $$ 
  which satisfies, for some $C>1$,
$$
0<a_- \leq a(|v|)\leq C(1+|v|) \,.
$$

\begin{Prop}\label{coercivity} 
The linear collision
operator $\cL$ defined by {\rm(\ref{lBoltzq})} is a nonnegative unbounded 
self-adjoint operator on $L^2(Mdv)$  with domain
$$
\cD(\cL)=\{g\in L^2(Mdv)\,|\, a g\in L^2(Mdv)\}=L^2(\R^d;a M(v)dv)
$$
and nullspace
$$
\Ker(\cL)=\Span\{1,v_1,\dots,v_d,|v|^2\}\,.
$$
Moreover the following coercivity estimate holds: there exists $C>0$ such that, for 
each~$g$ in~$\cD(\cL)\cap(\Ker(\cL))^\perp$
$$
\int g\cL g(v)M(v)dv\geq C\|g\|_{L^2(a M dv)}^2\,.
$$
\end{Prop}

\begin{proof}[Sketch of proof]

$\bullet$
The first step consists in characterizing the nullspace of $\cL$. It must contain
the collision invariants since the integrand in
$
\cL g$
vanishes identically if $g(v) =1,v_1,v_2 , \dots ,v_d$ or $|v|^2$. Conversely, from the identity,
$$
 \int\psi\cL g Mdv=
\frac14\int (\psi\!+\!\psi_1\!-\!\psi'\!-\!\psi'_1)(g\!+\!g_1\!-\!g'\!-\!g'_1) 
\big((v_{1}-v) \cdot \nu\big)_+Mdvdv_1d\nu\,,
$$
where we have used the classical notation
$$
g_1:= g(v_1) \, , \quad g' = g(v') \ , \quad g'_1=g(v'_1) \, , 
$$
we deduce that,  if $g $ belongs to the nullspace of $\cL$, then
$$
g+g_1=g'+g'_1 \, ,
$$
which
entails that $g$ is a linear combination of $1,v_1$, $v_2,\dots ,v_d$ and $|v|^2$ (see for instance~\cite{perthame}).

Note that the same identity shows that $\cL$ is self-adjoint.

$\bullet$
In order to establish the coercivity of the linearized collision operator 
$\cL$, the key step is then  to introduce
 Hilbert's decomposition \cite{hilbert}, showing that $\cL$ is   a compact perturbation of 
a multiplication operator~:
$$
\cL g (v)=a(|v|)g(v)-\cK g (v) \, .
$$

  Proving that $\cK$ is a compact integral operator on $L^2(Mdv)$  relies on  intricate computations using Carleman's parametrization of collisions (which we   also use in this paper for the study of  recollisions). We shall not perform them here (see~\cite{hilbert}).

  Because $a$ is bounded from below, $\cL$ has  a spectral gap, which provides the coercivity estimate.
\end{proof}
 Proposition  \ref{coercivity}, along with classical results on maximal accretive operators, imply the following statement.
\begin{Prop}\label{Cauchy-pb} Let $g_0 \in L^2(M  dvdx)$. Then, for any fixed $\alpha$, there exists a unique solution $g_\alpha \in C(\R^+, L^2(Mdvdx))\cap C^1( \R^+_*, L^2(M dvdx)) \cap  C( \R^+_*, L^2(M a dvdx)) $ to the linearized Boltzmann equation {\rm(\ref{lBoltzq})}. It satisfies  the scaled energy inequality
\begin{equation}\label{scaledenergyinequality} \| g_\alpha(t)\|^2_{L^2(Mdvdx)} +\alpha^{1+q} \int_0^t \int g_\alpha \cL g_\alpha (t') Mdvdx dt'\leq \| g_0\|_{L^2(Mdv)}^2\,.
\end{equation}
\end{Prop}

\subsection{The acoustic and Stokes limit}
The starting point for the study of hydrodynamic limits is the   energy inequality~(\ref{scaledenergyinequality}).
The uniform $L^2$ bound on $(g_\alpha)$ implies that, up to extraction of a subsequence,
\begin{equation}
\label{L2-cv}
g_\alpha \rightharpoonup g \hbox{ weakly in } L^2_{loc} (dt,L^2(Mdvdx))\,.
\end{equation}
Let $\Pi$ be the orthogonal projection on the kernel of $\cL$.
The dissipation, together with the coercivity estimate in Proposition  \ref{coercivity}, further provides
$$ \| g_\alpha -\Pi g_\alpha \|_{L^2(Ma dvdxdt)} = O(\alpha^{-(q+1)/2}) \, ,$$
from which we deduce that
\begin{equation}
\label{ansatz}
 g (t,x,v) =\Pi g(t,x,v) \equiv \rho(t,x) + u(t,x)\cdot v +\theta(t,x) {|v|^2-d\over 2}\, \cdotp
 \end{equation}

\bigskip
  If the Mach number $\alpha^q$ is of order 1, i.e. for $q=0$, one obtains asymptotically the acoustic equations. Denoting by $\la \cdot \ra$ the average with respect to the measure $Mdv$, we indeed have the following conservation laws
$$
\begin{aligned}
\d_t \la g_\alpha \ra +\nabla_x \cdot \la g_\alpha  v\ra=0 \, ,\\
\d_t \la g_\alpha v\ra +\nabla_x \cdot \la g_\alpha  v\otimes v\ra=0 \, ,\\
\d_t \la g_\alpha |v|^2 \ra +\nabla_x \cdot \la g_\alpha  v|v|^2\ra=0\,.
\end{aligned}
$$
From (\ref{L2-cv}) and (\ref{ansatz}) we then deduce that~$(\rho, u , \theta)$ satisfy
\begin{equation}
\label{acoustic}
\begin{aligned}
\d_t\rho  +\nabla_x \cdot u=0 \, ,\\
\d_t u +\nabla_x (\rho+\theta) =0 \, ,\\
\d_t \theta  +  \frac{2}{d}  \nabla_x \cdot u =0\,.
\end{aligned}
\end{equation}
By uniqueness of the limiting point, we get the convergence of the whole family $(g_\alpha)_{\alpha>0}$.

  Since the limiting distribution $g$ satisfies the energy equality
$$ \| g\|^2_{L^2(Mdvdx)} = \|  g_0\|_{L^2(Mdv dx)}^2$$
or equivalently  $$ \| g\|^2_{L^2(Mdvdx)} +\alpha \int_0^t \int g \cL g Mdvdx = \| \Pi g_0\|_{L^2(Mdv dx)}^2\,,$$
convergence is strong as soon as $g_0 =\Pi g_0$. 
We thus have the following result (see  \cite{golse-levermore} and references therein).

\begin{Prop}
Let $g_0\in L^2(Mdvdx)$. For all $\alpha$, let $g_\alpha$ be a solution to the scaled linearized Boltzmann equation {\rm(\ref{lBoltzq})} with $q=0$. Then, as $\alpha \to \infty$, $g_\alpha$ converges weakly in~$L^2_{loc} (dt,L^2(Mdvdx))$ to the infinitesimal Maxwellian 
$\displaystyle g=\rho+u\cdot v +\frac12 \theta (|v|^2-d)$ 
where  $(\rho, u,\theta)$ is the solution of the acoustic equations {\rm(\ref{acoustic})} with initial datum 
$\displaystyle (\la g_0\ra, \la g_0 v\ra, \la g_0 \frac1d(|v|^2-d)\ra)$. 

The convergence holds strongly in $L^\infty_t(L^2(Mdvdx))$ provided that $g_0=\Pi g_0$.
\end{Prop}

\bigskip
In the diffusive regime, i.e. for $q=1$, the moment equations state
$$
\begin{aligned}
{1\over \alpha} \d_t \la g_\alpha \ra +  \nabla_x \cdot \la g_\alpha  v\ra=0 \, ,\\
{1\over \alpha} \d_t \la g_\alpha v\ra + \nabla_x \cdot \la g_\alpha  v\otimes v\ra=0 \, ,\\
{1\over \alpha} \d_t \la g_\alpha |v|^2 \ra + \nabla_x \cdot \la g_\alpha  v|v|^2\ra=0\,.
\end{aligned}
$$
From (\ref{L2-cv}) and (\ref{ansatz}) we deduce that
$$\nabla_x \cdot u = 0 \, , \quad\nabla_x (\rho+\theta) =0\,,$$
referred to as incompressibility and Boussinesq constraints.

To characterize the mean motion, we then have to filter acoustic waves, i.e. to project on the kernel of the acoustic operator
$$
\begin{aligned}
\d_t P \la g_\alpha v\ra + \alpha P\nabla_x \cdot \la g_\alpha ( v\otimes v-\frac12 |v|^2 Id)\ra=0 \, ,\\
 \d_t \la g_\alpha({|v|^2 \over d+2} - 1 )\ra +\alpha \nabla_x \cdot \la g_\alpha   v({|v|^2 \over d+2} - 1 )\ra=0\,,
\end{aligned}
$$
where $P$ is the Leray projection on divergence free vector fields.
Define the kinetic momentum flux $\displaystyle\Phi(v) = v\otimes v-\frac{1}{d} |v|^2 Id$ and the kinetic energy flux $\displaystyle \Psi(v) =  \frac1{d+2} v(|v|^2-d-2)$. As~$\Phi, \Psi $ belong to~$ (\Ker \cL)^\perp$, and $\cL$ is a Fredholm operator, there exist
pseudo-inverses $\tilde \Phi, \tilde \Psi \in (\Ker \cL)^\perp$ such that $\Phi = \cL \tilde \Phi$ and $\Psi = \cL \tilde \Psi$. Then,
$$
\begin{aligned}
\d_t P \la g_\alpha v\ra + \alpha P\nabla_x \cdot \la \cL g_\alpha \tilde \Phi \ra=0 \, ,\\
 \d_t  \la g_\alpha ({|v|^2 \over d+2} - 1 )\ra +\alpha \nabla_x \cdot \la \cL g_\alpha  \tilde \Psi\ra=0\,.
\end{aligned}
$$
Using the equation
$$\alpha \cL g_\alpha = -v\cdot \nabla_x g_\alpha -\frac1\alpha \d_t g_\alpha $$
the Ansatz (\ref{ansatz}), and taking limits in the sense of distributions, we get
\begin{equation}
\label{Stokes}
\begin{aligned}
\nabla_x \cdot u = 0, \quad \nabla_x (\rho+\theta) =0 \, ,\\
\d_t  u - \mu  \Delta_x u=0 \, ,\\
 \d_t \theta  - \kappa  \Delta_x \theta =0\,.
\end{aligned}
\end{equation}
These are  exactly the Stokes-Fourier equations with 
$$
\mu = \frac{1}{(d-1)(d+2)} \la \Phi : \tilde \Phi \ra 
\quad \text{and} \quad 
\kappa = \frac{2}{d(d+2)} \la \Psi \cdot \tilde \Psi \ra.
$$

As previously, the limit is unique and the convergence is strong provided that the initial datum is well-prepared, i.e. if 
\begin{equation}
\label{wellp}
g_0 (x,v) =u_0\cdot v +\frac12 \theta_0(|v|^2-(d+2)) \hbox{ with } \nabla_x \cdot  u_0 = 0 \,.
\end{equation}
One can therefore prove the following result.
\begin{Prop}
Let $g_0\in L^2(Mdvdx)$. For all $\alpha$, let $g_\alpha$ be a solution to the scaled linearized Boltzmann equation {\rm(\ref{lBoltzq})} with $q=1$. Then, as $\alpha \to \infty$, $g_\alpha$ converges weakly in~ $L^2_{loc} (dt,L^2(Mdvdx))$ to the infinitesimal Maxwellian $\displaystyle g=u\cdot v +\frac12 \theta (|v|^2-(d+2) )$ where~$( u,\theta)$ is the solution of~{\rm(\ref{Stokes})} with initial datum $(P \la g_0 v\ra, \la g_0 ({|v|^2 \over d+2} - 1 )\ra)$. 

  The convergence holds in $L^\infty_t(L^2(Mdvdx))$ provided that the initial datum is well-prepared in the sense of {\rm(\ref{wellp})}.
\end{Prop}

\begin{Rmk}
In both cases, the defect of strong convergence for ill-prepared initial data can be described precisely.

  If the initial profile in $v$ is not an infinitesimal Maxwellian, i.e. if $g_0\neq \Pi g_0$, one has a relaxation layer of size $\alpha^{-(1+q)}$ governed essentially by the homogeneous equation
$$\d_t \Pi_\perp g_\alpha = -\alpha^{q+1} \cL g_\alpha\,.$$

  In the incompressible regime $q=1$, if the initial moments do not satisfy the incompressibility and Boussinesq constraints, one has to superpose a fast oscillating component (with a time scale~$\alpha^{-1}$).
For each eigenmode of the acoustic operator, the slow evolution is given by a diffusive equation.

  A straightforward energy estimate then shows that the asymptotic behavior of $g_\alpha$ is well described by the sum of these three contributions (main motion, relaxation layer and acoustic waves in incompressible regime).
\end{Rmk}


\section{Geometrical lemmas}\label{geometricallemmasappendix}
 
In this appendix, we prove several technical lemmas (namely Lemmas~\ref{v-constraint}, \ref{doublerecollision},~\ref{Murphy} and~\ref{self-recoll-periodic}) which were key steps in Sections~\ref{convergencepartieprincipales} and~\ref{recollisions} in proving Propositions~\ref{recoll1-prop}
and~\ref{lemrecoll}.

\smallskip

In the following we adopt the notation of those sections.

\subsection{A preliminary estimate}

Recall Equation~(\ref{rec-equation}) for the first recollision between $i,j$
\begin{equation}
\label{rec-equation'} 
v_i  - v_{j} = \frac1{\tau_{rec} } \delta x_\perp - \frac{\tau_1}{\tau_{rec}} (\bar v_{i} - v_{j})-  \frac1{\tau_{rec} } \nu_{rec} \,  , 
\end{equation}
with the notations \eqref{eq: tau1}
\begin{equation}
\label{eq: tau1'}
\tau_1 :=- \frac1\eps (t_{1^*} - t_{2^*}+\lambda ) \, ,\qquad  \tau_{rec} :=-\frac1\eps (t_{rec} - t_{1^*}) \, ,
\end{equation}
where
$$
 \frac1\eps ( x_{i} - x_{j}-q) =  {\lambda \over \eps} (\bar v_{i} - v_{j})  + \delta x_\perp \quad \hbox{ with } \quad\delta x_\perp \cdot (\bar v_i - v_{j}) =0 \, .
$$

\medskip

The  distance between particles $i,j$ at the collision time $t_{1^*}$ is given by 
\begin{equation*}
\big|  x_i(  t_{1^*}) - x_j (  t_{1^*}) \big|   = \eps  \big|   \delta x_\perp - \tau_1 (\bar v_{i} - v_{j})  \big|  
= \eps  \sqrt{\big|  \delta x_\perp\big|^2  + \big| \tau_1 (\bar v_{i} - v_{j})  \big|^2 }
\, .
\end{equation*}
The distance between the particles varies with the collision time $t_{1^*}$ and 
the closer they are, the easier  it is to aim (at the collision time $t_{1^*}$) to create  a recollision
at the later time $t_{rec}$.
The key idea is that for relative velocities $\bar v_{i} - v_{j} \not = 0$, the particles will never remain close for a long time so that integrating over $t_{1^*}$ allows us to recover some smallness uniformly over the initial positions at time $t_{2^*}$. 

\medskip
 Suppose $|\tau_1| |\bar v_i - v_j| \leq M$.
Since $v_{1^*}$ is in  a ball of size~$R$, and $\nu_{1^*}$ belongs to ${\mathbb S}$, we have
\begin{equation}
\label{eq: small distance cut off 0}
\int 
 \indc_{ \{ |\tau_1| |\bar v_i - v_j | \leq M \}} 
 |(v_{1^*} -  \bar v_i) \cdot \nu_{1^*}| |\bar v_i-v_j| d \tau_1  dv_{1^*} d\nu_{1^*} \leq C R^2  M \,.
\end{equation}

\medskip

For later purposes, it will be useful to evaluate the integral \eqref{eq: small distance cut off 0} in terms of the integration parameter $t_{1^*}$: we get by the change of variable 
$\tau_1 = ( t_{1^*} - t_{2^*} - \lambda)/\eps$
\begin{equation*}
\int 
\, \indc_{ \{ |\tau_1| |\bar v_i - v_j| \leq M \}} \,
\, |(v_{1^*} -  \bar v_i) \cdot \nu_{1^*}|  d t_{1^*}  dv_{1^*} d\nu_{1^*} \leq C R^2  M 
\frac{\eps}{|\bar v_i-v_j|} \,\cdotp
\end{equation*}
The singularity in $|\bar v_i-v_j|$ translates the fact that the distance between the particles may remain small during a long time if their relative velocity is small.
This singularity can then be integrated out, up to a loss of a $|\log \eps|$, using an additional parent of~$i$ or~$j$ thanks to~(\ref{carleman2})  in Lemma~\ref{scattering-lem}: we obtain
\begin{equation}
\label{eq: small distance cut off prime}
\int \indc_{ \{ |\tau_1| |\bar v_i - v_j| \leq M \}}
 \prod_{k \in \{1^*,2^\}} |(v_k -  v_{a(k)}) \cdot \nu_{k}|  d t_{k}  dv_{k} d\nu_{k }\leq C MR^5 t \eps |\log \eps|\,.
\end{equation}
To get rid of the logarithmic loss, one may use  two extra degrees of freedom associated with the parents 
$2^*, 3^*$ of~$i$ or~$j$:
from \eqref{carleman3}  and \eqref{alpha3}, we   obtain the upper bound 
\begin{equation}
\label{eq: small distance cut off}
\int \indc_{ \{ |\tau_1| |\bar v_i - v_j| \leq M \}}
 \prod_{k \in \{1^*,2^*,3^*\}} |(v_k -  v_{a(k)}) \cdot \nu_{k}|  d t_{k}  dv_{k} d\nu_{k }\leq C MR^8 t^2 \eps \,.
\end{equation}
Note that in the case when~$i$ and~$j$ are colliding at time~$t_{2^*}$, Lemma \ref{joint-scattering-lem} shows that only two integrations are necessary.


\subsection{A recollision with a constraint on the outgoing velocity}

The following lemma deals with the cost of the first  recollision when one of the outgoing velocities is constrained to lie in a given ball. More precisely, if the first recollision occurs between particles $i,j$, and~$k  $ is a given label, we will impose that (see Figure~\ref{v-picture})
\begin{equation}
\label{eq: contrainte proximite}
|v'_i - v_k| \leq \eps^{3/4}
\quad \text{or} \quad |v'_j - v_k| \leq \eps^{3/4}.
\end{equation}
\begin{figure} [h] 
   \centering
  \includegraphics[width=5cm]{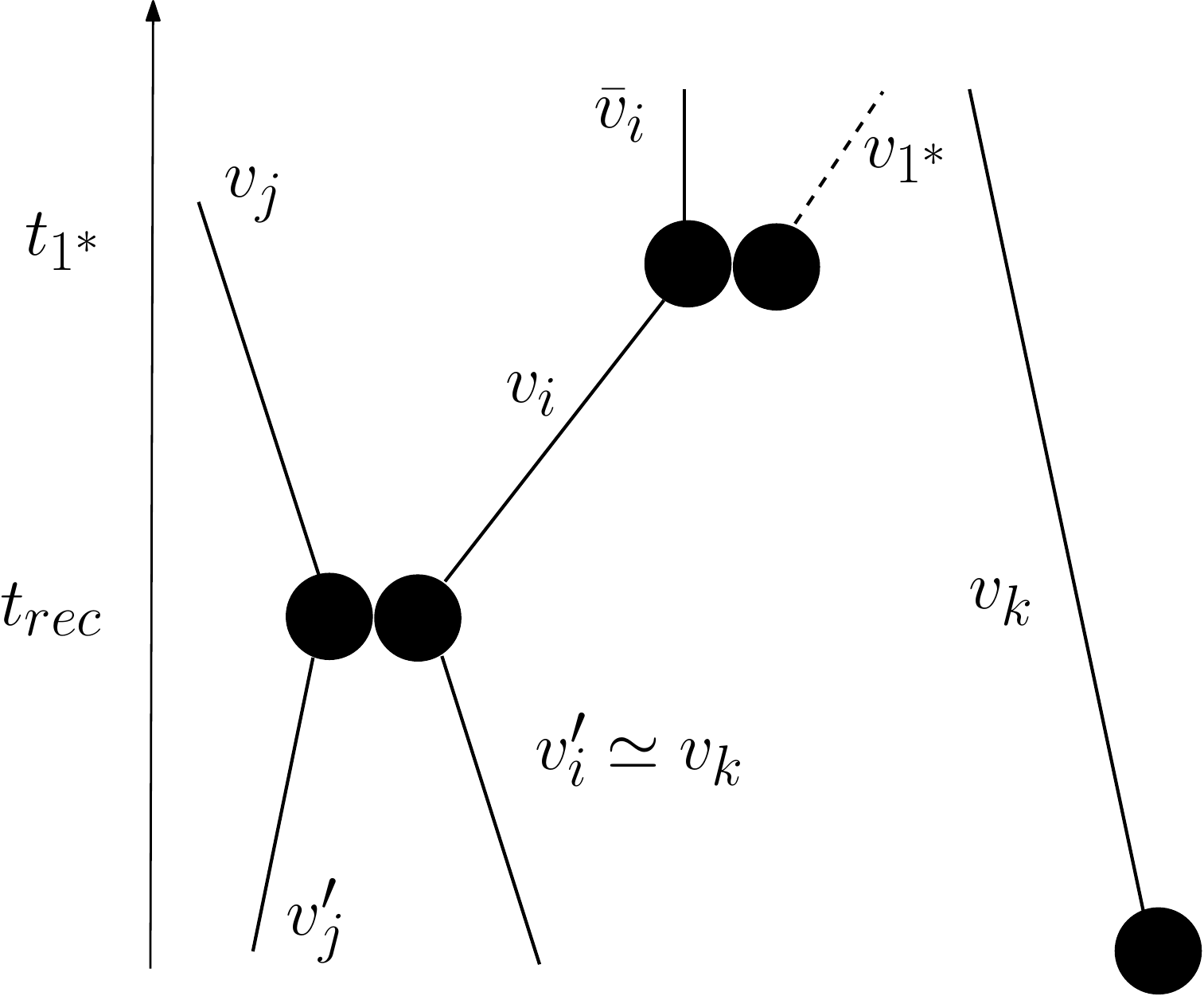} 
   \caption{Small relative velocities $|v'_i - v_k| \leq \eps^{3/4}$ after the first recollision. }
   \label{v-picture}
\end{figure}  
\begin{Lem}
\label{v-constraint}
Fix a final configuration of bounded energy~$z_1 \in \T^2 \times  B_R$ with~$1 \leq R^2 \leq C_0 |\log \eps|$,  a time $1\leq t \leq C_0|\log \eps|$ and  a collision tree~$a \in \cA_s$ with~$s \geq 2$.

There exist sets of bad parameters
$\cP_2 (a, p,\sigma)\subset \cT_{2,s} \times {\mathbb S}^{s-1} \times \R^{2(s-1)}$ for $4 \leq p \leq p_1$ (for some integer $p_1$) and $\sigma \subset \{2,\dots, s\}$ of cardinal $|\sigma|\leq 4$ such that
\begin{itemize}
\item
$\cP_2 (a, p,\sigma)$ is parametrized only in terms of   $( t_m, v_m, \nu_m)$ for $m \in \sigma$ and  $m < \min \sigma$;
\begin{equation}
\label{eq: A recollision with a constraint on the outgoing velocity}
\int \indc _{\cP_2 (a,p,\sigma)  } \displaystyle  \prod_{m\in \sigma} \,
\big | \big(v_{m}-v_{a(m)} ( t_{{m}} ) )\cdot \nu_{{m}}  \big | d  t_{{m}} d \nu_{{m}}dv_{{m}}   
\leq C(Rt)^r s^2 \eps \,,
\end{equation}
for some constant $r$,
\item  and any pseudo-trajectory starting from $z_1$ at $t$, with total energy bounded by $R^2$ and such that the first recollision produces a small relative velocity as in \eqref{eq: contrainte proximite} is parametrized by 
$$( t_n, \nu_n, v_n)_{2\leq n\leq s }\in  \bigcup _{4 \leq p \leq p_1} \bigcup _\sigma  \cP_2(a, p,\sigma)\,.$$
\end{itemize}
\end{Lem}

Before launching into the   details of the proof, we first give the gist of it.
Since the first recollision occurs between $i,j$,    we get a  constraint as in \eqref{rec-equation'} on the velocity 
\begin{equation*}
v_i  - v_{j} 
= \frac1{\tau_{rec} } \delta x_\perp - \frac{\tau_1}{\tau_{rec}} (\bar v_{i} - v_{j})-  \frac1{\tau_{rec} } \nu_{rec} \, ,
\end{equation*}
meaning that $v_i$ belongs to a rectangle of width $ \frac{R}{|\tau_1| |\bar v_{i} - v_{j}|}$  thanks to~(\ref{sizetaurec}),
which after integration over two parents leads to estimate of the type  $\eps | \log \eps|^3$ 
(see Lemma~\ref{rec-eq}). Imposing an extra condition on the velocity $v'_i$ after the first recollision means that the recollision angle $\nu_{rec}$ can take values only in a small set. Thus the constraint above will be stronger and 
 $v_i$ has to take values in a reduced set, much thinner than the rectangle considered in Lemma~\ref{rec-eq}. 
The core of the proof of Lemma \ref{v-constraint} is to identify this reduced set and to show that after integrating over some parents its measure is less than $O(\eps)$: here and in the following, we do not try to keep track of the powers of~$R$ and~$t$ coming up in the estimates.

\begin{proof} 
Throughout the proof, we suppose that the parameters associated with the first recollision between $i,j$
satisfy 
\begin{equation}
\label{eq: gamma}
| (\bar v_{i} - v_{j}) \tau_1| \geq R^{4 \over 4 - 5 \gamma} \geq R^3
\quad \text{for some} \ 
\gamma \in ]\frac{2}{3}, \frac{4}{5}[  \ \text{to be fixed later}.
\end{equation}
Otherwise, the estimate~\eqref{eq: small distance cut off}
applied with $M = R^{4 \over 4 - 5 \gamma}$ leads to a suitable  upper bound of order $\eps$.
As a consequence of \eqref{eq: gamma}, we deduce that $| \tau_1|$ is large enough
\begin{equation}
\label{eq: tau R}
| \tau_1| \geq  R^2 .
\end{equation}

\medskip
After the first recollision, $v'_i$ is given by one of the following formulas
\begin{equation}\label{formulav"i}
\begin{aligned}
v'_i = v_i - (v_i-v_j) \cdot \nu_{rec} \, \nu_{rec}\, ,\\
\mbox{or} \quad v'_i = v_j + (v_i-v_j) \cdot \nu_{rec} \, \nu_{rec} \, .
\end{aligned}
\end{equation}
Note that the second choice is the value $v_j'$ and we use this abuse of notation to describe the case when 
$|v'_j - v_k| \leq \eps^{3/4}$.

\smallskip

We   expect the condition  \eqref{eq: contrainte proximite} 
to impose a strong constraint on 
the recollision angle $\nu_{rec}$. We indeed find from~(\ref{formulav"i}) that this condition implies
\begin{equation}
\label{eq: 2 options}
\begin{aligned}
\mbox{either} \qquad v_k - v_j  =  (v_i-v_j) \cdot \nu_{rec}^\perp \,  \nu_{rec}^\perp  +O(\eps^{3 /4}) \, ,\\
\mbox{or} \qquad  v_k - v_j  =  (v_i-v_j) \cdot \nu_{rec}  \, \nu_{rec} +O(\eps^{3 /4}) \, . 
\end{aligned}
\end{equation}
We  consider now three different cases according to the label $k$.
Each different case listed below will be associated with  scenarios, labelled by some $p$ which will take   values in $4, \dots, p_1$ for some~$p_1$ we shall not attempt to compute.

\medskip

\goodbreak 
 
\noindent
\underbar{Case $k\neq j$ and $k\neq 1^*$}

\medskip
 
\noindent
$\bullet$ If $| v_j-v_k |> \eps^{5/8} \gg  \eps^{3/4}$, we deduce from the constraint \eqref{eq: 2 options} that 
the recollision angle is in a small angular sector
\begin{equation}
\label{eq: plus d'options}
\nu_{rec} =\pm {(v_j - v_k)^\perp\over |v_k - v_j|} + O( \eps ^{1/8}) 
\quad  \hbox{ or }\quad 
\nu_{rec} =\pm  { v_k - v_j  \over |v_k - v_j |} + O( \eps ^{1/8}) \,.
\end{equation}
Plugging this Ansatz in (\ref{rec-equation'}), we get 
$$ 
v_i  - v_{j} = \frac1{\tau_{rec} } \delta x_\perp - \frac{\tau_1}{\tau_{rec}} (\bar v_{i} - v_{j})-  \frac1{\tau_{rec} }{\cR_{n' \pi/2}(v_k - v_j )  \over |v_k - v_j |} + O\left( {\eps ^{1/8}\over \tau_{rec}}\right) \, ,
$$
denoting by $\cR_\theta$ the rotation of angle $\theta$ and  $n' = 0,1,2,3$ depending on the identity in 
\eqref{eq: plus d'options}.
This implies that~$v_i  - v_{j} $ lies in  a finite union of thin rectangles of size~
$2R \times  4R\eps^{1/8} \min \left(1,\frac{1}{|\tau_1||\bar v_i-v_j|} \right)$,  recalling again~(\ref{sizetaurec}).
We thus conclude by integrating in $(t_{1^*}, v_{1^*}, \nu_{1^*})$ and $(t_{2^*}, v_{2^*}, \nu_{2^*})$, exactly as in the proof of Proposition~\ref{recoll1-prop}, that  these configurations are encoded in a set~$\cP_2(a,p, \sigma)$ (with $|\sigma |\leq 2$)   of size~$O( R^7 s t^3 \eps^{9/8}  |\log \eps|^3)$.   

\medskip
 
\noindent
$\bullet$ If $|v_j-v_k| \leq \eps^{5/8}$, we  forget about all other constraints: we conclude by combining \eqref{preimage-sphere1}-\eqref{carleman2}, as in the case 1.2(c) in Section \ref{classificationsection}, that these pseudodynamics are encoded in a set~$\cP_2(a,p,\sigma)$ (with $|\sigma |\leq 2$) of size $O( R^5 s t^2 \eps^{5/4}  |\log \eps|)$.

\bigskip

 \noindent
\underbar{Case $k= j$ }

\medskip
 If  $k= j$, then $|v_k - v_i' |=| v_j'- v_i'| = |v_i- v_j| \leq \eps^{3/4}$, and we conclude exactly as in the previous case by combining \eqref{preimage-sphere1}-\eqref{carleman2}, that these pseudodynamics are encoded in a set~$\cP_2(a,p,\sigma)$ (with $|\sigma |\leq 2$) of size $O( R^5 s t^2 \eps^{3/2}  |\log \eps|)$.

\bigskip

 \noindent
\underbar{Case $k= 1^*$ }

This is the most delicate case as $v_i'$ and $v_k$ are linked through the same collision.
 We stress the fact that the label $i$ refers to a pseudo particle, thus 
 many cases have to be considered (see Figure \ref{fig: k = 1*}).
\begin{figure} [h] 
\includegraphics[width=4cm]{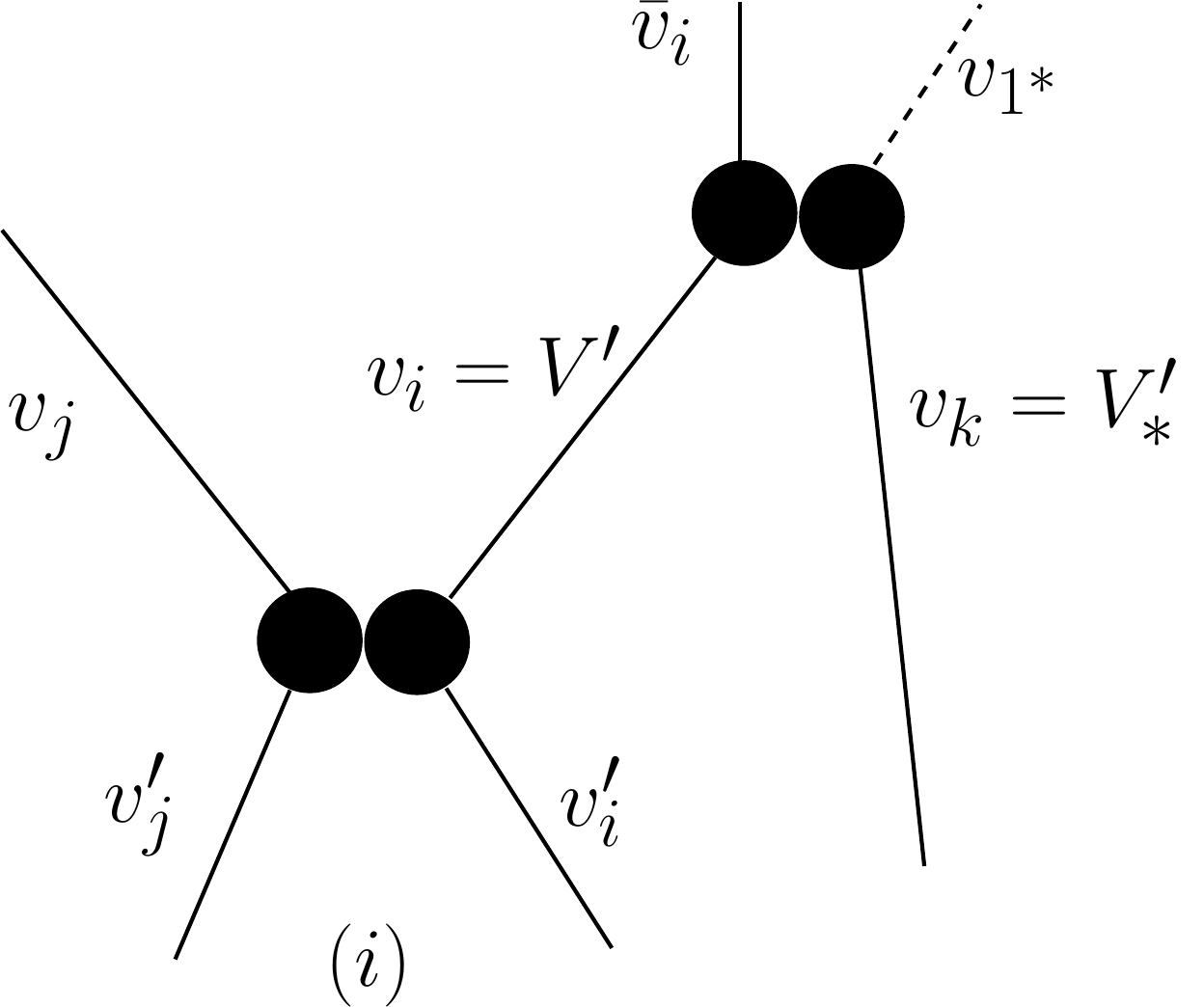} 
\hskip1.5cm
\includegraphics[width=3.2cm]{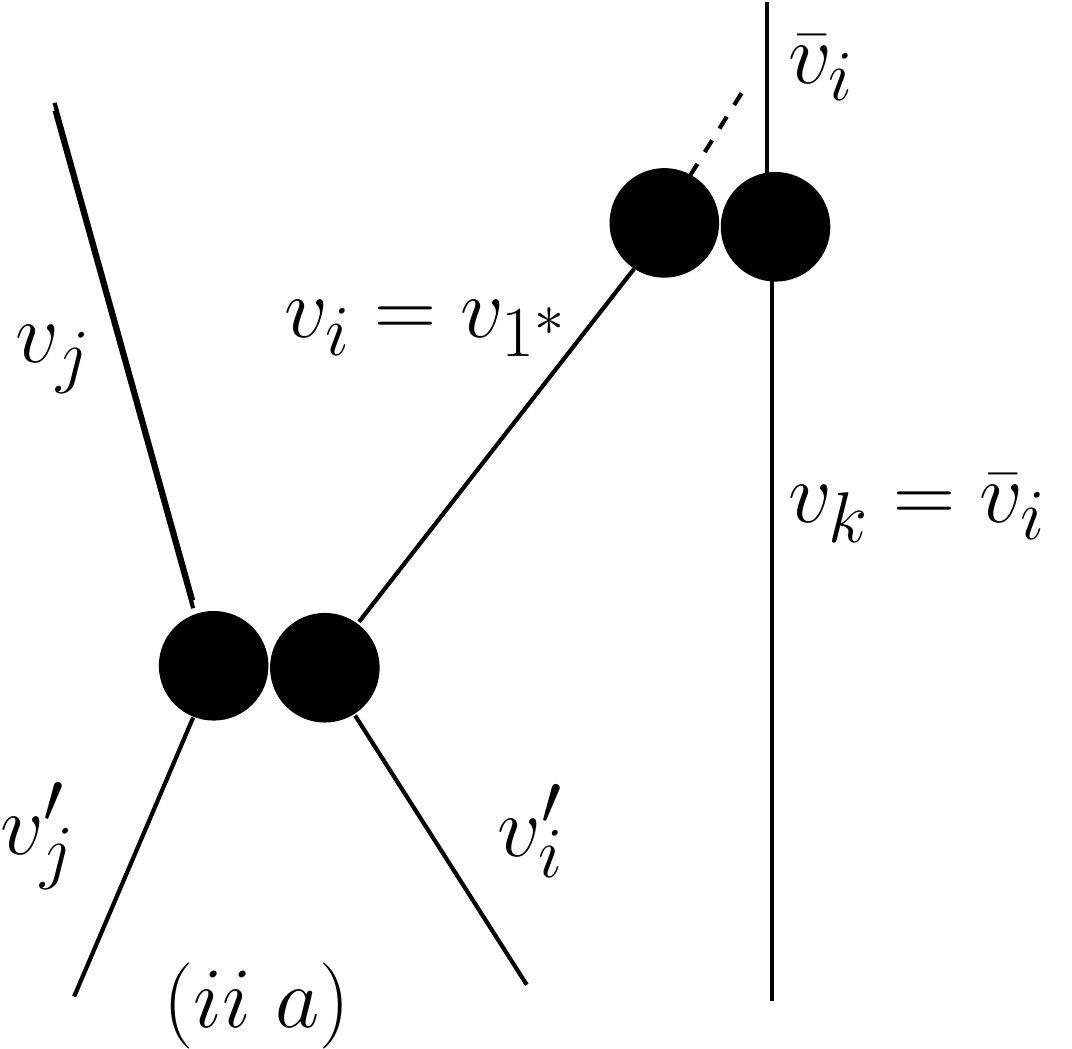} 
\hskip1.5cm
\includegraphics[width=4cm]{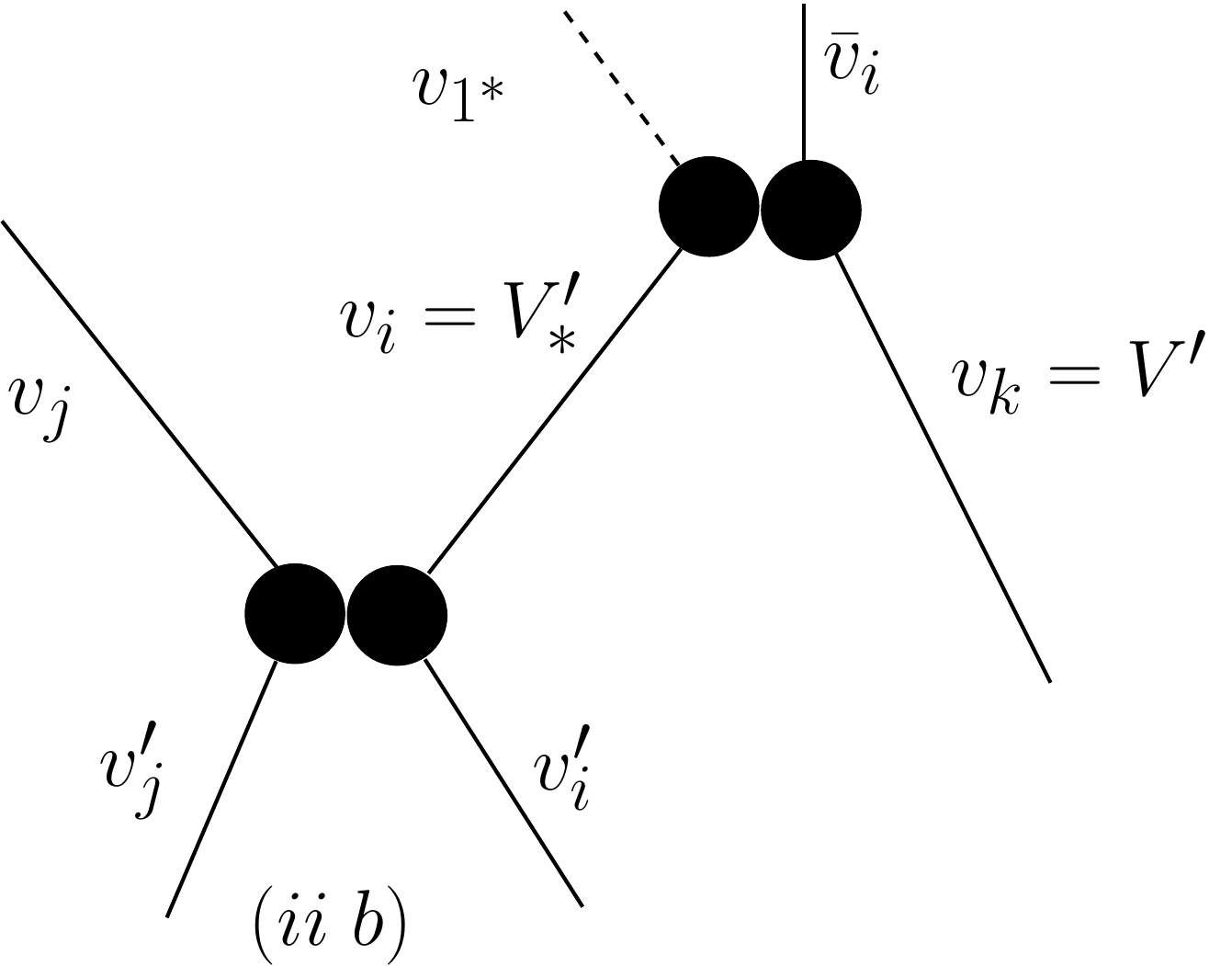} 
\caption{
The different scenarios associated with the case $k= 1^*$ are depicted. 
In all cases, $i$ refers to the pseudo particle recolliding with $j$; thus the labels $i,k$ can be switched after the collision. We used the notation \eqref{defcarleman} for the velocities $V',V'_*$ after scattering.}
\label{fig: k = 1*}
\end{figure}

We start by the case depicted in Figure \ref{fig: k = 1*} (i).
Denote by $v_k = V'_* $ the velocity of particle~$1^*$ after collision at $t_{1^*}^-$. 
The constraint \eqref{eq: contrainte proximite} that $|v_k  - v'_i| \leq \eps^{3/4}$ or that $|v_k  - v'_j| \leq \eps^{3/4}$ implies that 
one of the following identities then holds
$$
\begin{aligned}
 & v_i - (v_i-v_j) \cdot \nu_{rec}\,  \nu_{rec} = v_k + O(\eps^{3 /4})\,  ,\\
 \mbox{or}\quad & v_j + (v_i-v_j) \cdot \nu_{rec}\,  \nu_{rec}=v_i - (v_i-v_j) \cdot \nu^\perp_{rec} \nu^\perp_{rec} = v_k + O(\eps^{3/ 4}) \, ,
\end{aligned}
$$
and we further have that $|v_i - v_k| = | \bar v_i - v_{1^*}|$.

\medskip
 If $|\bar v_i - v_{1^*}|\leq \eps^{5/8}$, then~$v_{1^*}$ has to be in a ball of radius $\eps^{5/8}$ so we find a bound~$O(R\eps^{5/4}t )$ on integration over~$1^*$.

\medskip
 If $|\bar v_i - v_{1^*}|\geq \eps^{5/8}$, then
$$ \nu_{rec} =\pm  {v_i-  v_k \over |v_i - v_k |  }+O(\eps^{1/8}) 
\qquad \hbox{or} \qquad   
\nu_{rec} =\pm  {(v_i-  v_k )^\perp   \over |v_i - v_k  |}+O(\eps^{1/8}) \,.$$
Plugging this Ansatz in (\ref{rec-equation'}), we get
\begin{equation}
\label{eqansatz}
v_{i} - v_{j} = \frac1{\tau_{rec}  }\delta x_\perp - \frac{\tau_1}{\tau_{rec}} (\bar v_{i} - v_{j})-  \frac1{\tau_{rec} }\cR_{n \frac{\pi}{2} } { v_i - v_k \over | v_i- v_k |} + O\left( {\eps^{1/8} \over\tau_{rec}}\right) \, ,
\end{equation}
 with $n\in \{0,1,2,3\}$.
Compared with the formulas of the same type encountered in the proof of Proposition \ref{recoll1-prop}, this one has the additional difficulty that the ``unknown''~$v_i $ is on both sides of the equation. 
Furthermore the direction 
of~$v_k -v_i$ may have very fast variations  when~$|v_k - v_i| = |v_{1^*} - \bar v_i|$ is small.
To take this into account, we will consider different cases.

\medskip

Using the notation \eqref{eq: tau1'}, we define
$$
w:= \delta x_\perp  -  (\bar v_{i} - v_{j}) \tau_1, \quad \text{ and} \quad u:=|w|/\tau_{rec} \, .
$$  
By construction
\begin{equation}
\label{eq: estimation w}
|w| \geq | (\bar v_{i} - v_{j}) \tau_1| 
\quad \text{and} \quad u\leq 4R\,  ,
\end{equation}
where the latter inequality follows from (\ref{eqansatz}).
Recall that we can restrict to values such that 
$|w| \geq | (\bar v_{i} - v_{j}) \tau_1| \geq R^{4 \over 4 - 5 \gamma }$ due to~\eqref{eq: gamma}.
With these new variables, the condition (\ref{eqansatz}) may be rewritten
\begin{equation}
\label{eq: master equation B1}
v_{i} - \bar v_i = v_j-\bar v_i +u {w\over |w|} 
- {u\over |w|} \cR_{n \frac{\pi}{2} } { v_i - {v_k}\over | v_i- v_k|} 
+ O \Big( F ( |\bar v_{i} - v_{j}|, \tau_1, \eps) \Big) \,,
\end{equation}
where the error term 
$$
F( |\bar v_{i} - v_{j}|, \tau_1, \eps) = {\eps^{1/8} R\over | \bar v_{i} - v_{j}| |\tau_1|}
$$
has been estimated thanks to \eqref{eq: estimation w}.

The other cases depicted in Figure \ref{fig: k = 1*} obey the same equation (see  \eqref{withoutscattering} for $(ii \; a)$ and 
\eqref{Vprimestar} for $(ii \; b)$).
We will analyze the solutions of this equation for all cases of Figure \ref{fig: k = 1*}.

\bigskip

\noindent
$\blacktriangleright$ The easiest case is $(ii\, a)$ when the collision at $t_{1^*}$ has no scattering, i.e. 
$v_i = v_{1^*}$ and $v_k = \bar v_i$. We split the analysis into two more subcases.

\medskip

$\bullet$
If $| v_{1^*} - \bar v_i | \geq {1\over |w|^\gamma }$, the condition \eqref{eq: master equation B1} reads
\begin{equation}
\label{withoutscattering}
v_{1^*} - \bar v_i = v_j-\bar v_i +u {w\over |w|} - {u\over |w|} \cR_{n \frac{\pi}{2} } { v_{1^*} - \bar v_i\over | v_{1^*} - \bar v_i|} + O \Big( F ( |\bar v_{i} - v_{j}|, \tau_1, \eps) \Big)  \,,
\end{equation}
with $n \in \{ 0,1,2,3 \}$.
To analyse this equation, we  implement a fixed point method and consider $u \in [-4R,4R]$ as a parameter, forgetting its dependency on $\tau_{rec}$ (i.e. on $v_i$). 
Due to the assumption that~$| v_{1^*} - \bar v_i | \geq {1/ |w|^\gamma }$, this imposes a constraint on~$u$ since one needs to ensure that~$| v_j-\bar v_i +u {w\over |w|} | \geq c/ |w|^\gamma$. Depending on the angle between~$ v_j-\bar v_i $ and~$w$ and depending on the size of~$|v_j-\bar v_i|$, this implies that~$u$ should belong to one or two intervals in~$u \in [-4R,4R]$. Given~$u$ in one of those admissible intervals, we first look for a solution of the equation without the error term.
The mapping
\begin{align}
\label{eq: mapping theta}
\Theta : B_R \setminus B_{|w|^{-\gamma}} (\bar v_i)&\to {\mathbb S}\\
v_{1^*} & \mapsto {v_{1^*} - \bar v_i\over | v_{1^*} - \bar v_i|}\nonumber
\end{align}
is Lipschitz continuous with constant $ |w|^\gamma$. We deduce by a fixed point argument that, for any admissible $u$
$$
v_{1^*} - \bar v_i = v_j-\bar v_i +u {w\over |w|} 
- {u\over |w|} \cR_{n \frac{\pi}{2} } { v_{1^*} - \bar v_i\over | v_{1^*} - \bar v_i|}
$$
has a unique solution $\hat v_{1^*}(u)$ (which is clearly Lipschitz  in $u$).
Thus any solution of (\ref{withoutscattering}) satisfies
\begin{equation}
\label{eq: condition suffisante}
| v_{1^*} - \hat v_{1^*}(u)  | \leq  O \Big( F ( |\bar v_{i} - v_{j}|, \tau_1, \eps) \Big) .
\end{equation}
Note that among the solutions $v_{1^*}$ of \eqref{withoutscattering}, we are looking only for the solutions which are compatible with the constraint $u = |w|/\tau_{rec}$ where $\tau_{rec}$ is a function of~$v_{1^*}$. In particular~$\hat v_{1^*}(u)$ will not correspond to a velocity compatible with  this constraint.
Nevertheless it is enough use the bound \eqref{eq: condition suffisante} as a sufficient condition and to retain only the information that  $v_{1^*}$ has to belong to a tube $T( \delta x_\perp , v_j - \bar v_i , q,\tau_1)$ located around the curve $u \to \hat v_{1^*}(u)$ with $|u| \leq 4 R$ and of width 
$O \big( F ( |\bar v_{i} - v_{j}|, \tau_1, \eps) \big)$.
Integrating first with respect to the collision with $1^*$, we get
\begin{equation}
\label{eq: integrale 1*}
\int \indc_{ \{ v_{1^*}\in T( \delta x_\perp , v_j - \bar v_i , q,\tau_1) \} } \,b(\nu_{1^*},v_{1^*}) \,  dv _{1^*} d\nu_{1^*}  \leq C R^2  F ( |\bar v_{i} - v_{j}|, \tau_1, \eps)
= C {R^3 \eps^{1/8}  \over |\tau_1| \,  |v_j-\bar v_i| }  ,
\end{equation}
where we replaced $F ( |\bar v_{i} - v_{j}|, \tau_1, \eps)$ by its value.
Next we integrate  with respect to $|\tau_1| $ in the set~$ [R, \frac{T R}{\eps}]$
(see \eqref{eq: tau R}) and we get after a change of variable in $t_1^*$
\begin{equation}
\label{eq: integrale t1*}
\int \indc_{ \{ v_{1^*}\in T( \delta x_\perp , v_j - \bar v_i , q,\tau_1) \}} \,b(\nu_{1^*},v_{1^*}) \,  dv _{1^*} d\nu_{1^*} dt_{1^*} \leq C R^3{\eps^{9/8} | \log \eps |\over  |v_j-\bar v_i| } \, \cdotp
\end{equation}
It remains then to integrate the singularity $ |v_j-\bar v_i|$.
As the first recollision involves $i,j$,  there exists always an additional parent of $(i,j)$ 
to provide a degree of freedom. Thus after integration, the singularity $ |v_j-\bar v_i|$ can be controlled 
up to a loss $O(|\log \eps|)$ by application of Lemma~\ref{scattering-lem}  (see also page~\pageref{integrationsingulatity}).
This situation will be referred to as a new scenario $\cP_2(a,p,\sigma)$ for some $p$, and~$|\sigma| \leq 2$. Summing over all possible $q$ and all possible $j$ provides the estimate
 $$
\int \indc_{\cP_2(a,p,\sigma) }  \prod_{m\in \sigma} \,\big | \big(v_{m}-v_{a(m)} ( t_{{m}} ) )\cdot \nu_{{m}} \big| 
d  t_{{m}} d \nu_{{m}}dv_{{m}}   
\leq C s R^5 t^2 \eps^{9/8}  |\log \eps|^2 \, .
$$

\bigskip

$\bullet$ If $| v_{1^*} - \bar v_i | \leq {1\over |w|^\gamma }$, then 
$$
\begin{aligned}
\int \indc_{ \{ | v_{1^*} - \bar v_i  |\leq \frac{1}{|w|^\gamma} \}} 
\,b(\nu_{1 ^*},v_{1 ^*}) \,  dv _{1^*} d\nu_{1^*} 
\leq {C R \over |w|^{2 \gamma}} 
\leq {C R \over \tau_1^{2 \gamma}  |v_j-\bar v_i|^{2 \gamma}},
\end{aligned}
$$
by~(\ref{eq: estimation w}).  By \eqref{eq: tau R}, we know that 
$| \tau_1| \geq  R^2$, thus the singularity is integrable in $\tau_1$.
Changing to the variable $t_{1^*}$ by using~\eqref{eq: tau1'},
we find that
 $$
\int \indc_{\{ | v_{1^*} - \bar v_i  |\leq \frac{1}{|w|^\gamma} \} } \,b(\nu_{1^*},v_{1^*}) \,  dv_{1^*} d\nu_{1^*} dt_{1^*} 
\leq \eps {C \over  |v_j-\bar v_i|^{2 \gamma} } \, \cdotp
$$
It remains then to integrate the singularity $ |v_j-\bar v_i|^{- 2 \gamma }$, which can be done with two additional parents of $(i,j)$ since $\gamma<1$ (see Lemma \ref{scattering-lem}).
This situation will be referred to as a new scenario $\cP_2(a,p,\sigma)$ for some $p$, and~$|\sigma| \leq 3$. We have
\begin{equation}
\label{eq: sous cas d'un sous cas}
\int \indc_{\cP_2(a,p,\sigma) }  \prod_{m\in \sigma} \,
\big | \big(v_{m}-v_{a(m)} ( t_{{m}} ) )\cdot \nu_{{m}} \big| d  t_{{m}} d \nu_{{m}}dv_{{m}}   
 \leq C s (Rt)^r \eps\,  .
\end{equation}

\bigskip

\begin{figure} [h] 
\centering
\includegraphics[width=4cm]{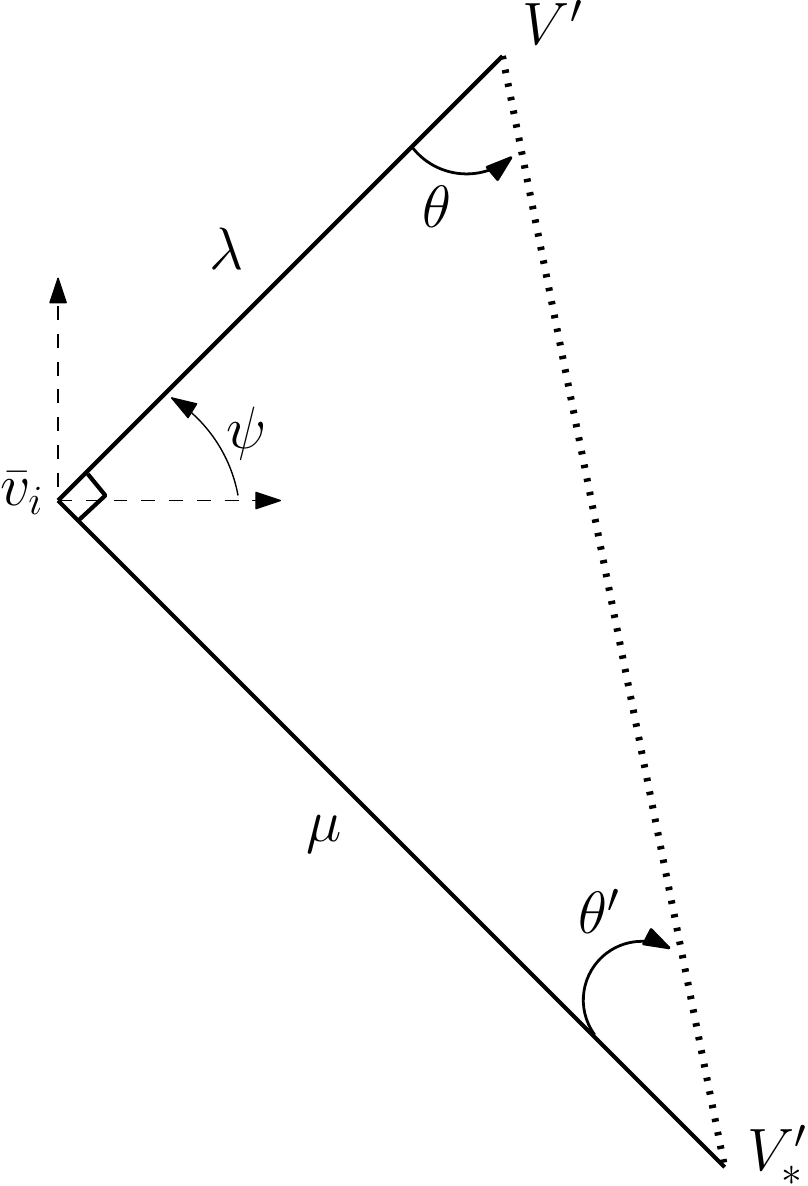} 
\caption{
Carleman's parametrization $(V',V_*')$ can be evaluated in terms of the measure $dV' d\mu$ 
 or alternatively by parametrizing $V' - \bar v_i$ in polar coordinates by the measure 
$\lambda d \lambda \, d \psi d\mu$ with $\lambda = | V' -  \bar v_i |$.
The direction ${V'-V_*'\over |V'-V_*'| }$ can be  recovered by a rotation from $V' - \bar v_i$ by an angle $\theta$ \eqref{eq: rotation dilatation} or $V_*' - \bar v_i$ by an angle $\theta'$  \eqref{eq: rotation dilatation 2}.
}
\label{fig: parametrisation carleman}
\end{figure}

\bigskip

\noindent
$\blacktriangleright$ If there is scattering at time $t_{1^*}$ of the type (i) such that the velocities are given by $v_i = V'$ and $v_k = v'_{1^*} = V'_*$,
 with the notation in \eqref{defcarleman}, then we consider two cases depending as above on the size of~$| v_i - \bar v_i |$  compared to~$ |w|^{-\gamma} $.
 
 \medskip
 
$\bullet$ If $| v_i - \bar v_i | = | V' - \bar v_i |  \geq {1/ |w|^\gamma }$, then 
Identity \eqref{eq: master equation B1} leads to
\begin{equation}
\label{Vprime}
V' - \bar v_i = v_j-\bar v_i +u {w\over |w|} - {u\over |w|} \cR_{n \frac{\pi}{2} } { V'- V'_*\over | V'-V'_*|} 
+ O \Big( F ( |\bar v_{i} - v_{j}|, \tau_1, \eps) \Big)  \,.
\end{equation}
We will use Carleman's parametrization  
and denote $\mu = | V'_* - \bar v_i|$ (see Figure \ref{fig: parametrisation carleman}). 
For a given~$V'_*$, we have to compute the regularity of the map which associates to 
$V'$ the direction of $V'- V'_*$. From the mapping $V' \to \Theta(V')$ defined in~\eqref{eq: mapping theta}, we first determine the direction of $V'- \bar v_i$ and then  rotate this direction by $\theta$ to get 
\begin{equation}
\label{eq: rotation dilatation}
 { V'- V'_*\over | V'-V'_*|} = \cR_\theta [ \Theta (V')]  
 \quad \hbox{  with } \quad \theta = \arctan { \mu \over \lambda } \,\cdotp
\end{equation}

If $\mu >  {1 \over |w|^{\gamma /4 } }$, the mapping $V' \to \cR_\theta [\Theta (V')]$ is  
 continuous  with Lipschitz constant less than
\begin{equation}
\label{eq: Lipschitz}
|w|^\gamma \times \frac{\mu}{\mu^2 + \lambda^2} \leq
\frac{|w|^\gamma}{\mu} \leq  |w|^{5 \gamma/4}.
\end{equation}
As $\gamma < \frac{4}{5}$ and $|w| \geq  R^{4 \over 4 -  5 \gamma}$ \eqref{eq: gamma}, we deduce by a fixed point argument that, for any $u\leq 4R$ 
$$
V' - \bar v_i = v_j-\bar v_i +u {w \over |w|} - {u\over |w|} \cR_{n \frac{\pi}{2} } { V'- V'_* \over | V'-V'_*|}
$$
has a unique solution $\hat V' (u)$ (which is Lipschitz in  $u$).
Thus for a given u, any solution of (\ref{Vprime}) satisfies
$$| V'-\hat V' (u)  | \leq   O \Big( F ( |\bar v_{i} - v_{j}|, \tau_1, \eps) \Big)  \,\cdotp$$
In other words, $V'$ has to belong to a tube $T( \delta x_\perp, v_j - \bar v_i , q,\tau_1)$
of width $ O \Big( F ( |\bar v_{i} - v_{j}|, \tau_1, \eps) \Big)$ around the curve  $u \to \hat V' (u)$. 
By Carleman's parametrization, we can then integrate over $d V' d\mu$ and get an estimate of the form 
\eqref{eq: integrale 1*} when replacing $F$ by its value
\begin{equation}
\label{eq: integrale 1* bis}
\int \indc_{ \{ \mu >  {1\over |w|^{\gamma /4} } \} } \indc_{ \{ V' \in T( \delta x_\perp , v_j - \bar v_i , q,\tau_1) \}} \,  d V' d \mu
\leq C R^2  F ( |\bar v_{i} - v_{j}|, \tau_1, \eps)
= C {R^3 \eps^{1/8}  \over |\tau_1| \,  |v_j-\bar v_i| }  .
\end{equation}
Integrating then over $\tau_1$ leads to an upper bound analogous to \eqref{eq: integrale t1*}
\begin{equation*}
\int \indc_{ \{ \mu >  {1\over |w|^{\gamma /4} } \} } \indc_{ \{ V' \in T( \delta x_\perp , v_j - \bar v_i , q,\tau_1) \}} \,  d V' d \mu  dt_{1^*} \leq C R^3{\eps^{9/8} | \log \eps| \over  |v_j-\bar v_i| } \, \cdotp
\end{equation*}
We then conclude by integrating as usual the singularity $1/ |v_j-\bar v_i|$ thanks to Lemma \ref{scattering-lem}.
This provides  a new scenario $\cP_2(a,p,\sigma) $ for some $p$ and~$|\sigma| \leq 2$.

\medskip

If $\mu <  {1\over |w|^{\gamma /4} }$, we only use  the condition \eqref{rec-equation'}
which reads in this case
\begin{equation*} 
V'  - \bar v_i =  v_{j} - \bar v_i + \frac1{\tau_{rec} }  w -  \frac1{\tau_{rec} } \nu_{rec} 
\quad \text {with} \quad
\Big | \frac1{\tau_{rec} } \Big| \leq { 4 R \over  |\tau_1| \, |v_j-\bar v_i| } \, \cdotp
\end{equation*}
As a consequence, $V'$ has to be in the rectangle $\cR( \delta x_\perp , \bar v_i-v_j, q,\tau_1)$ of size $ R \times { 4 R \over  |\tau_1| \, |v_j-\bar v_i| }$. 
Together with the condition on  $\mu$,  this leads to
$$
\int \indc_{\{ V' \in \cR( \delta x_\perp , v_j -  \bar v_i, q,\tau_1) \} } \indc _{ \{ \mu <\frac{1}{|w|^{\gamma/4} }  \} } 
\,  d V' d \mu   
\leq  { C R^2 \over  |\tau_1|^{1+\gamma/4} \; |v_j-\bar v_i|^{1+\gamma/4} } .
$$
Since $ 1 + \frac{\gamma}{4} >1$, we can integrate with respect to $t_{1^*}$ and gain a factor $\eps$. It remains then to integrate the singularity $ |v_j-\bar v_i|^{-(1+\gamma/4)}$.
Since $1+\frac{\gamma}{4} < 2$, this can be done by using~\eqref{alpha1} and~\eqref{alpha2} with two additional parents of $(i,j)$. This provides another contribution to $\cP_2(a,p,\sigma) $ with~$|\sigma| \leq 3$.

\medskip

$\bullet$ If $| V' - \bar v_i |  \leq {1\over |w|^\gamma }$. This can be dealt as in the case 
\eqref{eq: sous cas d'un sous cas}.

\bigskip

\noindent
$\blacktriangleright$
 If there is scattering at time $t_{1^*}$ of type $(ii \; b)$  such that the velocities are given by 
$v_i = V'$ and $v'_{1^*} = V'_*$ with the notation in \eqref{defcarleman}, then as above we separate the analysis into two sub-cases, depending on the relative size of~$| v_i - \bar v_i |$  and~$ |w|^{-\gamma} $.

\medskip

$\bullet$ Suppose that $| V'_* - \bar v_i |  \geq {1\over |w|^\gamma }$. 
The analogue of identity \eqref{eq: master equation B1} reads 
\begin{equation}
\label{Vprimestar}
V' _*- \bar v_i = v_j-\bar v_i +u {w\over |w|} - {u\over |w|} \cR_{n \frac{\pi}{2} } { V'- V'_*\over | V'-V'_*|} 
+ O \Big( F ( |\bar v_{i} - v_{j}|, \tau_1, \eps) \Big) \,,
\end{equation}
with $n \in \{ 0,1,2,3 \}$.
Denote by $\lambda = | V' - \bar v_i|$ and consider two cases.

\medskip

If $\lambda >  {1\over |w|^{\gamma /4} }$, we proceed as in \eqref{eq: rotation dilatation} and recover the direction of $V'- V'_*$  by
\begin{equation}
\label{eq: rotation dilatation 2}
{ V'- V'_*\over | V'-V'_*|} = 
\cR_{\theta'} [\Theta (V'_*)] 
\qquad   \hbox{  with } \qquad 
\theta' = \arctan {\lambda \over \mu} \,\cdotp
\end{equation}
Given $\lambda >  {1 \over  |w|^{\gamma /4} }$, the map~$V'_* \to \cR_{\theta'} [\Theta (V'_*)] $ has a Lipschitz constant 
$ |w|^{5 \gamma/4}$ as~ $| V'_* - \bar v_i |  \geq {1\over |w|^\gamma }$ (see \eqref{eq: Lipschitz}).
Since $\gamma < \frac{4}{5}$, we deduce by a fixed point argument that, for any $u\leq 4R$ 
$$
V'_* - \bar v_i = v_j-\bar v_i +u {w\over |w|} - {u\over |w|} \cR_{n \frac{\pi}{2} } { V'- V'_*\over | V'-V'_*|}
$$
has a unique solution $\hat V'_*(u)$ (which is Lipschitz in $u$) 
and that any solution of \eqref{Vprimestar} takes values close to this solution
$$
| V'_*-\hat V'_* (u) | \leq O \Big( F ( |\bar v_{i} - v_{j}|, \tau_1, \eps) \Big) \,\cdotp
$$
In other words, $V'_*$ has to belong to a tube $T( \delta x_\perp, v_j-  \bar v_i, q,\tau_1)$ of size 
$\delta = O \Big( F ( |\bar v_{i} - v_{j}|, \tau_1, \eps) \Big)$ around the smooth curve $u \to \hat V'_*(u)$ 
which stretches in the direction ${w \over |w|}$. 
Mimicking the proof of \eqref{rectangle0}, we can decompose the tube $T( \delta x_\perp, v_j-  \bar v_i, q,\tau_1)$ into small blocks of side length $\delta$. 
Summing over these blocks, we recover that the measure of 
$\{ V'_* \in T( \delta x_\perp, v_j-  \bar v_i, q,\tau_1)\}$ is less than
$$
R^2 \delta \big| \log \delta \big|
\leq 
{R^3 \eps^{1/8} \, \big| \log \eps \big| \over | \bar v_{i} - v_{j}| |\tau_1|}.
$$
Integrating with respect to $t_{1^*}$ and then integrating  the singularity $|\bar v_i - v_j|$, we get  a contribution of order 
$R^3 \eps^{1/8} |\log \eps|^2$ which  controls the occurence of a new scenario~$\cP_2 (a,p,\sigma)$.

\medskip

If $\lambda <  {1 \over |w|^{\gamma /4} }$, we only use the condition  \eqref{rec-equation'} which implies that $V'_*$ has to be in the rectangle $\cR( \delta x_\perp , v_j - \bar v_i, q,\tau_1)$ of size $ R \times \frac{4R}{|w|}$.
With the notation of Figure \ref{fig: parametrisation carleman}, Carleman's parametrization $(V',V_*')$ can be evaluated in terms of the measure  $\lambda d \lambda \, d \psi d\mu$.
Thus we get 
$$ 
\int \indc_{ \{ V' _* \in \cR( \delta x_\perp , v_j - \bar v_i , q,\tau_1) \} } d \mu d\psi   \leq {R|\log w| \over |w|}\,,
$$
together with the condition on  $\lambda$ which is independent
$$ 
\int \indc _{ \{ | \lambda |< \frac{1}{|w|^{\gamma/4}} \}  }  \lambda  d \lambda  \leq  {1\over |w|^{\gamma/2} } \, \cdotp
$$
As a consequence
$$
\int \indc_{ \{ V' \in \cR( \delta x_\perp , v_j -  \bar v_i, q,\tau_1) \}} 
\indc _{ \{ | V'_*  - \bar v_i|< \frac{1}{|w|^{\gamma/4}} \} } \,b(\nu_{1 ^*},v_{1 ^*}) \,  dv _{1^*} d\nu_{1 ^*} dt_{1^*}  
\leq C { \eps  \log |v_j-\bar v_i| \over  |v_j-\bar v_i|^{1+\gamma/2} } \, \cdotp
$$
It remains then to integrate the singularity $ |v_j-\bar v_i|^{-(1+\gamma/2)}$, which can be done by using \eqref{alpha1} and 
\eqref{alpha2} with two additional parents of $(i,j)$. This leads to another scenario to $\cP_2(a,p,\sigma) $ for some $p$.

\medskip

$\bullet$ Suppose that $| V'_* - \bar v_i |  \leq {1\over |w|^\gamma }$. By integration with respect to $b(\nu_{1 ^*},v_{1 ^*}) \,  dv _{1^*} d\nu_{1 ^*}$, we will gain only one power of $|w|^{-\gamma}$ due to the scattering (see Lemma \ref{scattering-lem}). This is not integrable with respect to time $\tau_1$ and we have therefore to use also the fact that there is a recollision between $(i,j)$ to regain some control.
Because $i$ has only a very small deflection at time  $t_{1^*}$, this implies that a ``kind of recollision" 
has to be triggered already before the collision with~$1^*$, i.e. at time $t_{2^*}$.

\begin{figure} [h] 
\centering 
\includegraphics[width=7cm]{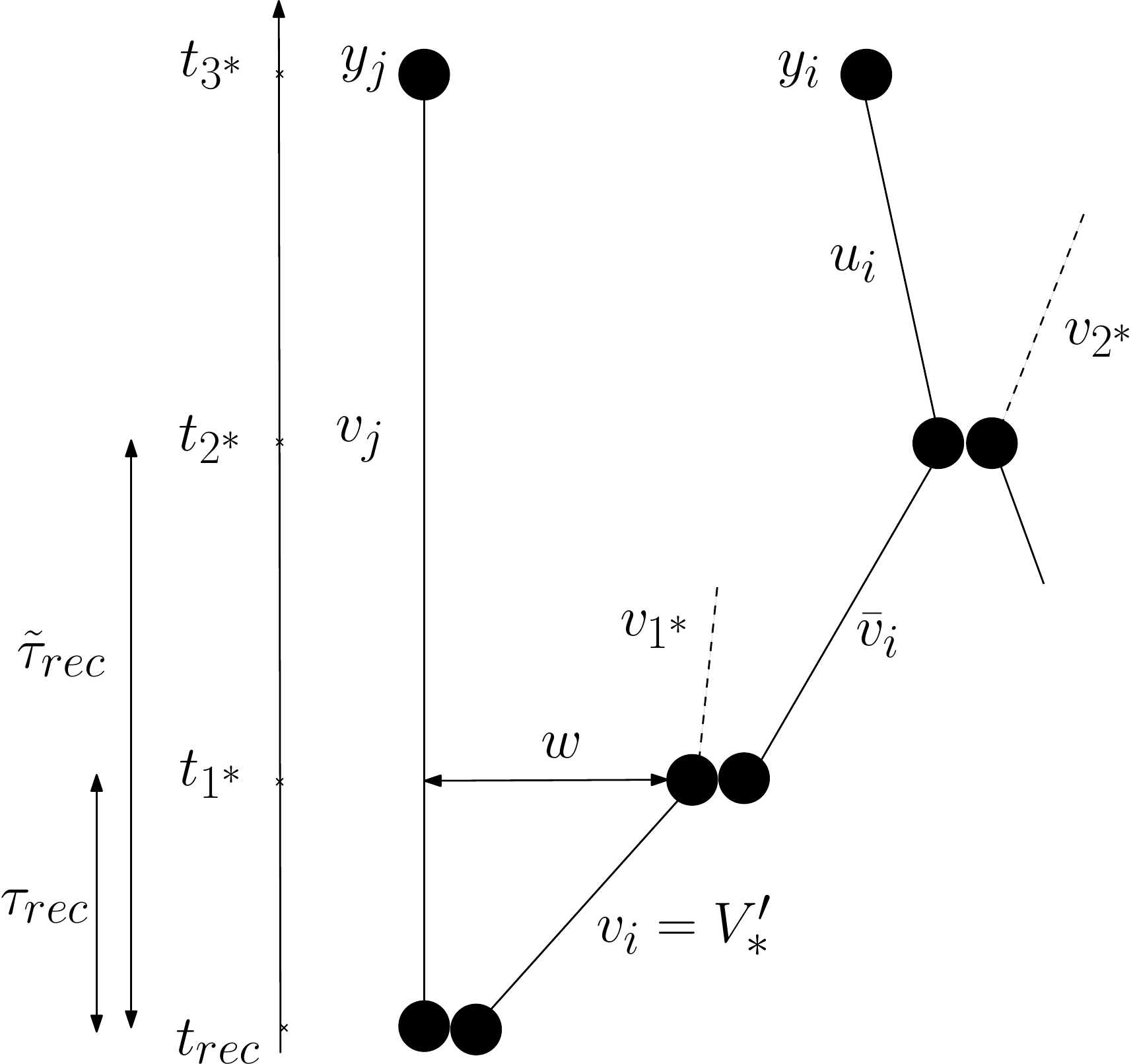} 
\caption{In the case $| V'_* - \bar v_i |  \leq {1\over |w|^\gamma }$, one has to use as well  the degree of freedom from the collision with $2^*$. In this figure, the particle $j$ is not deviated and $u_j = v_j$.
The parameter $w$ stands for the distance between $i,j$ at time $t_{1^*}$.}
\label{lemmeB1_cas3}
\end{figure}

Denote by $(y_i, u_i)$ and $(y_j,u_j)$ the positions and velocities  of the pseudo-particles $i$ and $j$ at time 
$t_{3^*}$ (see Figure \ref{lemmeB1_cas3}) and set
$$
\frac{y_i - y_j}{\eps} = \delta  y _\perp + \frac{\lambda'}{\eps} ( u_i - u_j)
\quad \text{with} \quad  \delta  y _\perp \cdot ( u_{i} - u_{j}) = 0 \, .
$$
We define  also
$$
\tau_2 = - \frac{1}{\eps}  ( t_{2^*} - t_{3^*} + \lambda')\,, \qquad 
\tilde \tau_{rec} = \frac{t_{2^*}-t_{rec}}{\eps}\,,
\qquad 
\tau_{rec} = \frac{t_{1^*}-t_{rec}}{\eps} \leq \tilde \tau_{rec}\,.
$$ 
By analogy with equation \eqref{rec-equation'}, we get
\begin{equation}
\label{fausserecoll}
\bar v_i  - v_{j} 
= \frac1{\tilde  \tau_{rec} } \delta y _\perp - \frac{\tau_2}{\tilde \tau_{rec}} ( u_{i} - u_{j}) - \frac{\tau_{rec}}{\tilde \tau_{rec}}  ( v_i - \bar v_i) -  \frac1{\tilde \tau_{rec}  } \nu_{rec} \, ,
\end{equation}
where the additional term comes from the small deflection at time $t_1^*$. By assumption, this term is less or equal than $|w|^{-\gamma}$.
As previously, we have that
\begin{equation*}
4 R \geq |\bar v_i  - v_j| +  \frac{\tau_{rec}}{\tilde \tau_{rec}}  |v_i - \bar v_i|
\geq   \frac{1}{\tilde \tau_{rec}}  \Big( |\tau_2| |u_{i} - u_{j}|  - 1 \Big) 
\geq   \frac{1}{ 2 \tilde \tau_{rec}}  |\tau_2| |u_{i} - u_{j}| \,,
\end{equation*}
as it is enough to consider $ |\tau_2| |u_{i} - u_{j}|  \gg 1$ (see \eqref{eq: small distance cut off}). Thus we get 
$$\tilde  \tau_{rec}  \geq \frac{1}{8 R} |\tau_2| |u_i-u_j| \,.$$
In the remaining of the proof, we fix the   parameter~$\gamma$ and a new parameter~$\alpha$ as follows
\begin{equation}
\label{eq: parametres lemme B1}
\gamma = \frac{3}{4}\,,  \quad \alpha =\frac{4}{7} 
\end{equation}
and  consider two cases according to large and small values of $\tau_1$.

\medskip

- If $ |\tau_1| \geq  {1 \over |\bar v_i - v_j|^{\gamma/(\gamma-\alpha)}}$, then   we get a control on the size 
of~$1/ |w|^\gamma$
$$ 
\frac{1}{|w|^\gamma} \leq 
\frac{1}{(|\tau_1| |\bar v_i - v_j|)^\gamma}  \leq \frac{1}{ |\tau_1|^\alpha} ,
$$
as $|w| \geq |\tau_1| |\bar v_i - v_j|$ from \eqref{eq: estimation w}.

Equation \eqref{fausserecoll}, imposes the condition that, at time $t_{2^*}$,
$\bar v_i - v_j$ has to be in a domain which is a kind of rectangle $\cK$ with axis
$\delta y _\perp - \tau_2 ( u_{i} - u_{j})$ and varying width 
$$ 
\frac{1}{|w|^\gamma} + \frac{1}{|\tau_2| |u_i-u_j|} \leq \frac{1}{|\tau_1|^\alpha} + \frac{1}{|\tau_2| |u_i-u_j|} \, \cdotp
$$
Recall from \eqref{eq: tau R} that $|\tau_1|\geq R$.
Combined with the condition $| V'_* - \bar v_i |  \leq {1\over |w|^\gamma }$ at time $t_{1^*}$, we get 
thanks to  Lemma \ref{scattering-lem}
\begin{align*}
\int \indc_{ \{ |\tau_1|\geq R \} } \, & \indc_{\{ | V'_* - \bar v_i |  \leq {1\over |w|^\gamma } \}}  \indc_{\{ \bar v_i -v_j \in \cK \}} 
 \prod_{\ell = 1^*,2^*}b(\nu_\ell ,v_\ell)  \,  dv_\ell d\nu_\ell   d t_{\ell}
\\
&\qquad  \leq  \int  \indc_{ \{ |\tau_1|\geq R \} } \,  {\indc_{\{ \bar v_i -v_j \in \cK \}} \over |\tau_1|^\alpha }
 b(\nu_{2^*},v_{2^*})  \,  dv_{2^*} d\nu_{2^*}  d t_{2^*} d t_1.
\end{align*}
At this stage, one has to be careful as $\tau_1$ was defined in  \eqref{eq: tau1'} by 
\begin{equation*}
\tau_1 :=- \frac1\eps (t_{1^*} - t_{2^*}+\lambda )
\quad  \text{with} \quad 
\lambda = \big( x_i (t_{2^*}) - x_j (t_{2^*}) - q  \big) \cdot \frac{\bar v_{i} - v_{j}}{|\bar v_{i} - v_{j}|},
\end{equation*}
thus~$|\tau_1| $ depends on~$\bar v_{i} - v_{j}$, i.e. also on $v_{2^*}$. 
In order to simplify the dependency between the variables $v_{2^*}$ and $t_1$, we replace $t_{1^*}$ with the variable $\tau_1$.
This boils down to integrating with respect to $\eps d\tau_1$.
The geometric structure implies that $\tau_1$ takes now values in a complicated domain which we will estimate from above by keeping only the constraint $ |\tau_1| \in[ R, \frac{R^2}{\eps}]$.
This decouples the variables in the integral and we finally get
\begin{align}
\label{eq: intricated}
\int \indc_{ \{ |\tau_1|\geq R \} } \, & \indc_{\{ | V'_* - \bar v_i |  \leq {1\over |w|^\gamma } \}}  \indc_{\{ \bar v_i -v_j \in \cK \}} 
 \prod_{\ell = 1^*,2^*}b(\nu_\ell ,v_\ell)  \,  dv_\ell d\nu_\ell   d t_{\ell}
\\
&\qquad  \leq \eps \int  \indc_{ \{ |\tau_1| \in [ R , \frac{R^2}{\eps}]\} } \,  {\indc_{\{ \bar v_i -v_j \in \cK \}} \over |\tau_1|^\alpha }
 b(\nu_{2^*},v_{2^*})  \,  dv_{2^*} d\nu_{2^*}  d t_{2^*} d \tau_1  \nonumber \\
   &  \qquad \leq 
CR^2  \eps\int \indc_{ \{ |\tau_1| \in [ R , \frac{R^2}{\eps}] \} } \,  \left( \frac{1}{|\tau_1|^{2 \alpha}} + \frac{1}{|\tau_1|^\alpha |\tau_2| |u_i-u_j|} \right)
 d t_{2^*}  d \tau_1 \, . \nonumber
\end{align}
The first term is integrable in $|\tau_1|\geq R$ as $2 \alpha >1$.  The second term can be integrated first with respect to $R \leq |\tau_1| \leq \frac{R^2}{\eps}$ which provides  a factor $\eps^\alpha$, then with respect to $t_{2^*}$ which provides an additional $\eps | \log \eps|$. The singularity with respect to small relative velocities can be controlled by two additional integrations. Thus the second term leads to an upper bound less than $\eps$ and the corresponding scenarios are indexed by sets $\sigma$ with cardinal 4.

\medskip

- If $|\tau_1| \leq  {1 \over |\bar v_i - v_j|^{\gamma/(\gamma-\alpha)}}$, then we can forget about (\ref{fausserecoll}). We indeed have that 
$$
\begin{aligned}
& \int \indc_{\{ | V'_* - \bar v_i |  \leq {1\over |w|^\gamma } \}}  \indc_{ \{ |\tau_1| \leq  {1 \over |\bar v_i - v_j|^{\gamma/(\gamma-\alpha)}}  \} }  b(\nu_{1^*} ,v_{1^*})  \,  dv_{1^*} d\nu_{1^*} d\tau_1\\
& \qquad \qquad  \leq \int \indc_{ \{ |\tau_1| \leq  {1 \over |\bar v_i - v_j|^{\gamma/(\gamma-\alpha)}}  \} }
{1\over (\tau_1 |\bar v_i - v_j|) ^{ \gamma}}  d\tau_1 \\
& \qquad \qquad  \leq {1\over  |\bar v_i - v_j| ^{ \gamma}} 
\int  \indc_{ \{ |\tau_1| \leq  {1 \over |\bar v_i - v_j|^{\gamma/(\gamma-\alpha)}}  \} }
\frac{1}{ |\tau_1|^\gamma} d\tau_1  
\leq {1\over  |\bar v_i - v_j| ^{ \gamma + \frac{\gamma(1-\gamma)}{\gamma-\alpha} }} \, \cdotp
\end{aligned}
$$
As $\frac{ \gamma (1-\alpha)}{\gamma-\alpha} <2$,
the singularity at small relative velocities is integrable by using Lemma \ref{scattering-lem}.
Thus the change of variable to $t_{1^*}$ allows us to recover an upper bound of order $\eps$.

\bigskip

Throughout the proof, the bad sets were analyzed in terms of the recolliding particles, thus we have to reindex these sets in terms of the labels $\sigma$ of the parents. A similar procedure has been done already at the end of the proof of  Proposition \ref{recoll1-prop}. Given a set $\sigma$ of parents, it may only determine the particle $i$, so that an extra factor $s^2$ has to be added in \eqref{eq: A recollision with a constraint on the outgoing velocity}
to take into account the choice of $j,k$.
\end{proof}


\subsection{Parallel recollisions}

The following result was used  in Section~\ref{secondcase1=1} page~\pageref{secondcase1=1}
to deal with parallel recollisions when~$t_{1^* } = t_{\tilde 1}$. The setting is recalled in Figure~\ref{cherriesagain}.
\begin{figure} [h] 
  \centering
    \includegraphics[width=5cm]{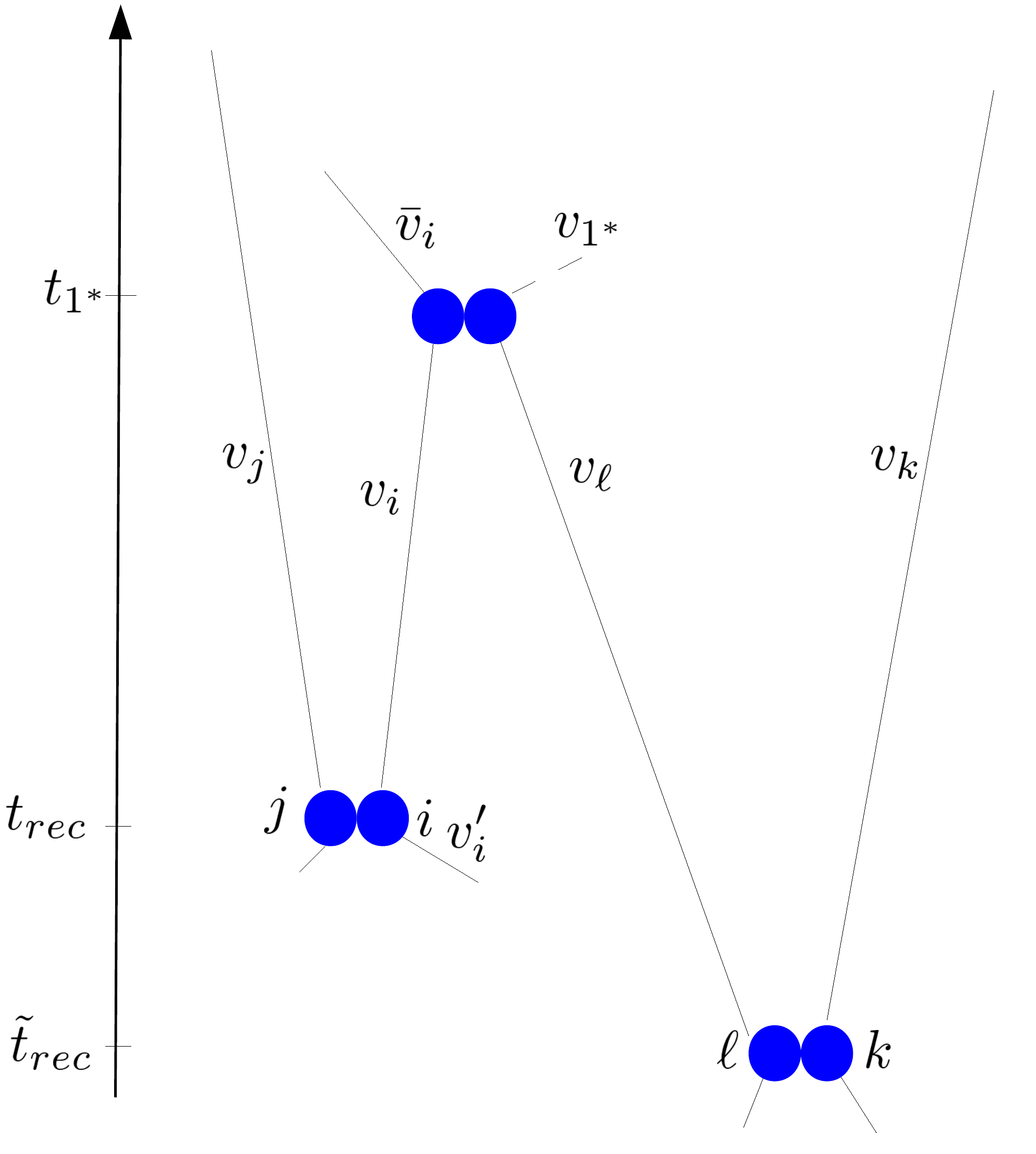}  
   \caption{Parallel recollisions.}
   \label{cherriesagain}
  \end{figure}

\begin{Lem}
\label{doublerecollision}
Fix a final configuration of bounded energy~$z_1 \in \T^2 \times  B_R$ with~$1 \leq R^2 \leq C_0 |\log \eps|$,  a time $1\leq t \leq C_0|\log \eps|$ and  a collision tree~$a \in \cA_s$ with~$s \geq 2$.

There exist sets of bad parameters
$\cP_2 (a, p,\sigma)\subset \cT_{2,s} \times {\mathbb S}^{s-1} \times \R^{2(s-1)}$ for $p_1 < p\leq p_2$ and $\sigma \subset \{2,\dots, s\}$ of cardinal $|\sigma|\leq 5$ such that
\begin{itemize}
\item
$\cP_2 (a, p,\sigma)$ is parametrized only in terms of   $( t_m, v_m, \nu_m)$ for $m \in \sigma$ and  $m < \min \sigma$;
\begin{equation}
\label{eq: cerises statement}
\int \indc _{\cP_2 (a,p,\sigma)  } \displaystyle  \prod_{m\in \sigma} \,
\big | \big(v_{m}-v_{a(m)} ( t_{{m}} ) )\cdot \nu_{{m}}  \big | d  t_{{m}} d \nu_{{m}}dv_{{m}}   
\leq C(Rt)^r s^2 \eps \,,
\end{equation}
for some constant $r \geq 1$,
\item  and any pseudo-trajectory starting from $z_1$ at $t$, with total energy bounded by $R^2$ and such that the first two recollisions involve two disjoint pairs of particles having the same first parent is parametrized by 
$$( t_n, \nu_n, v_n)_{2\leq n\leq s }\in  \bigcup _{p_1 < p\leq p_2} \bigcup _\sigma  \cP_2(a, p,\sigma)\,.$$
\end{itemize}
\end{Lem}

\begin{proof}
As in the previous section, we suppose from now on that the parameters associated with the first recollision are such that 
$|\tau_1| |\bar v_i - v_j| \geq R$. Otherwise, the estimate~\eqref{eq: small distance cut off}
applied with $M = R$ leads to the expected upper bound.

In the following, the parents of $i,j$ will be denoted by the superscript $\ ^*$ and those of 
$k, \ell$ by the superscript $\tilde \ $.
Denote by $t_* := \min (t_{2^*}, t_{\tilde 2})$ the first time (before $t_{1^*}$) when one of the particles $i,j $ or $k$ has been deviated. 
Without loss of generality (up to exchanging $j$ and $k$), we can assume that $i$ and $k$ are not colliding together at time $t_*$. 

\medskip
We  describe the  recollision between $(i,j)$ by the identity
\begin{equation}
\label{firstrecoll}
v_{i} - v_{j} = \frac1{t_{rec}-t_{1^*}}  \big(
x_i(t_{1^*})
-x_j(t_{1^*})+q+\eps  \nu_{rec}
\big) \, ,
\end{equation}
with~$q$ an element in~$\Z^2$ which we fix from now on (in the end the estimates will be multiplied by~$R^2t^2$ to take this fact into account). Similarly the recollision between~$(k,\ell)$ can be written
\begin{equation}
\label{secondrecoll}
v_\ell - v_{k} = \frac1{\tilde  t_{rec}-t_{1^*}}  \big(
x_i(t_{1^*})+\eps \nu_{1^*}
-x_k(t_{1^*})+\tilde q+\eps  \tilde\nu_{rec}
\big)\end{equation}
with~$\tilde q$ an element in~$\Z^2$ which again we fix from now on, up to mutiplying again the estimates by~$R^2t^2$ at the end.

\medskip

We introduce the notation
$$
\tilde x_{i,k}(t_{1^*}):= x_i(t_{1^*}) -x_k(t_{1^*})+\tilde q  \quad \mbox{and} \quad x_{i,j}(t_{1^*}):= x_i(t_{1^*}) 
-x_j(t_{1^*})+  q \, .
$$
Equation~(\ref{firstrecoll}) implies that $v_i - v_j$   lies in a rectangle $\cR_1$ of main axis~$x_{i,j}(t_{1^*})$, and of size~$CR\times ( R\eps /|x_{i,j}(t_{1^*})| )$. We recall
that an integration of this constraint in the collision parameters of particle~$1^*$ gives a bound of the type~$\min (1, \eps |\log \eps|^2 / |\bar v_i - v_j|)$.
On the other hand, Equation~(\ref{secondrecoll}) implies that~$v_\ell- v_k$ lies in a rectangle $\cR_2$
of main axis~$\tilde x_{i,k}(t_{1^*})$ and
of size $CR\times (R\eps / |\tilde x_{i,k}(t_{1^*})|)$.

\bigskip

Let us give the main ideas of the argument.
We can rewrite these conditions with Carleman's parametrization \eqref{defcarleman}, with either $(v_i,v_\ell) = (V', V'_*)$ or $(v_i, v_\ell) = (V'_*, V')$. We will actually focus on the second situation which is the worst one.
We will use the  parametrization in polar coordinates as in Figure \ref{fig: parametrisation carleman}.

The first condition states that $V'_*$ lies in a small rectangle of size $CR\times (R\eps /|x_{i,j}(t_{1^*})| )$, which  we shall eventually integrate with the measure $d\mu d \psi $.
We can show that this integral  provides a contribution
$(R\eps /|x_{i,j}(t_{1^*})| )(1+ |\log (\eps /|x_{i,j}(t_{1^*})| ) |)$.

The second condition tells us that $V'$ has to be in the intersection of the line orthogonal to $(V'_*-\bar v_i)$ passing through $\bar v_i$ and the rectangle $v_k +\cR_2$. 
{  We have therefore to evaluate the length of this intersection which appears when we integrate with respect to $\lambda d\lambda$.}

\begin{figure} [h] 
  \centering
 \includegraphics[width=4in]{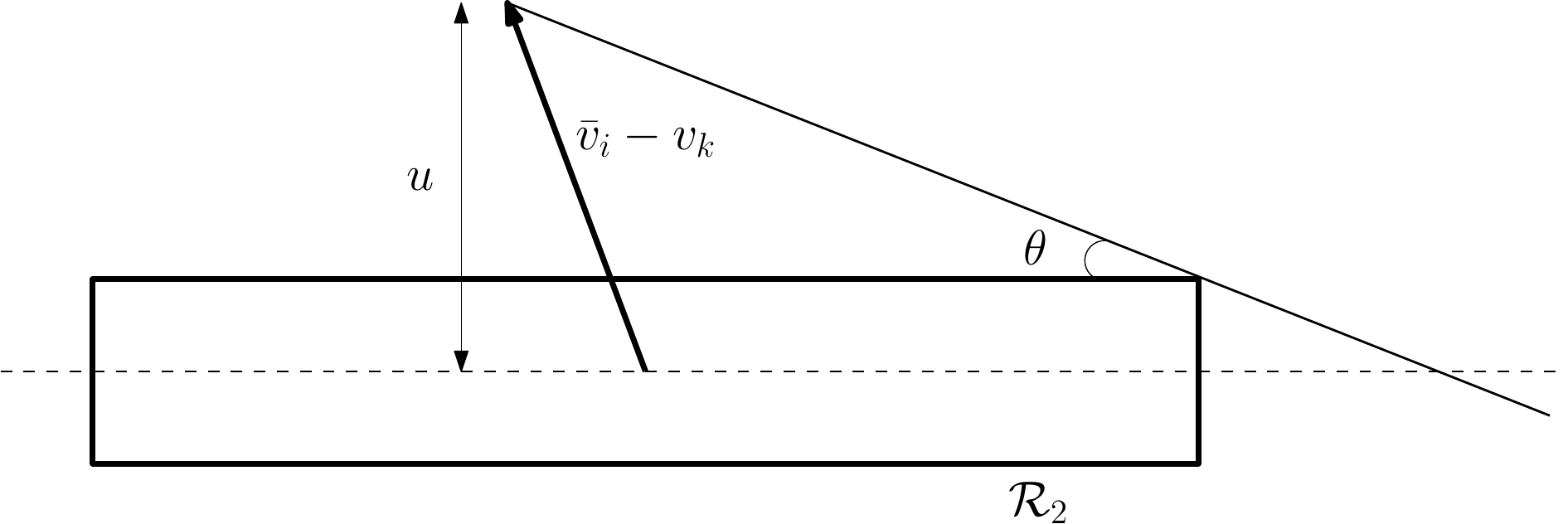} 
 \caption{\small The dashed line represents the main axis of the rectangle $\mathcal R_2$, oriented in the direction $\tilde x_{i,k}(t_{1^*})$. 
The angle $\theta$ is  the smallest angle between the axis of $\mathcal R_2$ and 
 any line passing through  $v_i$ and intersecting the axis of $\mathcal R_2$.}
\label{fig: angle}
\end{figure}

   Denote by $u$ the distance from $\bar v_i$ to the rectangle $v_k + \mathcal R_2$~:
   $$ u := \big |(  \bar v_i - v_k) \wedge \tilde x_{i,k}(t_{1^*}) |/ |  \tilde x_{i,k}(t_{1^*})|\,$$
    \begin{itemize}
   \item if this distance is large enough, we expect the length of the intersection to be small;
   \item
   if the distance $u$ is small, this imposes an additional constraint on $\bar v_i - v_k$, that we will analyse with different arguments depending on the size of $\tilde x_{i,k}(t_{1^*})$.
   \end{itemize}

   \bigskip
    \noindent
\underbar{Case $u \geq \eps^{3/4} $ }

The intersection of  the line orthogonal to $(V'_*-\bar v_i)$ passing through $\bar v_i$ and the rectangle $v_k +\cR_2$ 
of width $\frac{C R\eps}{ |\tilde x_{i,k}(t_{1^*})|}$ (see Figure \ref{fig: angle}) is a segment of size at most
\begin{equation}
\label{eq: distance d}
d \leq \min \Big(  \frac{C\eps R }{|\tilde x_{i,k}(t_{1^*})| \sin \theta } , CR\Big),
\end{equation}
where  $\theta $ is the minimal angle between the axis of $\cR_2$ and any line passing through $\bar v_i$ and intersecting $v_k+ \cR_2$. 
{ 
We have
$$
\sin \theta \geq \frac{u}{2 R} \geq \frac{\eps^\frac34}{2 R} \, \cdotp
$$
}
It follows from \eqref{eq: distance d} that
$$
d 
\leq  \frac{C\eps^\frac14 R^2  }{ |\tilde x_{i,k}(t_{1^*})| } \, \cdotp
$$

Multiplying this estimate by the size of $\cR_1$, we get the following upper bound  for  the measure in~$|(v_{1^*}-\bar v_i) \cdot \nu_{1^*} | \, dv_{1^*} d\nu_{1^*}$
$$
\frac{C  R^5 \eps^{\frac54}|\log \eps| }{   \,  \big |\tilde x_{i,k}(t_{1^*})\big | \,  \big |x_{i,j}(t_{1^*})\big |} \big(1+ \big |\log  \frac \eps { |x_{i,j}(t_{1^*})|}  \big |\big )  \, \cdotp$$

\medskip
$\bullet$  If $|\tilde x_{i,k}(t_{1^*})\big | \geq \eps^{1/8}$ then the bound becomes
$$
\frac{C R^5 \eps^{\frac98}|\log \eps| }{    \big |x_{i,j}(t_{1^*})\big |} \, ,$$
and we are back to the usual computations as in the proof of Proposition~\ref{recoll1-prop}: we integrate over $t_{1^*}$ then  over one parent of~$(i,j)$ to kill the singularity at small relative velocities,  and this gives rise in the end to 
$$
C   R^5 (R^3t)^2 \eps^{\frac98}  |\log \eps|^3  \, .
$$ 

\medskip
$\bullet$  If $|\tilde x_{i,k}(t_{1^*})\big | \leq \eps^{1/8}$, we have a  kind of ``recollision" between particles~$i$ and~$k$ at time~$t_1^*$. Denote by $\tilde 2, \tilde 3$ the first two parents of $(i,k)$. We therefore get that $\bar v_i - v_k$ has to belong to the union of $(Rt)^2$ rectangles $\cR_3$ of size $CR\times ( R\eps^{1/8} /|\tilde x_{i,k}(t_{\tilde 2 })| )$, with
$$
\tilde x_{i,k}(t_{\tilde 2}) : = x_i(t_{\tilde 2}) -x_k(t_{\tilde 2})  +\tilde q \,.
$$
 Combined with the condition that $v_i \in \cR_1$, one has to integrate 
$$ R^3 \indc_{ \{ \bar v_i- v_k \in \cR_3 \}}  \min \left( {\eps \, |\log \eps|^2\over |\bar v_i - v_j|}, 1 \right) .$$
Denote by $\sigma =\{1^*, 2^*, 3^*\} \cup \{\tilde 2, \tilde 3 \}$ so that the cardinal $|\sigma|$ can be 3, 4 or 5.
Integrating over~$1^*$ leads to an inequality involving constraints on the pairs $(i,k)$ and $(i,j)$
\begin{align}
\label{eq: holder}
& \int \indc_{\{ \bar v_i -v_k \in \cR_3 \}} \indc_{ \{ v_i - v_j \in \cR_1 \} } 
\,  \prod_{m \in \sigma  } b(\nu_m,v_m)  \,  dv_m d\nu_m dt_m  \nonumber \\
& \qquad \leq  R^3 \eps |\log \eps|^2 \int \frac{ \indc_{\{ \bar v_i -v_k \in \cR_3 \} } }{ |  \bar v_{i}- v_{j}|} 
\,  \prod_{m \in \{ 2^*, 3^*, \tilde 2, \tilde 3 \} } b(\nu_m,v_m)  \,  dv_m d\nu_m dt_m   \nonumber \\
&\qquad \leq     R^3 \eps |\log \eps|^2
\left( (R^3t)^2 \int \indc_{ \{ \bar v_i -v_k \in \cR_3 \} }     \prod_{m= \tilde 2, \tilde 3} b(\nu_m,v_m)  \,  dv_m d\nu_m dt_m \right)^{1/4}\\
& \qquad \qquad \qquad \times 
\left((R^3t)^2  \int \frac{ 1 }{ |  \bar v_{i}- v_{j}|^{4/3}} 
 \prod_{m= 2^*,3^*} b(\nu_m,v_m)  \,  dv_m d\nu_mdt_m\right)^{3/4}  \nonumber \\
 &\qquad \leq     C(Rt)^r  \eps ^{33/32} |\log \eps|^{11/4} \,,   \nonumber 
 \end{align}
where we used H\"older's inequality in order to decouple both terms: the first one provides as previously a bound~$\eps^{\frac18}|\log \eps|^{3}$ and the second one is bounded thanks to~(\ref{alpha1})-(\ref{alpha2})
as the singularity in the relative velocities is less than 2.
Note that H\"older's inequality was performed over the 4 variables $\{ 2^*, 3^*, \tilde 2, \tilde 3 \}$, but 
only two variables are relevant for each integral, thus the contribution of the two others is bounded from above by the factor $(R^3t)^2$.

\bigskip
    \noindent
\underbar{Case $u \leq \eps^{3/4} $ }

We   recall that 
$$
\begin{aligned}
\tilde x_{i,k}(t_{1^*}) &: = x_i(t_{1^*}) -x_k(t_{1^*})+\tilde q
= x_i(t_{\tilde 2}) -x_k(t_{\tilde 2})  +\tilde q+  ( \bar v_i - v_k) (t_{1^*}-t_{\tilde 2}) \, .
\end{aligned}
$$
Recalling that 
$$
\tilde x_{i,k}(t_{\tilde 2}) : = x_i(t_{\tilde 2}) -x_k(t_{\tilde 2})  +\tilde q \, ,
$$
the constraint~$u = \frac{ |(  \bar v_i - v_k) \wedge \tilde x_{i,k}(t_{1^*}) |}{|  \tilde x_{i,k}(t_{1^*})|   } \leq  \eps^\frac34$ implies
\begin{equation}
\label{constraintrectangle3}
\big |(  \bar v_i - v_k) \wedge   \tilde x_{i,k}(t_{\tilde 2})   \big | \leq  C\eps^\frac34 Rt \, .
\end{equation}
Recall  that the constraint \eqref{firstrecoll} on the rectangle~${\mathcal R}_1$ produces a singularity
$\frac{R\eps}{|x_{i,j}(t_{1^*})|} (1+ |\log ( \frac{\eps}{|x_{i,j}(t_{1^*})| }) |) $, and
we argue as follows: 

\medskip

\noindent
$\bullet$  If~$| \tilde x_{i,k}(t_{\tilde 2})|  \leq \eps^{\frac58}$, we have a  kind of ``recollision" between particles~$i$ and~$k$ at time~$t_{\tilde 2} $. We thus proceed as in Case 1 of Section 6.

- For small relative velocities, we integrate the constraint $|\bar v_i-v_j | \leq \eps^\frac9{16}$ over two parents of~$\{ i, j \}$ using \eqref{preimage-sphere1}, \eqref{carleman2} and we find directly a bound~$ \eps^\frac98 |\log \eps|^2$.

- When the relative velocities are bounded from below $|\bar v_i-v_j | \geq \eps^\frac9{16}$,    the 
contribution  of rectangle~${\mathcal R}_1$ gives a bound of the order~$CR^2 \eps^\frac{7}{16} |\log \eps|^2$. By integrating the  ``recollision"~$(i,k)$  over $\tilde 2, \tilde 3$, we find   a bound~$CR^7 t^3 \eps^\frac{5}{8} |\log \eps|^3$ so finally this case produces as usual (see Proposition~\ref{recoll1-prop}), after integration over three parameters, the error~$CR^9 t^3 \eps^\frac{17}{16} |\log \eps|^5$.

\medskip
\noindent
$\bullet$ If ~$| \tilde x_{i,k}(t_{\tilde 2})  | \geq \eps^{\frac58}$ then  according to~(\ref{constraintrectangle3}),~$\bar v_i-v_k$ must lie in the union of $(Rt)^2$ rectangles~${\mathcal R}_4$ with axis $\tilde x_{i,k}(t_{\tilde 2})$ and size~$CR \times CRt\eps^{\frac18}$.
This condition has to be coupled with the singularity~$\eps |\log \eps|^2/|\bar v_i-v_j| $ due to the constraint from the rectangle $\cR_1$.
We therefore have to integrate
$$ 
R^3 \indc_{ \{ \bar v_i- v_k \in \cR_4 \} }  \min \left( {\eps   |\log \eps|^2 \over |\bar v_i - v_j|}, 1 \right) .
$$
Denote by $\sigma =\{1^*, 2^*, 3^*\} \cup \{\tilde 2\}$ where $\tilde 2$ is the first parent of $(i,k)$.
In this case, the cardinal of $\sigma$ is 3 or 4.
Integrating first over $1^*$ and then using H\"older's inequality as in \eqref{eq: holder}, we  have
$$
\begin{aligned}
& 
\int \indc_{ \{ \bar v_i -v_k \in \cR_4 \} } \indc_{ \{ v_i - v_j \in \cR_1 \} } 
\,  \prod_{m \in \sigma  }b(\nu_m,v_m)  \,  dv_m d\nu_mdt_m \\
& \qquad 
\leq R^3 \eps   |\log \eps|^2   \int {  \indc_{ \{ \bar v_i- v_k \in \cR_4 \} } \over |\bar v_i - v_j|}
\,  \prod_{m \in \sigma \setminus \{ 1^*\} }b(\nu_m,v_m)  \,  dv_m d\nu_mdt_m \\
&\qquad 
\leq     R^3 \eps |\log \eps|^2
\left( (R^3t)^2 \int \indc_{ \{ \bar v_i -v_k \in \cR_4 \}}     b(\nu_{\tilde 2},v_{\tilde 2})  \,  dv_{\tilde 2} d\nu_{\tilde 2} dt_{\tilde 2} \right)^{1/4}
\\
& \qquad \qquad \qquad \times \left((R^3t)  \int \frac{ 1 }{ |  \bar v_{i}- v_{j}|^{4/3}} 
\prod_{m= 2^*,3^*} b(\nu_m,v_m)  \,  dv_m d\nu_mdt_m \right)^{3/4}\\
 &\qquad \leq     C(Rt)^r  \eps ^{33/32} |\log \eps|^{9/4} \,.\nonumber
\end{aligned} 
$$

\bigskip

Given a set $\sigma$ of parents, it may only determine the particle $i$, so that an extra factor $s^2$ has to be added in \eqref{eq: cerises statement}
to take into account the choice of $j,k$. This completes the proof of   Lemma~\ref{doublerecollision}.
\end{proof}


\subsection{Recollisions in chain}

The following Lemma was used in Section~\ref{secondcase1=1} page~\pageref{secondcase1=1}  to deal with the case when recollisions occur in chain, with~$t_{\tilde 1} = t_{1^*}$, i.e. both recollisions occur without any intermediate collisions 
as depicted in Figure~\ref{murphyfigagain}.

\begin{figure} [h] 
  \centering
  \includegraphics[width=5cm]{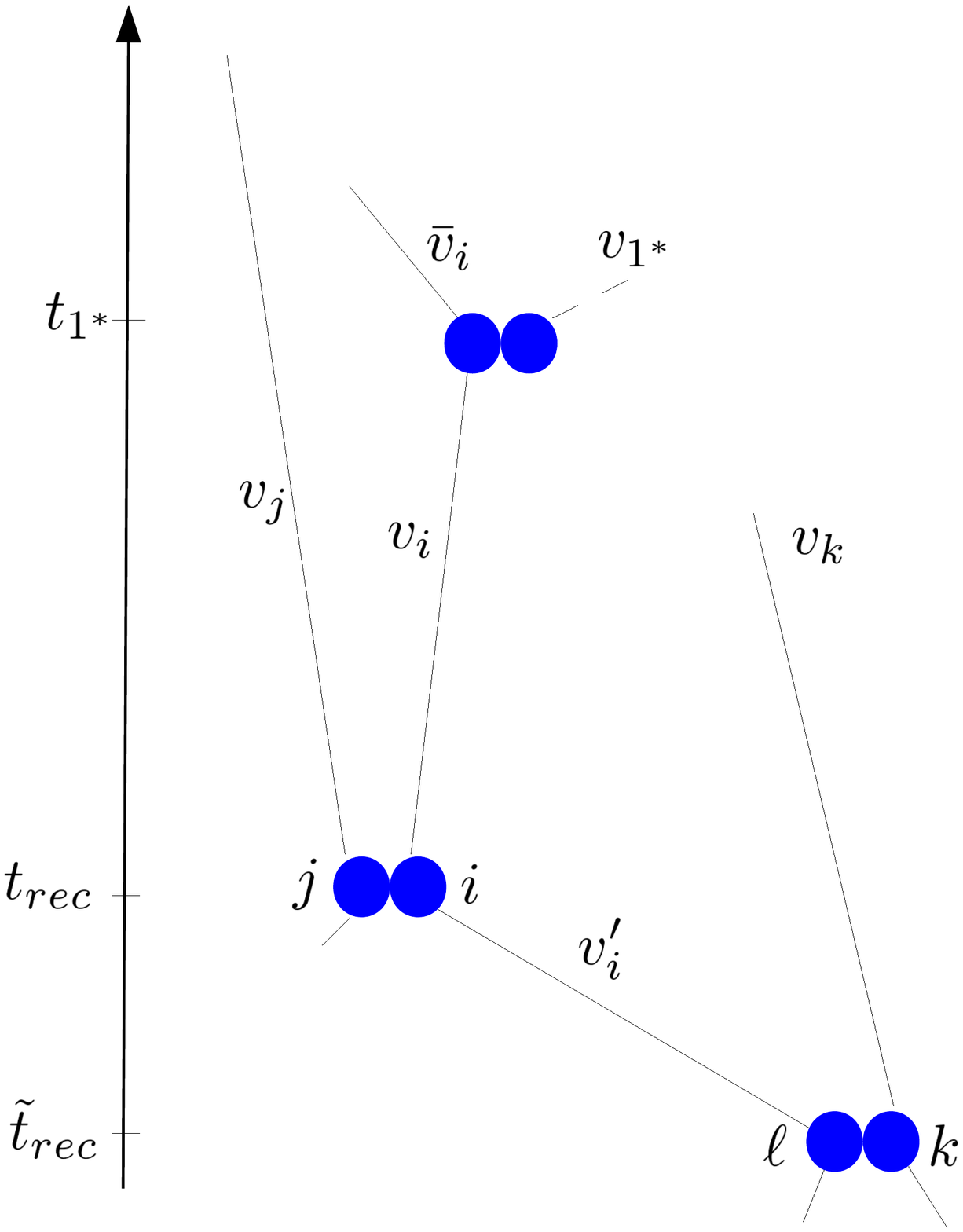}  
  \caption{ Recollisions in chain.}\label{murphyfigagain}
\end{figure}

\begin{Lem}
\label{Murphy}
Fix a final configuration of bounded energy~$z_1 \in \T^2 \times  B_R$ with~$1 \leq R^2 \leq C_0 |\log \eps|$,  a time $1\leq t \leq C_0|\log \eps|$ and  a collision tree~$a \in \cA_s$ with~$s \geq 2$.

There exist sets of bad parameters
$\cP_2 (a, p,\sigma)\subset \cT_{2,s} \times {\mathbb S}^{s-1} \times \R^{2(s-1)}$ for $p_2 < p\leq p_3$ and $\sigma \subset \{2,\dots, s\}$ of cardinal $|\sigma|\leq 3$ such that
\begin{itemize}
\item
$\cP_2 (a, p,\sigma)$ is parametrized only in terms of   $( t_m, v_m, \nu_m)$ for $m \in \sigma$ and  $m < \min \sigma$;
\begin{equation}
\label{eq: Recollisions in chain}
\int \indc _{\cP_2 (a,p,\sigma)  } \displaystyle  \prod_{m\in \sigma} \,
\big | \big(v_{m}-v_{a(m)} ( t_{{m}} ) )\cdot \nu_{{m}}  \big | d  t_{{m}} d \nu_{{m}}dv_{{m}}   
\leq C(Rt)^r s^2 \eps \,,
\end{equation}
for some constant $r \geq 1$,
\item  and any pseudo-trajectory starting from $z_1$ at $t$, with total energy bounded by $R^2$ and such that the first two recollisions  occur in chain as in Figure~{\rm\ref{murphyfigagain}} is parametrized by 
$$( t_n, \nu_n, v_n)_{2\leq n\leq s }\in  \bigcup _{p_2 < p\leq p_3} \bigcup _\sigma  \cP_2(a, p,\sigma)\,.$$
\end{itemize}
\end{Lem}

\begin{proof}
Recall that the condition for the first recollision states
\begin{equation}
\label{firstcondmurphy}
v_{i} - v_{j} = \frac1{\tau_{rec}}  \delta x_\perp - \frac{\tau_1}{\tau_{rec}} (\bar v_{i} - v_{j})-  \frac1{\tau_{rec}}  \nu_{rec} \, ,
\end{equation}
with $x_i, x_j$ the positions at time $t_{2^*}$
\begin{equation}
\label{changevarmurphy}
\begin{aligned}
\delta x &:= \frac1\eps (x_i-x_j-q)=  \lambda  ( \bar v_{i} - v_{j})  + \delta x_\perp
\qquad  \hbox{with} \quad \delta x_\perp \cdot (  \bar  v_i - v_{j}) =0 \, ,\\
\tau_1 & := \frac1\eps (t_{1^*} - t_{2^*} - \lambda  )  \, ,\qquad  \tau_{rec}: =\frac1\eps (t_{rec} - t_{1^*}) \, ,
\end{aligned}
\end{equation}
for some $q$ in $\Z^2$ of norm smaller than~$O(Rt)$ to take into account the periodicity.

\medskip

When $|\tau_1| |  \bar v_i-v_j| \leq R^2$,   estimate \eqref{eq: small distance cut off} is enough to obtain an upper bound of order $\eps$ without taking into account the second recollision.
Our goal here  is to prove that the constraint of having a second recollision produces an integrable 
function of~$|\tau_1| \, |  \bar v_i-v_j| \geq R^2$, hence a bound~$O(\eps)$ after integration over~$1^*$.

\medskip

From \eqref{firstcondmurphy}, we deduce as in \eqref{sizetaurec} that
\begin{equation}
\label{sizetaurec bis}
\frac1{|\tau_{rec} |} \leq  {4R\over |\tau_1| |  \bar  v_{i} - v_{j}|}
\quad \text{which implies that} \quad 
|\tau_{rec} | \geq R/4 \gg 1\, .
\end{equation}

\medskip


Two cases have to be considered: $k= 1^*$ and~$k\neq 1^*$.

\bigskip
    \noindent
\underbar{Case $k=1^* $.}

 The equation for the second recollision states
$$
\tau'_{rec} (v'_{i} - {v_{1^*}'})=\pm  \nu_{1^*} - \tau_{rec} ( v_i - {v_{1^*}'}) (+\nu_{rec})-\tilde \nu_{rec},
$$
where
$$
\tau'_{rec} :=\frac1\eps (\tilde  t_{rec} -  t_{rec}) \, , 
$$
and
where the $\pm$ and the translation by~$\nu_{rec}$ depend on the possible exchanges in the labels of the particles at collision times.
It can be rewritten, thanks to~(\ref{formulav"i}),
\begin{equation}
\label{rec2bis-equation}
\begin{aligned}
 \tau'_{rec} (v_j - v_i )\cdot \nu_{rec} \,  \nu_{rec} 
& =\pm  \nu_{1^*} - (\tau_{rec}+ \tau'_{rec}) ( v_i - {v'_{1^*}}) (+\nu_{rec})-\tilde \nu_{rec} \\
\hbox{ or } \qquad  \tau'_{rec} (v_j - v_i )\cdot \nu_{rec}^\perp \, \nu_{rec}^\perp 
&=\pm  \nu_{1^*} - (\tau_{rec}+ \tau'_{rec}) ( v_i - {v'_{1^*}}) (+\nu_{rec})-\tilde \nu_{rec}\,.
\end{aligned}
\end{equation}
We further know that $|v_i - {v'_{1^*}}| = | \bar v_i - v_{1^*}|$.

\medskip
 
- If $| \bar v_i - v_{1^*}| \geq  R \, |\tau_{rec}|^{-3/4}$, then 
the vector in the right-hand side of \eqref{rec2bis-equation} has a magnitude of order 
$$
|\tau_{rec}+ \tau'_{rec}| \, | v_i - {v'_{1^*}}| \geq |\tau_{rec}|  \, | v_i - {v'_{1^*}}|
\geq R \, |\tau_{rec}|^{1/4} \, .
$$ 
It follows that the vector $\nu_{rec}$ has to be aligned in the direction of $v_i - v_{1^*}'$ with a controlled error
$$
\nu_{rec} = \cR_{n\pi/2} { v_i - {v_{1^*}'}\over | \bar v_i- v_{1^*}|} 
+ O \left( {1\over| \tau_{rec}|^{1 /4}} \right)\,,
$$
for $n=0,1,2,3$, recalling that $\cR_\theta$ is the rotation of angle $\theta$.

\medskip
 
Plugging the  formula for $\nu_{rec}$ into~(\ref{firstcondmurphy})
and using  \eqref{sizetaurec bis}, we  get 
$$
v_{i} - v_{j} = \frac1{\tau_{rec} } \delta x_\perp - \frac{\tau_1}{\tau_{rec}} ( \bar v_{i} - v_{j})
-  \frac1{\tau_{rec}} \cR_{n\pi/2}{ v_i - {v_{1^*}'}\over | \bar v_i- v_{1^*}|}  
+ O \left(  {R^{5 /4} \over |\tau_1|^{5 /4}  \;   |\bar  v_{i} - v_{j}|^{5 /4}}  \right)\,.
$$
This equation has the same structure as \eqref{eqansatz}. Thus using 
the same arguments as in the proof of Lemma \ref{v-constraint}, we get 
\begin{itemize}
\item  a contribution of size  $O( 
 |\tau_1|^{-5 /4}  \;   |\bar  v_{i} - v_{j}|^{-5 /4}|\log| \tau_1 (\bar v_i - v_j)| )$ when the mapping  $v_i \mapsto  {v_i - v_{1^*}' \over |v_i - v_{1^*}'|}$ is Lipschitz with constant strictly less than $|w|^{\gamma}$ for some $\gamma  \in (0,1)$;
 \item the  same  integrable contribution as in Lemma \ref{v-constraint}  in degenerate cases when some velocities are close to each other (typically at a distance  $O(|w|^{-\gamma})$).
 \end{itemize}
Thus integrating with respect to $t_{1^*}$ we recover the factor $\eps$ and the singularity in $|\bar v_{i} - v_{j}|$ is removed as usual by integration over the parents of $i,j$.

\medskip
 
- If $| \bar v_i - v_{1^*}| \leq  R|\tau_{rec} |^{-3/4}$, we find that $v_{1^*}$ has to belong to a domain of 
size less than $(  |\tau_1| \,  | \bar v_i-v_j| )^{-3/2}$ as $|\tau_{rec}| \geq |\tau_1| \,  | \bar v_i-v_j|$.
Hence again, we obtain an integrable function of $ |\tau_1| \,  | \bar v_i-v_j|$, with no extra gain in~$\eps$.

\bigskip
    \noindent
\underbar{Case $k\neq1^* $.} In the following we denote by~$1^*,2^*\dots$ the parents of the set~$(i,j,k,\ell)$ at time~$t_{rec}$: contrary to previous cases, and since they both have the same first parent  we do not distinguish the parents of~$(i,j)$ and~$(k,\ell)$ but consider them as a whole.

 The position of particle $k$ at the time $\tilde  t_{rec}$ of the second recollision is given by
$$x_k (t_{rec}) = x_k +  v_k (\tilde  t_{rec} -t_{2^*}) \, .$$
We have written~$x_k$ for the position of particle~$k$ at time~$t_{2^*}$. We end up with the condition for the second recollision
\begin{equation}
\label{eq: b22}
(\tilde  t_{rec} -  t_{rec}) (v'_{i} - v_{k}) =  (x_j-x_k )(t_{rec})   -\eps\tilde\nu_{rec}(+\eps\nu_{rec})+ \tilde q \, ,
\end{equation}
for some~$\tilde q \in \Z^2$ not larger than~$O(R t)$, and where the translation $\eps \nu_{rec}$ arises only if the labels of particles are exchanged at  $t_{rec}$. 
In the following, we fix~$q$ and~$\tilde q $ and will multiply the final estimate by~$(R^2 t^2)^2$ to take into account the periodicity in both recollisions.
Using the notation \eqref{changevarmurphy},
we then rescale in~$\eps$ and write 
$$
\tau_{rec}:=\frac{t_{rec} -  t_{1^*}}\eps \, , 
\quad 
 \tau'_{rec}:=\frac{ \tilde  t_{rec} -  t_{rec}}\eps \, \cdotp
$$
Then  Equation \eqref{eq: b22} for the second recollision becomes
\begin{equation}
\label{rec2-equation}
 \tau'_{rec}(v'_i - v_k) 
=  \tilde x_{jk} (t_{rec}) - \tilde\nu_{rec} ( + \nu_{rec}) 
 \, ,
 \end{equation}
where $\eps \tilde  x_{jk} (t_{rec})$ stands for the relative position between $j,k$ at time $t_{rec}$.

As  in the proof of Lemma \ref{rec-eq}, the equation \eqref{rec2-equation} implies that $v'_i - v_k$  belongs to a rectangle $\cR$ of size  $2R \times \frac{2R}{|\tilde x_{jk} (t_{rec})|}$
and axis $ \tilde x_{jk} (t_{rec})$.
Furthermore $v'_i$  belongs as well to the circle of diameter $[v_j, v_i]$ by definition.
Computing the intersection of the rectangle and of the circle, we obtain a constraint on the angle $\nu_{rec}$. Then plugging this constraint in the equation for the first recollision, we will conclude as in Lemma B.1 that $v_i$ has to belong to a very small set.

\begin{figure} [h] 
   \centering
  \includegraphics[width=6cm]{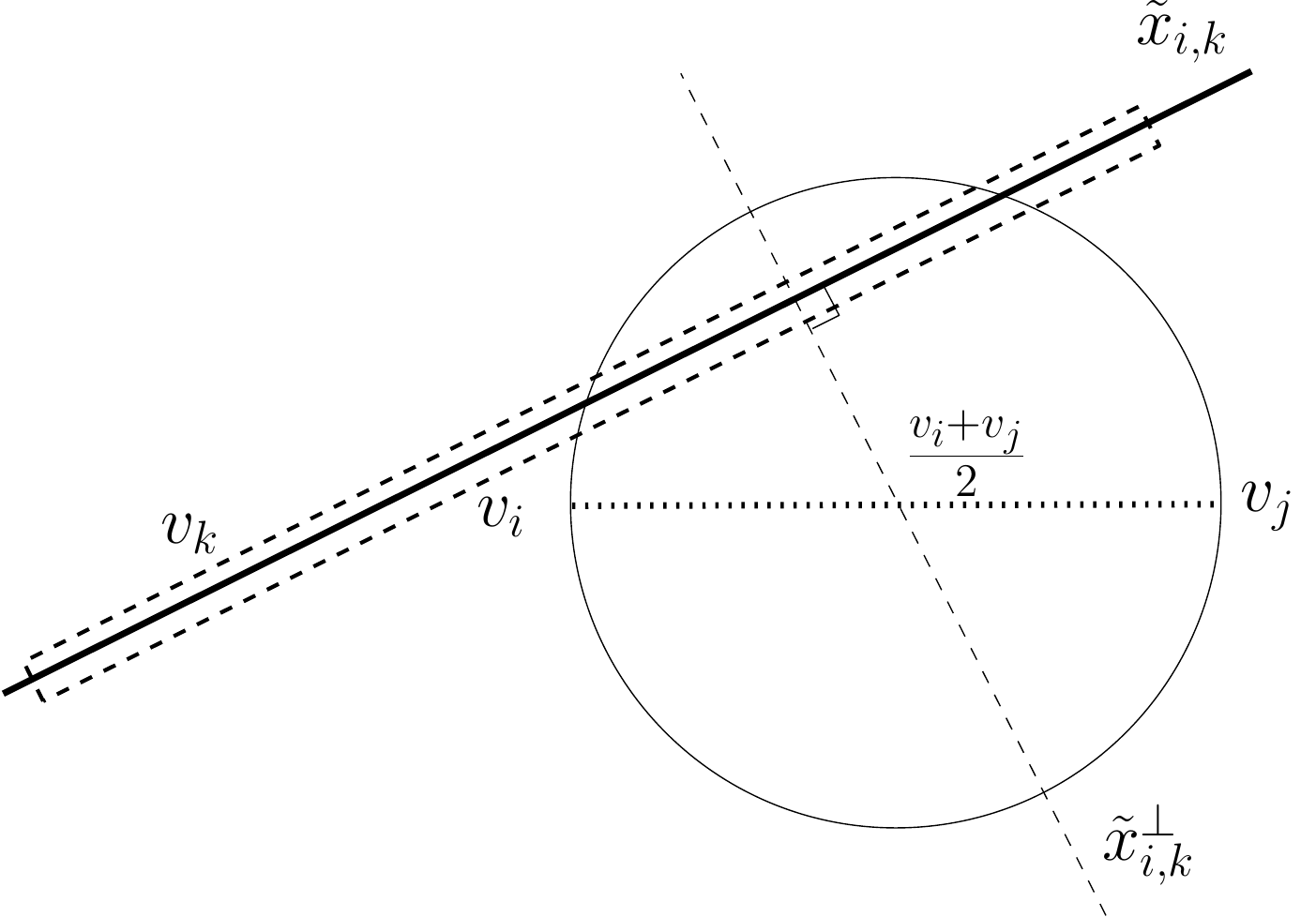} 
  \caption{The velocity $v_i'$ belongs to the rectangle of axis $\tilde  x_{jk} (t_{rec})$ as well as
  to the circle of diameter $[v_j, v_i]$.}
\label{fig: LemmeB3_tangence}
\end{figure}

This strategy can be applied in most situations. We have however to deal separately with the  two following geometries~:
\begin{itemize}
\item[-]
   if the relative velocity 
$v_i- v_j $ is small, the  rectangle can contain a macroscopic part of the circle~: we forget about the second recollision and just study the constraint of  small relative velocities;
\item[-]
 if the distance $| \tilde  x_{jk} (t_{rec}) |$ is small, then $i$ will be close to $k$ at the first recollision time and this will facilitate the second recollision~: we then forget about the second recollision and write  two independent constraints at the first recollision time.
\end{itemize}

\medskip

\noindent
$\bullet $ Suppose  that 
\begin{equation}
\label{bound-case1}
| v_i- v_j | \geq {1 \over |\tau_{rec}|^{5 /8}}
\quad \text{and} \quad 
| \tilde  x_{jk} (t_{rec}) |  \geq| \tau_{rec} |^{ 3 / 4} .
\end{equation}


From (\ref{rec2-equation}), we deduce that a necessary condition for the second recollision to hold is that
$$  
(v_i' - v_k) \cdot {\tilde  x_{jk}^\perp  (t_{rec})\over
| \tilde  x_{jk} (t_{rec})|}  
=
{(v_i - \big(  (v_i-v_j) \cdot \nu_{rec} \big) \, \nu_{rec} - v_k ) \cdot \tilde  x_{jk}^\perp  (t_{rec})\over
| \tilde  x_{jk} (t_{rec})|}  
= O\Big( {1\over |\tau'_{rec}|}\Big) ,
$$
where $|\tau'_{rec}|$ can be bounded from below thanks to \eqref{bound-case1}
$$
|\tau'_{rec}| \geq {| \tilde  x_{jk} (t_{rec})|\over 4R} \geq {|\tau_{rec}|^{3/4} \over 4R}\,.
$$
Using the bound from below on the relative velocity $|v_i - v_j|$, we finally get
$$ \Big( {v_i-v_j\over |v_i - v_j|}  \cdot \nu_{rec} \Big) \Big( {\tilde  x_{jk}^\perp  (t_{rec})\over |\tilde  x_{jk} (t_{rec})|}  \cdot \nu_{rec} \Big) ={(v_i-v_k) \cdot\tilde  x_{jk}^\perp  (t_{rec}) \over |v_i-v_j| |\tilde  x_{jk}^\perp  (t_{rec})|} +  O\Big( {1\over |\tau_{rec}|^{1/8} }\Big)\,.$$
Define the angles  $\theta= < \tilde  x_{jk}^\perp  (t_{rec}),  \nu_{rec} > $ and  $ \alpha =< \tilde  x_{jk}^\perp  (t_{rec}), v_i-v_j>$. We have
$$ 
\cos\theta \cos(\theta - \alpha) = \frac12 \Big( \cos(2\theta - \alpha) +\cos\alpha\Big) = {(v_i-v_k) \cdot\tilde  x_{jk}^\perp  (t_{rec}) \over |v_i-v_j| |\tilde  x_{jk}^\perp  (t_{rec})|}+  O\Big( {1\over |\tau_{rec}|^{1/8} }\Big),
$$
so that 
$$ 
\cos(2\theta - \alpha)  = {(v_i + v_j - 2v_k) \cdot\tilde  x_{jk}^\perp  (t_{rec}) \over |v_i-v_j| |\tilde  x_{jk}^\perp  (t_{rec})|}+  O\Big( {1\over |\tau_{rec}|^{1/8} }\Big).
$$
Recall the notation of the proof of Lemma \ref{v-constraint}
$$
w:= \delta x^\perp_{i,j}  -  (\bar v_{i} - v_{j}) \tau_1, \quad \text{ and} \quad u:={|w| \over \tau_{rec}} 
\leq 4R,
$$  
where $\eps w$ is the distance between $x_i,x_j$ at time $t_{1^*}$ and it is enough to consider $|w| \geq R^2$ thanks to \eqref{eq: small distance cut off}.
As the derivative of $\arccos$ is singular at $\pm1$, we will consider an approximation $\arccos_{|w|}$ which coincides with $\arccos$ on $[-1 + \frac{1}{|w|^{2 \delta}}  ,1- \frac{1}{|w|^{2 \delta}} ]$ (for a given $\delta \in (0,\frac{1}{16})$) and is constant in the rest of $[-1,1]$ so that 
$$
\big| \partial_x \arccos_{|w|} (x) \big| \leq  |w|^\delta
\quad \text{and} \quad 
\|  \arccos_{|w|} - \arccos  \|_\infty \leq \frac{1}{|w|^\delta} .
$$
Thus the angle $\theta$ can be approximated by 
$$ 
\theta = \bar \theta _\pm +  O\Big( \left( \frac{u}{|w|} \right)^{1/8}|w|^\delta
 +  \frac{1}{|w|^\delta} \Big),
$$
with
\begin{equation}
\label{bartheta-def}
\bar \theta _\pm = \pm \frac12 \arccos_{|w|} \left( {(v_i + v_j - 2v_k) \cdot\tilde  x_{jk}^\perp  (t_{rec}) \over |v_i-v_j| |\tilde  x_{jk}^\perp  (t_{rec})|}\right) 
+ \frac12 < \tilde  x_{jk}^\perp  (t_{rec}), v_i-v_j> .
\end{equation}
Plugging this constraint in the equation for the first recollision, we get
\begin{equation}
\label{rec1-lemB3}
v_{i} -v_j = u {w\over |w|} 
- {u\over |w|} \cR_{\bar \theta _\pm} {\tilde  x_{jk}^\perp  (t_{rec}) \over |\tilde  x_{jk} (t_{rec})|} 
+ O \Big(  u^\delta \left( \frac{u}{|w|} \right)^{9/8-\delta} 
+   \frac{u}{|w|} \, \frac{1}{|w|^\delta} \Big) .
\end{equation}
As $|w| \gg 1$, the leading term of this equation is $v_{i} -v_j \simeq u {w\over |w|}$, but we have to analyse carefully the corrections.
Compared with the formulas of the same type encountered in the proof of Lemma \ref{v-constraint}, this one has the additional difficulty that the dependence with respect to $u$ is very intricate.
Instead of solving \eqref{rec1-lemB3}, we are going to look at sufficient conditions satisfied by the solutions of \eqref{rec1-lemB3}. 
In particular, $u$ will be considered as a parameter independent of $|\tau_{rec}|$.
For a given $u$, we are going to solve the equation
\begin{equation}
\label{barv-def}
v_{i} -v_j = u {w\over |w|} 
- {u\over |w|} \cR_{\bar \theta _\pm} {\tilde  x_{jk}^\perp  (t_{rec}) \over |\tilde  x_{jk} (t_{rec})|} 
\quad \text{with} \quad 
|v_{i} -v_j| \geq  \frac{1}{2}  \left( {u \over |w|} \right)^{5/8},
\end{equation}
where $\tilde  x_{jk} (t_{rec})$ was originally defined in  \eqref{rec2-equation} as the relative position between $j,k$ at time $t_{rec}$, but is now simply a function of $u$
$$
\tilde  x_{jk} (t_{rec}) = x_{jk} (t_{1^*} ) - {|w| \over u} (v_j-v_k).
$$
The solutions of \eqref{rec1-lemB3} are such that $u \simeq |v_{i} -v_j|$, 
thus from the condition \eqref{barv-def} on the relative velocities,
 it is enough to restrict the range  of the 
parameter $u$ to  $ u \geq  \frac{1}{4}  \left( {u \over |w|} \right)^{5/8}$, i.e. $u \in [\frac{1}{4 |w|^{5/3}}, 4R]$.
We will first show that for any such~$u$, there is a unique solution $\hat v_i (u)$ of \eqref{barv-def}. 
The solution of \eqref{rec1-lemB3} will be located close to the curve $u \to \hat v_i(u)$, thus 
we will then need to control the regularity of the curve $u \to \hat v_i(u)$ to estimate the size of the tubular neighborhood around this curve.

\medskip

- For fixed $u$, note  that $\tilde  x_{jk}^\perp  (t_{rec})$ is also fixed.  The only dependence with respect to $v_i$ in the right-hand side of (\ref{barv-def}) is via $\bar \theta_\pm$~:
\begin{align}
d\bar\theta_\pm  \leq &\frac12  |w|^\delta \, { |v_i-v_j|^2  \tilde  x_{jk}^\perp  (t_{rec}) \cdot dv_i  - ((v_i+v_j - 2v_k) \cdot  \tilde  x_{jk}^\perp  (t_{rec})) \;(v_i-v_j)\cdot dv_i \over |v_i-v_j|^3 | \tilde  x_{jk} (t_{rec})|} 
\nonumber \\
&- \frac12 d< \tilde  x_{jk}^\perp  (t_{rec}), v_i-v_j>\,.
\label{eq: 2nd partie derivee}
\end{align}
where we used the Lipschitz bound satisfied by $\arccos_{|w|}$. 
Note that second term in \eqref{eq: 2nd partie derivee} controls the variation of the angle 
$< \tilde  x_{jk}^\perp  (t_{rec}), v_i-v_j>$ and has Lipschitz constant less than $\frac{1}{|v_i-v_j|}$.
Together with the bounds (\ref{bound-case1}), this implies that $ v_i \mapsto \bar \theta_\pm(v_i,u) $ is Lipschitz continuous with constant 
$$
C |w|^{ \delta}  \; \max \frac{1}{|v_i-v_j|} \leq C\left( {|w| \over u} \right)^{5/8 + \delta} u^\delta \ll {|w| \over u} .
$$
We therefore conclude by Picard's fixed point  theorem that there is a unique solution $\hat v_i = \hat v_i (u)$.
As $\delta < 1/16$, we  further have that any solution to (\ref{rec1-lemB3}) satisfies
$$ 
v_i = \hat v_i (u) + O \Big(  u^\delta \left( \frac{u}{|w|} \right)^{9/8-\delta} 
+   \frac{u}{|w|} \, \frac{1}{|w|^\delta} \Big),
$$
for  $u \in [\frac{1}{4 |w|^{5/3}}, 4R]$.

\medskip

- Let us now study the regularity of $ u \mapsto \hat v_i (u)$. In (\ref{barv-def}), we have both an explicit dependence with respect to $u$ and a dependence via the direction of 
$\tilde  x_{jk} (t_{rec})$.
To take into account the condition \eqref{bound-case1}, we further restrict the range of $u$ to 
\begin{equation}
\label{eq: conditions sur u}
u \in [\frac{1}{4 |w|^{5/3}}, 4R]
\quad \text{and} \quad
| \tilde  x_{jk} (t_{rec}) |  \geq \left( {|w| \over u} \right)^{ 3 / 4} .
\end{equation}
The derivative of ${ \tilde  x_{jk} (t_{rec}) \over | \tilde  x_{jk} (t_{rec})|}$ 
with respect to $u$ is controlled by
\begin{equation}
\label{eq: angle xik}
{|w| |v_j-v_k| \over  u^2| \tilde  x_{jk} (t_{rec})|} \leq C(R) \left( {|w| \over u} \right)^{1+5/8-3/4}
= C(R)\left( {|w| \over u} \right)^{7/8}
\end{equation}
as $u \geq {1 \over 2} \left( u  \over {|w|} \right)^{5/8}$ thanks to \eqref{eq: conditions sur u}.
Thus the Lipschitz constant of $ u\mapsto \bar\theta_\pm (v_i,u)$ is less than 
$\big( |w| / u \big)^{7/8} \, |w|^\delta$.
Gathering both estimates, we finally get by differentiating (\ref{barv-def}) with respect to $v_i$ and $u$ that 
$u\mapsto \hat v_i(u)$ is Lipschitz continuous with constant $1+ C \big( |w| / u \big)^{-1/8 + \delta} u^\delta$, which is bounded as  $\delta<1/16$.
The solutions of \eqref{rec1-lemB3} are at a distance at most ${R \over  |w|^{1+\delta} }$ from the curve $u\mapsto \hat v_i(u)$.
Thus under the condition \eqref{bound-case1}, any recollision in chain will belong to the tubular  
neighborhood of $u\mapsto \hat v_i(u)$.

In order to estimate the measure that $v_i$ belongs to this tubular neighborhood, we proceed as in 
\eqref{rectangle0} and  cover this tube by $O( |w|^{1+\delta} )$ balls of size ${R \over  |w|^{1+\delta} }$.
Integrating with respect to $dv_{1^*} d\nu_{1^*}$, we get an estimate 
$O \left( \frac{|\log|w||}{|w|^{1+\delta} } \right)$.
By construction $|w| \geq |\tau_1| |\bar v_i - v_j|$ so that the remainder
 can be integrated with respect to $\tau_1$. Changing to the variable $t_{1^*}$, we obtain an upper bound of order $\eps$.
 We then kill the singularity $\frac{\big| \log |\bar v_i - v_j| \big| }{ |\bar v_i - v_j|^{1+\delta}}$ at small relative velocities by integrating with respect to two additional parents, applying \eqref{alpha1} and then \eqref{alpha2}.

\bigskip

\noindent
$\bullet $ Suppose  that $| v_i- v_j | \leq |\tau_{rec}|^{-5 /8} $. 
We obtain by~(\ref{preimage-sphere1}) that 
\begin{align}
& \int  \indc_{ \{ | v_i- v_j | \leq| \tau_{rec}|^{-5 /8} \}} |(v_{1^*} - v_i)\cdot \nu_{1^*}| d\nu_{1^*} dv_{1^*} \\
& \qquad  \leq  \int  \indc_{\{ | v_i- v_j | \leq \frac{1}{|\tau_1|^{5 /8}   \, |\bar v_i - v_j|^{5 /8}} \}} |(v_{1^*} - v_i)\cdot \nu_{1^*}| d\nu_{1^*} dv_{1^*}  \nonumber \\
& \qquad  \leq \frac{C R^2}{|\tau_1|^{5 /8}   \, |\bar v_i - v_j|^{5 /8} } 
\min \left(  \frac{1}{|\tau_1|^{5 /8}   \, |\bar v_i - v_j|^{13 /8}}, 1\right) 
  \leq   \frac{C R^2}{|\tau_1|^{9 /8}   \, |\bar v_i - v_j|^{77 /40}} , \nonumber
\end{align}
where in the last inequality, we used that $\min (\delta, 1) \leq \delta^{4/5}$.
This produces an integrable function of~$|\tau_1|$ and leads to an upper bound of order $\eps$.
The singularity in $|\bar v_i-v_j|$ can be integrated out by applying \eqref{alpha1} and then \eqref{alpha2}
on the parents of $i,j$.

\bigskip

\noindent
$\bullet$  Suppose  that  $| \tilde x_{j,k} (t_{rec}) | \leq| \tau_{rec}| ^{ 3 / 4} $, this condition can be interpreted as a ``kind of recollision" between $j$ and $k$ at time $t_{rec}$. 
Note that this situation is similar to the last case studied in Lemma \ref{v-constraint}, where the size of the error depends on $|\tau_1| |\bar v_i - v_j|$.

\begin{figure} [h] 
   \centering
  \includegraphics[width=6cm]{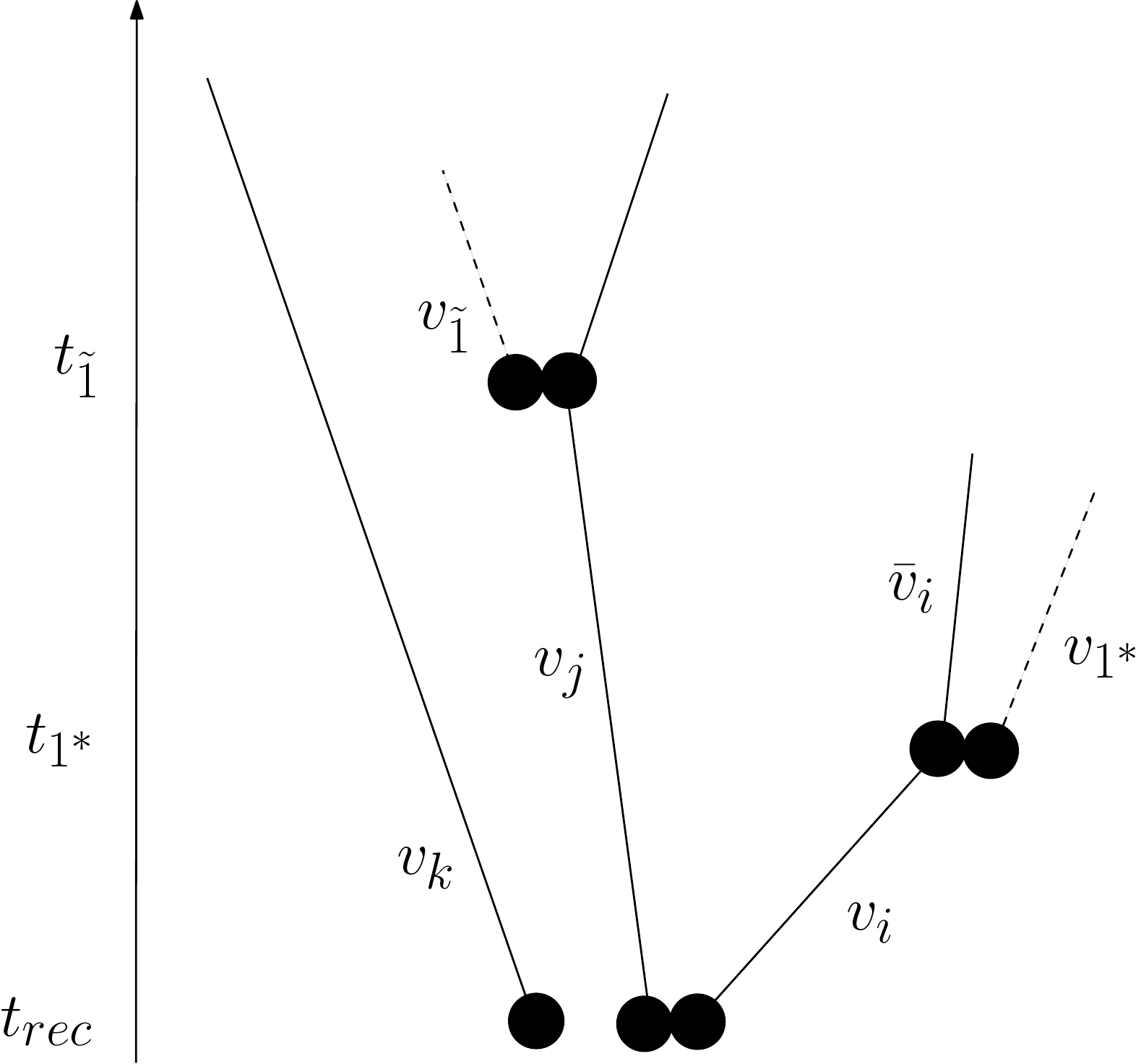} 
  \caption{In the case $| \tilde x_{j,k} (t_{rec}) | \leq| \tau_{rec}| ^{ 3 / 4} $, we will forget about the recollision between $i,k$ and use instead  that $j$ and $k$ are close at time $t_{rec}$.}
\label{fig: LemmeB3_cas3}
\end{figure}

The first recollision between $i,j$ imposes that $ v_i - v_j$ belongs to a rectangle $\cR$.
Integrating  first the condition for the recollision between $(i,j)$ with respect to 
$b(\nu_{1 ^*},v_{1 ^*}) \,  dv _{1^*} d\nu_{1 ^*}$, we  gain  a factor $(\tau_1 |\bar v_i - v_j|)^{-1} $. 
We will not use the recollision between $i,k$ and focus on the additional constraint that the distance between $j,k$ is less than $\eps |\tau_{rec}| ^{ 3 / 4}$ at  time $t_{rec}$.

Denote by $\tilde 1$ the first parent of $(j,k)$.  By analogy with equation \eqref{rec-equation'}, 
the constraint $| \tilde x_{j,k} (t_{rec}) | \leq| \tau_{rec}| ^{ 3 / 4} $ reads
\begin{equation*}
(x_j - x_k) (t_{\tilde 1} )+ ( t_{rec} - t_{\tilde 1}) (  v_j - v_k) 
= \eps \eta + q,
\end{equation*}
with $|\eta| \leq  |\tau_{rec}|^{ 3 / 4}$ and  a given $q \in \Z^2$ with modulus less than $Rt$.
With the notation $\tilde \tau_{rec} = \frac{t_{rec} - t_{\tilde 1}}{\eps}$, this can be rewritten
\begin{equation}
\label{fausserecoll2}
v_j - v_k = \frac{(x_j - x_k) (t_{\tilde 1}  + q)}{ \eps \tilde \tau_{rec}} 
+  \frac{ \eta}{\tilde \tau_{rec}} .
\end{equation}
Since $\tilde \tau_{rec} \geq \tau_{rec}$, we get
$$
\frac{ \eta}{|\tilde \tau_{rec}|} \leq \frac{ 1}{ |\tau_{rec}|^{1/4}} \leq \frac{1}{( |\tau_1| |\bar v_i - v_j|)^{1/4}} ,
$$
so that  $v_j - v_k$ has to belong to a rectangle $\tilde \cR$ of width less than  $\frac{1}{( |\tau_1| |\bar v_i - v_j|)^{1/4}}$.

As in the last case of Lemma \ref{v-constraint}, we split the proof according to the size of $|\tau_1|$.

\noindent
- If $|\tau_1| \geq  {1 \over |\bar v_i - v_j|^6}$, we deduce  that 
$$\frac{1}{|\tau_1| |\bar v_i - v_j|}   \leq \frac{1}{|\tau_1|^{5/6}}.$$ 
Then, we compute the cost of satisfying the previous constraints
\begin{align*}
\int  \indc_{\{ v_i - v_j \in  \cR \}}  \indc_{\{ v_j -v_k \in \tilde \cR \}} 
\prod_{\ell =1^*,\tilde 1} b(\nu_\ell,v_\ell)  \,  dv_\ell d\nu_\ell  d t_\ell 
& \leq \int {\indc_{\{ v_j -v_k \in \tilde \cR \}} \over |\tau_1| |\bar v_i - v_j| }  
 b(\nu_{\tilde 1},v_{\tilde 1 })  \, dt_{1^*} \,   dv_{\tilde 1} d\nu_{\tilde 1} d t_{\tilde 1} \\
 &   \leq \eps \int {\indc_{\{ v_j -v_k \in \tilde \cR \}} \over |\tau_1| |\bar v_i - v_j| }  
 b(\nu_{\tilde 1},v_{\tilde 1 })  \, d \tau_1 \,   dv_{\tilde 1} d\nu_{\tilde 1} d t_{\tilde 1}.
  \end{align*}
As in the case of \eqref{eq: intricated}, the change of variable from $t_{1^*}$ to $\tau_1$ 
leads to a factor $\eps$ and decouples the dependence between the variable 
$t_{1^*}$ and $v_{\tilde 1}$ by  keeping only the constraint $ |\tau_1| \geq R$.
We can then complete the upper bound as usual
\begin{align*}
\int  \indc_{\{ v_i - v_j \in  \cR \}}  \indc_{\{ v_j -v_k \in \tilde \cR \}} 
\prod_{\ell =1^*,\tilde 1} b(\nu_\ell,v_\ell)  \,  dv_\ell d\nu_\ell  d t_\ell 
 &  \leq \eps \int {\indc_{\{ v_j -v_k \in \tilde \cR \}}  \over |\tau_1|^{5/6} } 
  b(\nu_{\tilde 1},v_{\tilde 1 })  \, d \tau_1 dv_{\tilde 1} d\nu_{\tilde 1} d t_{\tilde 1} \\
  &  \leq \eps C(R)  \int  {\log |\tau_1| \over |\tau_1|^{25/24}  } d \tau_1, 
  \end{align*}
where the singularity is  integrable in $|\tau_1| \in[ R, + \infty]$.

\medskip

\noindent
- If $|\tau_1| \leq  {1 \over |\bar v_i - v_j|^6}$, we forget about (\ref{fausserecoll2}). We indeed have that 
$$
\int {1\over |\tau_1| \, |\bar v_i - v_j| }  d\tau_1 \leq {1\over  |\bar v_i - v_j| }  \int \frac{1}{|\tau_1|} d\tau_1   \leq {C |\log  |\bar v_i - v_j||\over  |\bar v_i - v_j| }\,.
$$
The singularity at small relative velocities is controlled with two additional integration.

\bigskip

Given a set $\sigma$ of parents, it may only determine the particle $i$, so that an extra factor $s^2$ has to be added in \eqref{eq: Recollisions in chain}
to take into account the choice of $j,k$. This concludes the proof of Lemma~\ref{Murphy}.
\end{proof}

\subsection{Two   particles recollide twice in chain due to  periodicity} 

We have seen in Proposition  \ref{recoll1-prop} that a self-recollision between two particles created at the same collision has a cost $\eps$. It may happen also that two particles have a recollision and then a second self-recollision due to periodicity (see Figure~\ref{coccinelleagain}). This is a very constrained case which is treated in 
the following Lemma.

 \begin{figure} [h] 
\centering 
\includegraphics[width=6cm]{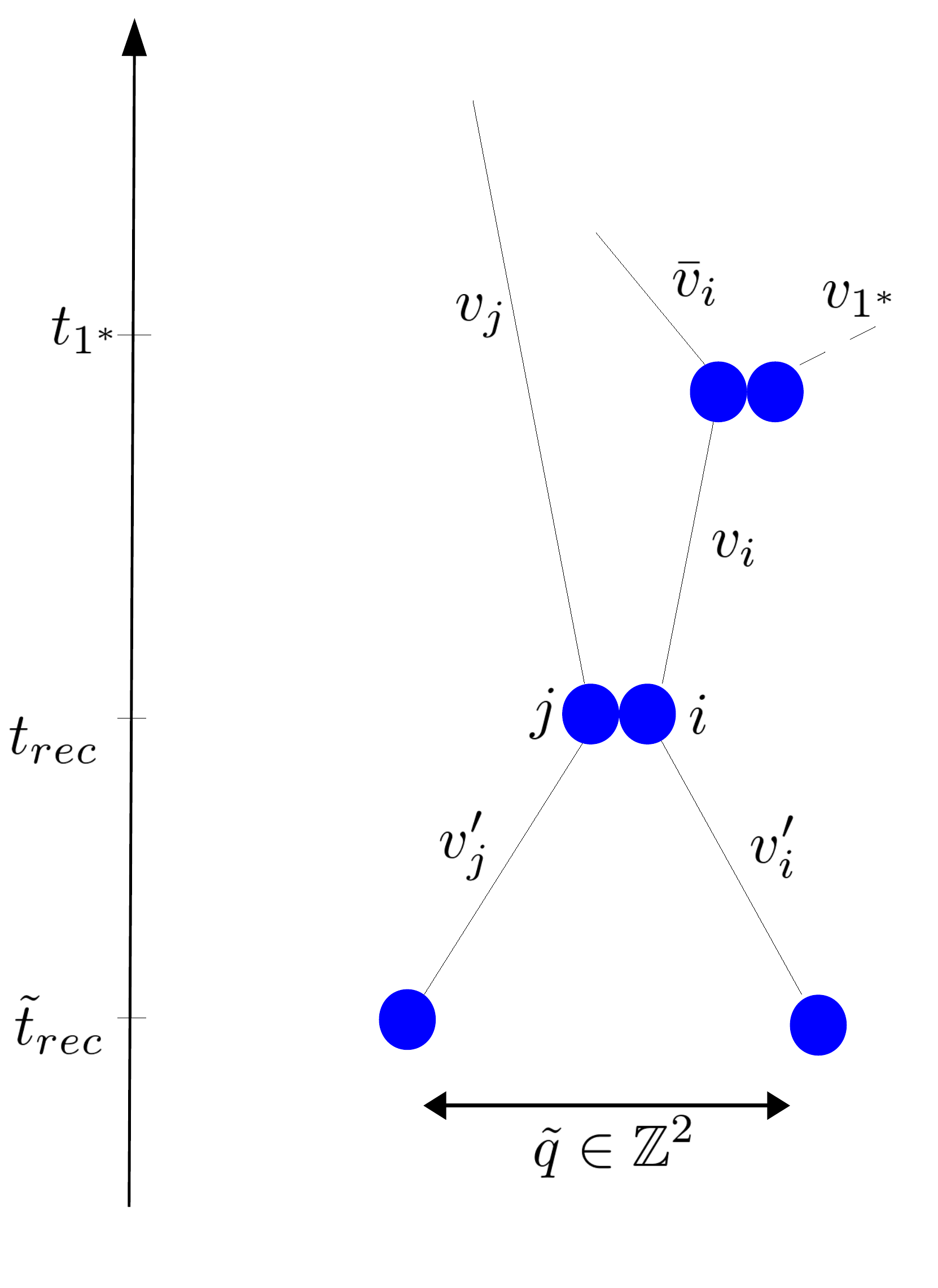} 
\hskip2cm  \includegraphics[width=6cm]{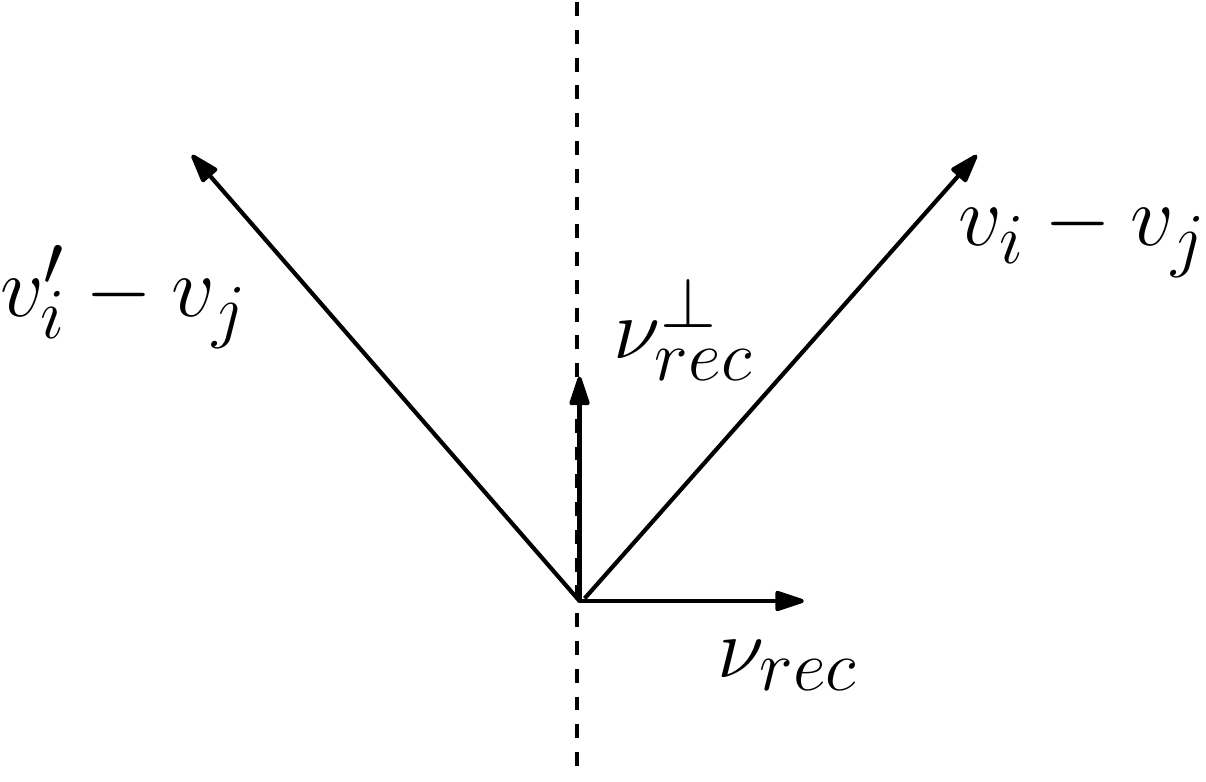}   
   \caption{On the left, two recollisions in chain due to  periodicity.
   On the right, the symmetry argument \eqref{eq: symmetry argument}. }
   \label{coccinelleagain}
\end{figure}

    \begin{Lem}
    \label{self-recoll-periodic}
    Fix a final configuration of bounded energy~$z_1 \in \T^2 \times  B_R$ with~$1 \leq R^2 \leq C_0 |\log \eps|$,  a time $1\leq t \leq C_0|\log \eps|$ and  a collision tree~$a \in \cA_s$ with~$s \geq 2$.

There exists a  set of bad parameters
$\cP_2 (a, p_4,\sigma)\subset \cT_{2,s} \times {\mathbb S}^{s-1} \times \R^{2(s-1)}$  and $\sigma \subset \{2,\dots, s\}$ of cardinal $|\sigma| \leq  3$ such that
\begin{itemize}
\item
$\cP_2 (a, p_4 ,\sigma)$ is parametrized only in terms of   $( t_m, v_m, \nu_m)$ for $m \in \sigma$ and  $m < \min \sigma$;
\begin{equation}
\label{eq: triple salto vrille} 
\int \indc _{\cP_2 (a,p_4 ,\sigma)  } \displaystyle  \prod_{m\in \sigma} \,
\big | \big(v_{m}-v_{a(m)} ( t_{{m}} ) )\cdot \nu_{{m}}  \big | d  t_{{m}} d \nu_{{m}}dv_{{m}}   
\leq C(Rt)^r s \eps \,,
\end{equation}
for some constant $r$,
\item  and any pseudo-trajectory starting from $z_1$ at $t$, with total energy bounded by $R^2$, and such that the first two recollisions involve the same two particles which recollide twice in chain is parametrized by 
$$( t_n, \nu_n, v_n)_{2\leq n\leq s }\in  \bigcup _\sigma  \cP_2(a, p_4 ,\sigma)\,.$$
\end{itemize}
\end{Lem}
\begin{proof}
We recall the equation \eqref{rec-equation} on the first recollision
\begin{equation}
\label{rec-equationrecol}
v_i-v_j = \frac1{\tau_{rec}} (\delta x_\perp - \tau_1 (\bar v_i-v_j) - \nu_{rec}) 
\quad \text{with} \quad 
\frac1{|\tau_{rec} |} \leq  {4R\over |\tau_1| | \bar v_{i} - v_{j}|}\, .
\end{equation}
The equation on the second recollision is
\begin{equation}
\label{rec-equationrecolperiodic}
(v'_i - v'_j) (\tilde  t_{rec}-  t_{rec}) = \eps \tilde \nu_{rec} + \eps \nu_{rec} + \tilde q
\end{equation}
for some~$\tilde  t_{rec}\geq 0$, $ \tilde \nu_{rec} \in {\mathbb S}$, and~$\tilde q \in \Z^2 \setminus \{ 0 \}$.
Note that $\tilde q  \not = 0$ as the second recollision  occurs from the periodicity.
As usual we fix~$\tilde q$ and multiply the estimates in the end by~$O(R^2t^2)$ to take that into account.

The condition \eqref{rec-equationrecolperiodic} implies that the vector $v'_i - v'_j$ is located in a cone of axis $\tilde q$ and angular sector $2\eps$.
By definition, we have
\begin{equation}
\label{eq: symmetry argument}
v_i'- v_j'=( v_i-v_j ) - 2 (v_i- v_j)\cdot \nu_{rec} \; \nu _{rec},
\end{equation}
which means that $\nu_{rec}^\perp$ is the bisector of $v_i- v_j$ and $v_i'- v_j'$ (see Figure \ref{coccinelleagain}).

From \eqref{rec-equationrecol}, we deduce that the direction of $v_i- v_j$ is
$$
{\delta x_\perp - \tau_1 (\bar v_i-v_j) \over |\delta x_\perp - \tau_1 (\bar  v_i-v_j)|} 
+ O\left( {1\over |\tau_1 (\bar v_i-v_j)|}\right)\,.
$$
From \eqref{rec-equationrecolperiodic}, we deduce that the direction of $v'_i-v'_j$ is
$$ {\tilde q\over |\tilde q|} + O( \eps)\,.$$
Finally we get that $\nu_{rec}^\perp$ is known up to an error term which can be bounded by
$$
\eta =  \eps + {1\over \sqrt{ |\tau_1 ( \bar  v_i-v_j)|}} \,\cdotp
$$
Note that we have introduce the square root as in the proof of Lemma  \ref{Murphy} for integrability purposes of the singularity $|\bar  v_i-v_j|$.

Plugging this constraint on $\nu_{rec}$ in \eqref{rec-equationrecol}, we get that $v'_i- v_j$ has to belong, for each given $q, \tilde q$, to a rectangle $\cR$ of axis $\delta x_\perp - \tau_1 (\bar v_i-v_j)$
and  size $R \times R \frac{\eta}{|\tau_1( \bar v_i-v_j)|}$.
By Lemma \ref{scattering-lem 2}, we obtain 
$$
\int  \indc_{ \{ v_i - v_j \in \cR \} }  \;  \big | \big ( v_1^*-v_j ) \cdot \nu_1^*\big) \big | dv_1^* d\nu_1^* 
\leq CR^3 {\eps |\log \eps| \over \tau_1| \bar v_i-v_j|} +  { CR^3 \over \tau_1^{3/2} | \bar v_i-v_j|^{3/2}}\,\cdotp
$$
Taking the union of the previous rectangles for the different choices of $q,\tilde q$, we define
the set $\cP_2(a, p_4 ,\sigma)$ associated with the scenario of two particles recolliding twice in chain due to  periodicity.
By integration with respect to time, we then get
$$ 
\int 
\indc_{\cP_2(a, p_4 ,\sigma)}
 \;   \big | \big ( v_1^*-v_j ) \cdot \nu_1^*\big) \big |dv_1^* d\nu_1^*dt_1^* 
 \leq CR^3 {\eps^2 |\log \eps| ^2\over |\bar v_i-v_j|} + CR^3 
 { \eps \over  |\bar v_i-v_j|^{3/2}}
 \,\cdotp$$
We then apply twice Lemma \ref{scattering-lem} on two parents of $i,j$ to integrate the singularities at small relative velocities.

\bigskip

Given $\sigma$, there are at most $s$ choices for the pair $(i,j)$ as $\sigma$ determines at least one of the labels.
Thus the previous scenario leads to the set $\cP_2(a, p_4 ,\sigma)$ with measure controlled by \eqref{eq: triple salto vrille}.
\end{proof}


\section{Carleman's parametrization and scattering estimates}
\label{sectionappendixcarleman}

In Sections \ref{convergencepartieprincipales}, \ref{recollisions} and Appendix~\ref{geometricallemmasappendix},  we were faced with integrals containing singularities in relative velocities~$v_i-v_j$
and with a multiplicative factor of the type~$(v^* -  \bar v_{i}) \cdot \nu^*$ where~$v_i$ is recovered from~$v^*$, $\nu^*$ and~$\bar v_{i}$ through a scattering condition. 
This appendix is devoted to the proof of ``tool-box'' lemmas for computing these singular integrals.
These lemmas are used many times in this paper.

\begin{figure} [h] 
   \centering
  \includegraphics[width=8cm]{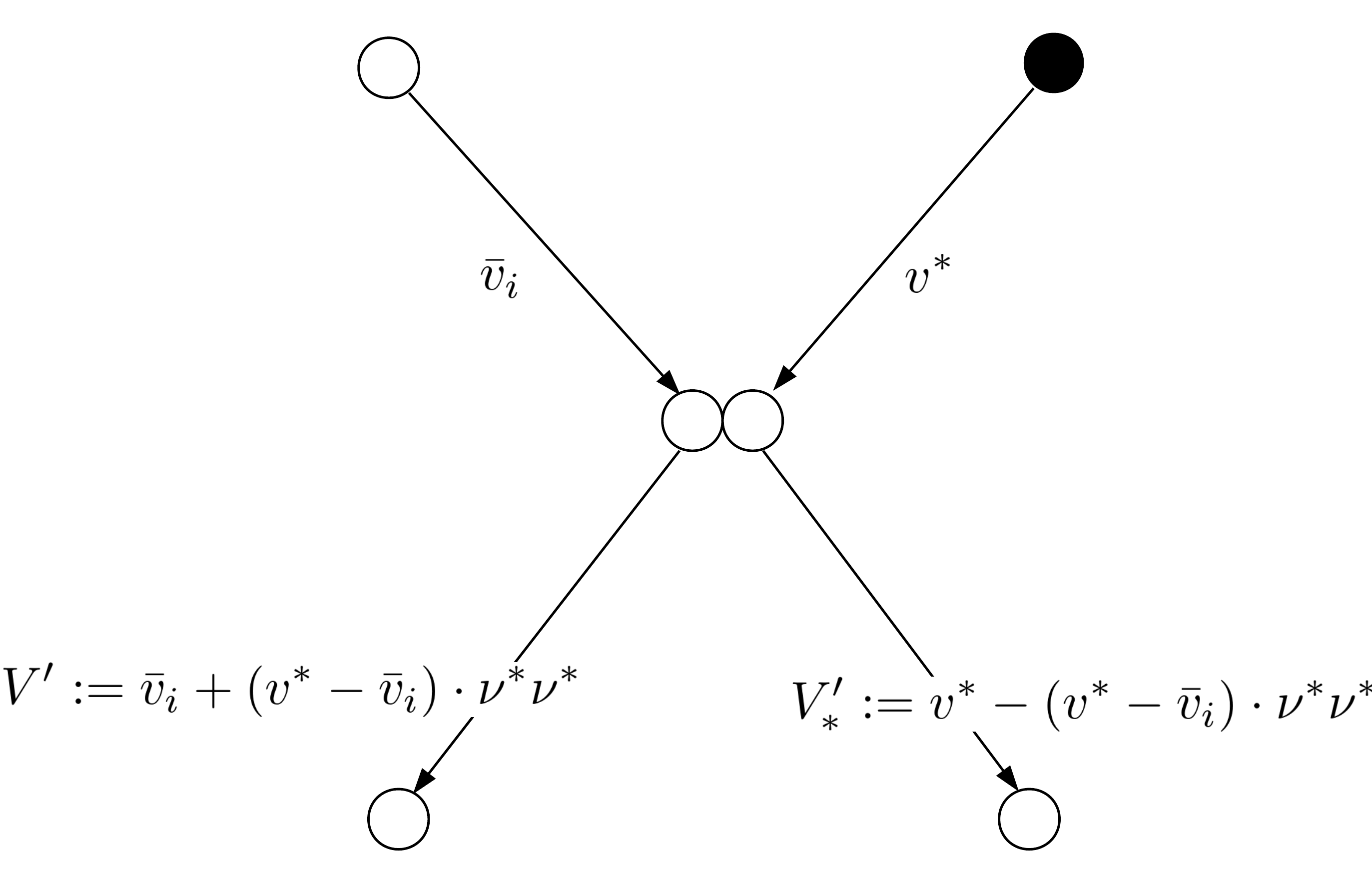} 
  \caption{Scattering relations.}
\label{fig: carlemanpicture}
\end{figure}

\begin{Lem}
\label{joint-scattering-lem} 
Fix a velocity $\bar v_i$ and let $v_i, v_j$  be the velocities after a collision (with or without scattering)
\begin{align*}
(v_i, v_j)  = (\bar v_i, v^*)
\quad \text{or} \quad  
\begin{cases}
v_i = \bar v_{i} + (v^*-  \bar v_{i}) \cdot \nu^* \nu^*, \\
v_j  = v^*- (v^* -  \bar v_{i}) \cdot \nu^* \nu ^*,
\end{cases}
\end{align*}
with~$ \nu^* \in {\mathbb S}$ and~$v^* \in \R^2$ (see Figure~{\rm\ref{fig: carlemanpicture}}). 
Assume all the velocities are bounded by~$R$ then  
\begin{eqnarray}
\int   { 1\over |v_i - v_j| }  \,  \big | \big( v ^* -  \bar v_i)\cdot   \nu ^* \big |\,   dv^* d\nu^* \leq  CR^2 . 
\label{carleman3 joint}  
\end{eqnarray} 
\end{Lem}
\begin{proof}
In both cases, the velocities before and after the collision are related by
$|v_i - v_j|  = |v ^* -  \bar v_i|$. Inequality \eqref{carleman3 joint}   follows from the fact that the 
singularity $1/ |v ^* -  \bar v_i|$ is integrable.
\end{proof}

\begin{Lem}
\label{scattering-lem} 
Fix $\bar v_i$ and $v_j$,  and define $v_i$ to be one of the following velocities
\begin{align}
\label{eq: scattering relation}
v_i  = v^*- (v^* -  \bar v_{i}) \cdot \nu^* \nu ^*,\\
\mbox{or} \quad v_i = \bar v_{i} + (v^*-  \bar v_{i}) \cdot \nu^* \nu^*\,, \nonumber
\end{align}
with~$ \nu^* \in {\mathbb S}$ 
and~$v^* \in B_R \subset \R^2$ (see Figure~{\rm\ref{fig: carlemanpicture}}). 
Assume all {  the velocities}  are bounded by~$R>1$ and fix~$\delta\in ]0,1[$. Then the following estimates hold, denoting~$b(\nu ^*,v ^*):= |(v ^* -  \bar v_i) \cdot \nu ^*| $:
\begin{eqnarray}
\int \indc_{|v_i -v_j|\leq \delta} \,b(\nu ^*,v ^*) \,  dv ^* d\nu ^*  \! \! \! &\leq &  \!\!  \! C R^2 \delta \min \left( {\delta \over |v_j-\bar v_i|}, 1 \right) \, ,
\label{preimage-sphere1} \\
  \int\! \!    \min \left(  { \delta \over |v_i - v_j| }, 1 \right)    b(\nu ^* ,v ^*)  \,   dv^* d\nu^* &\leq& CR^2\delta |\log\delta | {  + C R^3 \delta} \, ,
  \label{carleman2}\\
      \int   { 1\over |v_i - v_j| }  \,   b(\nu ^*,v ^*) \,   dv^* d\nu^* \! \! \! &\leq &  \!\!  \! CR^2 \Big( \big |\log |\bar v_{i}- v_{j}|\big|  +R\Big)  \, , \label{carleman3}  \\
 \int   { 1\over |v_i - v_j|^\gamma }  \,  b(\nu ^*,v ^*)  \,    dv^* d\nu^* \! \! \! &\leq &  \!\!  \!  {CR^2\over  |\bar v_{i}- v_{j}|^{\gamma-1}} +CR^3 \quad \mbox{ for } \,  \gamma \in ]1,2[  \,,\label{alpha1}\\
 \int   { 1\over |v_i - v_j|^\gamma }  \,  b(\nu ^*,v ^*)  \,  dv^* d\nu^*\! \! \! &\leq &  \!\!  \!  CR^3 \quad \mbox{ for } \, \gamma \in ]0,1[  \,,\label{alpha2}\\
 \int  \big |\log  |v_i - v_j|\big|   \, b(\nu ^*,v ^*) \,    dv^* d\nu^* \! \! \! &\leq &  \!\!  \!  CR^3 \,.\label{alpha3}  
\end{eqnarray} 
\end{Lem}

\begin{proof}
We start by recalling  Carleman's parametrization, which we shall be using many times in this Appendix: it is defined by
\begin{equation}
\label{defcarleman}
(v^*,\nu^* )\in \R^2 \times {\mathbb S} \mapsto 
\begin{cases}
 V' _* := v ^*- (v ^* - \bar  v_{i}) \cdot \nu ^* \nu  ^*\\
 V' := \bar v_{i} + (v ^*-  \bar v_{i}) \cdot \nu ^* \nu ^*
\end{cases}
\end{equation}
where $(V',V'_*)$ belong to the set ${\mathcal C}$ defined by 
$${\mathcal C}:=\Big \{(V',V'_*)\in \R^2\times \R^2 \,/\, (V'-\bar v_i)\cdot (V'_*-\bar v_i) =0\Big\} \, .$$
This map sends the measure $ b(\nu ^*,v ^*)  \,  dv^*d\nu^*$ on the 
measure $dV'dS(V'_*)$, where $dS$ is the Lebesgue measure  on the line
orthogonal to $(V'-\bar v_i)$ passing through $\bar v_i$.

\medskip
  Now let   us consider  the case when~$|v_i-v_j| \leq \delta$ and prove~(\ref{preimage-sphere1}).  
What we need here is  to estimate the measure of the pre-image  of the small ball of center $v_j$ and radius $\delta$ by the scattering operator: let us study how for fixed~$v_j$, the set~$\{|v_i-v_j| \leq \delta\}$ is transformed by the inverse scattering map. Notice that the most singular case concerns the case when~$v_i = V' _*$ belongs to the small ball of radius~$\delta$: indeed in the case when it is~$V'$ then the measure~$  b(\nu ^*,v ^*)  \,  dv^*d\nu^*$ will have support in a domain of size~$O(\delta^2)$. So now assume that~$V' _*$ satisfies~$|V' _* -v_j|\leq \delta$.

\begin{figure} [h] 
   \centering
  \includegraphics[width=3in]{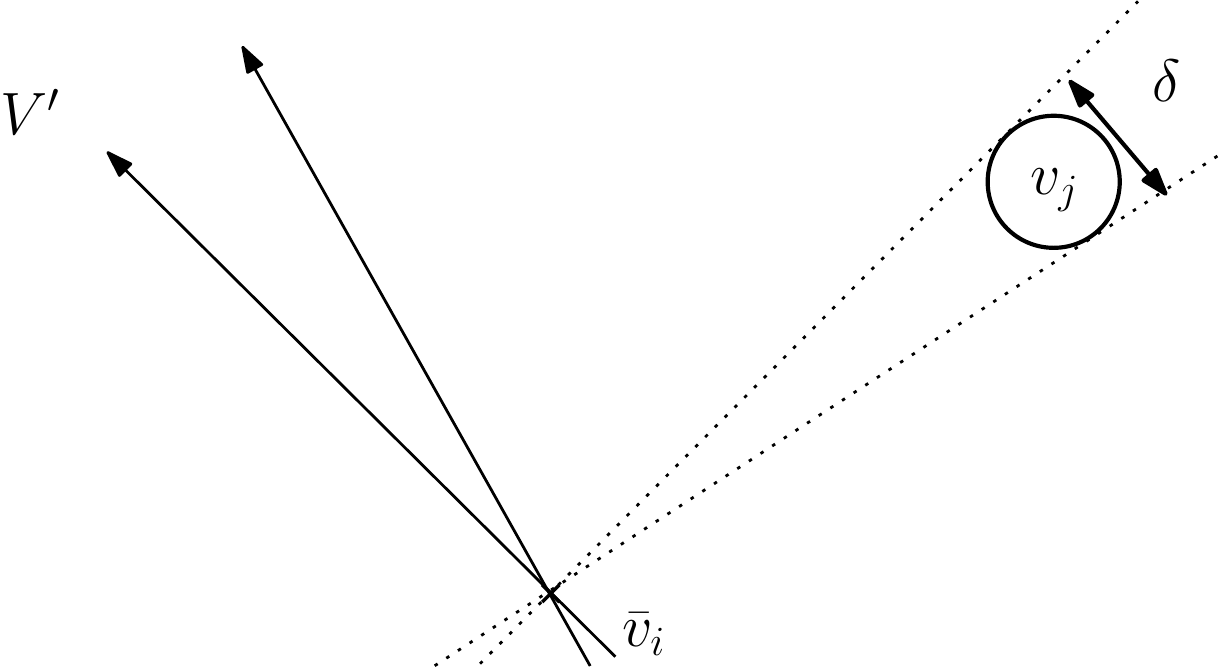} 
\caption{\small 
 $V'_*$ has to belong to the ball of radius $\delta$ around $v_j$, thus it has to be in the cone with the doted lines.
By Carleman's parametrization, this imposes constraints on the angular sector of $V'-\bar v_i$.
}  
\label{fig: carleman0}
\end{figure}

\medskip

$ \bullet $ $ $  If  $|v_j - \bar v_i| \leq \delta $, meaning that $\bar v_i$ is itself in the same ball, then for any $V' \in B_R$, the intersection between the small ball and the line $ \bar v_i + \R (V'-\bar v_i)^\perp$ is  a segment, the length of which is at most 
$\delta$. We therefore find 
$$ \int \indc_{|V' _* -v_j|\leq \delta}  \, b(\nu ^*,v ^*)  \,    dv ^* d\nu ^* \leq  C R^2 \delta \,.$$

\medskip

$ \bullet $ $ $  If $|v_j - \bar v_i| >\delta$, in order for  the intersection between the ball and the line $ \bar v_i + \R (V'-\bar v_i)^\perp$ to be non empty, we have the additional condition that $V'-\bar v_i$ has to be in an angular sector of size $\delta /| v_j-\bar v_i|$ (see Figure \ref{fig: carleman0}). We therefore have 
$$ \int \indc_{|V' _*  -v_j|\leq \delta}  \, b(\nu ^*,v ^*)  \,   dv ^* d\nu ^* \leq  C R^2 {\delta^2\over |v_j-\bar v_i|}  \,\cdotp$$
Thus~(\ref{preimage-sphere1}) holds.

\medskip

The other estimates provided in Lemma \ref{scattering-lem} then come from  Fubini's theorem: let us start with~(\ref{carleman2}). We write
$$
\begin{aligned}
\int \min\big( {\delta\over |v_i-v_j|}  ,1\big) \, b(\nu ^*,v ^*)  \,    dv ^* d\nu ^* &= \int \indc_{ |v_i-v_j| \leq \delta} 
 \, b(\nu ^*,v ^*)  \,    dv ^* d\nu ^* \\
& \qquad +  \int \indc_{ |v_i-v_j| > \delta} {\delta\over |v_i-v_j|}
 \, b(\nu ^*,v ^*)  \,   dv ^* d\nu ^*\\
& \leq CR^2 \delta +  \int \indc_{ |v_i-v_j| > \delta} {\delta\over |v_i-v_j|}
 \, b(\nu ^*,v ^*)  \,    dv ^* d\nu ^*
\end{aligned}
$$
thanks to~(\ref{preimage-sphere1}). 
The contribution of the velocities such that $|v_i - v_j| \geq 1$ can be bounded by $R^3 \delta$. Thus 
 it is enough to consider 
$$
\begin{aligned}
 \int {\delta \indc_{ 1 \geq  |v_i-v_j| > \delta}\over |v_i-v_j|}
 \, b(\nu ^*,v ^*)  \,     dv ^* d\nu ^*
& = \delta  \int  \Big( \int_{|v_i-v_j|} ^1 \frac{dr}{r^{2}}   +1 \Big)   \indc_{ 1 \geq  |v_i-v_j| > \delta} \, b(\nu ^*,v ^*)  \,   dv ^* d\nu ^*\\
& \leq\delta \int_{\delta}^1  \frac{dr}{r^2} \int \indc_{ |v_i-v_j| \leq r} \, b(\nu ^*,v ^*)  \,  dv ^* d\nu ^*+CR^3 \delta \, ,
\end{aligned}
$$
so using~(\ref{preimage-sphere1}) again we get
$$
\begin{aligned}
 \int \indc_{ 1 \geq |v_i-v_j| > \delta} {\delta\over |v_i-v_j|}
 \, b(\nu ^*,v ^*)  \,    dv ^* d\nu ^*
&\leq CR^2\delta \int_{\delta}^1  \frac{dr}{r} + CR^3 \delta\, ,
\end{aligned}
$$
from which~(\ref{carleman2}) follows.

\medskip

Next let us prove~(\ref{carleman3})-(\ref{alpha2}).
We have 
$$
\begin{aligned}
\int {1\over |v_i-v_j|^\gamma}  \, b(\nu ^*,v ^*)  \,    dv ^* d\nu ^* &=  \gamma \int  \Big( \int_{|v_i-v_j|} ^1 \frac1{r^{1+\gamma}} dr  +1 \Big)   \, b(\nu ^*,v ^*)  \,   dv ^* d\nu ^*\\
& =\gamma  \int_0^1 \frac1{r^{\gamma+1}} \Big( \int \indc_{|v_i -v_j|\leq  r}   \, b(\nu ^*,v ^*)  \,  dv ^* d\nu ^* \Big) dr + CR^3\\
& \leq C_\gamma R^2 \Big(  \int_0^{ |v_j-\bar v_i|}  \frac1{|v_j-\bar v_i|}  r^{1-\gamma} dr 
+ \int_{ |v_j-\bar v_i|} ^1 \frac1{r^{\gamma} } dr \Big)  + C R^3
\end{aligned}
$$ which gives the expected estimates. Similarly
$$
\begin{aligned}
\int \big |\log |v_i-v_j|\big| \, b(\nu ^*,v ^*)  \,   dv ^* d\nu ^* &=  \int   \int_{|v_i-v_j|} ^1 \frac1{r} dr    \, b(\nu ^*,v ^*)  \,   dv ^* d\nu ^*\\
& \leq CR^2 \left(  \int_0^{ |v_j-\bar v_i|}  \frac r{|v_j-\bar v_i|}  dr + \int_{ |v_j-\bar v_i|} ^1  dr  \right) \leq CR^3 \, .
\end{aligned}
$$
This ends the proof of Lemma~\ref{scattering-lem}.
\end{proof}

\begin{Rmk}\label{rmkcarleman}
The proof of Lemma~{\rm\ref{scattering-lem}} shows that in order to keep control on the collision integral the power~$\gamma$ of the singularity must not be too large (namely smaller than~$2$).
\end{Rmk}

\medskip

Finally the following result 
describes the size of a collision integral when relative velocities are prescribed to lie in a given rectangle.
\begin{Lem}
\label{scattering-lem 2} 
Consider two pseudo-particles $i$, $j$ as well as their first   parent~$1^*$.  
Denote by~$ \nu_{1^*} \in {\mathbb S}$ and~$v_{1^*} \in \R^2$ their scattering parameters.
We assume also that all the velocities  are bounded by~$R>1$.
Let $\cR$ be a rectangle with sides of length $\delta, \delta' $, then 
\begin{align}
\int    \indc_{v_i -v_j \in \cR}  \; \big| (v_{1^*} - v_{a(1^*)} ) \cdot \nu_{1^*} \big|  \,  dv_{1^*} d\nu_{1^*}  
  \leq C R^2 \min (\delta, \delta')     \big( |\log \delta| + |\log \delta' | + 1\big) ,
  \label{rectangle0}
\end{align} 
\end{Lem}

\begin{proof}

Note that if $i,j$ are generated by the same collision, then   better 
estimates can be obtained from Lemma
\ref{joint-scattering-lem}. The case without scattering is also straightforward.
Thus from now, we assume that $v_i$ is given by  \eqref{eq: scattering relation}.

\medskip

To derive \eqref{rectangle0}, we suppose that $\delta \leq \delta' \leq 1$ and that the collision with $1^*$ takes place with~$i$ which had a velocity $\bar v_i$.
We cover the rectangle $v_j + \cR$ into~$\lfloor \delta'/\delta \rfloor$ balls of radius~$2 \delta$.
Let $\omega$ be the axis of the rectangle  $v_j + \cR$ and denote by $w_k = w_0 +  \delta k \, \omega$ the centers of the balls which are indexed by the integer $k \in \{ 0,  \dots,  \lfloor \delta'/\delta \rfloor \}$.
Applying \eqref{preimage-sphere1} to each ball, we get 
\begin{eqnarray}
 \int    \indc_{v_i -v_j \in \cR}  \; b(\nu_{1^*},v_{1^*})  \,  dv_{1^*} d\nu_{1^*}  
&\leq&  \sum_{k = 0}^{ \lfloor \delta'/\delta \rfloor}
\int    \indc_{ |v_i - w_k| \leq 2 \delta}  \; b(\nu_{1^*},v_{1^*})  \,  dv_{1^*} d\nu_{1^*}  \nonumber
\\
&\leq&   C R^2 \sum_{k = 0}^{ \lfloor \delta'/\delta \rfloor} \delta \min \left( {\delta \over |w_k -\bar v_i|}, 1 \right) ,
\nonumber\\
&\leq&   C R^2 \delta \sum_{k = 0}^{ \lfloor \delta'/\delta \rfloor}   {\delta \over |w_k -\bar v_i| + \delta }
\leq C R^2 \delta \left( \log (\frac{\delta'}{\delta}) +1 \right), \nonumber
\end{eqnarray}  
where the log divergence in the last inequality follows by summing over $k$.
This completes the proof of~\eqref{rectangle0}.

%
%
%
%

This completes the proof of Lemma \ref{scattering-lem 2}.
\end{proof}


\section{Initial data estimates}
\label{sec: appendix Initial data estimates}

This section is devoted to the proof  of Proposition \ref{exclusion-prop2} stated page~\pageref{exclusion-prop2}.

Using the notation $X_{k,N} := \{ x_k, \dots, x_N \}$,
we write
\begin{align*}
\Big| \left(f_{N }^{0(s)}  -  f^{(s)} _0 \right) (Z_s)\indc_{{\mathcal D}_{\eps}^s}(X_s) \Big|
\leq M_\beta^{\otimes s}(V_s) \sum_{i=1}^s \big|g_{\alpha,0}(z_i)\big| \,
\Big| {\mathcal Z}_N^{-1}     \int \indc_{{\mathcal D}_{\eps}^N}(X_N) \, dX_{s+1,N}-1\Big| \\
+{}{\mathcal Z}_N^{-1}
M_\beta^{\otimes s}(V_s) \sum_{i=s+1}^N \Big|\int M_\beta (v_i) g_{\alpha,0}(z_i)
\indc_{{\mathcal D}_{\eps}^N}(X_N) \, dv_idX_{s+1,N}\Big| \, ,
\end{align*}
where ${\mathcal D}_{\eps}^N$ stands for the exclusion constraint on the positions (with a slight abuse of notation compared to \eqref{eq: phase space}).
The first term is estimated as in the proof of Proposition 3.3 in~\cite{GSRT}
$$
M_\beta^{\otimes s}(V_s) \sum_{i=1}^s \big|g_{\alpha,0}(z_i)\big| \,
\Big| {\mathcal Z}_N^{-1}     \int \indc_{{\mathcal D}_{\eps}^N}(X_N) \, dX_{s+1,N}-1\Big|
\leq
C ^s \eps \alpha    M_\beta ^{\otimes s} (V_s) \|g_{\alpha,0}\|_{L^\infty} \, .
$$
The exchangeability of the variables allows us to rewrite the second term as
\begin{align*}
I(Z_s)&:= {\mathcal Z}_N^{-1}
M_\beta^{\otimes s}(V_s) \sum_{i=s+1}^N \Big|\int M_\beta (v_i) g_{\alpha, 0}(z_i)
\indc_{{\mathcal D}_{\eps}^N}(X_N) \, dv_i dX_{s+1,N}\Big|  \\
&\,   \leq
(N-s) M_\beta^{\otimes s}(V_s){\mathcal Z}_N^{-1}
\\
& \qquad \qquad \Big| \int M_\beta (v_{s+1}) g_{\alpha, 0}(z_{s+1})
\Big( \prod_{k \neq s+1} \indc_{|x_k-x_{s+1} |>  \eps} \Big) \;
\chi_{s+2}(X_N) \, dz_{s+1} dX_{s+2,N}\Big|   \, ,
\end{align*}
where we used  
the notation
\begin{equation}
\label{eq: chi s+2}
\chi_{s+2}(X_N):= \hat \chi^+_{s+2}(X_{s+2,N}) \, \hat \chi^-_{s+2}(X_N)
\end{equation}
which distinguishes the interaction of the particles $X_{s+2,N}$ with themselves and with $X_s$, defining
$$
\hat \chi^+_{s+2}(X_{s+2,N}):=
\prod_{s+2 \leq \ell < k \leq N}  \indc_{|x_k-x_\ell|> \eps}
\quad \text{and} \quad
\hat \chi^-_{s+2}(X_N):=\prod_{s+2 \leq \ell \leq N \atop 1 \leq k \leq s}
\indc_{|x_k-x_\ell|> \eps} \,.
$$
The exclusion between $s+1$ and the rest of the system is also decomposed into a term for the interaction with  $X_s$ and another one for the interaction with $X_{s+2,N}$. Defining
$$
\chi^-_{s+1}(X_{s+1}) := \prod_{k \leq  s} \indc_{|x_k-x_{s+1} | > \eps}\quad \text{and} \quad  \chi^+_{s+1}(X_{s+1,N}):=\prod_{k \geq  s+2} \indc_{|x_k-x_{s+1} | > \eps}
$$
we have
\begin{align*}
\prod_{k \neq s+1} \indc_{|x_k-x_{s+1} | > \eps}
& = \chi^-_{s+1}(X_{s+1}) \, \chi^+_{s+1}(X_{s+1,N}) \\
&= \chi^+_{s+1}(X_{s+1,N})  - \big( 1- \chi^-_{s+1}(X_{s+1}) \big)  \chi^+_{s+1}(X_{s+1,N}) \,.
\end{align*}
We deduce that
\begin{equation*}
I(Z_s)  \leq  M_\beta^{\otimes s}(V_s) \Big( I_1 (Z_s)  + I_2 (Z_s) \Big)
\end{equation*}
with
\begin{align*}
\begin{cases}
\displaystyle I_1 (Z_s): =  {\mathcal Z}_N^{-1} N \int M_\beta (v_{s+1})  \big| g_{\alpha, 0}(z_{s+1})  \big|
\big( 1- \chi^-_{s+1}(X_{s+1}) \big)  \hat  \chi^+_{s+2}(X_{s+2,N}) \, dz_{s+1} dX_{s+2,N}\,  , \\
\displaystyle I_2 (Z_s) := {\mathcal Z}_N^{-1}N  \Big| \int M_\beta (v_{s+1}) g_{\alpha, 0}(z_{s+1})
\chi^+_{s+1}(X_{s+2,N})  \; \chi_{s+2}(X_N) \, dz_{s+1} dX_{s+2,N}\Big|\,.
\end{cases}
\end{align*}
From \eqref{eq: ZNZN-k} and the assumption $N \e = \alpha \ll 1 /\eps$, we get
$$
{\mathcal Z}_N^{-1}\int \hat \chi^+_{s+2}(X_{s+2,N}) \,  dX_{s+2,N}
=
\frac{{\mathcal Z}_{N-s-2} }{ {\mathcal Z}_N} \leq
\exp  \big(C s \alpha \eps \big)  \leq \exp  \big(C s \big) \,.
$$
We infer that the term $I_1$ is bounded by the fact that $x_{s+1}$ is close to $X_s$
\begin{align*}
I_1 (Z_s)
\leq s N \eps^2    \exp  \big(C s \big) \|g_{\alpha,0}\|_{L^\infty} \leq s
\exp \big(C s \big) \alpha \eps     \|g_{\alpha,0}\|_{L^\infty}\,.
\end{align*}
Using the assumption  $\displaystyle \int_\D M_\beta g_{\alpha,0} (z) dz = 0$, the second term is rewritten as
\begin{align*}
I_2 (Z_s) ={\mathcal Z}_N^{-1} N  \big|
\int M_\beta (v_{s+1}) g_{\alpha, 0}(z_{s+1})
\big( 1 - \chi^+_{s+1}(X_{s+1,N}) \big)  \;
\chi_{s+2}(X_N) \, dz_{s+1} dX_{s+2,N}  \big|\,.
\end{align*}
Plugging the identity \eqref{eq: chi s+2}
$$
\chi_{s+2}(X_N) = \hat \chi^+_{s+2}(X_{s+2,N}) - \big( 1 - \hat \chi^-_{s+2}(X_N) \big)
\hat  \chi^+_{s+2}(X_{s+2,N})
$$
we  distinguish two more contributions $I_2 (Z_s) \leq I_{2,1} (Z_s) + I_{2,2} (Z_s)$ with
\begin{align*}
\begin{cases}
\displaystyle I_{2,1} (Z_s) :=   {\mathcal Z}_N^{-1}  N \|g_{\alpha,0}\|_{L^\infty} \,
\int  \big( 1 -   \chi^+_{s+1}(X_{s+1,N}) \big)  \; \big( 1 - \hat \chi^-_{s+2}(X_N) \big)
\hat  \chi^+_{s+2}(X_{s+2,N}) \, dX_{s+1,N} \, ,\\
\displaystyle I_{2,2} (Z_s): =
 {\mathcal Z}_N^{-1} N  \Big| \int M_\beta (v_{s+1}) g_{\alpha, 0}(z_{s+1})
\big( 1 -   \chi^+_{s+1}(X_{s+1,N}) \big)  \;
\hat  \chi^+_{s+2}(X_{s+2,N}) \, dz_{s+1} dX_{s+2,N}  \Big|\,. \\
\end{cases}
\end{align*}
The term $I_{2,1}$ takes into account two constraints : $s+1$ is close to a particle in $X_{s+2,N}$ and one particle in
$X_{s+2,N}$ is close to $X_s$. Since $N \eps = \alpha$, we deduce that
$$
I_{2,1} (Z_s) \leq N  s \eps^2 \, (N-s-1) ^2  \eps^2 \, \frac{{\mathcal Z}_{N-s-3} }{ {\mathcal Z}_N} \|g_{\alpha,0}\|_{L^\infty}
\leq s \alpha^3 \eps \exp( C s)  \|g_{\alpha,0}\|_{L^\infty}\,.
$$
The term $I_{2,2}$ does not depend on  $X_s$, thus one can integrate over $z_{s+1}$ and use again the
assumption $\displaystyle\int_\D M_\beta g_{\alpha,0} (z) dz = 0$. To see this, it is enough to note that the function
$$
x_{s+1} \mapsto \int \big( 1 - \chi^+_{s+1}(X_{s+1,N}) \big)  \;
\hat \chi^+_{s+2}(X_{s+2,N}) dX_{s+2,N}
$$
is independent of $x_{s+1}$ thanks to the periodic structure of $\D^{N- s-2}$. Thus $I_{2,2} (Z_s) = 0$.

\medskip
Combining the previous estimates, we conclude Proposition \ref{exclusion-prop2}. \qed


\end{document}